\documentclass[11pt]{amsart}

\usepackage{amsmath}
\usepackage{array}

\usepackage{xcolor}
\usepackage{graphicx,color}
\usepackage{tensor}
\usepackage{mathabx}
\usepackage{verbatim}

\usepackage[margin=0.745in]{geometry}

\usepackage{tikz}

 \usepackage{mathrsfs}

\newtheorem{theorem}{Theorem}[section]
\newtheorem{lemma}[theorem]{Lemma}

\newtheorem{remark}[theorem]{Remark}

\newcommand{\eps}{\epsilon}
\newcommand{\rif}{\varrho_{\infty}}

\newcommand{\cA}{\mathcal A}
\newcommand{\cL}{\mathcal L}
\newcommand{\bbI}{\mathbb I}
\newcommand{\cO}{\mathcal O}
\newcommand{\cR}{\mathcal R}
\newcommand{\cS}{\mathcal S}
\newcommand{\cD}{\mathcal D}
\newcommand{\bbN}{\mathbb N}

\definecolor{darkred}{rgb}{0.6,0.1,0.1}
\definecolor{darkgreen}{rgb}{0.1,0.6,0.1}
\definecolor{darkblue}{rgb}{0.1,0.1,0.6}

\newcommand{\pd}{\partial}
\newcommand{\pdv}{{\partial}_{v_j}}
\newcommand{\pdr}{{\partial}_{r_j}}

\newcommand{\R}{\mathbb{R}}
\newcommand{\E}{\mathbb{E}}

\newcommand{\bbR}{\mathbb{R}}

\newcommand{\I}{\mathbb{I}}

\newcommand{\dd}{\,\mathrm{d}}

\newcommand{\be}{\begin{equation}}
\newcommand{\en}{\end{equation}}

\newcommand{\hrho}{\widehat\varrho}
\newcommand{\orho}{\bar{\rho}}

\newcommand{\hrhok}{\hrho^{(k)}}
\newcommand{\uk}{u^{(k)}}
\newcommand{\KK}{\mathbb{K}}
\def\Xint#1{\mathchoice
{\XXint\displaystyle\textstyle{#1}}%
{\XXint\textstyle\scriptstyle{#1}}%
{\XXint\scriptstyle\scriptscriptstyle{#1}}%
{\XXint\scriptscriptstyle\scriptscriptstyle{#1}}%
\!\int}
\def\XXint#1#2#3{{\setbox0=\hbox{$#1{#2#3}{\int}$ }
\vcenter{\hbox{$#2#3$ }}\kern-.6\wd0}}

\def\dashint{\Xint-}

\begin{document}

\numberwithin{equation}{section}

\title[Bead-spring-chain models for dilute polymeric fluids]{McKean--Vlasov diffusion and the well-posedness of the Hookean bead-spring-chain model for dilute polymeric fluids:\\
small-mass limit and equilibration in momentum space}
\author[Endre S\"uli and Ghozlane Yahiaoui]
{Endre S\"uli and Ghozlane Yahiaoui}
\address{Mathematical Institute, University of Oxford, Woodstock Road, Oxford OX2 6GG, United Kingdom;
{\tt Endre.Suli@maths.ox.ac.uk}, {\tt Ghozlane.Yahiaoui@maths.ox.ac.uk}}

\begin{abstract}
We reformulate a general class of classical bead-spring-chain models for dilute polymeric fluids, with
Hookean spring potentials, as McKean--Vlasov diffusion. This results in a coupled system of partial
differential equations involving the unsteady incompressible linearized Navier--Stokes equations,
referred to as the Oseen system, for the velocity and the pressure of the fluid, with a
source term which is a nonlinear function of the probability density function, and a second-order
degenerate parabolic Fokker--Planck equation, whose transport terms depend on the velocity field,
for the probability density function. We show that this coupled Oseen--Fokker--Planck
system has a large-data global weak solution.
We then perform a rigorous passage to the limit as the masses of the beads in the bead-spring-chain converge to zero,
which is shown in particular to result in equilibration in momentum space. The limiting problem is then used
to perform a rigorous derivation of the Hookean bead-spring-chain model for dilute polymeric fluids, 
which has the interesting feature that, if the flow domain is bounded, then so is the associated 
configuration space domain and the associated Kramers stress tensor is defined by integration over this bounded configuration domain.

\vspace{0.13pc}

\vspace{0.13pc}

\end{abstract}
%

\maketitle

\vspace{-10mm}

\section{Introduction}
\label{sec:intro}
This paper is concerned with the mathematical analysis of a set of partial differential equations arising in a
class of bead-spring-chain models for dilute polymeric fluids, where long polymer molecules immersed in a viscous incompressible
Newtonian fluid are idealized as linear chains of $J+1$ beads $\mathfrak{B}_1, \dots,\mathfrak{B}_{J+1}$, each with small mass $\epsilon$,  which are
considered to be points positioned at
$r_1, \dots, r_{J+1}$, respectively, in the flow domain $\Omega \subset \R^d$, $d \in \{2,3\}$; $\mathfrak{B}_{j+1}$ and $\mathfrak{B}_j$ are assumed to be
connected with an elastic spring with spring force $F = F(q_j)$, where $q_j = r_{j+1}-r_j$, $j=1,\dots,J$.
Models of this type involve the coupling of the incompressible Navier--Stokes equations with a Fokker--Planck equation.
For the derivation of polymeric flow models of this kind we refer to the monographs \cite{BCAH} and \cite{Larson1}.
The mathematical analysis of coupled Navier--Stokes--Fokker--Planck systems that model dilute polymeric fluids has been a subject of active research in recent years; for a survey of recent developments in this field
the reader may wish to consult, for example,
\cite{LM2, CFTZ, Masmoudi, M10, ChenL13},
or the papers
\cite{BSS, BS, BS:M3AS, BS2011-fene, BS2010-hookean, BS-2012-density-JDE, BS-2015-compressible-M3AS, BS2016, BS2018, BMS}, and the references therein.

Here we pursue an alternative line of investigation, which has to the best of our knowledge not, so far, been considered in connection with models of dilute polymeric fluids: we shall recast the model in terms of McKean--Vlasov diffusion, in the sense that the stochastic differential equation appearing in the model will have coefficients that depend on the distribution of the solution itself. As our objective here is to understand the impact of the McKean--Vlasov diffusion on the model rather than dealing with the usual technical difficulties
associated with the presence of the nonlinear convection term in the Navier--Stokes equation, we shall consider instead
a linearization of the Navier--Stokes equation about a bounded divergence-free velocity field $b$, resulting
in a linearized Navier--Stokes equation, usually referred to as the Oseen equation, whose right-hand side contains the divergence of
an elastic extra stress tensor $\KK$, representing the contribution of the polymeric stress to the Cauchy stress.

More precisely, we shall consider the following unsteady Oseen system on the space-time domain $\overline\Omega \times [0,T]$, where $\Omega$ is a
bounded open convex domain in $\R^d$, $d \in \{2,3\}$, with a $\mathcal{C}^2$ boundary,
and $T>0$:
\begin{subequations}\label{eq:NSe}
\begin{alignat}{2}
\pd_t u + (b\cdot \nabla) u - \mu \triangle u + \nabla \pi &= \nabla \cdot \KK &&\qquad \mbox{for
$(x, t) \in {\Omega} \times (0,T]$}, \\
\nabla \cdot u &= 0 &&\qquad \mbox{for $(x, t) \in {\Omega} \times (0, T]$},\\
u(x,t) &=0 &&\qquad \mbox{for $(x,t) \in \partial\Omega \times (0,T]$},\\
u(x,0) & = u_0(x)&&\qquad \mbox{for $x \in {\Omega}$},
\end{alignat}
\end{subequations}
with
\begin{alignat}{2}\label{eq:K}
\KK(x,t) &:= \E^x \left(\sum_{j=1}^J F(q_j) \otimes q_j\right) &&\qquad\mbox{for
$(x, t) \in \Omega \times (0, T]$, $J \geq 1$},
\end{alignat}
%

\noindent
where $F$ is a spring force vector and $\E^x$ denotes conditional expectation in a sense to be made precise below. 
We shall assume without loss of generality that $0 \in \mathbb{R}^d$ is the centroid, $\frac{1}{|\Omega|}\int_\Omega x \dd x$, of $\Omega$.

In the equations \eqref{eq:NSe}, $u: \overline{\Omega} \times [0,T] \rightarrow \mathbb{R}^d$ denotes the velocity field,
and $\pi: {\Omega} \times (0,T] \rightarrow \mathbb{R}$ is the pressure; $b$ is a divergence-free (in the sense of distributions on ${\Omega}$)
vector field, $b \in L^\infty(0,T;L^{\infty}({\Omega})^d)$ (see, however, Remark \ref{rem:relax} concerning the weakening
of this assumption); $u_0 \in W^{1-2/z,z}_0({\Omega})^d$, with $z=d+\vartheta$ for some $\vartheta \in (0,1)$, is a divergence-free
(in the sense of distributions on ${\Omega}$) initial velocity field; $\mu>0$ is the viscosity coefficient; $\KK: {\Omega} \times (0,T] \rightarrow \mathbb{R}^{d \times d}_{\rm symm}$ is the elastic extra stress tensor, involving the conditional expectation $\E^x$, which we now define.
To this end we introduce the following notations:
\begin{alignat*}{2}
r&:=(r_1^{\mathrm T}, \ldots, r_{J+1}^{\mathrm T})^{\mathrm T}, \quad && \mbox{where $r_j \in {\Omega}$ for $j=1,\dots,J+1$},\\
v&:=(v_1^{\mathrm T}, \ldots, v_{J+1}^{\mathrm T})^{\mathrm T}, \quad && \mbox{where $v_j \in \R^d$ for $j=1,\dots,J+1$},\\
q=q(r)&:=(q_1^{\mathrm T}, \ldots, q_{J}^{\mathrm T})^{\mathrm T}, \quad && \mbox{where $q_j = q_j(r):= r_{j+1} -  r_j$ for $j=1,\ldots, J$}.
\end{alignat*}
We note here that
\[q_j \in D:={\Omega}-{\Omega}=\{\omega_1 - \omega_2\,:\, \omega_1, \omega_2 \in {\Omega}\},\qquad \mbox{for $j=1,\dots,J$};\]
by definition, $D$ is a bounded, balanced, convex neighbourhood of $0 \in \R^d$, and $D \subset [-L,L]^d$ for some $L>0$. Furthermore, we let
\[
x :=  \frac{1}{J+1}\sum_{j=1}^{J+1} r_j.
\]
Thanks to the assumed convexity of $\Omega$, $x \in \Omega$ for any $r_1, \dots, r_{J+1} \in \Omega$.

Let
\[ \varrho\,:\, (r,v,t)\in \Omega^{J+1}\times \R^{(J+1)d} \times [0,T] \mapsto \varrho(r,v,t) \in \R_{\geq 0}\]
be the probability density function associated with the law of a diffusion process
for $(r,v)$, which we shall define below; the law depends on $\varrho$ itself through the function
$u$ and is therefore a McKean--Vlasov diffusion process.

Now, given $F \in L^\infty(D;\R^{d})$, we define $\E \big(\sum_{j=1}^J F(q_j)\otimes q_j\big) \,:\, (0,T] \rightarrow \R^{d \times d}_{\rm symm}$ by
\[
\bigg(\E\bigg(\sum_{j=1}^J F(q_j)\otimes q_j\bigg)\bigg)(t):=\int_{\Omega^{J+1}\times \R^{(J+1)d}} \bigg(\sum_{j=1}^J F(q_j(r))\otimes q_j(r)\bigg)\,\varrho(r,v,t)\dd r \dd v, \qquad t \in (0,T],
\]
and we perform a change of variables in this integral, replacing integration over
$r \in \Omega^{J+1}$ by integration over $(q,x) \in D^{J} \times {\Omega}$. To this end, we note that the mapping
$r \in \Omega^{J+1} \mapsto (q,x) \in D^{J} \times \Omega$ is one-to-one and onto.
Denoting by $B$ the linear transformation
from $(q,x)\in D^{J} \times \Omega$ to $r\in \Omega^{J+1}$, so that $r=B(q,x)$, and letting $\mathrm{D} B$ denote the Jacobian matrix of the transformation, we have that
\[
\bigg(\E\bigg(\sum_{j=1}^J F(q_j)\otimes q_j\bigg)\bigg)(t)=\int_{D^{J} \times {\Omega} \times \R^{(J+1)d}} \bigg(\sum_{j=1}^J F(q_j)\otimes q_j\bigg)\,\varrho\bigl(B(q,x),v,t\bigr)\,|{\rm det\,} \mathrm{D} B|\dd q \dd x \dd v,
\qquad t \in (0,T].
\]
Henceforth $|\cdot|$ will signify the absolute value of a real number, the Euclidean norm of a vector,
or the Frobenius norm of a square matrix, depending on the context.

We note that the Cartesian product of $K\geq 1$ bounded open convex sets in $\mathbb{R}^d$ is a bounded open convex set in
$\mathbb{R}^{Kd}$ (cf. \cite{HUL}, p.23), and that by Corollary 1.2.2.3 in \cite{Gris}, a bounded open convex set in a Euclidean space
has Lipschitz boundary, so $\Omega^{J+1}$ and $D^J$ are (convex) Lipschitz domains in $\mathbb{R}^{(J+1)d}$ and $\mathbb{R}^{Jd}$, respectively.

The class of spring forces under consideration here are of the form
\[ F(q_j) = H\,U'(|q_j|)\,q_j,\qquad \mbox{for $q_j \in D$, $\quad j = 1,\dots, J$,}\]
where
$H>0$ is a \textit{spring constant}, characteristic of the stiffness of the spring, and $U$ is a given
spring potential, $U \in \mathcal{C}^{0,1}([0,b];\R)$, with $b:=\sup_{p \in D}|p|$. For example, $U(s)= s$
corresponds to a model with Hookean springs, which we shall hereafter focus on in the rest of the paper.
Clearly, since $\Omega$ is bounded, the same is true of $D$ and therefore $0<b<\infty$.

The conditional expectation $\E^x$, which is the expectation under $\E$ conditional
on
\[
\frac{1}{J+1}\sum_{j=1}^{J+1} r_j=x,
\]
is then defined as follows: for $(x,t) \in {\Omega} \times (0,T]$,
\begin{align*}
\left(\E^x\left(\sum_{j=1}^J F(q_j) \otimes q_j\right)\right)(x,t):=\frac{\int_{D^{J}\times \R^{(J+1)d}} \big(\sum_{j=1}^J F(q_j)\otimes q_j\big)\,\varrho\bigl(B(q,x),v,t\bigr)\,|{\rm  det\,} \mathrm{D} B|\dd q  \dd v}
{\int_{D^{J}\times \R^{(J+1)d)}} \varrho\,\bigl(B(q,x),v,t\bigr)\,|{\rm det\,} \mathrm{D} B|\dd q  \dd v}.
\end{align*}
Since $\mathrm{D} B$ is independent of $q$ and $v$, the factor $|{\rm det\,} \mathrm{D} B|$ cancels in the numerator, which is a $d \times d$ symmetric positive semidefinite matrix function, and in the denominator, and the expression for the above
conditional expectation is thereby simplified to
\begin{align*}
\left(\E^x\left(\sum_{j=1}^J F(q_j) \otimes q_j\right)\right)(x,t)=\frac{\int_{D^{J}\times \R^{(J+1)d}} \big(\sum_{j=1}^J F(q_j)\otimes q_j\big)\,\varrho\bigl(B(q,x),v,t\bigr) \dd q  \dd v}
{\int_{D^{J}\times \R^{(J+1)d}} \varrho\,\bigl(B(q,x),v,t\bigr) \dd q  \dd v},\qquad\!\!
(x,t) \in {\Omega} \times (0,T].
\end{align*}
We note that if the denominator vanishes at a point $(x_0,t_0) \in \Omega \times (0,T]$, then, since $\varrho$ is a
nonnegative function, necessarily $\varrho\,\bigl(B(q,x_0),v,t_0\bigr) = 0$ for a.e. $(q,v) \in D^J \times \R^{(J+1)d}$, and therefore the numerator also vanishes at $(x_0,t_0)$. We shall adopt the convention that the ratio $0/0$ is, by definition, equal to $0$.

Hence, now with $|\cdot|$ signifying the Frobenius matrix norm on $\R^{d \times d}$,
\[ \left|\left(\E^x\left(\sum_{j=1}^J F(q_j) \otimes q_j\right)\right)(x,t)\right| \leq \mbox{ess.sup}_{q \in D^{J}} \sum_{j=1}^J |F(q_j)\otimes q_j| \qquad \forall\,(x,t) \in {\Omega} \times (0,T] ,\]
whereby, recalling \eqref{eq:K},
\begin{equation}\label{eq:ce}
 \|\KK\|_{L^\infty(0,T;L^\infty({\Omega}))} \leq \sum_{j=1}^J\|F(q_j) \otimes q_j\|_{L^\infty(D^{J})} < \infty.
\end{equation}
We note that, given $\varrho$, we may write $u(x,t)=(\cA \varrho)(x,t)$, where the nonlinear operator $\cA$ involves composition of the ratio of two integral operators (as in the definition of the conditional expectation $\E^x$ above), the divergence operator $\nabla \cdot$,
and  the solution operator for the time-dependent Oseen problem. As the velocity field $u=\cA \varrho$ appears as a coefficient in the Fokker--Planck equation for the probability density function $\varrho$, it follows that it is, in fact, in the present context, a \textit{nonlinear} partial differential equation for $\varrho$.

The aim of the paper is two-fold: we will show that this coupled Oseen--Fokker--Planck system has a large-data global
weak solution; having done so, we shall perform a rigorous analysis of the small-mass limit, $\epsilon \rightarrow 0_+$,
corresponding to passage to the limit as the masses of the beads $\mathfrak{B}_1,\dots, \mathfrak{B}_{J+1}$
in the bead-spring-chain converge to zero, leading to a rigorous derivation of the Hookean bead-spring-chain model.

We proceed to define the McKean--Vlasov diffusion. Let
\begin{align*}
\mathcal{U}(r,t;\varrho)&:=\Bigl( u(r_1,t)^{\rm T}, \cdots, u(r_{J+1},t)^{\rm T}\Bigr)^{\rm T}\\
&~=\Bigl( (\cA \varrho)(r_1,t)^{\rm T}, \cdots, (\cA \varrho)(r_{J+1},t)^{\rm T}\Bigr)^{\rm T},
\end{align*}
with $\mathcal{A}$ as indicated above, and consider the SDE
$$\eps^2 \ddot{r}={\mathcal L}r+\zeta\big(\mathcal{U}(r,t;\varrho)-\dot{r}\big)+\sqrt{2\beta}\,\dot{W}.$$
Here $\eps^2>0$ signifies the mass of an individual bead in the chain, $\beta=k\mathtt{T}\zeta>0$,
where $k$ is the Boltzmann constant, $\mathtt{T}$ is the absolute temperature and $\zeta$ is the drag coefficient.
Furthermore, $\cL$ is the following $(J+1)\times(J+1)$ block-matrix (analogous to a discrete Laplacian, corresponding to a homogeneous
Neumann boundary condition) with $d \times d$ matrices as its entries, associated with a Hookean bead-spring-chain:
\begin{eqnarray*}
\lambda\left(
\begin{array}{ccccc}
    -\mathbb{I} & \mathbb{I}  & \mathbb{O} & \dots  &\mathbb{O} \\
    \mathbb{I} & -2\mathbb{I} & \mathbb{I} & \ddots &\mathbb{O} \\
    \mathbb{O} & \mathbb{I} & -2\mathbb{I} & \mathbb{I} & \mathbb{O}\\
    \vdots & \ddots & \ddots & \ddots & \ddots \\
    \mathbb{O} & \dots & \mathbb{I} & -2\mathbb{I}&\mathbb{I}\\
    \mathbb{O} & \dots & \mathbb{O} & \mathbb{I}&\mathbb{-I}
\end{array}
\right),
\end{eqnarray*}
where $\lambda>0$ is a constant factor characteristic of the stiffness of the springs,
the block $\mathbb{I} \in \mathbb{R}^{d \times d}$ is the $d \times d$ identity matrix, and the block $\mathbb{O} \in \mathbb{R}^{d \times d}$ is the $d \times d$ zero matrix. Thus, $\cL$ is a $(J+1)d \times (J+1)d$ matrix, in fact.

As the parameter $\lambda$ plays no role in the discussion that will follow, we set $\lambda=1$; similarly, we set $\zeta=1$.
The SDE may then be rewritten as the first-order system
\begin{equation}
\begin{aligned}
\label{eq:sde}
\eps \dot{r}  &=  v,\\
\eps \dot{v}  &=  {\mathcal L}r+\mathcal{U}(r,t;\varrho)-\eps^{-1} v+\sqrt{2\beta}\,\dot{W}.
\end{aligned}
\end{equation}
Then, $\varrho(r,v,t)$ solves the Fokker--Planck equation, stated in the next section, associated
with this system. For \eqref{eq:sde} to be meaningful, it is clearly necessary that the function
$(r,t) \in \Omega \times [0,T] \mapsto u(r,t) \in \mathbb{R}^d$ satisfies the Carath\'eodory condition:
i.e., it is \textit{continuous} with respect to $r$ for a.e. $t \in [0,T]$ and \textit{measurable} with respect to $t$
for every $r \in \Omega$. This requirement is consistent with the underlying modelling assumption
that the background fluid (i.e. the solvent), in which the polymer molecules are immersed, represents
a `continuum' relative to the scale of the polymer molecules.

The paper is structured as follows: In Section \ref{sec:use} we formulate the Fokker--Planck equation.
In Section \ref{sec:FP} we show, for a fixed velocity field $u$, the existence of a global weak solution to the
Fokker--Planck equation, subject to a specular boundary condition. The argument is based on a parabolic
regularization of the (hypoelliptic) Fokker--Planck equation, and passage to the limit with the parabolic
regularization parameter. In Section \ref{sec:coupled} we then return to the original coupled
Oseen--Fokker--Planck system and use an iterative process between the Oseen equation and the Fokker--Planck equation
to show the existence of large-data global weak solutions to the coupled problem for any nonnegative $L^1$ initial datum
with finite initial relative entropy for the Fokker--Planck equation, and any (distributionally) divergence-free initial datum $u_0 \in W^{1-2/z,z}_0(\Omega)^d$,
with  $z=d+\vartheta$ for some $\vartheta \in (0,1)$, for the Oseen equation. 
The latter regularity hypothesis on $u_0$ will then ensure the continuity of the velocity field $u$ with respect to
its spatial variable, alluded to in the last sentence
of the previous paragraph, via maximal regularity theory for the unsteady Stokes system. Indeed, the fact
that in the case of the Oseen system we are able to guarantee, through the above regularity hypothesis on $u_0$,
that the velocity field $u$ belongs to the function space $u \in L^2(0,T;L^\infty(\Omega)^d)$ plays a crucial role
in our proofs; it is unclear to us, in particular, how replacement of the Oseen system by the full Navier--Stokes 
system would impact on the arguments presented herein.

The proofs use a variety of compactness arguments for infinite sequences of approximate solutions. Passage to the limit in the
extra stress tensor $\mathbb{K}$, whose divergence appears on the right-hand side of the Oseen equation, is nontrivial
as $\mathbb{K}$ depends nonlinearly on the probability density function; to this end, we shall show the strong convergence
of the sequence of approximating probability density functions using techniques developed by DiPerna \& Lions
for the Fokker--Planck--Boltzmann system and related hypoelliptic PDEs (see, in particular, the Appendix in \cite{DL}).
In Section \ref{sec:trace} we show, by using the existence of a trace on the boundary of our domain, that the solution to the Fokker--Planck equation attains 
the weakly imposed specular boundary condition in a strong sense.
In Section \ref{sec:SM} we then focus on the second objective of the paper: we rigorously identify the small-mass limit of the system, as $\epsilon \rightarrow 0_+$. Once again, passage to the limit in the extra stress tensor $\mathbb{K}$,
whose divergence appears on the right-hand side of the Oseen equation, is the main source of technical difficulties, as we require
strong convergence of the approximating sequence of probability density functions, as $\epsilon \rightarrow 0_+$.
Motivated by an argument in the work of Carrillo \& Goudon \cite{CG}, which first appeared in the context of diffusion
asymptotics for hyperbolic problems in the work of Marcati and Milani \cite{MM1990}, and was then applied in the
framework of kinetic equations by Lions \& Toscani \cite{LT1997} and Goudon \& Poupaud \cite{GP2001}, we shall use a compensated compactness
argument based on the Div-Curl lemma to prove weak convergence, which we then strengthen to the desired strong convergence
result, enabling us to identify the small-mass limit, as $\epsilon \rightarrow 0_+$. We prove in particular that passage to the small-mass limit 
results in equilibration in momentum space, in a sense to be made precise in Remark \ref{rem:equilib}.
This enables us to make mathematically rigorous various formal asymptotic calculations from the polymer physics literature asserting that passage 
to the small-mass limit implies equilibration in momentum space. In the final section we relate the resulting small-mass-limit
model to the classical Hookean bead-spring-chain model for dilute solutions of polymeric fluids.

\section{Statement of the Fokker--Planck equation}
\label{sec:use}

To define the Fokker--Planck equation we mimic the procedure in \cite{PS08} and introduce the following
differential operators, noting that those with suffix $0$ are
independent of $u$ (which is considered to be fixed for the moment), whilst those with suffix $1$ are not:
\begin{align*}
\cL_{0,j}\varphi &:=-v_j \cdot \pdv\varphi+\beta\, \pdv^2\varphi,\qquad j=1,\dots, J+1,\\
\cL_{1,j}(u)\varphi &:= v_j \cdot \pdr\varphi+(({\mathcal L}r)_j+u(r_j,t))\cdot \pdv \varphi,\qquad j=1,\dots, J+1,\\
\cL_{0,j}^*\varphi &:=\pdv\cdot (v_j \varphi)+\beta\, \pdv^2\varphi,\qquad j=1,\dots, J+1,\\
\cL_{1,j}^*(u)\varphi &:= -v_j \cdot \pdr\varphi-(({\mathcal L}r)_j+u(r_j,t))\cdot \pdv \varphi,\qquad j=1,\dots, J+1,\\
\cL_0 \varphi&:=\sum_{j=1}^{J+1}\cL_{0,j}\varphi,\\
\cL_1(u)\varphi&:=\sum_{j=1}^{J+1}\cL_{1,j}(u)(\varphi).
\end{align*}
In these expressions $\pdv$ denotes the ($d$-component) gradient operator with respect to $v_j \in \R^d$, $\pdv\cdot$
denotes the divergence operator with respect to $v_j$, and $\pdv^2 = \pdv \cdot \pdv$ is the Laplace operator
with respect to $v_j$.
We further note that $\cL_{0,j}$ has a one-dimensional null-space spanned
by the real-valued constant function that is identically equal to $1$ with respect to $v_j$, denoted by $\bbI(v_j)$, and its adjoint $\cL_{0,j}^*$ has null-space spanned by the function
$$g(v_j):=(2\pi \beta)^{-\frac12}\exp(-|v_j|^2/2\beta).$$
Observe also that, for $g(s)=(2\pi \beta)^{-\frac12}\exp(-s^2/2\beta)$, $s \in \R$, and with $'$ denoting differentiation with respect to the variable $s$, we have that
$$(sg(s))'+\beta g''(s)=0,$$
implying that
\begin{equation}
\label{eq:need}
(sg'(s))'+\beta(g'(s))''=-g'(s).
\end{equation}
Finally, we note that $\cL_{0}$ has a one-dimensional null-space spanned by the constant function with
respect to $v=(v_1^{\rm T}, \dots, v_{J+1}^{\rm T})^{\rm T}$, denoted by $\bbI(v)=\prod_{j=1}^{J+1} \bbI(v_j)$, and its adjoint $\cL_{0}^*$ has null-space spanned by the function
\begin{align*}
\rif(v)=\prod_{j=1}^{J+1} g(v_j).
\end{align*}

\section{Existence of solutions to the Fokker--Planck equation}
\label{sec:FP}

The probability density function associated with \eqref{eq:sde} is denoted by $\varrho=\varrho(r,v,t)$; formally it solves the \textit{nonlinear} partial differential equation
\begin{equation}
\label{eq:FP}
\pd_t \varrho=\frac{\beta}{\eps^2}\cL_{0}^* \varrho+\frac{1}{\eps}\cL_{1}(u)^* \varrho.
\end{equation}
In case it is not apparent, we emphasize that the nonlinearity enters into the equation through the
dependence of the velocity field $u$ on the probability density function $\varrho$, since $u$ is the solution
of the Oseen equation whose right-hand side depends on $\varrho$ through the presence of the
conditional expectation there.
Substituting the defining expressions for $\cL_{0}^*$ and $\cL_{1}(u)^*$ into \eqref{eq:FP} yields
\begin{alignat}{2}
\label{eq:FP-eq}
\pd_t \varrho - \frac{\beta}{\eps^2}\left(\sum_{j=1}^{J+1} \pdv \cdot (v_j \,\varrho)+\beta\, \pdv^2 \varrho \right) + \frac{1}{\eps}
\left(\sum_{j=1}^{J+1} v_j \cdot \pdr \varrho + (({\mathcal L}r)_j+u(r_j,t)) \cdot \pdv \varrho \right) = 0,\\
~\hfill \mbox{for all $(r,v,t) \in \Omega^{J+1} \times \R^{(J+1)d} \times (0,T]$},\nonumber\\
\varrho(r,v,0)=\varrho_0(r,v)~\hfill \qquad \mbox{for all $(r,v) \in \Omega^{J+1} \times \R^{(J+1)d}$}.
\label{eq:FP-ini}
\end{alignat}

The equation \eqref{eq:FP-eq} should be supplemented with a boundary condition; here,
for the sake of simplicity of the exposition, we shall consider a specular boundary condition with respect to
the independent variable $r$,  which we shall state below. More complicated boundary conditions can of course
be used to model the interaction between the wall $\partial\Omega$ and the beads in the
bead-spring-chain; for example, a Maxwell-type boundary condition (proposed by Maxwell \cite{Maxwell1879} in 1879
as a phenomenological law by splitting the reflection operator into a local
reflection operator and a diffuse reflection operator) may be
considered, as in \cite{Mischler2010}: it involves a boundary trace operator that
is a convex linear combination of a specular boundary trace operator, describing local
reflection by the wall, and a diffuse reflection operator.

Before formulating the specular boundary condition considered here, we require
some additional notation. We let
\[\partial\Omega^{(j)}:=\Omega \times
\cdots \times \Omega \times \partial \Omega \times \Omega \times \cdots \times \Omega,\qquad j=1,\dots, J+1,\]
with $\partial \Omega$
appearing at the $j$-th position in this $(J+1)$-fold Cartesian product. Clearly, $\bigcup_{j=1}^{J+1}
\overline{\partial\Omega^{(j)}} = \partial\Omega^{J+1}$.
Let, further,
\[\nu^{(j)}(r) := (0^{\rm T},\dots,0^{\rm T}, (\nu(r_j))^{\rm T},0^{\rm T},\dots,0^{\rm T})^{\rm T}\in \R^{(J+1)d},\]
where, for $r=(r_1,\dots,r_{J+1}) \in \partial\Omega^{(j)}$, the nonzero entry $\nu(r_j) \in \R^d$
appearing at the $j$-th position is the unit outward normal (column-)vector to $\partial \Omega$ at
$r_j \in \partial\Omega$, for $j=1,\dots, J+1$, and $0$ is a $d$-component zero (column-)vector. With this notation, we then impose the following \textit{specular boundary
condition} for $\varrho$ on $\partial\Omega^{(j)}$, $j=1,\dots, J+1$:
\begin{align}\label{eq:rho-spec}
\varrho(r,v,t) = \varrho(r,v_*^{(j)},t)\qquad \mbox{for all $(r,v,t) \in \partial\Omega^{(j)} \times \R^{(J+1)d} \times
 (0,T]$, with $v \cdot \nu^{(j)}(r)<0$},
\end{align}
where
\[ v_*^{(j)} =v_*^{(j)}(r,v) := v - 2(v\cdot \nu^{(j)}(r))\,\nu^{(j)}(r), \qquad j=1,\dots, J+1,\]
is the \textit{specular velocity}; clearly, $v_*^{(j)} \cdot \nu^{(j)}(r) = - v \cdot \nu^{(j)}(r)$. This boundary condition on $\varrho$
means that if the $j$-th bead in the chain $(r_1, \dots, r_{J+1})$ hits the boundary with velocity vector $v_j \in \R^d$ it is reflected with
velocity vector $v_j - 2(v_j \cdot \nu(r_j))\,\nu(r_j) \in \R^d$.  With respect to the independent variable $v=(v_1^{\rm T},\dots,v_{J+1}^{\rm T})^{\rm T}$
the domain of definition of $\varrho$ is $\R^{(J+1)d}$. The behaviour of $\varrho$ as a function of $v$ in
the limit of $|v| \rightarrow \infty$ is
dictated by the requirement that $\varrho(\cdot,\cdot,t) \in L^1(\Omega^{J+1}\times\bbR^{(J+1)d};\R_{\geq 0})$,
for each fixed $t \in (0,T]$.

In order to state the weak formulation of this problem we consider the Maxwellian $M(v):= \varrho_\infty(v)$ and define
\[ \hrho : = \frac{\varrho}{M}\quad \mbox{and}\quad \hrho_0 := \frac{\varrho_0}{M}.\]
Further, we define $\mathcal{F} \in \mathcal{C}(\R_{\geq 0};\R_{\geq 0})$, by
\[ \mathcal{F}(s):=s(\log s - 1) + 1, \quad s \in \R_{>0}, \quad\mbox{with $\mathcal{F}(0):=1$}.\]
The function $\mathcal{F}$ is nonnegative, strictly convex, and has superlinear growth as $s \rightarrow +\infty$, i.e.
$$\lim_{s \rightarrow + \infty}
\frac{\mathcal{F}(s)}{s}=+\infty.$$

\smallskip

We shall assume that the initial datum $\varrho_0$ (cf. \eqref{eq:FP-ini}) satisfies
\begin{align}\label{eq:ini-cond}
\begin{aligned}
\varrho_0 \in L^1(\Omega^{J+1} \times \R^{(J+1)d};\R_{\geq 0}), &\qquad \int_{\Omega^{J+1} \times \R^{(J+1)d}}\varrho_0(r,v) \dd r \dd v  = 1,\\
M\mathcal{F}(\hrho_0) \in L^1&(\Omega^{J+1} \times \R^{(J+1)d});
\end{aligned}
\end{align}
in other words, the initial probability density function $\varrho_0$ is assumed to have finite relative entropy with respect to the Maxwellian $M$.

We shall also \textit{assume} throughout this section that
$u \in L^2(0,T;W^{1,\sigma}_0({\Omega})^d)$ for some $\sigma>d$, and $\nabla \cdot u =0$  a.e. in $\Omega \times (0,T)$,
and that $u$ is \textit{given and held fixed}. We shall show later on that, under the assumptions on $u_0$ (cf. the paragraph
following eq. \eqref{eq:K}), the
function $u$ does indeed possess this regularity; in fact, we will see that $\sigma = \min(\hat\sigma,z)$, where $\hat\sigma:=2+\frac{4}{d}>d$ for $d=2,3$,
and $z=d+\vartheta$ for some $\vartheta \in (0,1)$, whereby $\sigma > d$ for $d=2,3$, as is being assumed here.
As a consequence of the assumed regularity of $u$, by Sobolev embedding,
$u \in L^2(0,T;L^\infty(\Omega)^d)$.


We (formally) multiply the equation \eqref{eq:FP-eq} by a function $\varphi \in W^{1,1}(0,T; \mathcal{C}^\infty(\overline{\Omega}^{J+1}; \mathcal{C}^\infty_0(\R^{(J+1)d})))$ and, assuming for the moment that $\varrho$ is sufficiently smooth and satisfies the specular
boundary condition \eqref{eq:rho-spec}, we  integrate the resulting equality over $\Omega^{J+1} \times \R^{(J+1)} \times [0,T]$, and then integrate by parts with respect to each of the independent variables. Hence,
\begin{align}\label{eq:prep1}
&\int_{\Omega^{J+1}} \int_{\R^{(J+1)d}} M(v)\,\hrho(r,v,T)\,\varphi(r,v,T)\dd v \dd r\nonumber\\
&\qquad - \int_0^T \int_{\Omega^{J+1}} \int_{\R^{(J+1)d}} M(v)\,\hrho(r,v,\tau)\,\pd_\tau \varphi(r,v,\tau)\dd v \dd r \dd \tau \nonumber\\
&\qquad+ \frac{\beta^2}{\eps^2}\left(\sum_{j=1}^{J+1} \int_0^T \int_{\Omega^{J+1}} \int_{\R^{(J+1)d}}
M(v)\,\pdv \hrho \cdot\pdv \varphi \dd v \dd r \dd \tau \right) \nonumber\\
&\qquad- \frac{1}{\eps} \left(\sum_{j=1}^{J+1} \int_0^T \int_{\Omega^{J+1}} \int_{\R^{(J+1)d}} M(v)\, v_j \hrho\cdot \pdr \varphi \dd v \dd r \dd \tau \right) \nonumber\\
&\qquad+ \frac{1}{\eps} \left(\sum_{j=1}^{J+1} \int_0^T \int_{\partial\Omega^{(j)}} \int_{\R^{(J+1)d}} M(v)\, (v_j
\cdot \nu(r_j))\, \hrho\,\varphi \dd v \dd s(r) \dd \tau \right) \nonumber\\
&\qquad- \frac{1}{\eps} \left(\sum_{j=1}^{J+1} \int_0^T \int_{\Omega^{J+1}} \int_{\R^{(J+1)d}}M(v)\, (({\mathcal L}r)_j+u(r_j,\tau))\,\hrho\cdot \pdv \varphi \dd v \dd r \dd \tau \right) \nonumber\\
&\qquad\quad  = \int_{\Omega^{J+1}} \int_{\R^{(J+1)d}} M(v)\,\hrho_0(r,v)\,\varphi(r,v,0)\dd v \dd r
\qquad \forall\, \varphi \in W^{1,1}(0,T; \mathcal{C}^\infty(\overline{\Omega}^{J+1}; \mathcal{C}^\infty_0(\R^{(J+1)d}))).
\end{align}

We focus our attention on the fifth integral on the left-hand side:
%
\begin{align*}
&\int_{\partial\Omega^{(j)}} \int_{\R^{(J+1)d}} M(v)\, (v_j
\cdot \nu(r_j))\, \hrho\,\varphi \dd v \dd s(r)  = \int_{\R^{(J+1)d}} \int_{\partial\Omega^{(j)}}  M(v)\, (v_j
\cdot \nu(r_j))\, \hrho\,\varphi \dd s(r) \dd v  \\
& \quad = \int_{\R^{(J+1)d}} \int_{\partial\Omega^{(j)}\,:\,v_j \cdot \nu(r_j)>0}\!\! M(v)\, (v_j
\cdot \nu(r_j))\, \hrho\,\varphi \dd s(r) \dd v \\
&\qquad + \int_{\R^{(J+1)d}} \int_{\partial\Omega^{(j)}\,:\,v_j \cdot \nu(r_j)<0}
 \!\!M(v)\, (v_j \cdot \nu(r_j))\, \hrho\,\varphi  \dd s(r) \dd v.
\end{align*}
Now, since $|v_*^{(j)}|^2=|v|^2$ and $v_*^{(j)} \cdot \nu^{(j)}(r) = -v \cdot \nu^{(j)}(r) = - v_j \cdot \nu(r_j)$, and using the specular boundary
condition satisfied by $\hrho$, we have for the second integral
on the right-hand side of the last equality that
\begin{align*}
&\int_{\R^{(J+1)d}}\int_{\partial\Omega^{(j)}\,:\,v_j \cdot \nu(r_j)<0}  M(v)\, (v_j
\cdot \nu(r_j))\, \hrho(r,v,t)\,\varphi(r,v,t) \dd s(r) \dd v \\
&= \int_{\R^{(J+1)d}} \int_{\partial\Omega^{(j)}\,:\,v_j \cdot \nu(r_j)<0}  M(v)\, (v_j
\cdot \nu(r_j))\, \hrho(r,v_*^{(j)},t)\,\varphi(r,v,t) \dd s(r) \dd v\\
&= - \int_{\R^{(J+1)d}} \int_{\partial\Omega^{(j)}\,:\,-v_j \cdot \nu(r_j)>0} M(v)\, (-v_j
\cdot \nu(r_j))\, \hrho(r,v_*^{(j)},t)\,\varphi(r,v,t) \dd s(r) \dd v\\
&= - \int_{\R^{(J+1)d}}\int_{\partial\Omega^{(j)}\,:\, v_*^{(j)} \cdot \nu^{(j)}(r) >0}
 M(v_*^{(j)})\, (v_*^{(j)} \cdot \nu^{(j)}(r))\, \hrho(r,v_*^{(j)},t)\,\varphi(r,v,t)
\dd s(r) \dd v.
\end{align*}
Assuming that the test function $\varphi$ satisfies the specular boundary condition:
\begin{equation}\label{eq:phi-spec}
\varphi(r,v,t) = \varphi(r,v_*^{(j)},t)\qquad \begin{array}{l} \forall\,(r,v,t) \in \partial\Omega^{(j)} \times \R^{(J+1)d} \times
 (0,T],\\
\mbox{such that } v \cdot \nu^{(j)}(r)<0,\quad j=1,\dots, J+1,
\end{array}
\end{equation}
we then have, for all $j=1, \dots, J+1$,  that
\begin{align*}
&\int_{\R^{(J+1)d}}  \int_{\partial\Omega^{(j)}\,:\,v_j \cdot \nu(r_j)<0} M(v)\, (v_j
\cdot \nu(r_j))\, \hrho(r,v,t)\,\varphi(r,v,t) \dd s(r) \dd v \\
&= -\int_{\R^{(J+1)d}}\int_{\partial\Omega^{(j)}\,:\, v_*^{(j)} \cdot \nu^{(j)}(r) >0}
M(v_*^{(j)})\, (v_*^{(j)} \cdot \nu^{(j)}(r))\, \hrho(r,v_*^{(j)},t)\,\varphi(r,v_*^{(j)},t)
\dd s(r) \dd v\\
&= -\int_{\R^{(J+1)d}}\int_{\partial\Omega^{(j)}}
M(v_*^{(j)})\, (v_*^{(j)} \cdot \nu^{(j)}(r))_{+}\, \hrho(r,v_*^{(j)},t)\,\varphi(r,v_*^{(j)},t)
\dd s(r) \dd v\\
&= -\int_{\partial\Omega^{(j)}}\int_{\R^{(J+1)d}}
M(v_*^{(j)})\, (v_*^{(j)} \cdot \nu^{(j)}(r))_{+}\, \hrho(r,v_*^{(j)},t)\,\varphi(r,v_*^{(j)},t)
 \dd v \dd s(r).
\end{align*}
Since, for $r \in \partial\Omega^{(j)}$ fixed,
the absolute value of the Jacobian $\mathrm{D}\Phi$ of the (bijective) mapping $$\Phi\,:\,v\in \R^{(J+1)d} \mapsto
v_*^{(j)}(r,v) \in \R^{(J+1)d}$$
is equal to $1$, whereby, for $r \in \partial\Omega^{(j)}$ fixed, $\dd v_*^{(j)} = |\mathrm{D}\Phi| \dd v = \dd v$, by treating $v_*^{(j)}$ as
a dummy variable in the last integral and renaming it into $v$, and noting again that $v \cdot \nu^{(j)}(r) =  v_j \cdot \nu(r_j)$, it follows that, for all $j=1, \dots, J+1$,
\begin{align*}
& \int_{\R^{(J+1)d}}\int_{\partial\Omega^{(j)}\,:\,v_j \cdot \nu(r_j)<0} M(v)\, (v_j
\cdot \nu(r_j))\, \hrho(r,v,t)\,\varphi(r,v,t)  \dd s(r) \dd v \\
&= -\int_{\partial\Omega^{(j)}} \int_{\R^{(J+1)d}} M(v)\, (v_j \cdot \nu(r_j))_+\, \hrho(r,v,t)\,\varphi(r,v,t) \dd v \dd s(r)\\
&= -\int_{\R^{(J+1)d}} \int_{\partial\Omega^{(j)}}  M(v)\, (v_j \cdot \nu(r_j))_+\, \hrho(r,v,t)\,\varphi(r,v,t) \dd s(r) \dd v \\
&= - \int_{\R^{(J+1)d}}\int_{\partial\Omega^{(j)}\,:\,v_j \cdot \nu(r_j)>0} M(v)\, (v_j \cdot \nu(r_j))\, \hrho(r,v,t)\,\varphi(r,v,t) \dd s(r)  \dd v.
\end{align*}
Hence, provided that the test function $\varphi \in W^{1,1}(0,T; \mathcal{C}^\infty(\overline{\Omega}^{J+1}; \mathcal{C}^\infty_0(\R^{(J+1)d})))$
appearing in \eqref{eq:prep1} satisfies the specular boundary condition \eqref{eq:phi-spec}, the fifth integral in
\eqref{eq:prep1} will vanish. We shall therefore assume that this is indeed the case and will work with such test functions
$\varphi$, whereby the absence of the fifth integral from \eqref{eq:prep1} can be seen as a weak imposition of the
specular boundary condition \eqref{eq:rho-spec} for $\hrho$ (and, equivalently, for $\varrho$). The imposition of the
specular boundary condition on all functions that belong to a certain function space will be indicated by the subscript
$_*$ in our notation for the particular function space. For example,
\begin{align*}
\mathcal{C}^\infty_*(\overline{\Omega}^{J+1}; \mathcal{C}^\infty_0(\R^{(J+1)d})) &= \bigg\{ \varphi \in \mathcal{C}^\infty(\overline{\Omega}^{J+1}; \mathcal{C}^\infty_0(\R^{(J+1)d}))\,:\, \varphi(r,v) = \varphi(r,v_*^{(j)})\\
&\qquad \mbox{for all $(r,v) \in \partial\Omega^{(j)} \times \R^{(J+1)d}$, with $v \cdot \nu^{(j)}(r)<0$},\quad j=1,\dots, J+1\bigg\}.
\end{align*}
Thus, by eliminating the fifth integral from \eqref{eq:prep1}, we are led to the following problem:
for a fixed divergence-free function $u \in L^2(0,T;W^{1,\sigma}_0({\Omega})^d)$ with $\sigma>d$, we seek a function $\hrho \geq 0$ such that
\[ M \hrho \in \mathcal{C}_w([0,T]; L^1(\Omega^{J+1} \times \R^{(J+1)d})),\]
\[ M\mathcal{F}(\hrho) \in L^\infty(0,T; L^1(\Omega^{J+1} \times \R^{(J+1)d})),\qquad \sqrt{\hrho} \in  L^2(0,T; L^2(\Omega^{J+1}; W^{1,2}_M(\R^{(J+1)d}))),\]
with $\mathcal{C}_w([0,T]; L^1(\Omega^{J+1} \times \R^{(J+1)d}))$ being the linear space of weakly continuous mappings from $[0,T]$ into $L^1(\Omega^{J+1} \times \R^{(J+1)d})$,  and $W^{1,2}_M(\R^{(J+1)d})$ signifying the Maxwellian-weighted Sobolev space on $\R^{(J+1)d}$:
\[W^{1,2}_M(\R^{(J+1)d}) := \bigg\{ \varphi \in L^1_{\rm loc}(\R^{(J+1)d})\,:\, \|\varphi\|^2_{W^{1,2}_M(\R^{(J+1)d})}:=
\int_{\R^{(J+1)d}} M(v) \bigg(|\varphi(v)|^2 +
\sum_{j=1}^{J+1}|\partial_{v_j} \varphi(v)|^2\bigg) \dd v < \infty\bigg\}\]
(with analogous notation for all other Maxwellian-weighted Lebesgue and Sobolev spaces), such that
{\small
\begin{align}\label{eq:weak}
&\int_{\Omega^{J+1}} \int_{\R^{(J+1)d}} M(v)\,\hrho(r,v,T)\,\varphi(r,v,T)\dd v \dd r- \int_0^T \int_{\Omega^{J+1}} \int_{\R^{(J+1)d}} M(v)\,\hrho(r,v,\tau)\,\pd_\tau \varphi(r,v,\tau)\dd v \dd r \dd \tau \nonumber\\
&\qquad+ \frac{\beta^2}{\eps^2}\left(\sum_{j=1}^{J+1} \int_0^T \int_{\Omega^{J+1}} \int_{\R^{(J+1)d}}
M(v)\,\pdv \hrho \cdot\pdv \varphi \dd v \dd r \dd \tau \right)\nonumber\\
&\qquad- \frac{1}{\eps} \left(\sum_{j=1}^{J+1} \int_0^T \int_{\Omega^{J+1}} \int_{\R^{(J+1)d}} M(v)\, v_j \hrho\cdot \pdr \varphi \dd v \dd r \dd \tau \right)\nonumber\\
&\qquad- \frac{1}{\eps} \left(\sum_{j=1}^{J+1} \int_0^T \int_{\Omega^{J+1}} \int_{\R^{(J+1)d}}M(v)\, (({\mathcal L}r)_j+u(r_j,\tau))\,\hrho\cdot \pdv \varphi \dd v \dd r \dd \tau \right)\nonumber\\
&\qquad\qquad  = \int_{\Omega^{J+1}} \int_{\R^{(J+1)d}} M(v)\,\hrho_0(r,v)\,\varphi(r,v,0)\dd v \dd r
\qquad \forall\, \varphi \in W^{1,1}(0,T; W^{s,2}_*(\Omega^{J+1} \times \R^{(J+1)d})),
\end{align}
}

\noindent
where $s>(J+1)d+1$. We note that for $s>(J+1)d+1$, by Sobolev embedding,
\[W^{s,2}_*(\Omega^{J+1} \times \R^{(J+1)d}) \hookrightarrow W^{1,\infty}_*(\Omega^{J+1} \times \R^{(J+1)d}).\]
We emphasize here again that the specular boundary condition is imposed weakly, through the \textit{omission} of the fifth integral from \eqref{eq:prep1}
(and, thereby, through the \textit{absence} of the corresponding term from \eqref{eq:weak}) and the choice of the test functions $\varphi$
in $ W^{1,1}(0,T; W^{s,2}_*(\Omega^{J+1} \times \R^{(J+1)d}))$.
This helps us to circumvent at this point the question whether $\hrho$ is
regular enough to satisfy \eqref{eq:rho-spec} in the (stronger) sense of a trace theorem on $\partial\Omega$. The existence of a trace in a stronger sense
will be shown later, in Section \ref{sec:trace}.

\subsection{Existence of solutions to a parabolic regularization of \eqref{eq:weak}}
We begin by considering a parabolic regularization of the weak formulation \eqref{eq:weak}: for a fixed divergence-free
function $u \in L^2(0,T;W^{1,\sigma}_0({\Omega})^d)$ with $\sigma>d$, 
and with $\alpha \in (0,1]$ a regularization parameter that will be eventually sent to $0$, we shall seek a function
\[ \varrho_\alpha \in \mathcal{C}([0,T]; L^2_M(\Omega^{J+1} \times \R^{(J+1)d}))\cap L^2(0,T; W^{1,2}_{*,M}(\Omega^{J+1} \times \R^{(J+1)d}))\]
such that
{\small
\begin{align}\label{eq:weak-a}
&\int_{\Omega^{J+1}} \int_{\R^{(J+1)d}} M(v)\,\hrho_\alpha(r,v,T)\,\varphi(r,v,T)\dd v \dd r- \int_0^T \int_{\Omega^{J+1}} \int_{\R^{(J+1)d}} M(v)\,\hrho_\alpha(r,v,\tau)\,\pd_\tau \varphi(r,v,\tau)\dd v \dd r \dd \tau\nonumber\\
&\qquad+ \frac{\beta^2}{\eps^2}\left(\sum_{j=1}^{J+1} \int_0^T \int_{\Omega^{J+1}} \int_{\R^{(J+1)d}} M(v)\,\pdv \hrho_\alpha \cdot \pdv \varphi \dd v \dd r \dd \tau \right)\nonumber\\
&\qquad- \frac{1}{\eps} \left(\sum_{j=1}^{J+1} \int_0^T \int_{\Omega^{J+1}} \int_{\R^{(J+1)d}} M(v)\, v_j \hrho_\alpha\cdot \pdr \varphi \dd v \dd r \dd \tau \right)\nonumber\\
&\qquad + \alpha \sum_{j=1}^{J+1} \int_0^T \int_{\Omega^{J+1}} \int_{\R^{(J+1)d}} M(v)\,\pdr \hrho_\alpha \cdot \pdr \varphi \dd v \dd r \dd \tau \nonumber\\
&\qquad- \frac{1}{\eps} \left(\sum_{j=1}^{J+1} \int_0^T \int_{\Omega^{J+1}} \int_{\R^{(J+1)d}} M(v)\,(({\mathcal L}r)_j+u(r_j,\tau))\,\hrho_\alpha\cdot \pdv \varphi \dd v \dd r \dd \tau \right)\nonumber\\
&\qquad\qquad  = \int_{\Omega^{J+1}} \int_{\R^{(J+1)d}} M(v)\,\hrho_{0}(r,v)\,\varphi(r,v,0)\dd v \dd r
\qquad \forall\, \varphi \in W^{1,2}(0,T; W^{1,2}_{*,M}(\Omega^{J+1} \times \R^{(J+1)d})),
\end{align}
}

\noindent
where, \textit{in addition} to our earlier assumption \eqref{eq:ini-cond} on the initial datum, we shall \textit{temporarily assume} that
$$\hrho_0 \in L^2_M(\Omega^{J+1} \times \R^{(J+1)d}).$$
This additional assumption will be required in order to enable passage to the limit $\alpha \rightarrow 0_+$.
In the final step of the existence proof, discussed in Section \ref{sec:coupled}, this additional assumption
on $\hrho_0$ will be removed,
and the final global existence result for the coupled Oseen--Fokker--Planck system will be shown to hold assuming  \eqref{eq:ini-cond} only.

To show the existence of a solution to \eqref{eq:weak-a}, note that $W^{1,2}_{*,M}(\Omega^{J+1} \times \R^{(J+1)d})$, the normed linear space of all functions contained in the Maxwellian-weighted Sobolev space $W^{1,2}_{M}(\Omega^{J+1} \times \R^{(J+1)d})$
satisfying the specular boundary condition on $\partial\Omega$ in the sense of the trace theorem, is a separable Hilbert space, as it is a closed linear subspace of $W^{1,2}_{M}(\Omega^{J+1} \times \R^{(J+1)d})$, which is a separable Hilbert space
(cf. Theorem 8.10.2 on p.418 in the monograph of Kufner, John \& Fu\v{c}ik \cite{KJF}).
Furthermore, since $W^{1,2}_{M}(\Omega^{J+1} \times \R^{(J+1)d})$ is compactly embedded into
the space $L^2_M(\Omega^{J+1} \times \R^{(J+1)d})$ (cf. Appendix D in \cite{BS2010-hookean}), $W^{1,2}_{*,M}(\Omega^{J+1} \times \R^{(J+1)d})$ is also compactly embedded into $L^2_M(\Omega^{J+1} \times \R^{(J+1)d})$.
Thus, by a variant of the Hilbert--Schmidt theorem (cf. Lemma 5.1 in \cite{FS}), there exists a complete
orthogonal basis $(\psi_k)_{k \geq 1}$ in
$W^{1,2}_{*,M}(\Omega^{J+1} \times \R^{(J+1)d})$, which is complete and orthonormal in $L^2_M(\Omega^{J+1} \times \R^{(J+1)d})$; the function $\psi_k \in W^{1,2}_{*,M}(\Omega^{J+1} \times \R^{(J+1)d})$ solves the following eigenvalue problem:
\begin{align*}
(\psi_k, \eta)_{W^{1,2}_M(\Omega^{J+1} \times \R^{(J+1)d})} = \lambda_k(\psi_k,\eta)_{L^2_M(\Omega^{J+1} \times \R^{(J+1)d})} \qquad &\forall\, \eta \in W^{1,2}_{*,M}(\Omega^{J+1} \times \R^{(J+1)d}),\quad k=1,2,\dots;\\
\|\psi_k\|_{L^2_M(\Omega^{J+1} \times \R^{(J+1)d})} &= 1.
\end{align*}

Let $\mathcal{X}_N := \mbox{span}\{\psi_1,\ldots, \psi_N\}$ and denote by $P_N$ the orthogonal projector in $L^2_M(\Omega^{J+1} \times \R^{(J+1)d})$ onto $\mathcal{X}_N$. Suppose further that $w \in W^{1,2}_{*,M}(\Omega^{J+1} \times \R^{(J+1)d})$, with
\[ w = \sum_{k=1}^\infty \alpha_k \psi_k.\]
As $(w- P_Nw , \psi_j)_{L^2_M(\Omega^{J+1} \times \R^{(J+1)d})}=0$ for all $j=1,\dots, N$, thanks to the orthonormality of the functions $\psi_k$, $k\geq 1$, in $L^2_M(\Omega^{J+1} \times \R^{(J+1)d})$, it follows
that
\[ P_N w = \sum_{k=1}^N \alpha_k \psi_k.\]
Thus, by the orthogonality of the $\psi_k$ in $W^{1,2}_{*,M}(\Omega^{J+1} \times \R^{(J+1)d})$, Parseval's identity implies that
\begin{align}\label{ortho}
\|P_N w\|^2_{W^{1,2}_M(\Omega^{J+1} \times \R^{(J+1)d})} &= \sum_{k=1}^N \alpha_k^2 \|\psi_k\|^2_{W^{1,2}_M(\Omega^{J+1} \times \R^{(J+1)d})} \leq \sum_{k=1}^\infty \alpha_k^2\|\psi_k\|^2_{W^{1,2}_M(\Omega^{J+1} \times \R^{(J+1)d})} \nonumber\\
&= \|w\|^2_{W^{1,2}_M(\Omega^{J+1} \times \R^{(J+1)d})} \qquad \forall\, w \in W^{1,2}_{*,M}(\Omega^{J+1} \times \R^{(J+1)d}).
\end{align}

We shall seek a function
\begin{equation}\label{eq:expand}
\hrho_{\alpha,N}(r,v,t) = \sum_{k=1}^N A_{k,N}(t)\, \psi_k(r,v)
\end{equation}
such that
{\small
\begin{align}
&\left(\int_{\Omega^{J+1}} \int_{\R^{(J+1)d}} M(v)\,\hrho_{\alpha,N}(r,v,T)\, \psi_\ell(r,v)\dd v \dd r \right)\phi(T)\nonumber\\
&\qquad
- \int_0^T \int_{\Omega^{J+1}} \int_{\R^{(J+1)d}} M(v)\,\hrho_{\alpha,N}(r,v,\tau)\, \psi_\ell(r,v)\, \pd_\tau \phi(\tau)\dd v \dd r \dd \tau\nonumber\\
&\qquad+ \frac{\beta^2}{\eps^2}\left(\sum_{j=1}^{J+1} \int_0^T \int_{\Omega^{J+1}} \int_{\R^{(J+1)d}} M(v)\,\pdv \hrho_{\alpha,N} \cdot \pdv \psi_\ell(r,v)\, \phi(\tau) \dd v \dd r \dd \tau \right)\nonumber\\
&\qquad- \frac{1}{\eps} \left(\sum_{j=1}^{J+1} \int_0^T \int_{\Omega^{J+1}} \int_{\R^{(J+1)d}} M(v)\, v_j \hrho_{\alpha,N}\cdot \pdr \psi_\ell(r,v)\, \phi(\tau) \dd v \dd r \dd \tau \right)\nonumber
\end{align}
\begin{align}\label{eq:weak-a-gal}
&\qquad + \alpha \sum_{j=1}^{J+1} \int_0^T \int_{\Omega^{J+1}} \int_{\R^{(J+1)d}} M(v)\,\pdr \hrho_{\alpha,N} \cdot \pdr \psi_\ell(r,v)\, \phi(\tau)\dd v \dd r \dd \tau \nonumber\\
&\qquad- \frac{1}{\eps} \left(\sum_{j=1}^{J+1} \int_0^T \int_{\Omega^{J+1}} \int_{\R^{(J+1)d}} M(v)\,(({\mathcal L}r)_j+u(r_j,\tau))\,\hrho_{\alpha,N}\cdot \pdv \psi_\ell(r,v)\, \phi(\tau) \dd v \dd r \dd \tau \right)\nonumber\\
&\qquad\qquad  = \left(\int_{\Omega^{J+1}} \int_{\R^{(J+1)d}} M(v)\,\hrho_{0}(r,v)\, \psi_\ell(r,v)\dd v \dd r\right) \phi(0) \quad \forall\, \ell \in \{1,\dots,N\} \mbox{ and $\forall\,\phi \in W^{1,2}(0,T)$.}
\end{align}
}

Substitution of \eqref{eq:expand} into \eqref{eq:weak-a-gal} yields
{\small
\begin{align*}
&A_{\ell,N}(T)\,\phi(T) - \int_0^T A_{\ell,N}(\tau)\, \pd_\tau \phi(\tau) \dd \tau\nonumber\\
&\quad+ \int_0^T \sum_{k=1}^N A_{k,N}(\tau) \left(\frac{\beta^2}{\eps^2} \sum_{j=1}^{J+1} \int_{\Omega^{J+1}} \int_{\R^{(J+1)d}} M(v)\,\pdv \psi_k(r,v) \cdot \pdv \psi_\ell(r,v)\dd v \dd r \right)\phi(\tau) \dd \tau \nonumber\\
&\quad+ \int_0^T \sum_{k=1}^N A_{k,N}(\tau) \left(-\frac{1}{\eps} \sum_{j=1}^{J+1}  \int_{\Omega^{J+1}} \int_{\R^{(J+1)d}} M(v)\, v_j \psi_k(r,v) \cdot \pdr \psi_\ell(r,v)\dd v \dd r \right) \phi(\tau) \dd \tau\nonumber\\
&\quad + \int_0^T \sum_{k=1}^N A_{k,N}(\tau) \left(\alpha \sum_{j=1}^{J+1}  \int_{\Omega^{J+1}} \int_{\R^{(J+1)d}} M(v)\,\pdr \psi_k(r,v) \cdot \pdr \psi_\ell(r,v) \dd v \dd r\right) \phi(\tau) \dd \tau \nonumber\\
&\quad+ \int_0^T\sum_{k=1}^N A_{k,N}(\tau) \left(-\frac{1}{\eps}\sum_{j=1}^{J+1}  \int_{\Omega^{J+1}} \int_{\R^{(J+1)d}} M(v)\,(({\mathcal L}r)_j+u(r_j,\tau))\,\psi_k(r,v) \cdot \pdv \psi_\ell(r,v) \dd v \dd r\right) \phi(\tau) \dd \tau\nonumber\\
&\qquad  = \left(\int_{\Omega^{J+1}} \int_{\R^{(J+1)d}} M(v)\,\hrho_{0}(r,v)\,\psi_\ell(r,v)\dd v \dd r\right) \phi(0)\qquad \forall\, \ell \in \{1,\dots,N\} \mbox{ and $\forall\,\phi \in W^{1,2}(0,T)$}.
\end{align*}
}

Denoting the sum of the terms in the brackets in the second, third and fourth line by $G_{\ell,k}$,
and the term in the outer pair of brackets in the fifth line by $H_{\ell,k}(\tau)$, we have that
\begin{align}\label{eq:weak-a-gal2}
A_{\ell,N}(T)\,\phi(T) - \int_0^T A_{\ell,N}(\tau)\, \pd_\tau \phi(\tau) \dd \tau &+ \int_0^T \sum_{k=1}^N (G_{\ell,k} + H_{\ell,k}(\tau))\,A_{k,N}(\tau)\,  \phi(\tau) \dd \tau\nonumber\\
&= \int_{\Omega^{J+1}} \int_{\R^{(J+1)d}} M(v)\,\hrho_{0}(r,v)\,\psi_\ell(r,v)\dd v \dd r \, \phi(0). 
\end{align}
As it will transpire from the discussion that follows, $|G_{\ell,k}|<\infty$ and $|H_{\ell,k}(\tau)|<\infty$ for a.e.
$\tau \in (0,T]$, and for all $\ell, k = 1,\dots, N$.

The above is the weak form of the following initial-value problem for a system of linear ODEs:
\begin{equation}\label{eq:odesystem}
\begin{aligned}
\frac{\dd}{\dd t} A_{\ell,N}(t) +  \sum_{k=1}^N (G_{\ell,k} + H_{\ell,k}(t))\,A_{k,N}(t) = 0, \qquad t \in (0,T],\\
A_{\ell,N}(0) = \int_{\Omega^{J+1}} \int_{\R^{(J+1)d}} M(v)\,\hrho_{0}(r,v)\,\psi_\ell(r,v)\dd v \dd r, \quad \ell=1,\dots,N.
\end{aligned}
\end{equation}
As $(G_{\ell,k})_{\ell,k=1}^N$ is a constant matrix, the existence of a solution to this system of linear ODEs will follow from Carath\'eodory's theorem once we have shown that $t \in (0,T) \mapsto H_{\ell,k}(t) \in \mathbb{R}$ is measurable and
a (matrix) norm of the matrix $(H_{\ell,k}(t))_{\ell,k=1}^N$ is dominated by $h(t)$, where $h \in L^1(0,T)$. As a matter of fact, once this has been shown, the uniqueness of the solution to this system of ODEs will also follow, by Gronwall's lemma,
thanks to the linearity of the system.

To this end, it suffices to note that, since by hypothesis $u \in L^2(0,T;W^{1,\sigma}_0({\Omega})^d)$ for some $\sigma>d$, Fubini's theorem implies that all entries of the matrix $(H_{\ell,k}(t))_{\ell,k=1}^N$ are measurable functions of $t \in (0,T]$; furthermore,
there exists a positive constant $C_0=C_0(J,N)$ such that
%
\begin{align*}
&\int_0^T \left|\sum_{j=1}^{J+1} \int_{\Omega^{J+1}} \int_{\R^{(J+1)d}} M(v)\, u(r_j,\tau)\,\psi_k(r,v) \cdot \pdv \psi_\ell(r,v) \dd v \dd r\right| \dd \tau\\
& \leq \sum_{j=1}^{J+1}\int_0^T \left(\int_{\Omega^{J+1}} \int_{\R^{(J+1)d}} M(v)\, |u(r_j,\tau)\,\psi_k(r,v)|^2 \dd v \dd r \right)^{\frac{1}{2}} \left(\int_{\Omega^{J+1}} \int_{\R^{(J+1)d}} M(v)\, |\pdv \psi_\ell(r,v)|^2 \dd v \dd r \right)^{\frac{1}{2}} \dd \tau\\
& \leq \|u\|_{L^1(0,T;L^\infty(\Omega))} \sum_{j=1}^{J+1} \max_{1 \leq \ell \leq N} \left(\int_{\Omega^{J+1}} \int_{\R^{(J+1)d}} M(v)\, |\pdv \psi_\ell(r,v)|^2 \dd v \dd r \right)^{\frac{1}{2}}\\
& =   C_0 \|u\|_{L^1(0,T;L^\infty(\Omega))}.
\end{align*}
%


\noindent
This then implies the existence of a measurable function $h \in L^1(0,T)$ such that the (matrix) norm of the matrix
$(H_{\ell,k}(t))_{\ell,k=1}^N$ is dominated by $h(t)$, where $h \in L^1(0,T)$; take, for example,
$h(t):=\frac{C}{\eps}(1+\|u(t)\|_{L^{\infty}({\Omega})})$, where $C$ is a sufficiently large constant.
Hence,
Carath\'eodory's theorem implies the existence of a solution $A_{\ell,N} \in W^{1,1}(0,T)$ (and, consequently, absolutely continuous on $[0,T]$), $\ell=1,\dots, N$, to \eqref{eq:weak-a-gal2}, and by Gronwall's lemma the solution to \eqref{eq:weak-a-gal2} is unique. In fact, since $H_{\ell,k} \in L^\infty(0,T)$, $\ell,k = 1,\dots, N$, it follows that $A_{\ell,N} \in W^{1,\infty}(0,T)$, $\ell=1,\dots, N$; cf. \eqref{eq:odesystem}. Thus, by noting \eqref{eq:expand}, we deduce that the finite-dimensional problem \eqref{eq:weak-a-gal} has a unique solution
\[ \hrho_{\alpha,N} \in W^{1,\infty}(0,T;W^{1,2}_{*,M}(\Omega^{J+1} \times \R^{(J+1)d})).\]

Next, for any $t \in (0,T)$ fixed, and $h \in (0,T-t)$, consider the function
\[
 \chi_{t,h}(\tau) := \min\left\{1, \left(\frac{1}{h}(t-\tau) + 1\right)_{\!+}\right\}, \qquad \tau \in
 [0,T].
\]
Clearly, $\tau \mapsto \chi_{t,h}(\tau)$ is a continuous piecewise linear function defined on $[0,T]$, which is identically $1$ on $[0,t]$, identically $0$ on $[t+h,T]$, and has slope $-1/h$ on $[t,t+h]$. Taking $\phi =
\chi_{t,h}\, A_{\ell,N} \in W^{1,\infty}(0,T)$ in \eqref{eq:weak-a-gal} with $t \in (0,T)$ fixed and passing to the limit $h \rightarrow 0_+$, we have that
{\small
\begin{align}\label{eq:weak-a-gal2a}
&\int_{\Omega^{J+1}} \int_{\R^{(J+1)d}} M(v)\,\hrho_{\alpha,N}(r,v,t)\, \psi_\ell(r,v)\,A_{\ell,N}(t)\dd v \dd r\nonumber\\
&\qquad
- \int_0^t \int_{\Omega^{J+1}} \int_{\R^{(J+1)d}} M(v)\,\hrho_{\alpha,N}(r,v,\tau)\, \psi_\ell(r,v)\, \pd_\tau A_{\ell,N}(\tau)\dd v \dd r \dd \tau\nonumber\\
&\qquad+ \frac{\beta^2}{\eps^2}\left(\sum_{j=1}^{J+1} \int_0^t \int_{\Omega^{J+1}} \int_{\R^{(J+1)d}} M(v)\,\pdv \hrho_{\alpha,N} \cdot \pdv \psi_\ell(r,v)\, A_{\ell,N}(\tau) \dd v \dd r \dd \tau \right)\nonumber\\
&\qquad- \frac{1}{\eps} \left(\sum_{j=1}^{J+1} \int_0^t \int_{\Omega^{J+1}} \int_{\R^{(J+1)d}} M(v)\, v_j \hrho_{\alpha,N}\cdot \pdr \psi_\ell(r,v)\, A_{\ell,N}(\tau) \dd v \dd r \dd \tau \right)\nonumber\\
&\qquad + \alpha \sum_{j=1}^{J+1} \int_0^t \int_{\Omega^{J+1}} \int_{\R^{(J+1)d}} M(v)\,\pdr \hrho_{\alpha,N} \cdot \pdr \psi_\ell(r,v)\, A_{\ell,N}(\tau)\dd v \dd r \dd \tau \nonumber\\
&\qquad- \frac{1}{\eps} \left(\sum_{j=1}^{J+1} \int_0^t \int_{\Omega^{J+1}} \int_{\R^{(J+1)d}} M(v)\,(({\mathcal L}r)_j+u(r_j,\tau))\,\hrho_{\alpha,N}\cdot \pdv \psi_\ell(r,v)\, A_{\ell,N}(\tau) \dd v \dd r \dd \tau \right)\nonumber\\
&\qquad\qquad  = \left(\int_{\Omega^{J+1}} \int_{\R^{(J+1)d}} M(v)\,\hrho_{0}(r,v)\,\psi_\ell(r,v)\dd v \dd r\right) A_{\ell,N}(0) \quad \forall\, \ell \in \{1,\dots,N\}.
\end{align}
}

\noindent
Summing \eqref{eq:weak-a-gal2a} through $\ell=1,\dots,N$ and recalling \eqref{eq:expand} then yields
{\small
\begin{align}\label{eq:weak-a-gal3}
&\frac{1}{2}\int_{\Omega^{J+1}} \int_{\R^{(J+1)d}} M(v)\,\hrho^2_{\alpha,N}(r,v,t)\dd v \dd r\nonumber\\
&\qquad+ \frac{\beta^2}{\eps^2}\left(\sum_{j=1}^{J+1} \int_0^t \int_{\Omega^{J+1}} \int_{\R^{(J+1)d}} M(v)\,|\pdv \hrho_{\alpha,N}|^2  \dd v \dd r \dd \tau \right)\nonumber\\
&\qquad- \frac{1}{\eps} \left(\sum_{j=1}^{J+1} \int_0^t \int_{\Omega^{J+1}} \int_{\R^{(J+1)d}} M(v)\, v_j \hrho_{\alpha,N}\cdot \pdr \hrho_{\alpha,N} \dd v \dd r \dd \tau \right)\nonumber\\
&\qquad + \alpha \sum_{j=1}^{J+1} \int_0^t \int_{\Omega^{J+1}} \int_{\R^{(J+1)d}} M(v)\,|\pdr \hrho_{\alpha,N}|^2 \dd v \dd r \dd \tau \nonumber\\
&\qquad- \frac{1}{\eps} \left(\sum_{j=1}^{J+1} \int_0^t \int_{\Omega^{J+1}} \int_{\R^{(J+1)d}} M(v)\,(({\mathcal L}r)_j+u(r_j,\tau))\,\hrho_{\alpha,N}\cdot \pdv \hrho_{\alpha,N} \dd v \dd r \dd \tau \right)\nonumber\\
&\qquad\qquad  = \frac{1}{2}\left(\int_{\Omega^{J+1}} \int_{\R^{(J+1)d}} M(v)\,|\hrho_{0}(r,v)|^2 \dd v \dd r\right)\qquad \forall\, t \in (0,T).
\end{align}
}

Let us denote by ${\rm T}_1$ and ${\rm T}_2$ the terms in the third and fifth line of \eqref{eq:weak-a-gal3}, respectively; our objective is to bound these by quantities that can be absorbed into the remaining terms
on the left-hand side. That will then result in uniform-in-$N$ bounds on various norms of $\hrho_{\alpha,N}$,
which will allow us to pass to the limit $N \rightarrow \infty$ in the Galerkin approximation.

We shall show below that $M(v)\, v_j \hrho_{\alpha,N}\cdot \pdr \hrho_{\alpha,N} \in L^1(\Omega^{J+1} \times \R^{(J+1)d} \times (0,T))$. Taking this for granted for the moment, we have that
\begin{align*}
{\rm T}_1 &:= - \frac{1}{\eps} \sum_{j=1}^{J+1} \int_0^t \int_{\Omega^{J+1}}\int_{\R^{(J+1)d}} M(v)\, v_j\,   \hrho_{\alpha,N}\cdot \pdr \hrho_{\alpha,N} \dd v \dd r \dd \tau\\
&= - \frac{1}{2\eps} \sum_{j=1}^{J+1} \int_0^t \int_{\Omega^{J+1}} \int_{\R^{(J+1)d}} M(v)\, v_j\cdot \pdr (|\hrho_{\alpha,N}|^2) \dd v \dd r \dd \tau\\
&= - \frac{1}{2\eps} \sum_{j=1}^{J+1} \int_0^t \int_{\partial\Omega^{(j)}} \int_{\R^{(J+1)d}} M(v)\, (v_j\cdot \nu(r_j))\,  |\hrho_{\alpha,N}|^2 \dd v \dd s(r) \dd \tau  = 0,
\end{align*}
because $\hrho_{\alpha,N} \in W^{1,\infty}(0,T;W^{1,2}_{*,M}(\Omega^{J+1} \times \R^{(J+1)d}))$.
It therefore remains to show that $M(v)\, v_j \hrho_{\alpha,N}\cdot \pdr \hrho_{\alpha,N}$ belongs to  $L^1(\Omega^{J+1} \times \R^{(J+1)d}\times (0,T))$.
Since the function $\sqrt{M(v)}\,\pdr \hrho_{\alpha,N} \in L^2(\Omega^{J+1} \times \R^{(J+1)d}\times (0,T))$, it suffices to show
that $\sqrt{M(v)}\, v_j\, \hrho_{\alpha,N}$ belongs to $L^2(\Omega^{J+1} \times \R^{(J+1)d}\times (0,T))$.

To this end, we first recall the logarithmic Young's inequality
\[ ab \leq {\rm e}^a + b(\log b - 1) \qquad \forall\, a, b \in \mathbb{R}_{\geq 0}.\]
This follows from the following Fenchel--Young inequality:
\[ ab \leq g(a) + g^*(b) \qquad \forall\, a,b \in \mathbb{R}_{\geq 0}, \]
involving the convex function $g\,:\,a \in \R \mapsto g(a) \in (-\infty,+\infty]$ and its convex conjugate $g^*$, defined by $g^*(b):=\sup_{a \in \mathbb{R}}(ab-g(a))$, with $g(a) = {\rm e}^a$ and
\[ g^*(b)= \left\{\begin{array}{cl}
+ \infty & \mbox{if $b<0$};\\
0        & \mbox{if $b=0$};\\
b(\log b-1) & \mbox{if $b>0$},
                 \end{array}    \right.
\]
with the resulting inequality then restricted to $\R_{\geq 0}$. Consequently, recalling that $\mathcal{F}(s)=s(\log s - 1) +1$ for
$s>0$ and $\mathcal{F}(0):=0$, we have that
\begin{equation}\label{eq:ab}
ab \leq  {\rm e}^a - 1 + \mathcal{F}(b)\qquad \forall\, a, b \in \R_{\geq 0}.
\end{equation}

Hence, with $a=\frac{1}{4\beta}\, |v_j|^2$ and $b = \|\hrho_{\alpha,N}\|^2_{L^2(\Omega^{J+1}\times(0,T))}$, we have that
\[ \frac{1}{4\beta}\, |v_j|^2\, \|\hrho_{\alpha,N}\|^2_{L^2(\Omega^{J+1}\times(0,T))} \leq \mathcal{F}\left(\|\hrho_{\alpha,N}\|^2_{L^2(\Omega^{J+1}\times(0,T))}\right) +
{\rm e}^{ \frac{1}{4\beta} |v_j|^2}-1,\]
and therefore, upon multiplication by $M(v)$ and omitting the final, negative term from the right-hand side,
\begin{align*}
&\frac{1}{4\beta}\, M(v)\, |v_j|^2 \, \|\hrho_{\alpha,N}\|^2_{L^2(\Omega^{J+1}\times(0,T))} \leq M(v)\,\mathcal{F}\left(\|\hrho_{\alpha,N}\|^2_{L^2(\Omega^{J+1}\times(0,T))}\right) + (2\pi\beta)^{-\frac{(J+1)d}{2}}\,{\rm e}^{ -\frac{1}{4\beta} |v_j|^2} \prod_{{\stackrel{k=1}{k \neq j}}}^{J+1} {\rm e}^{-\frac{1}{2\beta}|v_k|^2}\nonumber\\
& = M(v)\, \left[\|\hrho_{\alpha,N}\|^2_{L^2(\Omega^{J+1}\times(0,T))} (\log \|\hrho_{\alpha,N}\|^2_{L^2(\Omega^{J+1}\times(0,T))} - 1) + 1\right] + (2\pi\beta)^{-\frac{(J+1)d}{2}}\,{\rm e}^{ -\frac{1}{4\beta} |v_j|^2} \prod_{{\stackrel{k=1}{k \neq j}}}^{J+1} {\rm e}^{-\frac{1}{2\beta}|v_k|^2}\nonumber\\
& \leq M(v)\, \|\hrho_{\alpha,N}\|^2_{L^2(\Omega^{J+1}\times(0,T))} \, \log \|\hrho_{\alpha,N}\|^2_{L^2(\Omega^{J+1}\times(0,T))}  + \bigg[M(v) + (2\pi\beta)^{-\frac{(J+1)d}{2}}\,{\rm e}^{ -\frac{1}{4\beta} |v_j|^2} \prod_{{\stackrel{k=1}{k \neq j}}}^{J+1} {\rm e}^{-\frac{1}{2\beta}|v_k|^2}\bigg].
\end{align*}
Integrating this over $\R^{(J+1)d}$ and applying Gross' logarithmic Sobolev inequality  to the first term on the right-hand side yields (c.f. \cite{Gross1975}, particularly (1.2) there multiplied by $2$, and (1.1) with $n=(J+1)d$):
\begin{align}\label{eq:log-estimate}
&\frac{1}{4\beta}\,\int_{\R^{(J+1)d}} M(v)\, |v_j|^2 \, \|\hrho_{\alpha,N}\|^2_{L^2(\Omega^{J+1}\times(0,T))} \dd v\nonumber\\
&\quad\leq \int_{\R^{(J+1)d}} M(v)\, \|\hrho_{\alpha,N}\|^2_{L^2(\Omega^{J+1}\times(0,T))} \, \log \|\hrho_{\alpha,N}\|^2_{L^2(\Omega^{J+1}\times(0,T))} \dd v \nonumber\\
&\qquad + \int_{\R^{(J+1)d}}\bigg[M(v) + (2\pi\beta)^{-\frac{(J+1)d}{2}}\,{\rm e}^{ -\frac{1}{4\beta} |v_j|^2} \prod_{{\stackrel{k=1}{k \neq j}}}^{J+1} {\rm e}^{-\frac{1}{2\beta}|v_k|^2}\bigg] \dd v
\nonumber
\\
&\quad\leq 2\sum_{j=1}^{J+1} \int_{\R^{(J+1)d}} M(v)\, |\pdv \|\hrho_{\alpha,N}\|_{L^2(\Omega^{J+1}\times(0,T))}|^2 \dd v\nonumber\\
&\qquad + \|\hrho_{\alpha,N}\|^2_{L^2_M(\R^{(J+1)d}; L^2(\Omega^{J+1}\times (0,T)))} \log \|\hrho_{\alpha,N}\|^2_{L^2_M(\R^{(J+1)d}; L^2(\Omega^{J+1}\times (0,T)))}\nonumber\\
&\qquad + \int_{\R^{(J+1)d}}\bigg[M(v) + (2\pi\beta)^{-\frac{(J+1)d}{2}}\,{\rm e}^{ -\frac{1}{4\beta} |v_j|^2} \prod_{{\stackrel{k=1}{k \neq j}}}^{J+1} {\rm e}^{-\frac{1}{2\beta}|v_k|^2}\bigg] \dd v
\nonumber
\\
&\quad\leq 2\sum_{j=1}^{J+1} \int_{\R^{(J+1)d}} M(v) \left(\int_{\Omega^{J+1}\times(0,T)} |\pdv \hrho_{\alpha,N}|^2 \dd r \dd \tau\right) \dd v \nonumber\\
&\qquad + \|\hrho_{\alpha,N}\|^2_{L^2(0,T; L^2_M(\Omega^{J+1}\times \R^{(J+1)d}))} \log \|\hrho_{\alpha,N}\|^2_{L^2(0,T; L^2_M(\Omega^{J+1}\times \R^{(J+1)d}))}\nonumber\\
&\qquad + \int_{\R^{(J+1)d}}\bigg[M(v) + (2\pi\beta)^{-\frac{(J+1)d}{2}}\,{\rm e}^{ -\frac{1}{4\beta} |v_j|^2} \prod_{{\stackrel{k=1}{k \neq j}}}^{J+1} {\rm e}^{-\frac{1}{2\beta}|v_k|^2}\bigg] \dd v.
\end{align}
The term in the square brackets on the right-hand side is trivially in $L^1(\R^{(J+1)d})$.
Furthermore, both $\sqrt{M(v)}\, \hrho_{\alpha,N}$ and $\sqrt{M(v)}\,\pdv \hrho_{\alpha,N}$ belong to
$L^2(\Omega^{J+1} \times \R^{(J+1)d}\times (0,T))$ for all $j=1,\dots, J+1$.
Thus we have shown that $M(v)\, |v_j|^2\, |\hrho_{\alpha,N}|^2  \in L^1(\Omega^{J+1} \times \R^{(J+1)d}\times (0,T))$; hence, $\sqrt{M(v)}\, v_j \, \hrho_{\alpha,N}$ belongs to $L^2(\Omega^{J+1} \times \R^{(J+1)d} \times (0,T))$, as required. This completes the proof of the assertion that ${\rm T}_1 = 0$.

Let us now turn our attention to the term
\[ {\rm T}_2:= - \frac{1}{\eps} \left(\sum_{j=1}^{J+1} \int_0^t \int_{\Omega^{J+1}} \int_{\R^{(J+1)d}} M(v)\,(({\mathcal L}r)_j+u(r_j,\tau))\,\hrho_{\alpha,N}\cdot \pdv \hrho_{\alpha,N} \dd v \dd r \dd \tau \right).\]
We have, by the Cauchy--Schwarz inequality, the triangle inequality, and noting that
$|({\mathcal L}r)_j| \leq 4\sqrt{d}\,L$, that
{\small
\begin{align*}
&{\rm T}_2 \leq \frac{1}{\eps} \left(\sum_{j=1}^{J+1} \int_0^t  \|\sqrt{M}\, (|({\mathcal L}r)_j|+|u(r_j,\tau)|)\,\hrho_{\alpha,N}\|_{L^2(\Omega^{J+1}\times \R^{(J+1)d})}\, \|\sqrt{M}\, \pdv \hrho_{\alpha,N}\|_{L^2(\Omega^{J+1}\times \R^{(J+1)d})} \dd \tau \right)\\
&\leq \frac{1}{\eps} \left(\sum_{j=1}^{J+1} \int_0^t  \|\sqrt{M}\, (|({\mathcal L}r)_j|+|u(r_j,\tau)|)\,\hrho_{\alpha,N}\|^2_{L^2(\Omega^{J+1}\times \R^{(J+1)d})}
\dd \tau\right)^{\frac{1}{2}} \\
&\hspace{2in} \times
\left(\sum_{j=1}^{J+1} \int_0^t \|\sqrt{M}\, \pdv \hrho_{\alpha,N}\|^2_{L^2(\Omega^{J+1}\times \R^{(J+1)d})} \dd \tau \right)^{\frac{1}{2}}\\
&\leq \frac{1}{\eps} \left[\left(\sum_{j=1}^{J+1} \int_0^t  \|\sqrt{M} \,|({\mathcal L}r)_j|\,\hrho_{\alpha,N}\|^2_{L^2(\Omega^{J+1}\times \R^{(J+1)d})}
\dd \tau\right)^{\frac{1}{2}} + \left(\sum_{j=1}^{J+1} \int_0^t  \|\sqrt{M} \,|u(r_j,\tau)|\,\hrho_{\alpha,N}\|^2_{L^2(\Omega^{J+1}\times \R^{(J+1)d})}\dd \tau
\right)^{\frac{1}{2}}\right]
\\
&\hspace{2in} \times
\left(\sum_{j=1}^{J+1} \int_0^t \|\sqrt{M}\, \pdv \hrho_{\alpha,N}\|^2_{L^2(\Omega^{J+1}\times \R^{(J+1)d})} \dd \tau \right)^{\frac{1}{2}}\\
&\leq \frac{1}{\eps} \left[4\sqrt{(J+1)d}\,L \left(\int_0^t  \|\sqrt{M} \,\hrho_{\alpha,N}\|^2_{L^2(\Omega^{J+1}\times \R^{(J+1)d})}
\dd \tau\!\right)^{\!\frac{1}{2}}
+ \left(\sum_{j=1}^{J+1} \int_0^t  \|\sqrt{M} \,|u(r_j,\tau)|\,\hrho_{\alpha,N}\|^2_{L^2(\Omega^{J+1}\times \R^{(J+1)d})} \dd \tau
\!\right)^{\!\frac{1}{2}}\right] \\
&\hspace{2in} \times
\left(\sum_{j=1}^{J+1} \int_0^t \|\sqrt{M}\, \pdv \hrho_{\alpha,N}\|^2_{L^2(\Omega^{J+1}\times \R^{(J+1)d})} \dd \tau \right)^{\frac{1}{2}}.
\end{align*}
}

\noindent
We shall focus our attention on the second term in the square brackets on the right-hand side:
\begin{align*}
& \int_0^t \|\sqrt{M} \,|u(r_j,\tau)|\,\hrho_{\alpha,N}\|^2_{L^2(\Omega^{J+1}\times \R^{(J+1)d})}\dd \tau
= \int_0^t \int_{\Omega^{J+1}\times \R^{(J+1)d}} M(v)\, |u(r_j,\tau)|^2\,\hrho_{\alpha,N}^2(r,v,\tau) \dd r \dd v \dd \tau\\
&\quad = \int_0^t \int_{\Omega^{J+1}} |u(r_j,\tau)|^2 \left(\int_{\R^{(J+1)d}} M(v)\,\hrho_{\alpha,N}^2(r,v,\tau) \dd v \right) \dd r \dd \tau\\
&\quad \leq \int_0^t \|u(\cdot,\tau)\|^2_{L^\infty({\Omega})} \left(\int_{\Omega^{J+1}} \int_{\R^{(J+1)d}} M(v)\,\hrho_{\alpha,N}^2(r,v,\tau) \dd v \right) \dd r \dd \tau.
\end{align*}

\noindent
Thus, we have the following bound:
\begin{align}\label{eq:u-bound}
&\left(\sum_{j=1}^{J+1} \int_0^t  \|\sqrt{M} \,|u(r_j,\tau)|\,\hrho_{\alpha,N}\|^2_{L^2(\Omega^{J+1}\times \R^{(J+1)d})}
\dd \tau\right)^{\frac{1}{2}}\nonumber\\
&\quad \leq \sqrt{J+1}\, \left(\int_0^t \|u(\cdot,\tau)\|^2_{L^\infty({\Omega})}\|\sqrt{M}\,\hrho_{\alpha,N}(\cdot,\cdot,\tau)\|^2_{L^2(\Omega^{J+1}\times \R^{(J+1)d})}\dd \tau\right)^{\frac{1}{2}}.
\end{align}
Consequently,
\begin{align*}
{\rm T}_2 &\leq \frac{1}{\eps} C(L,J)\, \left(\int_0^t (1+ \|u(\cdot,\tau)\|^2_{L^\infty({\Omega})})\,\|\sqrt{M}\, \hrho_{\alpha,N}\|^2_{L^2(\Omega^{J+1}\times
\R^{(J+1)d})} \dd \tau\right)^{\frac{1}{2}}\\
&\qquad \times  \left(\sum_{j=1}^{J+1}\int_0^t \|\sqrt{M}\, \pdv \hrho_{\alpha,N}\|^2_{L^2(\Omega^{J+1}
\times \R^{(J+1)d})} \dd \tau\right)^{\frac{1}{2}}.
\end{align*}
Returning with this bound to \eqref{eq:weak-a-gal3}, we have that
\begin{align}\label{eq:weak-a-gal4}
&\frac{1}{2}\|\sqrt{M}\, \hrho_{\alpha,N}(\cdot,\cdot,t)\|^2_{L^2(\Omega^{J+1}\times \R^{(J+1)d})}\nonumber\\
&\quad+ \frac{\beta^2}{2\eps^2} \sum_{j=1}^{J+1} \int_0^t \|\sqrt{M}\, \pdv \hrho_{\alpha,N}\|^2_{L^2(\Omega^{J+1}\times \R^{(J+1)d})} \dd \tau +
\alpha  \sum_{j=1}^{J+1} \int_0^t \|\sqrt{M}\, \pdr \hrho_{\alpha,N}\|^2_{L^2(\Omega^{J+1}\times \R^{(J+1)d})} \dd \tau \nonumber\\
&\quad\quad  \leq \frac{1}{2} \|\sqrt{M}\, \hrho_{0}\|^2_{L^2(\Omega^{J+1}\times \R^{(J+1)d})}\nonumber\\
&\quad\qquad + \frac{1}{2\beta^2}\,C(L,J)^2\,\left(\int_0^t (1+ \|u(\cdot,\tau)\|^2_{L^\infty({\Omega})}) \|\sqrt{M}\, \hrho_{\alpha,N}\|^2_{L^2(\Omega^{J+1}\times \R^{(J+1)d})} \dd \tau\right)
\qquad \forall\, t \in (0,T].
\end{align}

\noindent
Hence, by Gronwall's lemma,
\begin{align}\label{eq:energy-aN}
&\|\sqrt{M}\, \hrho_{\alpha,N}(\cdot,\cdot,t)\|^2_{L^2(\Omega^{J+1}\times \R^{(J+1)d})}\nonumber\\
&\quad+ \frac{\beta^2}{\eps^2} \sum_{j=1}^{J+1} \int_0^t \|\sqrt{M}\, \pdv \hrho_{\alpha,N}\|^2_{L^2(\Omega^{J+1}\times \R^{(J+1)d})} \dd \tau
+ 2\alpha \sum_{j=1}^{J+1} \int_0^t \|\sqrt{M}\, \pdr \hrho_{\alpha,N}\|^2_{L^2(\Omega^{J+1}\times \R^{(J+1)d})} \dd \tau \nonumber\\
&\quad\quad  \leq \|\sqrt{M}\, \hrho_{0}\|^2_{L^2(\Omega^{J+1}\times \R^{(J+1)d})} \,
\mbox{exp}\left(\frac{1}{\beta^2}\,C(L,J)^2\,(T+ \|u\|^2_{L^2(0,T;L^{\infty}({\Omega}))})  \right)\qquad \forall\, t \in (0,T].
\end{align}
Thus, for $\alpha \in (0,1]$ fixed, we deduce the following uniform bounds with respect to $N$:
\begin{equation}
\begin{aligned}\label{eq:uniform-space}
\|\hrho_{\alpha,N}\|_{L^\infty(0,T;L^2_M(\Omega^{J+1}\times \R^{(J+1)d}))} &\leq
C(L,J,T,\eps,\|u\|_{L^2(0,T;L^{\infty}({\Omega}))})\,\|\hrho_{0}\|_{L^2_M(\Omega^{J+1}\times \R^{(J+1)d})},\\
\|\partial_{v_j}\hrho_{\alpha,N}\|_{L^2(0,T;L^2_M(\Omega^{J+1}\times \R^{(J+1)d}))} &\leq C(L,J,T,\eps,\|u\|_{L^2(0,T;L^{\infty}({\Omega}))})\,\|\hrho_{0}\|_{L^2_M(\Omega^{J+1}\times \R^{(J+1)d})},\\
\sqrt{\alpha}\,\|\partial_{r_j}\hrho_{\alpha,N}\|_{L^2(0,T;L^2_M(\Omega^{J+1}\times \R^{(J+1)d}))} &\leq C(L,J,T,\eps,\|u\|_{L^2(0,T;L^{\infty}({\Omega}))})\,\|\hrho_{0}\|_{L^2_M(\Omega^{J+1}\times \R^{(J+1)d})},
\end{aligned}
\end{equation}
for all $j \in \{1,\dots, J+1\}$. Furthermore, by \eqref{eq:log-estimate},
\begin{equation}\label{eq:uniform-space-1}
\||v_j|\,\hrho_{\alpha,N}\|_{L^2(0,T;L^2_M(\Omega^{J+1}\times \R^{(J+1)d}))} \leq
C(L,J,T,\eps,\|u\|_{L^2(0,T;L^{\infty}({\Omega}))}, \|\hrho_{0}\|_{L^2_M(\Omega^{J+1}\times \R^{(J+1)d})}),
\end{equation}
for all $j \in \{1,\dots, J+1\}$; as $\beta>0$ is considered to be fixed throughout, the dependence of the constants on $\beta$ has not been (and will not be) indicated.

\smallskip

Next, we shall exploit the bounds stated in \eqref{eq:uniform-space} and \eqref{eq:uniform-space-1} to derive a uniform-in-$N$ bound on $\partial_t\hrho_{\alpha,N}$ in the function space $L^2(0,T;(W^{1,2}_{*, M}(\Omega^{J+1}\times \R^{(J+1)d}))')$. Let us first note that
\begin{align*}
\|\partial_t \hrho_{\alpha,N}(t)\|_{(W^{1,2}_{*,M}(\Omega^{J+1} \times \R^{(J+1)d}))'} &= \sup_{w \in W^{1,2}_{*,M}(\Omega^{J+1} \times \R^{(J+1)d}),\; \|w\|_{W^{1,2}_M(\Omega^{J+1} \times \R^{(J+1)d})}\leq 1} (M \partial_t \hrho_{\alpha,N}(t) , w)\\
&=\sup_{w \in W^{1,2}_{*,M}(\Omega^{J+1} \times \R^{(J+1)d}),\; \|w\|_{W^{1,2}_M(\Omega^{J+1} \times \R^{(J+1)d})}\leq 1} (M \partial_t \hrho_{\alpha,N}(t) , P_N w),
\end{align*}
where $(\cdot,\cdot)$ denotes the inner product of $L^2(\Omega^{J+1} \times \R^{(J+1)d})$. By
reversing the partial integration with respect to $\tau$ in  \eqref{eq:weak-a-gal2}, we deduce,
for all $t \in (0,T]$, that
%
\begin{align}\label{eq:galerkin-1}
\int_0^t (M \,\pd_\tau \hrho_{\alpha,N}(\cdot,\cdot,\tau), \psi_\ell(\cdot,\cdot))\, \phi(\tau) \dd \tau + \frac{\beta^2}{\eps^2}\sum_{j=1}^{J+1} \int_0^t (M\,\pdv\hrho_{\alpha,N}(\cdot,\cdot,\tau), \pdv \psi_\ell(\cdot,\cdot))\, \phi(\tau)  \dd \tau\qquad\qquad\nonumber\\
- \frac{1}{\eps} \sum_{j=1}^{J+1} \int_0^t (M\, v_j \hrho_{\alpha,N}(\cdot,\cdot,\tau), \pdr \psi_\ell(\cdot,\cdot))\, \phi(\tau)\dd \tau + \alpha \sum_{j=1}^{J+1} \int_0^t (M\,\pdr \hrho_{\alpha,N}(\cdot,\cdot,\tau), \pdr \psi_\ell(\cdot,\cdot))\, \phi(\tau) \dd \tau \nonumber\\
- \frac{1}{\eps} \sum_{j=1}^{J+1} \int_0^t (M\,(({\mathcal L}r)_j+u(r_j,\tau))\,\hrho_{\alpha,N}(\cdot,\cdot,\tau), \pdv \psi_\ell(\cdot,\cdot))\, \phi(\tau) \dd \tau = 0
\nonumber\\
\forall\, \ell \in \{1,\dots,N\} \mbox{ and $\forall\,\phi \in W^{1,2}(0,T)$.}
\end{align}
%

\noindent
Hence, thanks to the density of $W^{1,2}(0,T)$ in $L^p(0,T)$ for all $p \in [1,\infty)$, and recalling the fundamental lemma of the calculus of variations (du Bois-Reymond's lemma), we have that
\begin{align*}
&(M \,\pd_t \hrho_{\alpha,N}(\cdot,\cdot, t), \psi_\ell(\cdot,\cdot)) + \frac{\beta^2}{\eps^2}\sum_{j=1}^{J+1}(M\,\pdv\hrho_{\alpha,N}(\cdot,\cdot,t), \pdv \psi_\ell(\cdot,\cdot))\, \nonumber\\
&\qquad- \frac{1}{\eps} \sum_{j=1}^{J+1} (M\, v_j \hrho_{\alpha,N}(\cdot,\cdot,t), \pdr \psi_\ell(\cdot,\cdot)) + \alpha \sum_{j=1}^{J+1} (M\,\pdr \hrho_{\alpha,N}(\cdot,\cdot,t), \pdr \psi_\ell(\cdot,\cdot)) \nonumber\\
&\qquad- \frac{1}{\eps} \sum_{j=1}^{J+1} (M\,(({\mathcal L}r)_j+u(r_j,t))\,\hrho_{\alpha,N}(\cdot,\cdot,t), \pdv \psi_\ell(\cdot,\cdot)) \nonumber = 0
\qquad \forall\, \ell \in \{1,\dots,N\} \mbox{ and a.e. $t \in (0,T]$.}
\end{align*}
This then implies that
\begin{align*}
&(M \,\pd_t \hrho_{\alpha,N}(t), P_N w) = -\frac{\beta^2}{\eps^2}\sum_{j=1}^{J+1}(M\,\pdv\hrho_{\alpha,N}(t), \pdv P_N w) \nonumber\\
&\qquad+ \frac{1}{\eps} \sum_{j=1}^{J+1} (M\, v_j \hrho_{\alpha,N}(t), \pdr P_N w) - \alpha \sum_{j=1}^{J+1} (M\,\pdr \hrho_{\alpha,N}(t), \pdr P_N w) \nonumber\\
&\qquad+ \frac{1}{\eps} \sum_{j=1}^{J+1} (M\,(({\mathcal L}r)_j+u(r_j,t))\,\hrho_{\alpha,N}(t), \pdv P_N w)\nonumber\\
&\quad =: {\rm S}_1(t) + {\rm S}_2(t) + {\rm S}_3(t) + {\rm S}_4(t)
\qquad \forall\, \ell \in \{1,\dots,N\} \mbox{ and a.e. $t \in (0,T]$.}
\end{align*}

The terms ${\rm S}_1(t)$ and ${\rm S}_3(t)$ are easy to bound: for a.e. $t \in (0,T]$,
\[ |{\rm S}_1(t)| \leq \frac{\beta^2}{\eps^2}\left(\sum_{j=1}^{J+1}\|\sqrt{M}\,\pdv\hrho_{\alpha,N}(t)\|^2_{L^2(\Omega^{J+1}\times \R^{(J+1)d})} \right)^{\frac{1}{2}}
\left(\sum_{j=1}^{J+1}\|\sqrt{M}\,\pdv P_N w\|^2_{L^2(\Omega^{J+1}\times \R^{(J+1)d})} \right)^{\frac{1}{2}}\]
and
\[ |{\rm S}_3(t)| \leq \alpha\left(\sum_{j=1}^{J+1}\|\sqrt{M}\,\pdr\hrho_{\alpha,N}(t)\|^2_{L^2(\Omega^{J+1}\times \R^{(J+1)d})} \right)^{\frac{1}{2}}
\left(\sum_{j=1}^{J+1}\|\sqrt{M}\,\pdr P_N w\|^2_{L^2(\Omega^{J+1}\times \R^{(J+1)d})} \right)^{\frac{1}{2}}.\]
Thus, by $\mbox{\eqref{eq:uniform-space}}_2$ we have that,
\[ \int_0^T |{\rm S}_1(t)|^2\dd t \leq C \left(\sum_{j=1}^{J+1}\|\sqrt{M}\,\pdv P_N w\|^2_{L^2(\Omega^{J+1}\times \R^{(J+1)d})} \right)^{\frac{1}{2}}\]
and, by $\mbox{\eqref{eq:uniform-space}}_3$,
\[ \int_0^T |{\rm S}_3(t)|^2 \dd t \leq C \sqrt{\alpha} \left(\sum_{j=1}^{J+1}\|\sqrt{M}\,\pdr P_N w\|^2_{L^2(\Omega^{J+1}\times \R^{(J+1)d})} \right)^{\frac{1}{2}},\]
where $C$ is a positive constant, independent of $N$ and $\alpha$.

For the term ${\rm S}_2(t)$, we have, for a.e. $t \in (0,T]$, that
\begin{align*}
|{\rm S}_2(t)| \leq \alpha\left(\sum_{j=1}^{J+1}\|\sqrt{M}\, |v_j|\,\hrho_{\alpha,N}(t)\|^2_{L^2(\Omega^{J+1}\times \R^{(J+1)d})} \right)^{\frac{1}{2}}
\left(\sum_{j=1}^{J+1}\|\sqrt{M}\,\pdr P_N w\|^2_{L^2(\Omega^{J+1}\times \R^{(J+1)d})} \right)^{\frac{1}{2}}.
\end{align*}
Now, \eqref{eq:uniform-space-1} implies that
\begin{align}\label{eq:vbound}
\frac{1}{4\beta}\, \int_0^T \int_{\Omega^{J+1} \times \R^{(J+1)d}}
M(v)\, |v_j|^2\, |\hrho_{\alpha,N}|^2 \dd r \dd v \dd \tau
& \leq C,
\end{align}
where the constant $C$ is independent of $N$ and $\alpha$, and therefore
\[ \int_0^T |{\rm S}_2(t)|^2 \dd t \leq C \alpha \left(\sum_{j=1}^{J+1}\|\sqrt{M}\,\pdr P_N w\|^2_{L^2(\Omega^{J+1}\times \R^{(J+1)d})} \right)^{\frac{1}{2}},\]
where $C$ is independent of $N$ and $\alpha$.

Finally, thanks to \eqref{eq:u-bound} and \eqref{eq:uniform-space}, we have that
\[ \int_0^T |{\rm S}_4(t)|^2 \dd t \leq C \left(\sum_{j=1}^{J+1}\|\sqrt{M}\,\pdv P_N w\|^2_{L^2(\Omega^{J+1}\times \R^{(J+1)d})} \right)^{\frac{1}{2}},\]
where, again, $C$ is independent of $N$ and $\alpha$.

By collecting the bounds on ${\rm S}_1, \dots, {\rm S}_4$, noting \eqref{ortho}, and recalling that $\alpha \in (0,1]$,
we deduce that
\begin{align*}
\int_0^T \|\partial_t \hrho_{\alpha,N}(t)\|^2_{(W^{1,2}_{*,M}(\Omega^{J+1} \times \R^{(J+1)d}))'} \dd t &\leq C(L,J,T,\eps,\|u\|_{L^2(0,T;L^\infty({\Omega}))}).
\end{align*}
Hence,
\begin{align}\label{eq:uniform-time}
\|\partial_t \hrho_{\alpha,N}\|_{L^2(0,T;(W^{1,2}_{*,M}(\Omega^{J+1} \times \R^{(J+1)d}))')} &\leq C(L,J,T,\eps,\|u\|_{L^2(0,T;L^\infty({\Omega}))}),
\end{align}
as required.

The bounds \eqref{eq:uniform-space}, \eqref{eq:uniform-space-1}, \eqref{eq:uniform-time} in conjunction with the compact embedding of $W^{1,2}_{*,M}(\Omega^{J+1} \times \R^{(J+1)d})$ into  the function space $L^2_M(\Omega^{J+1} \times \R^{(J+1)d})$
and the Aubin--Lions lemma (cf. \cite{Simon})
imply the existence of a subsequence (not indicated) of $(\hrho_{\alpha,N})_{N \geq 1}$ and of an element
{\small
\[ \hrho_\alpha \in L^\infty(0,T;L^2_M(\Omega^{J+1} \times \R^{(J+1)d})) \cap L^2(0,T;W^{1,2}_{*,M}(\Omega^{J+1} \times \R^{(J+1)d})) \cap W^{1,2}(0,T;(W^{1,2}_{*,M}(\Omega^{J+1}\times \R^{(J+1)d}))')\]
}

\vspace{-4mm}
\noindent
such that, as $N \rightarrow \infty$,
\begin{alignat}{2}\label{eq:compactness-N}
\begin{aligned}
\hrho_{\alpha,N} &\rightharpoonup \hrho_\alpha &&\qquad \mbox{weakly* in $L^\infty(0,T;L^2_M(\Omega^{J+1} \times \R^{(J+1)d}))$},\\
\hrho_{\alpha,N} &\rightarrow \hrho_\alpha &&\qquad \mbox{strongly in $L^2(0,T;L^2_M(\Omega^{J+1} \times \R^{(J+1)d}))$},\\
\hrho_{\alpha,N} &\rightharpoonup \hrho_\alpha &&\qquad \mbox{weakly in $L^2(0,T;W^{1,2}_{*,M}(\Omega^{J+1} \times \R^{(J+1)d}))$},\\
|v_j|\,\hrho_{\alpha,N} &\rightharpoonup |v_j|\,\hrho_\alpha &&\qquad \mbox{weakly in $L^2(0,T;L^2_M(\Omega^{J+1} \times \R^{(J+1)d}))$},\\
\partial_t \hrho_{\alpha,N} &\rightharpoonup\partial_t\hrho_\alpha &&\qquad \mbox{weakly in
$L^2(0,T;(W^{1,2}_{*,M}(\Omega^{J+1} \times \R^{(J+1)d}))')$}.
\end{aligned}
\end{alignat}
Thanks to the density of $W^{1,2}_{0,M}(\Omega^{J+1} \times \R^{(J+1)d})$ in $L^2_M(\Omega^{J+1} \times \R^{(J+1)d})$
(cf. Appendix A in \cite{BS2010-hookean}) and noting that $W^{1,2}_{0,M}(\Omega^{J+1} \times \R^{(J+1)d}) \subset W^{1,2}_{*,M}(\Omega^{J+1} \times \R^{(J+1)d})$, it follows that $W^{1,2}_{*,M}(\Omega^{J+1} \times \R^{(J+1)d})$
is dense in the space $L^2_M(\Omega^{J+1} \times \R^{(J+1)d})$. Thus, the Hilbert space $V:=W^{1,2}_{*,M}(\Omega^{J+1} \times \R^{(J+1)d})$ is continuously and densely embedded into the Hilbert space $H:=L^2_M(\Omega^{J+1} \times \R^{(J+1)d})$. Hence, according to the function space interpolation result (2.41) in Lions \& Magenes \cite{LM1972},  $[V,V']_{1/2} = H$, and therefore Theorem 3.1 in \cite{LM1972}
yields the continuous embedding
\[ L^2(0,T;W^{1,2}_{*,M}(\Omega^{J+1} \times \R^{(J+1)d})) \cap W^{1,2}(0,T;(W^{1,2}_{*,M}(\Omega^{J+1} \times \R^{(J+1)d})'))
\hookrightarrow \mathcal{C}([0,T];L^2_M(\Omega^{J+1} \times \R^{(J+1)d}))\]
which then implies that
\begin{align}\label{eq:compactness-N1}
\begin{aligned}
\hrho_\alpha \in \mathcal{C}([0,T];L^2_M(\Omega^{J+1} \times \R^{(J+1)d}))&,\\
\lim_{N \rightarrow \infty}(\hrho_{\alpha,N}(\cdot,\cdot,t)-\hrho_\alpha(\cdot,\cdot,t), \eta)_{L^2_M(\Omega^{J+1} \times \R^{(J+1)d})} \rightarrow 0 &\qquad \forall\, \eta \in L^2_M(\Omega^{J+1} \times \R^{(J+1)d})\quad
\forall\, t \in [0,T].
\end{aligned}
\end{align}

By passing to the limit $N \rightarrow \infty$ in \eqref{eq:energy-aN}, using the weak convergence
results \eqref{eq:compactness-N} in conjunction with the weak lower-semicontinuity of the norm function, we deduce
that $\hrho_\alpha$ satisfies the following energy inequality:
\begin{align}\label{eq:energy-a}
&\|\sqrt{M}\, \hrho_{\alpha}(\cdot,\cdot,t)\|^2_{L^2(\Omega^{J+1}\times \R^{(J+1)d})}\nonumber\\
&\quad+ \frac{\beta^2}{\eps^2} \sum_{j=1}^{J+1} \int_0^t \|\sqrt{M}\, \pdv \hrho_{\alpha}\|^2_{L^2(\Omega^{J+1}\times \R^{(J+1)d})} \dd \tau +
2\alpha\, \sum_{j=1}^{J+1} \int_0^t \|\sqrt{M}\, \pdr \hrho_{\alpha}\|^2_{L^2(\Omega^{J+1}\times \R^{(J+1)d})} \dd \tau \nonumber\\
&\quad\quad  \leq \|\sqrt{M}\, \hrho_{0}\|^2_{L^2(\Omega^{J+1}\times \R^{(J+1)d})} \,
\mbox{exp}\left(\frac{1}{\beta^2}\,C(L,J)^2\,(T+ \|u\|^2_{L^2(0,T;L^{\infty}({\Omega}))})  \right)
\qquad \forall\, t \in (0,T].
\end{align}
Furthermore, by replacing $\varphi(r,v,\tau)$ with $\varphi(r,v,\tau)\,\chi_{t,h}(\tau)$ in \eqref{eq:weak-a}, for $t \in (0,T]$ fixed, and passing to the limit $h \rightarrow 0_+$, we deduce that
{\small
\begin{align}\label{eq:weak-aa}
&\int_{\Omega^{J+1}} \int_{\R^{(J+1)d}} M(v)\,\hrho_\alpha(r,v,t)\,\varphi(r,v,t)\dd v \dd r- \int_0^t \int_{\Omega^{J+1}} \int_{\R^{(J+1)d}} M(v)\,\hrho_\alpha(r,v,\tau)\,\pd_\tau \varphi(r,v,\tau)\dd v \dd r \dd \tau\nonumber\\
&\qquad+ \frac{\beta^2}{\eps^2}\left(\sum_{j=1}^{J+1} \int_0^t \int_{\Omega^{J+1}} \int_{\R^{(J+1)d}} M(v)\,\pdv \hrho_\alpha \cdot \pdv \varphi \dd v \dd r \dd \tau \right)\nonumber\\
&\qquad- \frac{1}{\eps} \left(\sum_{j=1}^{J+1} \int_0^t \int_{\Omega^{J+1}} \int_{\R^{(J+1)d}} M(v)\, v_j \hrho_\alpha\cdot \pdr \varphi \dd v \dd r \dd \tau \right)\nonumber\\
&\qquad + \alpha \sum_{j=1}^{J+1} \int_0^t \int_{\Omega^{J+1}} \int_{\R^{(J+1)d}} M(v)\,\pdr \hrho_\alpha \cdot \pdr \varphi \dd v \dd r \dd \tau \nonumber\\
&\qquad- \frac{1}{\eps} \left(\sum_{j=1}^{J+1} \int_0^t \int_{\Omega^{J+1}} \int_{\R^{(J+1)d}} M(v)\,(({\mathcal L}r)_j+u(r_j,\tau))\,\hrho_\alpha\cdot \pdv \varphi \dd v \dd r \dd \tau \right)\nonumber\\
&\qquad\qquad  = \int_{\Omega^{J+1}} \int_{\R^{(J+1)d}} M(v)\,\hrho_{0}(r,v)\,\varphi(r,v,0)\dd v \dd r
\qquad \forall\, \varphi \in W^{1,2}(0,T; W^{1,2}_{*,M}(\Omega^{J+1} \times \R^{(J+1)d})).
\end{align}
}

By letting $t \rightarrow 0_+$ in the weak formulation \eqref{eq:weak-aa},
recalling that $\hrho_\alpha \in \mathcal{C}([0,T];
L^2_M(\Omega^{J+1}\times \R^{(J+1)d}))$ and noting that $W^{1,2}(0,T; W^{1,2}_{*,M}(\Omega^{J+1} \times \R^{(J+1)d})) \hookrightarrow \mathcal{C}([0,T]; W^{1,2}_{*,M}(\Omega^{J+1} \times \R^{(J+1)d}))$, it follows that
\begin{align*}
&\lim_{t \rightarrow 0_+} \int_{\Omega^{J+1}} \int_{\R^{(J+1)d}} M(v)\,\hrho_\alpha(r,v,t)\,\varphi(r,v,t)\dd v \dd r  = \int_{\Omega^{J+1}} \int_{\R^{(J+1)d}} M(v)\,\hrho_\alpha(r,v,0)\,\varphi(r,v,0)\dd v \dd r\\
&\qquad = \int_{\Omega^{J+1}} \int_{\R^{(J+1)d}} M(v)\,\hrho_{0}(r,v)\,\varphi(r,v,0)\dd v \dd r
\qquad \forall\, \varphi \in W^{1,2}(0,T; W^{1,2}_{*,M}(\Omega^{J+1} \times \R^{(J+1)d})).
\end{align*}
As was noted in the paragraph preceding \eqref{eq:compactness-N1},
$W^{1,2}_{*,M}(\Omega^{J+1} \times \R^{(J+1)d})$ is continuously and densely embedded into $L^2_M(\Omega^{J+1} \times \R^{(J+1)d})$, so we deduce from the above, with $\varphi(\cdot,\cdot,t) \equiv \eta(\cdot,\cdot) \in L^2_M(\Omega^{J+1} \times \R^{(J+1)d})$, that
\begin{align}\label{eq:weak-ini}
\lim_{t \rightarrow 0_+} (\hrho_\alpha(t) - \hrho_{0},\eta)_{L^2_M(\Omega^{J+1} \times \R^{(J+1)d})}= 0
\qquad \forall\, \eta \in L^2_M(\Omega^{J+1} \times \R^{(J+1)d}).
\end{align}
This weak attainment of the initial datum $\hrho_{0}$ by the solution $\hrho_{\alpha}$ can be strengthened, in fact. By letting $t \rightarrow 0_+$ in \eqref{eq:energy-a}, it follows that
\[ \lim_{t \rightarrow 0_+} \|\hrho_{\alpha}(t)\|^2_{L^2_M(\Omega^{J+1} \times \R^{(J+1)d})}
\leq \|\hrho_{0}\|^2_{L^2_M(\Omega^{J+1} \times \R^{(J+1)d})}.
\]
Hence, and noting \eqref{eq:weak-ini},
\begin{align}
&\lim_{t \rightarrow 0_+} \|\hrho_{\alpha}(t) - \hrho_{0}\|^2_{L^2_M(\Omega^{J+1} \times \R^{(J+1)d})} = \lim_{t \rightarrow 0_+} (\hrho_{\alpha}(t) - \hrho_{0},\hrho_{\alpha}(t) - \hrho_{0})_{L^2_M(\Omega^{J+1} \times \R^{(J+1)d})} \nonumber\\
&= \lim_{t \rightarrow 0_+} (\hrho_{\alpha}(t),\hrho_{\alpha}(t) - \hrho_{0})_{L^2_M(\Omega^{J+1} \times \R^{(J+1)d})} \nonumber\\
&= \lim_{t \rightarrow 0_+} \|\hrho_{\alpha}(t)\|^2_{L^2_M(\Omega^{J+1} \times \R^{(J+1)d})} - \lim_{t \rightarrow 0_+} (\hrho_{\alpha}(t), \hrho_{0})_{L^2_M(\Omega^{J+1} \times \R^{(J+1)d})}
\nonumber\\
& = \lim_{t \rightarrow 0_+} \|\hrho_{\alpha}(t)\|^2_{L^2_M(\Omega^{J+1} \times \R^{(J+1)d})} - (\hrho_{0}, \hrho_{0})_{L^2_M(\Omega^{J+1} \times \R^{(J+1)d})} \leq 0,\nonumber
\end{align}
which, by the nonnegativity of the norm, then implies the following strong attainment of the initial datum:
\begin{align}\label{eq:strong-ini}
\lim_{t \rightarrow 0_+} \|\hrho_{\alpha}(t) - \hrho_{0}\|^2_{L^2_M(\Omega^{J+1} \times \R^{(J+1)d})} = 0.
\end{align}

Having thus shown that $\hrho_\alpha$ satisfies the given initial condition, we shall now pass to the limit $N \rightarrow \infty$ in the Galerkin approximation, in order to show that $\hrho_\alpha$ is in fact a weak solution to the parabolic regularization \eqref{eq:weak-a} of \eqref{eq:weak}.

Given any (fixed) $\varphi \in W^{1,2}(0,T; W^{1,2}_{*,M}(\Omega^{J+1} \times \R^{(J+1)d}))$, we consider the
function $\varphi_N \in W^{1,2}(0,T;\mathcal{X}_N)$, defined by
\[ \varphi_N(r,v,t):= \sum_{k=1}^N \beta_{k,N}(t)\, \psi_k(r,v),\]
where $\beta_{k,N} \in W^{1,2}(0,T)$ is defined by
\[\beta_{k,N}(t) = (\varphi(\cdot,\cdot,t), \psi_k(\cdot,\cdot))_{L^2_M(\Omega^{J+1} \times \R^{(J+1)d})}, \qquad k=1,\dots, N; \quad N \geq 1.\]
Hence,
\begin{equation}\label{eq:philimit}
\lim_{N\rightarrow \infty} \|\varphi - \varphi_N\|_{W^{1,2}(0,T;W^{1,2}_M(\Omega^{J+1} \times \R^{(J+1)d}))} = 0.
\end{equation}

Next, for $\varphi \in W^{1,2}(0,T;W^{1,2}_{*,M}(\Omega^{J+1} \times \R^{(J+1)d}))$ fixed and $\varphi_N
\in W^{1,2}(0,T;\mathcal{X}_N)$ as defined above, we rewrite \eqref{eq:weak-a-gal} in the following equivalent form:
\begin{align}
&\int_{\Omega^{J+1}} \int_{\R^{(J+1)d}} M(v)\,\hrho_{\alpha,N}(r,v,T)\,\varphi(r,v,T)\dd v \dd r \nonumber\\
&\qquad- \int_0^T \int_{\Omega^{J+1}} \int_{\R^{(J+1)d}} M(v)\,\hrho_{\alpha,N}(r,v,\tau)\,\pd_\tau \varphi(r,v,\tau)\dd v \dd r \dd \tau\nonumber\\
&\qquad+ \frac{\beta^2}{\eps^2}\left(\sum_{j=1}^{J+1} \int_0^T \int_{\Omega^{J+1}} \int_{\R^{(J+1)d}} M(v)\,\pdv \hrho_{\alpha,N} \cdot \pdv \varphi \dd v \dd r \dd \tau \right)\nonumber\\
&\qquad- \frac{1}{\eps} \left(\sum_{j=1}^{J+1} \int_0^T \int_{\Omega^{J+1}} \int_{\R^{(J+1)d}} M(v)\, v_j \hrho_{\alpha,N}\cdot \pdr \varphi \dd v \dd r \dd \tau \right)\nonumber\\
&\qquad + \alpha \sum_{j=1}^{J+1} \int_0^T \int_{\Omega^{J+1}} \int_{\R^{(J+1)d}} M(v)\,\pdr \hrho_{\alpha,N} \cdot \pdr \varphi \dd v \dd r \dd \tau \nonumber\\
&\qquad- \frac{1}{\eps} \left(\sum_{j=1}^{J+1} \int_0^T \int_{\Omega^{J+1}} \int_{\R^{(J+1)d}} M(v)\,(({\mathcal L}r)_j+u(r_j,\tau))\,\hrho_{\alpha,N}\cdot \pdv \varphi \dd v \dd r \dd \tau \right)\nonumber\\
&= \int_{\Omega^{J+1}} \int_{\R^{(J+1)d}} M(v)\,\hrho_{0}(r,v)\,\varphi(r,v,0)\dd v \dd r\nonumber
\\
&\qquad - \int_{\Omega^{J+1}} \int_{\R^{(J+1)d}} M(v)\,\hrho_{0}(r,v)\,(\varphi-\varphi_N)(r,v,0)\dd v \dd r \nonumber\\
&\qquad+ \int_{\Omega^{J+1}} \int_{\R^{(J+1)d}} M(v)\,\hrho_{\alpha,N}(r,v,T)\,(\varphi-\varphi_N)(r,v,T)\dd v \dd r \nonumber
\\
&\qquad- \int_0^T \int_{\Omega^{J+1}} \int_{\R^{(J+1)d}} M(v)\,\hrho_{\alpha,N}(r,v,\tau)\,\pd_\tau (\varphi-\varphi_N)(r,v,\tau)\dd v \dd r \dd \tau\nonumber
\end{align}
\begin{align}
\label{eq:phi-phiN}
&\qquad+ \frac{\beta^2}{\eps^2}\left(\sum_{j=1}^{J+1} \int_0^T \int_{\Omega^{J+1}} \int_{\R^{(J+1)d}} M(v)\,\pdv \hrho_{\alpha,N} \cdot \pdv (\varphi-\varphi_N) \dd v \dd r \dd \tau \right)\nonumber\\
&\qquad- \frac{1}{\eps} \left(\sum_{j=1}^{J+1} \int_0^T \int_{\Omega^{J+1}} \int_{\R^{(J+1)d}} M(v)\, v_j \hrho_{\alpha,N}\cdot \pdr (\varphi-\varphi_N) \dd v \dd r \dd \tau \right)\nonumber\\
&\qquad + \alpha \sum_{j=1}^{J+1} \int_0^T \int_{\Omega^{J+1}} \int_{\R^{(J+1)d}} M(v)\,\pdr \hrho_{\alpha,N} \cdot \pdr (\varphi-\varphi_N) \dd v \dd r \dd \tau \nonumber\\
&\qquad- \frac{1}{\eps} \left(\sum_{j=1}^{J+1} \int_0^T \int_{\Omega^{J+1}} \int_{\R^{(J+1)d}} M(v)\,(({\mathcal L}r)_j+u(r_j,\tau))\,\hrho_{\alpha,N}\cdot \pdv (\varphi-\varphi_N) \dd v \dd r \dd \tau \right).
\end{align}

\noindent
Now, using the convergence results \eqref{eq:compactness-N}, \eqref{eq:compactness-N1}, the uniform bounds \eqref{eq:u-bound}, \eqref{eq:uniform-space}, \eqref{eq:vbound},
together with the strong convergence \eqref{eq:philimit}, passage to the limit $N \rightarrow \infty$ in \eqref{eq:phi-phiN} yields that the function
{\small
\[ \hrho_\alpha \in \mathcal{C}([0,T];L^2_M(\Omega^{J+1} \times \R^{(J+1)d})) \cap L^2(0,T;W^{1,2}_{*,M}(\Omega^{J+1} \times \R^{(J+1)d})) \cap W^{1,2}(0,T;(W^{1,2}_{*,M}(\Omega^{J+1}\times \R^{(J+1)d}))')\]
}

\vspace{-4mm}
\noindent
satisfies \eqref{eq:weak-a}. Indeed, as $N \rightarrow \infty$, all terms on the right-hand side of \eqref{eq:phi-phiN}, except the first,
converge to zero, while each of the terms on the left-hand side converges to its counterpart with $\hrho_{N,\alpha}$ replaced by $\hrho_\alpha$.
Thus we have shown that, for any $\alpha \in (0,1]$, $\hrho_\alpha$ is a solution of \eqref{eq:weak-a}, and the energy inequality \eqref{eq:energy-a} holds. It is important to note for the purpose of the discussion in the next section that
the right-hand side of the inequality \eqref{eq:energy-a} is independent of $\alpha$; therefore, \eqref{eq:energy-a} implies that
\begin{alignat}{2}
\hrho_\alpha &\qquad &&\mbox{is bounded in~ $L^\infty(0,T;L^2_M(\Omega^{J+1} \times \R^{(J+1)d}))$} \label{eq:alphabound-1}\\
\partial_{v_j} \hrho_\alpha &\qquad &&\mbox{is bounded in~ $L^2(0,T;L^2_M(\Omega^{J+1} \times \R^{(J+1)d}))\quad$ for all $j=1,\dots,J+1$,} \label{eq:alphabound-2}\\
\sqrt{\alpha}\, \partial_{r_j} \hrho_\alpha &\qquad &&\mbox{is bounded in~ $L^2(0,T;L^2_M(\Omega^{J+1} \times \R^{(J+1)d}))\quad$ for all $j=1,\dots, J+1$}, \label{eq:alphabound-3}
\end{alignat}
provided that $\hrho_0 \in L^2_M(\Omega^{J+1} \times \R^{(J+1)d})$ and $u \in L^2(0,T;L^\infty(\Omega)^d)$. Furthermore,
weak lower semicontinuity of the norm function, \eqref{eq:uniform-space-1} and \eqref{eq:compactness-N}$_4$ imply that
\begin{equation}
~\hspace{0.5cm} |v_j|\,\hrho_\alpha \qquad \mbox{is bounded in~ $L^2(0,T;L^2_M(\Omega^{J+1} \times \R^{(J+1)d}))\quad $ for all $j=1,\dots,J+1$}. \label{eq:alphabound-4}
\end{equation}
Similarly, \eqref{eq:uniform-time} implies that
\begin{equation}\label{eq:uniform-time-bd}
\partial_t \hrho_{\alpha} \qquad \mbox{is bounded in~ $L^2(0,T;(W^{1,2}_{*,M}(\Omega^{J+1} \times \R^{(J+1)d}))')$}.
\end{equation}
We are now ready to pass to the limit $\alpha \rightarrow 0_+$.

\subsection{Passage to the limit with the parabolic regularization parameter}
\label{subsec_2:FP}
The next step in our argument is passage to the limit $\alpha \rightarrow 0_+$ in \eqref{eq:weak-a}. We begin by noting that
\eqref{eq:alphabound-1}--\eqref{eq:alphabound-4} imply that
\begin{subequations}
\begin{alignat}{2}
\hrho_\alpha &\rightharpoonup \hrho \qquad &&\mbox{weak$^*$ in~ $L^\infty(0,T;L^2_M(\Omega^{J+1} \times \R^{(J+1)d}))$} \label{eq:alphabound-1a}\\
\partial_{v_j} \hrho_\alpha &\rightharpoonup \partial_{v_j} \hrho \qquad &&\mbox{weakly in~ $L^2(0,T;L^2_M(\Omega^{J+1} \times \R^{(J+1)d}))\qquad$ for all $j=1,\dots,J+1$,} \label{eq:alphabound-2a}\\
\alpha\, \partial_{r_j} \hrho_\alpha &\rightarrow 0\qquad &&\mbox{strongly in~ $L^2(0,T;L^2_M(\Omega^{J+1} \times \R^{(J+1)d}))\quad\;\,$ for all $j=1,\dots, J+1$}, \label{eq:alphabound-3a}\\
|v_j|\,\hrho_\alpha &\rightharpoonup |v_j|\,\hrho \qquad &&\mbox{weakly in~ $L^2(0,T;L^2_M(\Omega^{J+1} \times \R^{(J+1)d}))\qquad$ for all $j=1,\dots,J+1$}, \label{eq:alphabound-4a}
\end{alignat}
\end{subequations}
provided that $\hrho_0 \in L^2_M(\Omega^{J+1} \times \R^{(J+1)d})$ and $u \in L^2(0,T;L^\infty(\Omega)^d)$.

Next, we shall prove that $\hrho_\alpha \geq 0$ a.e. on  $\Omega^{J+1} \times \R^{(J+1)d} \times [0,T]$ for all $\alpha \in (0,1]$, and that
\begin{equation}\label{eq:conservation=1}
 \int_{\Omega^{J+1} \times \R^{(J+1)d}} M(v)\, \hrho_\alpha(r,v,t) \dd r \dd v = 1\qquad \forall\, t \in [0,T].
\end{equation}
The proof of the latter assertion is straightforward: for $t=0$ it follows from \eqref{eq:ini-cond}; for $t \in
(0,T]$ (fixed), we take $\varphi(r,v,\tau) \equiv 1$ in \eqref{eq:weak-a} and note \eqref{eq:ini-cond}
to  deduce that
\begin{align*}
\int_{\Omega^{J+1}} \int_{\R^{(J+1)d}} M(v)\,\hrho_\alpha(r,v,t)\dd v \dd r
=  \int_{\Omega^{J+1}} \int_{\R^{(J+1)d}} M(v)\,\hrho_{0}(r,v) \dd v \dd r = 1\qquad \forall\,
t \in (0,T].
\end{align*}

Before embarking on the proof of the nonnegativity of $\hrho_\alpha$ we shall first
extend the set of test functions
$$W^{1,2}(0,T;W^{1,2}_{*,M}(\Omega^{J+1} \times \R^{(J+1)d}))$$
appearing in \eqref{eq:weak-aa} to $$L^2(0,T;W^{1,2}_{*,M}(\Omega^{J+1}\times \R^{(J+1)d}))$$
by using a density argument. We begin by rewriting \eqref{eq:weak-aa} as follows:
\begin{align}\label{eq:weak-duality}
&\int_0^t \big\langle M\,\pd_\tau\hrho_\alpha(\cdot,\cdot,\tau),\varphi(\cdot,\cdot,\tau)\big\rangle \dd \tau
+ \frac{\beta^2}{\eps^2}\left(\sum_{j=1}^{J+1} \int_0^t \int_{\Omega^{J+1}} \int_{\R^{(J+1)d}} M(v)\,\pdv \hrho_\alpha \cdot \pdv \varphi \dd v \dd r \dd \tau \right)\nonumber\\
&\qquad- \frac{1}{\eps} \left(\sum_{j=1}^{J+1} \int_0^t \int_{\Omega^{J+1}} \int_{\R^{(J+1)d}} M(v)\, v_j \hrho_\alpha\cdot \pdr \varphi \dd v \dd r \dd \tau \right)\nonumber\\
&\qquad + \alpha \sum_{j=1}^{J+1} \int_0^t \int_{\Omega^{J+1}} \int_{\R^{(J+1)d}} M(v)\,\pdr \hrho_\alpha \cdot \pdr \varphi \dd v \dd r \dd \tau \nonumber\\
&\qquad- \frac{1}{\eps} \left(\sum_{j=1}^{J+1} \int_0^t \int_{\Omega^{J+1}} \int_{\R^{(J+1)d}} M(v)\,(({\mathcal L}r)_j+u(r_j,\tau))\,\hrho_\alpha\cdot \pdv \varphi \dd v \dd r \dd \tau \right)
= 0 \nonumber \\
&\hspace{2.8in} \forall\, \varphi \in W^{1,2}(0,T; W^{1,2}_{*,M}(\Omega^{J+1} \times \R^{(J+1)d}))\quad \forall\,t \in (0,T],
\end{align}

\noindent
where $\langle \cdot , \cdot \rangle$ denotes the duality pairing between
$(W^{1,2}_{*,M}(\Omega^{J+1} \times \R^{(J+1)d}))'$ and $W^{1,2}_{*,M}(\Omega^{J+1} \times \R^{(J+1)d})$ with respect to
$L^2_M(\Omega^{J+1} \times \R^{(J+1)d}))$ as pivot space, into which $W^{1,2}_{*,M}(\Omega^{J+1} \times \R^{(J+1)d})$ is
continuously and densely embedded; hence $\langle \eta , \phi \rangle$ and $(\eta,\phi)_{L^2_M(\Omega^{J+1} \times \R^{(J+1)d})}$ are identified when $\eta \in L^2_M(\Omega^{J+1} \times \R^{(J+1)d}))$ and $\phi \in W^{1,2}_{*,M}(\Omega^{J+1} \times \R^{(J+1)d})$. We note that for
{\small
\[ \hrho_\alpha \in \mathcal{C}([0,T];L^2_M(\Omega^{J+1} \times \R^{(J+1)d})) \cap L^2(0,T;W^{1,2}_{*,M}(\Omega^{J+1} \times \R^{(J+1)d})) \cap W^{1,2}(0,T;(W^{1,2}_{*,M}(\Omega^{J+1}\times \R^{(J+1)d}))')\]
}

\vspace{-4mm}
\noindent
fixed, each of the terms in \eqref{eq:weak-duality} is a bounded linear functional of
$\varphi \in L^2(0,T;W^{1,2}_{*,M}(\Omega^{J+1} \times \R^{(J+1)d}))$.
As the Hilbert space $W^{1,2}(0,T;W^{1,2}_{*,M}(\Omega^{J+1} \times \R^{(J+1)d}))$ is continuously and densely embedded into the Hilbert space $L^2(0,T;W^{1,2}_{*,M}(\Omega^{J+1} \times \R^{(J+1)d}))$, we deduce that
{\small
\[ \hrho_\alpha \in \mathcal{C}([0,T];L^2_M(\Omega^{J+1} \times \R^{(J+1)d})) \cap L^2(0,T;W^{1,2}_{*,M}(\Omega^{J+1} \times \R^{(J+1)d})) \cap W^{1,2}(0,T;(W^{1,2}_{*,M}(\Omega^{J+1}\times \R^{(J+1)d}))')\]
}

\vspace{-4mm}
\noindent
satisfies the following weak formulation:

\begin{align}\label{eq:weak-duality-b}
&\int_0^t \big\langle M\,\pd_\tau\hrho_\alpha(\cdot,\cdot,\tau),\varphi(\cdot,\cdot,\tau)\big\rangle \dd \tau
+ \frac{\beta^2}{\eps^2}\left(\sum_{j=1}^{J+1} \int_0^t \int_{\Omega^{J+1}} \int_{\R^{(J+1)d}} M(v)\,\pdv \hrho_\alpha \cdot \pdv \varphi \dd v \dd r \dd \tau \right)\nonumber\\
&\qquad- \frac{1}{\eps} \left(\sum_{j=1}^{J+1} \int_0^t \int_{\Omega^{J+1}} \int_{\R^{(J+1)d}} M(v)\, v_j \hrho_\alpha\cdot \pdr \varphi \dd v \dd r \dd \tau \right)\nonumber\\
&\qquad + \alpha \sum_{j=1}^{J+1} \int_0^t \int_{\Omega^{J+1}} \int_{\R^{(J+1)d}} M(v)\,\pdr \hrho_\alpha \cdot \pdr \varphi \dd v \dd r \dd \tau \nonumber\\
&\qquad- \frac{1}{\eps} \left(\sum_{j=1}^{J+1} \int_0^t \int_{\Omega^{J+1}} \int_{\R^{(J+1)d}} M(v)\,(({\mathcal L}r)_j+u(r_j,\tau))\,\hrho_\alpha\cdot \pdv \varphi \dd v \dd r \dd \tau \right)
= 0 \nonumber \\
&\hspace{2.8in} \forall\, \varphi \in L^2(0,T; W^{1,2}_{*,M}(\Omega^{J+1} \times \R^{(J+1)d})) \quad \forall\,t \in (0,T].
\end{align}

We prove the nonnegativity of $\hrho_\alpha$ by using Stampacchia's truncation method. Let $[x]_{\pm}$ denote the nonnegative and nonpositive parts of $x$, i.e., $[x]_{\pm}:= \frac{1}{2}(x \pm |x|)$. Thus, $x = [x]_{+} + [x]_{-}$ and $x [x]_{-} = ([x]_{-})^2$.

By taking $\varphi = [\hrho_\alpha]_{-}$ in \eqref{eq:weak-duality-b} (which belongs to the function space
$L^2(0,T; W^{1,2}_{*,M}(\Omega^{J+1} \times \R^{(J+1)d}))$, because $\hrho_\alpha$ belongs to this space), we
have that
\begin{align*}
&\int_0^t \frac{1}{2} \frac{\dd}{\dd t} \|\sqrt{M}\,[\hrho_\alpha(\tau)]_{-}\|^2 \dd \tau
+ \frac{\beta^2}{\eps^2}\sum_{j=1}^{J+1} \int_0^t  \|\sqrt{M}\,\pdv [\hrho_\alpha]_{-}\|^2 \dd \tau
+ \alpha \sum_{j=1}^{J+1} \int_0^t \|\sqrt{M}\,\pdr [\hrho_\alpha]_{-}\|^2\dd \tau \nonumber\\
&\qquad = \frac{1}{\eps} \sum_{j=1}^{J+1} \int_0^t (M\, v_j [\hrho_\alpha]_{-} \,,\, \pdr [\hrho_\alpha]_{-}) \dd \tau\\
&\qquad\qquad + \frac{1}{\eps} \sum_{j=1}^{J+1} \int_0^t (M\,(({\mathcal L}r)_j+u(r_j,\tau))\,[\hrho_\alpha]_{-} \,,\, \pdv [\hrho_{\alpha}]_{-}) \dd \tau
\quad \forall\,t \in (0,T],
\end{align*}
subject to the initial condition $[\hrho_\alpha(0)]_{-} = [\hrho_0]_{-}=0$. Therefore, for all $t \in (0,T]$,
\begin{align}\label{eq:energy-ccc}
&\frac{1}{2} \|\sqrt{M}\,[\hrho_\alpha(t)]_{-}\|^2
+ \frac{\beta^2}{\eps^2}\sum_{j=1}^{J+1} \int_0^t  \|\sqrt{M}\,\pdv [\hrho_\alpha]_{-}\|^2 \dd \tau
+ \alpha \sum_{j=1}^{J+1} \int_0^t \|\sqrt{M}\,\pdr [\hrho_\alpha]_{-}\|^2\dd \tau \nonumber\\
&\qquad = \frac{1}{\eps} \sum_{j=1}^{J+1} \int_0^t (M\, v_j [\hrho_\alpha]_{-}\, ,\, \pdr [\hrho_\alpha]_{-}) \dd \tau \nonumber\\
&\qquad\qquad + \frac{1}{\eps} \sum_{j=1}^{J+1} \int_0^t (M\,(({\mathcal L}r)_j+u(r_j,\tau))\,[\hrho_\alpha]_{-}\, ,\, \pdv [\hrho_{\alpha}]_{-}) \dd \tau.
\end{align}

Next we apply the Cauchy--Schwarz inequality to each of the two terms on the right-hand side of \eqref{eq:energy-ccc}.
We then repeat the calculations that resulted in the bounds \eqref{eq:log-estimate} and \eqref{eq:u-bound}, but now
with $\hrho_{\alpha,N}$ replaced by $[\hrho_\alpha]_{-}$ in those bounds, insert the resulting bounds into the right-hand side
of \eqref{eq:energy-ccc}, absorb the terms containing norms of derivatives of $[\hrho_\alpha]_{-}$ into the left-hand side,
and apply Gronwall's lemma to deduce that
$ \|\sqrt{M}\,[\hrho_\alpha(t)]_{-}\|^2 = 0$ $ \mbox{for all $t \in [0,T]$}$.
Consequently $\hrho_\alpha \geq 0$ a.e. on $\Omega^{J+1} \times \R^{(J+1)d} \times [0,T]$, as has been asserted.
Finally we note that an identical procedure can be used to deduce that $\hrho_\alpha$ is the unique weak solution of \eqref{eq:weak-duality-b}
satisfying the initial condition $\hrho_\alpha(0)=\hrho_{0}$.

\smallskip

%
%

The expression appearing on the right-hand side of \eqref{eq:energy-a}
involves the $L^2_M(\Omega^{J+1} \times \R^{(J+1)d})$ norm of $\hrho_0$, whereas, ultimately,
we would like to make use of the weaker hypotheses, stated in \eqref{eq:ini-cond}, only.
As a matter of fact, in the next section we will require an analogous inequality
whose right-hand side involves the $L^1_M(\Omega^{J+1} \times \R^{(J+1)d})$ norm of $\mathcal{F}(\hrho_0)$
rather than the $L^2_M(\Omega^{J+1} \times \R^{(J+1)d})$ norm of $\hrho_0$.
Thus, before passing to the
limit $\alpha \rightarrow 0_+$ by using the weak convergence results stated in \eqref{eq:alphabound-1a}--\eqref{eq:alphabound-4a}, we shall derive additional
bounds
on $\hrho_{\alpha}$, which involve the $L^1_M(\Omega^{J+1} \times \R^{(J+1)d})$ norm of $\mathcal{F}(\hrho_0)$ rather than the
$L^2_M(\Omega^{J+1} \times \R^{(J+1)d})$ norm of $\hrho_0$. The resulting bounds will also play an important role in the
next section, where we focus on the coupled Oseen--Fokker--Planck system.
The argument is based on the relative entropy method.
Briefly, the procedure involves choosing $\mathcal{F}'(\hrho_\alpha + \gamma)$ as test function in \eqref{eq:weak-duality-b},
with $\gamma>0$, and passing to the limit $\gamma \rightarrow 0_+$; ideally, we would like to choose $\mathcal{F}'(\hrho_\alpha)$
as test function in \eqref{eq:weak-duality-b}, however since $\hrho_\alpha \geq 0$ a.e. on $\Omega^{J+1} \times \R^{(J+1)d} \times [0,T]$,
and $\hrho_\alpha$ is potentially equal to $0$ on a subset of positive measure, there is no guarantee that $\mathcal{F}'(\hrho_\alpha)
= \log \hrho_\alpha$ is a.e. finite. Thus we shall, instead, test with $\mathcal{F}'(\hrho_\alpha + \gamma) = \log (\hrho_\alpha + \gamma)$, and once we have obtained the
necessary bounds we shall pass to the limit $\gamma \rightarrow 0_+$, which will then be followed by passage to the limit with $\alpha \rightarrow 0_+$.

We begin by noting that since $\hrho_\alpha \in L^2(0,T;W^{1,2}_{*,M}(\Omega^{J+1} \times
\mathbb{R}^{(J+1)d}))$ and $\hrho_\alpha \geq 0$, also
\[\mathcal{F}'(\hrho_\alpha+\gamma) = \log(\hrho_\alpha + \gamma)
\in L^2(0,T;W^{1,2}_{*,M}(\Omega^{J+1} \times \mathbb{R}^{(J+1)d})).\]
Hence, for all $t \in (0,T]$, \eqref{eq:weak-duality-b} yields
\begin{align*}
&\int_0^t \big\langle M\,\pd_\tau\hrho_\alpha(\cdot,\cdot,\tau), \mathcal{F}'(\hrho_\alpha(\cdot,\cdot,\tau)+\gamma)\big\rangle \dd \tau
\nonumber\\
&\qquad + \frac{\beta^2}{\eps^2}\left(\sum_{j=1}^{J+1} \int_0^t \int_{\Omega^{J+1}} \int_{\R^{(J+1)d}} M(v)\,\pdv \hrho_\alpha \cdot \pdv \mathcal{F}'(\hrho_\alpha(r,v,\tau)+\gamma) \dd v \dd r \dd \tau \right)\nonumber\\
&\qquad- \frac{1}{\eps} \left(\sum_{j=1}^{J+1} \int_0^t \int_{\Omega^{J+1}} \int_{\R^{(J+1)d}} M(v)\, v_j \hrho_\alpha\cdot \pdr \mathcal{F}'(\hrho_\alpha(r,v,\tau)+\gamma) \dd v \dd r \dd \tau \right)\nonumber\\
&\qquad + \alpha \sum_{j=1}^{J+1} \int_0^t \int_{\Omega^{J+1}} \int_{\R^{(J+1)d}} M(v)\,\pdr \hrho_\alpha \cdot \pdr \mathcal{F}'(\hrho_\alpha(r,v,\tau)+\gamma) \dd v \dd r \dd \tau \nonumber\\
&\qquad- \frac{1}{\eps} \left(\sum_{j=1}^{J+1} \int_0^t \int_{\Omega^{J+1}} \int_{\R^{(J+1)d}} M(v)\,(({\mathcal L}r)_j+u(r_j,\tau))\,\hrho_\alpha\cdot \pdv \mathcal{F}'(\hrho_\alpha(r,v,\tau)+\gamma) \dd v \dd r \dd \tau \right)
= 0.
\end{align*}

Thus, we have that
\begin{align*}
&\big\langle M\,\mathcal{F}(\hrho_\alpha(\cdot,\cdot,t)+\gamma), 1 \big \rangle - \big\langle M\,\mathcal{F}(\hrho_{0}(\cdot,\cdot)+\gamma), 1 \big\rangle
\nonumber\\
&\qquad + \frac{\beta^2}{\eps^2} \sum_{j=1}^{J+1} \int_0^t \int_{\Omega^{J+1}} \int_{\R^{(J+1)d}} M(v)\,\frac{|\pdv \hrho_\alpha|^2}{\hrho_\alpha(r,v,\tau)+\gamma}\, \dd v \dd r \dd \tau\nonumber\\
&\qquad + \alpha \sum_{j=1}^{J+1} \int_0^t \int_{\Omega^{J+1}} \int_{\R^{(J+1)d}} M(v)\,\frac{|\pdr \hrho_\alpha|^2}{\hrho_\alpha(r,v,\tau)+\gamma}\, \dd v \dd r \dd \tau 
\end{align*}
\begin{align*}
&= \frac{1}{\eps} \sum_{j=1}^{J+1} \int_0^t \int_{\Omega^{J+1}} \int_{\R^{(J+1)d}} M(v)\, v_j \frac{\hrho_\alpha}{\hrho_\alpha(r,v,\tau)+\gamma}\cdot \pdr \hrho_\alpha \dd v \dd r \dd \tau \nonumber\\
&\qquad + \frac{1}{\eps} \sum_{j=1}^{J+1} \int_0^t \int_{\Omega^{J+1}} \int_{\R^{(J+1)d}} M(v)\,(({\mathcal L}r)_j+u(r_j,\tau))\,\frac{\hrho_\alpha}{\hrho_\alpha(r,v,\tau)+\gamma}\cdot \pdv \hrho_\alpha(r,v,\tau)\,  \dd v \dd r \dd \tau .
\end{align*}
Similarly to the term ${\rm T}_1$ encountered earlier in the argument following \eqref{eq:weak-a-gal3}, the first term on the right-hand side
is equal to zero. This can be seen by interchanging the order of the integrals over $\Omega^{J+1}$ and $\R^{(J+1)d}$, observing that
\[ \frac{\hrho_\alpha}{\hrho_\alpha(r,v,\tau)+\gamma}\,\, \pdr \hrho_\alpha  = \pdr [\hrho_\alpha - \gamma \log (\hrho_\alpha + \gamma)]\quad\mbox{and}\quad \hrho_\alpha - \gamma \log (\hrho_\alpha + \gamma) \in L^2(0,T;W^{1,2}_{*,M}(\Omega^{J+1} \times \mathbb{R}^{(J+1)d})),\]
and performing integration by parts with respect to $r_j$. The resulting equality can be rewritten as follows:
\begin{align*}
&\big\langle M\,\mathcal{F}(\hrho_\alpha(\cdot,\cdot,t)+\gamma), 1 \big \rangle + \frac{\beta^2}{\eps^2} \sum_{j=1}^{J+1} \int_0^t \int_{\Omega^{J+1}} \int_{\R^{(J+1)d}} M(v)\,\frac{|\pdv \hrho_\alpha|^2}{\hrho_\alpha +\gamma}\, \dd v \dd r \dd \tau \nonumber\\
&\qquad + \alpha \sum_{j=1}^{J+1} \int_0^t \int_{\Omega^{J+1}} \int_{\R^{(J+1)d}} M(v)\,\frac{|\pdr \hrho_\alpha|^2}{\hrho_\alpha +\gamma}\, \dd v \dd r \dd \tau \nonumber\\
&= \big\langle M\,\mathcal{F}(\hrho_{0}(\cdot,\cdot)+\gamma), 1 \big\rangle \\
&\qquad + \frac{1}{\eps} \sum_{j=1}^{J+1} \int_0^t \int_{\Omega^{J+1}} \int_{\R^{(J+1)d}} M(v)\,(({\mathcal L}r)_j+u(r_j,\tau))\,\frac{\hrho_\alpha}{\hrho_\alpha +\gamma}\cdot \pdv \hrho_\alpha  \dd v \dd r \dd \tau
=: {\rm R}_1 + {\rm R}_2.
\end{align*}
We begin by considering ${\rm R}_2$. As
\[ \left|\frac{\hrho_\alpha}{\hrho_\alpha +\gamma} \cdot \pdv \hrho_\alpha\right| \leq \frac{\hrho_\alpha}{(\hrho_\alpha +\gamma)^{\frac{1}{2}}}\,
\frac{|\pdv \hrho_\alpha|}{(\hrho_\alpha +\gamma)^{\frac{1}{2}}} \leq (\hrho_\alpha)^{\frac{1}{2}}\,\frac{|\pdv \hrho_\alpha|}{(\hrho_\alpha +\gamma)^{\frac{1}{2}}},
\]
it follows by the Cauchy--Schwarz inequality that
\begin{align*}
{\rm R}_2 \leq \left(\frac{1}{\beta^2}\sum_{j=1}^{J+1} \int_0^t \int_{\Omega^{J+1}} \int_{\R^{(J+1)d}} M(v)\,|({\mathcal L}r)_j+u(r_j,\tau)|^2
\,\hrho_\alpha \dd v \dd r \dd \tau\right)^{\frac{1}{2}}
\\
\hspace{0.5in} \times \left(\frac{\beta^2}{\eps^2} \sum_{j=1}^{J+1} \int_0^t \int_{\Omega^{J+1}} \int_{\R^{(J+1)d}} M(v)\,\frac{|\pdv \hrho_\alpha|^2}{\hrho_\alpha +\gamma}
\dd v \dd r \dd \tau\right)^{\frac{1}{2}}.
\end{align*}
Therefore,
\begin{align}\label{eq:pre-lim}
&\big\langle M\,\mathcal{F}(\hrho_\alpha(\cdot,\cdot,t)+\gamma), 1 \big \rangle + \frac{\beta^2}{2\eps^2} \sum_{j=1}^{J+1} \int_0^t \int_{\Omega^{J+1}} \int_{\R^{(J+1)d}} M(v)\,\frac{|\pdv \hrho_\alpha|^2}{\hrho_\alpha +\gamma}\, \dd v \dd r \dd \tau \nonumber\\
&\qquad + \alpha \sum_{j=1}^{J+1} \int_0^t \int_{\Omega^{J+1}} \int_{\R^{(J+1)d}} M(v)\,\frac{|\pdr \hrho_\alpha|^2}{\hrho_\alpha +\gamma}\, \dd v \dd r \dd \tau \nonumber\\
&\leq {\rm R}_1 + \frac{1}{2\beta^2}\left(\sum_{j=1}^{J+1} \int_0^t \int_{\Omega^{J+1}} \int_{\R^{(J+1)d}} M(v)\,|({\mathcal L}r)_j+u(r_j,\tau)|^2
\,\hrho_\alpha \dd v \dd r \dd \tau\right).
\end{align}
The second term on the right-hand side of \eqref{eq:pre-lim} is, thanks to \eqref{eq:conservation=1}, bounded as follows:
\begin{align}\label{eq:pre-lim-2}
&\frac{1}{2\beta^2}\left(\sum_{j=1}^{J+1} \int_0^t \int_{\Omega^{J+1}} \int_{\R^{(J+1)d}} M(v)\,|({\mathcal L}r)_j+u(r_j,\tau)|^2
\,\hrho_\alpha \dd v \dd r \dd \tau\right) \nonumber\\
&\qquad \leq \frac{1}{2\beta^2} \sum_{j=1}^{J+1}\int_0^T \left[{\rm{ess.sup}}_{r \in \Omega^{J+1}}
|({\mathcal L}r)_j+u(r_j,\tau)|^2  \int_{\Omega^{J+1}} \int_{\R^{(J+1)d}} M(v)\,\hrho_\alpha \dd v \dd r\right] \dd \tau
\nonumber\\
&\qquad = \frac{1}{2\beta^2} \sum_{j=1}^{J+1}\int_0^T \left[{\rm{ess.sup}}_{r \in \Omega^{J+1}}
|({\mathcal L}r)_j+u(r_j,\tau)|^2\right] \dd \tau \nonumber\\
&\qquad \leq C(J,T) (1 + \|u\|^2_{L^2(0,T;L^\infty(\Omega))}),
\end{align}
where, again, the dependence of the constant $C(J,T)$ on $\beta$ has been suppressed. Substituting \eqref{eq:pre-lim-2} into \eqref{eq:pre-lim} we deduce that
\begin{align}\label{eq:pre-lim-1}
&\big\langle M\,\mathcal{F}(\hrho_\alpha(\cdot,\cdot,t)+\gamma), 1 \big \rangle + \frac{\beta^2}{2\eps^2} \sum_{j=1}^{J+1} \int_0^t \int_{\Omega^{J+1}} \int_{\R^{(J+1)d}} M(v)\,\frac{|\pdv \hrho_\alpha|^2}{\hrho_\alpha +\gamma}\, \dd v \dd r \dd \tau \nonumber\\
&\qquad + \alpha \sum_{j=1}^{J+1} \int_0^t \int_{\Omega^{J+1}} \int_{\R^{(J+1)d}} M(v)\,\frac{|\pdr \hrho_\alpha|^2}{\hrho_\alpha +\gamma}\, \dd v \dd r \dd \tau \nonumber\\
&\leq {\rm R}_1 + C(J,T) (1 + \|u\|^2_{L^2(0,T;L^\infty(\Omega))}).
\end{align}

Let us now focus on the term ${\rm R}_1$. As
\begin{align*}
{\rm R}_1 &= \int_{\Omega^{J+1}} \int_{\R^{(J+1)d}} M(v)\,\mathcal{F}(\hrho_{0} + \gamma) \dd v \dd r\\
&= \int_{\Omega^{J+1}} \int_{\R^{(J+1)d}} M(v)\, [\hrho_{0} (\log(\hrho_{0} + \gamma) - 1) + 1] \dd v \dd r
\\
&\qquad
+ \gamma \int_{\Omega^{J+1}} \int_{\R^{(J+1)d}} M(v)\, (\log(\hrho_{0} + \gamma) - 1) \dd v \dd r,
\end{align*}
the dominated convergence theorem implies that the second summand on the right-hand side converges to 0 as $\gamma \rightarrow 0_+$,
while the first summand, again by the dominated convergence theorem, converges to
\[ \int_{\Omega^{J+1}} \int_{\R^{(J+1)d}} M(v)\, [\hrho_{0} (\log \hrho_{0} - 1) + 1] \dd v \dd r =
\int_{\Omega^{J+1}} \int_{\R^{(J+1)d}} M(v)\,\mathcal{F}(\hrho_{0}) \dd v \dd r.\]
Returning with this information to \eqref{eq:pre-lim} we can now pass to the limit $\gamma \rightarrow 0_+$ there
using, in the first term on the left-hand side, Fatou's lemma, and in the second and third term on the left-hand side the monotone convergence theorem.
Hence,
%
\begin{align}\label{eq:energy-alpha1}
&\int_{\Omega^{J+1}} \int_{\R^{(J+1)d}} M(v)\,\mathcal{F}(\hrho_{\alpha}(t)) \dd v \dd r  + \frac{\beta^2}{2\eps^2} \sum_{j=1}^{J+1} \int_0^t \int_{\Omega^{J+1}} \int_{\R^{(J+1)d}} M(v)\,\frac{|\pdv \hrho_\alpha|^2}{\hrho_\alpha}\, \dd v \dd r \dd \tau \nonumber\\
&\qquad + \alpha \sum_{j=1}^{J+1} \int_0^t \int_{\Omega^{J+1}} \int_{\R^{(J+1)d}} M(v)\,\frac{|\pdr \hrho_\alpha|^2}{\hrho_\alpha}\, \dd v \dd r \dd \tau \nonumber\\
&\leq  \int_{\Omega^{J+1}} \int_{\R^{(J+1)d}} M(v)\,\mathcal{F}(\hrho_{0}) \dd v \dd r +
C(\beta,J,T) (1 + \|u\|^2_{L^2(0,T;L^\infty(\Omega))}).
\end{align}

\noindent
and therefore \eqref{eq:energy-alpha1} is the desired bound on $\hrho_\alpha$ that is uniform in $\alpha$.

We shall also require a bound on the time derivative of $\hrho_\alpha$ that is uniform in $\alpha$,
which we shall now derive, using \eqref{eq:energy-alpha1}. Thanks to \eqref{eq:weak-duality}, we have that
\begin{align*}
&\left|\int_0^T \big\langle M\,\pd_\tau\hrho_\alpha(\cdot,\cdot,\tau),\varphi(\cdot,\cdot,\tau)\big\rangle \dd \tau\right|
\leq \frac{\beta^2}{\eps^2} \sum_{j=1}^{J+1} \int_0^T \int_{\Omega^{J+1}} \int_{\R^{(J+1)d}} M(v)\,|\pdv \hrho_\alpha| \; |\pdv \varphi| \dd v \dd r \dd \tau \nonumber\\
&\qquad + \frac{1}{\eps} \sum_{j=1}^{J+1} \int_0^T \int_{\Omega^{J+1}} \int_{\R^{(J+1)d}} M(v)\, |v_j|\, \hrho_\alpha\, |\pdr \varphi| \dd v \dd r \dd \tau \nonumber\\
&\qquad + \alpha \sum_{j=1}^{J+1} \int_0^T \int_{\Omega^{J+1}} \int_{\R^{(J+1)d}} M(v)\,|\pdr \hrho_\alpha|\, |\pdr \varphi| \dd v \dd r \dd \tau \nonumber\\
&\qquad + \frac{1}{\eps} \sum_{j=1}^{J+1} \int_0^T \int_{\Omega^{J+1}} \int_{\R^{(J+1)d}} M(v)\,(|({\mathcal L}r)_j|+|u(r_j,\tau)|)\,\hrho_\alpha\,|\pdv \varphi| \dd v \dd r \dd \tau
\nonumber\\
&=: {\rm Q}_1 + {\rm Q}_2 + {\rm Q}_3 + {\rm Q}_4 \qquad \forall\, \varphi \in L^2(0,T; W^{1,2}_{*,M}(\Omega^{J+1} \times \R^{(J+1)d})) \quad \forall\,t \in (0,T].
\end{align*}
Next, we shall bound each of the terms ${\rm Q}_1, \dots, {\rm Q}_4$. Thanks to \eqref{eq:conservation=1}, \eqref{eq:energy-alpha1} and Sobolev embedding,
\begin{align*}
{\rm Q}_1 &\leq \frac{2\beta^2}{\eps^2} \sum_{j=1}^{J+1} \int_0^T \int_{\Omega^{J+1}} \int_{\R^{(J+1)d}} M(v)\,\sqrt{\hrho_\alpha} \,|\pdv \sqrt{\hrho_\alpha}| \; |\pdv \varphi| \dd v \dd r \dd \tau \nonumber\\
&\leq \frac{2\beta^2}{\eps^2} \sum_{j=1}^{J+1} \int_0^T  \|\hrho_\alpha\|_{L^1_M(\Omega^{J+1} \times \R^{(J+1)d})}
\|\pdv \sqrt{\hrho_\alpha}\|_{L^2_M(\Omega^{J+1} \times \R^{(J+1)d})}\, \|\pdv \varphi\|_{L^\infty(\Omega^{J+1} \times \R^{(J+1)d})} \dd \tau\\
&\leq \frac{2\beta^2}{\eps^2} \,
\left(\sum_{j=1}^{J+1} \int_0^T \|\pdv \sqrt{\hrho_\alpha}\|_{L^2_M(\Omega^{J+1} \times \R^{(J+1)d})}^2 \dd \tau\right)^{\frac{1}{2}}
\left(\sum_{j=1}^{J+1} \int_0^T \|\pdv \varphi\|_{L^\infty(\Omega^{J+1} \times \R^{(J+1)d})}^2 \dd \tau\right)^{\frac{1}{2}}\\
& \leq C\|\varphi\|_{L^2(0,T;W^{s,2}(\Omega^{J+1} \times \R^{(J+1)d}))} \qquad \forall\, \varphi \in L^2(0,T;W^{s,2}_*(\Omega^{J+1} \times \R^{(J+1)d})),\quad
s> (J+1)d + 1,
\end{align*}
where $C$ is a positive constant, independent of $\alpha \in (0,1]$. By an identical argument,
\begin{align*}
{\rm Q}_3 & \leq C\|\varphi\|_{L^2(0,T;W^{s,2}(\Omega^{J+1} \times \R^{(J+1)d}))} \qquad \forall\, \varphi \in L^2(0,T;W^{s,2}_*(\Omega^{J+1} \times \R^{(J+1)d})),\quad
s> (J+1)d + 1,
\end{align*}
where $C$ is a positive constant, independent of $\alpha \in (0,1]$. Next,
\begin{align*}
{\rm Q}_2 & \leq \frac{1}{\eps} \sum_{j=1}^{J+1} \int_0^T  \||v_j|\, \hrho_\alpha\|_{L^1_M({\Omega^{J+1} \times \R^{(J+1)d}})}
\|\pdr \varphi\|_{L^\infty({\Omega^{J+1} \times \R^{(J+1)d}})} \dd \tau.
\end{align*}
Now, by Cauchy's inequality and the inequality \eqref{eq:ab}, we have
\begin{align*}
&\||v_j|\, \hrho_\alpha\|_{L^1_M({\Omega^{J+1} \times \R^{(J+1)d}})}  = \int_{\Omega^{J+1} \times \R^{(J+1)d}} M(v)\,|v_j|\,\hrho_\alpha
 \dd r \dd v\\
&\quad \leq \frac{1}{2}\int_{\Omega^{J+1} \times \R^{(J+1)d}} M(v)\,(1 + |v_j|^2)\,\hrho_\alpha \dd r \dd v
\leq \frac{1}{2}L^{(J+1)d} + 2\beta
\int_{\Omega^{J+1} \times \R^{(J+1)d}} M(v)\, \frac{|v_j|^2}{4\beta}\, \hrho_\alpha \dd r \dd v\\
&\quad \leq \frac{1}{2}L^{(J+1)d} + 2\beta
\int_{\Omega^{J+1} \times \R^{(J+1)d}} M(v)\big({\rm e}^{\frac{1}{4\beta}|v_j|^2}-1\big)\dd r \dd v
+ 2\beta
\int_{\Omega^{J+1} \times \R^{(J+1)d}} M(v)\, \mathcal{F}(\hrho_\alpha) \dd r \dd v,
\end{align*}
and hence, by \eqref{eq:energy-alpha1},
\begin{align*}
\||v_j|\, \hrho_\alpha\|_{L^\infty(0,T;L^1_M({\Omega^{J+1} \times \R^{(J+1)d}}))} \leq \frac{1}{2}\|(1 + |v_j|^2)\, \hrho_\alpha\|_{L^\infty(0,T;L^1_M({\Omega^{J+1} \times \R^{(J+1)d}}))} \leq C,
\end{align*}
where $C$ is a positive constant, independent of $\alpha$, which then implies that
\begin{align*}
{\rm Q}_2 & \leq C\|\varphi\|_{L^2(0,T;W^{s,2}(\Omega^{J+1} \times \R^{(J+1)d}))} \qquad \forall\, \varphi \in L^2(0,T;W^{s,2}_*(\Omega^{J+1} \times \R^{(J+1)d})),\quad
s> (J+1)d + 1,
\end{align*}
where $C$ is a positive constant, independent of $\alpha \in (0,1]$.

It remains to bound ${\rm Q}_4$; proceeding in the same way as in \eqref{eq:pre-lim-2}, we deduce that
\begin{align*}
{\rm Q}_4 &\leq \frac{1}{\eps} \sum_{j=1}^{J+1} \int_0^T
\|(|({\mathcal L}r)_j|+|u(r_j,\tau)|)\,\hrho_\alpha\|_{L^1_M(\Omega^{J+1} \times \R^{(J+1)d})}
\, \|\pdv \varphi\|_{L^\infty(\Omega^{J+1} \times \R^{(J+1)d})} \dd \tau \\
& \leq C(\epsilon, J, T)(1 + \|u\|_{L^2(0,T;L^\infty(\Omega))})
\left(\sum_{j=1}^{J+1} \int_0^T \|\pdv \varphi\|_{L^\infty(\Omega^{J+1} \times \R^{(J+1)d})}^2 \dd \tau\right)^{\frac{1}{2}}.
\nonumber
\end{align*}
Thus we have shown that
\begin{align*}
{\rm Q}_4 & \leq C\|\varphi\|_{L^2(0,T;W^{s,2}(\Omega^{J+1} \times \R^{(J+1)d}))} \qquad \forall\, \varphi \in L^2(0,T;W^{s,2}_*(\Omega^{J+1} \times \R^{(J+1)d})),\quad
s> (J+1)d + 1,
\end{align*}
where $C$ is a positive constant, independent of $\alpha \in (0,1]$.

\smallskip

By collecting the bounds on ${\rm Q}_1, \dots, {\rm Q}_4$, we have that
\begin{align*}
\left|\int_0^T \big\langle M\,\pd_\tau\hrho_\alpha(\cdot,\cdot,\tau),\varphi(\cdot,\cdot,\tau)\big\rangle \dd \tau\right|
&\leq C\|\varphi\|_{L^2(0,T;W^{s,2}(\Omega^{J+1} \times \R^{(J+1)d}))} \\
&\qquad \forall\, \varphi \in L^2(0,T;W^{s,2}_*(\Omega^{J+1} \times \R^{(J+1)d})),\quad
s> (J+1)d + 1,
\end{align*}
where $C$ is a positive constant, independent of $\alpha \in (0,1]$, which then implies the following uniform
bound on the time derivative of $\hrho_\alpha$:
\begin{align}\label{eq:time-bound}
\|M\,\pd_t \hrho_\alpha\|_{L^2(0,T;(W^{s,2}_*(\Omega^{J+1} \times \R^{(J+1)d}))')} \leq C,\qquad s> (J+1)d + 1,
\end{align}
where $C$ is a positive constant, independent of $\alpha \in (0,1]$.

\smallskip

For future reference, we collect here the various uniform bounds we have derived on $\hrho_\alpha$, $\alpha \in (0,1]$:
\begin{subequations}
\begin{alignat}{2}
\mathcal{F}(\hrho_\alpha) \quad & \mbox{is bounded in $L^\infty(0,T;L^1_M(\Omega^{J+1} \times \R^{(J+1)d}))$}, \label{eq:bound1}\\
\sum_{j=1}^{J+1} |\pdv \sqrt{\hrho_\alpha}|^2 \quad & \mbox{is bounded in $L^1(0,T;L^1_M(\Omega^{J+1} \times \R^{(J+1)d}))$},\label{eq:bound2}\\
\alpha \sum_{j=1}^{J+1} |\pdr \sqrt{\hrho_\alpha}|^2 \quad & \mbox{is bounded in $L^1(0,T;L^1_M(\Omega^{J+1} \times \R^{(J+1)d}))$},\label{eq:bound3}\\
M\,\pd_t \hrho_\alpha \quad & \mbox{is bounded in $L^2(0,T;(W^{s,2}_*(\Omega^{J+1} \times \R^{(J+1)d}))')$},\quad s>(J+1)d + 1,\label{eq:bound4}\\
\hrho_\alpha \geq 0\quad  & \mbox{and$\quad\|\hrho_\alpha(t)\|_{L^1_M(\Omega^{J+1} \times \R^{(J+1)d})} = \|\hrho_{0}\|_{L^1_M(\Omega^{J+1} \times \R^{(J+1)d}),}\quad$  $t \in [0,T]$},
\label{eq:bound5}\\
(1 + |v_j|^2)\, \hrho_\alpha \quad & \mbox{is bounded in $L^\infty(0,T;L^1_M(\Omega^{J+1} \times \R^{(J+1)d}))$,\quad $j=1,\dots,J+1$,} \label{eq:bound6}\\
|({\mathcal L}r)_j+u(r_j,\tau)|\,\hrho_\alpha \quad & \mbox{is bounded in $L^2(0,T;L^1_M(\Omega^{J+1} \times \R^{(J+1)d}))$,\quad $j=1,\dots,J+1$.} \label{eq:bound7a}
\end{alignat}
By writing $\hrho = (\sqrt{\hrho\,}\,)^2$, it then also follows from \eqref{eq:bound2}, \eqref{eq:bound3} and \eqref{eq:bound5} that
\begin{alignat}{2}
\nabla_v \hrho_\alpha \quad & \mbox{is bounded in $L^2(0,T;L^1_M(\Omega^{J+1} \times \R^{(J+1)d}))$}, \label{eq:bound7}\\
\alpha^{\frac{1}{2}} \,\nabla_r \hrho_\alpha \quad & \mbox{is bounded in $L^2(0,T;L^1_M(\Omega^{J+1} \times \R^{(J+1)d}))$}, \label{eq:bound8}
\end{alignat}
\end{subequations}
where $\nabla_v := (\pd_{v_1}^{\rm T},\dots, \pd_{v_{J+1}}^{\rm T})^{\rm T}$ and $\nabla_r := (\pd_{r_1}^{\rm T},\dots, \pd_{r_{J+1}}^{\rm T})^{\rm T}$
are $(J+1)d$-component column vectors.

We proceed by considering the Maxwellian-weighted Orlicz space $L^\Phi_M(\Omega^{J+1} \times \R^{(J+1)d})$, with Young's function $\Phi(r) = \mathcal{F}(1+|r|)$
(cf. Kufner, John \& Fu\v{c}ik \cite{KJF}, Sec. 3.18.2). This has a separable predual $E^\Psi_M(\Omega^{J+1} \times \R^{(J+1)d})$, with Young's function
 $\Psi(r) = {\rm e}^{|r|} - |r| - 1$; the (Banach) space $E^\Psi_M(\Omega^{J+1} \times \R^{(J+1)d})$ is defined as the closure, in the norm of the Orlicz space $L^\Psi_M(\Omega^{J+1} \times \R^{(J+1)d})$, of the set of all real-valued bounded measurable functions defined on $\Omega^{J+1} \times \R^{(J+1)d}$. As there exists a constant $K$ such that $\mathcal{F}(1+r) \leq K (1+\mathcal{F}(r))$ for all $r \geq 0$, it
 follows from \eqref{eq:bound1} that the sequence $(\mathcal{F}(1+\hrho_\alpha))_{\alpha >0}$ is bounded in $L^\infty(0,T;L^1_M(\Omega^{J+1} \times \R^{(J+1)d}))$. Hence,
 $\hrho_\alpha$ is bounded in $L^\infty(0,T;L^\Phi_M(\Omega^{J+1} \times \R^{(J+1)d})) = L^\infty(0,T;(E^\Psi_M(\Omega^{J+1} \times \R^{(J+1)d}))')=[L^1(0,T;E^\Psi_M(\Omega^{J+1} \times \R^{(J+1)d}))]'$.
 By the Banach--Alaoglu theorem, there exists a subsequence (not indicated) of the sequence $(\hrho_\alpha)_{\alpha>0}$ and a \begin{equation}\label{eq:Linfty-orlicz}
 \hrho \in
 L^\infty(0,T;L^\Phi_M(\Omega^{J+1} \times \R^{(J+1)d}))\qquad (\mbox{whereby also $~\mathcal{F}(\hrho) \in L^\infty(0,T;L^1(\Omega^{J+1} \times \R^{(J+1)d}))$~})
 \end{equation}
such that
\begin{alignat}{2}\label{eq:weak-orl1}
\hrho_\alpha &\rightharpoonup \hrho &&\quad \mbox{weakly$^*$ in $L^\infty(0,T;L^\Phi_M(\Omega^{J+1} \times \R^{(J+1)d})) = L^\infty(0,T;(E^\Psi_M(\Omega^{J+1} \times \R^{(J+1)d}))')$}.
\end{alignat}
As, by definition, $L^\infty(\Omega^{J+1} \times \R^{(J+1)d}) \subset E^\Psi_M(\Omega^{J+1} \times \R^{(J+1)d})$, it follows in particular that
\begin{alignat}{2}\label{eq:weak-orl2}
\hrho_\alpha &\rightharpoonup \hrho &&\quad \mbox{weakly$$ in $L^p(0,T;L^1_M(\Omega^{J+1} \times \R^{(J+1)d}))$}\qquad \forall\, p \in [1,\infty).
\end{alignat}

The convergence results \eqref{eq:alphabound-1a}--\eqref{eq:alphabound-4a} and \eqref{eq:bound1}--\eqref{eq:bound8} now imply the existence of
\[
\hrho \in L^\infty(0,T;L^2_M(\Omega^{J+1} \times \R^{(J+1)d})), \qquad \hrho\geq 0,
\]
with
\[\nabla_v \hrho \in L^2(0,T;L^2_M(\Omega^{J+1} \times \R^{(J+1)d}))\quad \mbox{and}\quad
M\,\pd_t \hrho \in L^2(0,T;(W^{s,2}(\Omega^{J+1} \times \R^{(J+1)d}))'),\quad s>(J+1)d + 1,
\]
such that
{\small
\begin{alignat*}{2}
\hrho_\alpha &\rightharpoonup \hrho &&\quad \mbox{weakly$^*$ in $L^\infty(0,T;L^2_M(\Omega^{J+1} \times \R^{(J+1)d}))$}, \\
\nabla_v \hrho_\alpha &\rightharpoonup \nabla_v \hrho &&\quad \mbox{weakly in $L^2(0,T;L^2_M(\Omega^{J+1} \times \R^{(J+1)d}))$},\\
\alpha \nabla_r \hrho_\alpha &\rightarrow 0 &&\quad \mbox{strongly in $L^2(0,T;L^2_M(\Omega^{J+1} \times \R^{(J+1)d}))$},\\
M\,\pd_t \hrho_\alpha &\rightharpoonup M\, \pd_t \hrho &&\quad \mbox{weakly in $L^2(0,T;(W^{s,2}(\Omega^{J+1} \times \R^{(J+1)d}))'),\!\quad\! s > (J+1)d\! +\! 1$},\\
v_j\, \hrho_\alpha &\rightharpoonup v_j\, \hrho &&\quad
\mbox{weakly in $L^2(0,T;L^1_M(\Omega^{J+1} \times \R^{(J+1)d}))$,\quad $j=1,\dots,J+1$},\\
(({\mathcal L}r)_j+u(r_j,\tau))\,\hrho_\alpha &\rightharpoonup (({\mathcal L}r)_j+u(r_j,\tau))\,\hrho &&\quad
\mbox{weakly in $L^2(0,T;L^1_M(\Omega^{J+1} \times \R^{(J+1)d}))$,\quad $j=1,\dots,J+1$.}
\end{alignat*}
}

Using these convergence results, passage to the limit $\alpha \rightarrow 0_+$ in \eqref{eq:weak-duality-b} implies the existence
of
\[
\hrho \in L^\infty(0,T;L^2_M(\Omega^{J+1} \times \R^{(J+1)d})),\qquad \hrho \geq 0,
\]
with
\[\nabla_v \hrho \in L^2(0,T;L^2_M(\Omega^{J+1} \times \R^{(J+1)d}))\;\; \mbox{and}\;\;
M\,\pd_t \hrho \in L^2(0,T;(W^{s,2}(\Omega^{J+1} \times \R^{(J+1)d}))'),\quad s>(J+1)d + 1,
\]
satisfying the following weak form of the Fokker--Planck equation: for all $t \in (0,T]$,
\begin{align}\label{eq:weak-duality-bb}
&\int_0^t \big\langle M\,\pd_\tau\hrho(\cdot,\cdot,\tau),\varphi(\cdot,\cdot,\tau)\big\rangle \dd \tau
+ \frac{\beta^2}{\eps^2}\left(\sum_{j=1}^{J+1} \int_0^t \int_{\Omega^{J+1}} \int_{\R^{(J+1)d}} M(v)\,\pdv \hrho \cdot \pdv \varphi \dd v \dd r \dd \tau \right)\nonumber\\
&\qquad\quad- \frac{1}{\eps} \left(\sum_{j=1}^{J+1} \int_0^t \int_{\Omega^{J+1}} \int_{\R^{(J+1)d}} M(v)\, v_j \hrho\cdot \pdr \varphi \dd v \dd r \dd \tau \right)\nonumber\\
&\qquad\quad- \frac{1}{\eps} \left(\sum_{j=1}^{J+1} \int_0^t \int_{\Omega^{J+1}} \int_{\R^{(J+1)d}} M(v)\,(({\mathcal L}r)_j+u(r_j,\tau))\,\hrho\cdot \pdv \varphi \dd v \dd r \dd \tau \right)
= 0 \nonumber \\
&\hspace{0.6in} \forall\, \varphi \in L^2(0,T; W^{1,2}_{*,M}(\Omega^{J+1} \times \R^{(J+1)d})\cap W^{s,2}_{*}(\Omega^{J+1} \times \R^{(J+1)d})), \quad s>(J+1)d+1.
\end{align}

\bigskip

It remains to discuss the attainment of the initial condition by $\hrho$. To this end, we require the following lemma.

\begin{lemma}\label{lemma-strauss}
Let $X$ and $Y$ be Banach spaces.
\begin{itemize}
\item[(a)] If the space $X$ is reflexive and is continuously embedded in the space $Y$, then
$$L^\infty(0,T; X) \cap \mathcal{C}_{w}([0,T]; Y) = \mathcal{C}_{w}([0,T];X).$$
\item[(b)] If $X$ has separable predual $E$ and $Y$ has predual $F$ such that
$F$ is continuously embedded in $E$, then
$$L^\infty(0,T;X) \cap \mathcal{C}_{w\ast}([0,T]; Y) = \mathcal{C}_{w\ast}([0,T];X).$$
\end{itemize}
\end{lemma}

Part (a) is due to Strauss \cite{Strauss} (cf. Lions \& Magenes \cite{Lions-Magenes}, Lemma 8.1, Ch. 3, Sec. 8.4); part (b) is proved analogously, {\em via}
the sequential Banach--Alaoglu theorem.

\smallskip

We shall prove that $\hrho \in \mathcal{C}_{w}([0,T];L^1_M(\Omega^{J+1} \times \R^{(J+1)d}))$. Let us first recall that, thanks to
\eqref{eq:Linfty-orlicz},
\begin{equation*}
 \hrho \in
 L^\infty(0,T;L^\Phi_M(\Omega^{J+1} \times \R^{(J+1)d}))\qquad (\mbox{whereby also $~\mathcal{F}(\hrho) \in L^\infty(0,T;L^1(\Omega^{J+1} \times \R^{(J+1)d}))$~}),
\end{equation*}
and, also,
\[\hrho \in W^{1,2}(0,T; M^{-1}(W^{s,2}(\Omega^{J+1} \times \R^{(J+1)d}))'), \quad s>(J+1)d+1.\]
We then apply Lemma \ref{lemma-strauss}(b) by taking:
\begin{itemize}
\item
$X:=L^\Phi_M(\Omega^{J+1} \times \R^{(J+1)d})$, the Maxwellian weighted Orlicz space with Young's function $$\Phi(r) = \mathcal{F}(1+|r|)$$
(cf. Kufner, John \& Fu\v{c}ik \cite{KJF}, Sec. 3.18.2) whose separable
predual $$E:=E^\Psi_M(\Omega^{J+1} \times \R^{(J+1)d})$$ has Young's function $$\Psi(r) = {\rm e}^{|r|} - |r| - 1;$$
\item and
$Y := M^{-1}(W^{s,2}(\Omega^{J+1} \times \R^{(J+1)d}))'$ whose predual with respect to the duality pairing
$$\langle M \cdot , \cdot \rangle_{W^{s,2}(\Omega^{J+1} \times \R^{(J+1)d})}, \quad s>(J+1)d+1,$$
is
$$F:=W^{s,2}(\Omega^{J+1} \times \R^{(J+1)d}), \quad s> (J+1)d + 1,$$
\end{itemize}
and noting that $\mathcal{C}_{w\ast}([0;T]; L^\Phi_M(\Omega^{J+1} \times \R^{(J+1)d})) \subset \mathcal{C}_{w}([0,T]; L^1_M(\Omega^{J+1} \times \R^{(J+1)d}))$.
This last inclusion and that $F \hookrightarrow E$ are proved by adapting
Def. 3.6.1. and Thm. 3.2.3 in Kufner, John \& Fu\v{c}ik \cite{KJF} to the measure $M(v)\dd v \dd r$ to show that
$L^\infty(\Omega^{J+1} \times \R^{(J+1)d}) \hookrightarrow L^\Xi_M(\Omega^{J+1} \times \R^{(J+1)d})$ for any Young's function $\Xi$,
and then adapting Theorem 3.17.7 {\em ibid.} to deduce that
$$F \hookrightarrow L^\infty(\Omega^{J+1} \times \R^{(J+1)d}) \hookrightarrow E^\Psi_M(\Omega^{J+1} \times \R^{(J+1)d})= E.$$
The abstract framework in Temam \cite{Temam}, Ch.\ 3, Sec.\ 4 then implies that $\hrho$ satisfies
$\hrho(\cdot,\cdot,0) = \hrho_0(\cdot,\cdot)$ in the sense of $\mathcal{C}_w([0,T]; L^1_M(\Omega^{J+1} \times \R^{(J+1)d}))$.

\bigskip

By taking $\varphi \equiv 1$ in \eqref{eq:weak-duality-bb}, we have that
\[\langle M\, \hrho(\cdot,\cdot,t),\, 1\big\rangle - \langle M\, \hrho(\cdot,\cdot,0),\, 1\big\rangle = \int_0^t \big\langle M\,\pd_\tau\hrho(\cdot,\cdot,\tau),\, 1\big\rangle \dd \tau = 0.\]
Hence,
\[ \langle M\, \hrho(\cdot,\cdot,t),\, 1\big\rangle - \langle M\, \hrho_0(\cdot,\cdot),\, 1\big\rangle = 0,\]
and this then gives
\[ \int_{\Omega^{J+1} \times \R^{(J+1)d}} M(v)\,\hrho(r,v,t)\dd r \dd v = \int_{\Omega^{J+1} \times \R^{(J+1)d}} M(v)\,\hrho_0(r,v)\dd r \dd v = 1
\qquad \forall\, t \in (0,T],\]
which, together with $\hrho \geq 0$, implies that $$\varrho=M\,\hrho \in L^\infty(0,T;L^1(\Omega^{J+1} \times \R^{(J+1)d}))$$ is a probability density function,
as required.


Noting that the function $\mathcal{F}$ is nonnegative and convex, for each fixed $\gamma \in (0,1]$ the first term on the left-hand side of \eqref{eq:pre-lim-1} is weakly lower-semicontinuous
in $L^1_M(\Omega^{J+1} \times \R^{(J+1)d})$, as $\alpha \rightarrow 0_+$ (cf. Theorem 3.20 in \cite{dacorogna}). Similarly,
since $\xi \in \mathbb{R} \mapsto |\xi|^2$, with $y \geq 0$, is a nonnegative convex function,
we have weak lower semicontinuity of the second term
on the left-hand side of \eqref{eq:pre-lim-1} (cf. Corollary 3.24 in \cite{dacorogna}) for each $\gamma \in (0,1]$. By passing to the limit $\alpha \rightarrow 0_+$ in \eqref{eq:pre-lim-1}, and then passing to the limit $\gamma \rightarrow 0_+$
using the dominated convergence theorem in the first term on the left-hand side and the monotone convergence
theorem in the second term on the left-hand side, we deduce that $\hrho$ satisfies the following energy inequality:
\begin{align}\label{eq:energy-hrho}
&\int_{\Omega^{J+1}} \int_{\R^{(J+1)d}} M(v)\,\mathcal{F}(\hrho(t)) \dd v \dd r
 + \frac{\beta^2}{2\eps^2} \sum_{j=1}^{J+1} \int_0^t \int_{\Omega^{J+1}} \int_{\R^{(J+1)d}} M(v)\,\frac{|\pdv \hrho\,|^2}{\hrho}
\, \dd v \dd r \dd \tau
\nonumber\\
&\leq  \int_{\Omega^{J+1}}  \int_{\R^{(J+1)d}} M(v)\,\mathcal{F}(\hrho_{0}) \dd v \dd r + \frac{16}{\beta}\,(J+1)d\, L^2\, T +  \frac{1}{\beta}(J+1)\, \|u\|_{L^2(0,T;L^\infty({\Omega}))}^2.
\end{align}

\noindent
It is important to note here that, although we had supposed that $\hrho_0 \in L^2_M(\Omega^{J+1} \times \R^{(J+1)d};\R_{\geq 0})$, the upper bound in \eqref{eq:energy-hrho} only depends
on the $L^1_M(\Omega^{J+1} \times \R^{(J+1)d})$ norm of $\mathcal{F}(\hrho_0)$, the $L^2(0,T;L^{\infty}({\Omega})^d)$ norm of $u$, and the constants $d, \beta, J, L, T$, all of which are independent of $\eps$.

\section{Existence of solutions to the coupled Oseen--Fokker--Planck system}
\label{sec:coupled}

We now return to the full system stated in the Introduction,
our objective being to prove the existence of large-data global weak solutions to the coupled Oseen--Fokker--Planck
system.
To this end, we formulate an iterative process, by defining the sequence of functions $(\uk,\hrhok)$,
for $k=1,2,\dots$, as follows.
We set $u^{(1)} \equiv 0$. Given a divergence-free $\uk \in L^2(0,T;W^{1,\sigma}_0(\Omega)^d)$, for some $k \geq 1$ and $\sigma>d$, we define $\hrhok$ as the weak solution (in a sense to be made precise below) of the Fokker--Planck equation:
\begin{alignat}{2}
\label{eq:FP-eq-k}
M\pd_t \hrhok - \frac{\beta^2}{\eps^2}\left(\sum_{j=1}^{J+1} \pdv \cdot (M \pdv \hrhok) \right) + \frac{1}{\eps}
\left(\sum_{j=1}^{J+1} M v_j\cdot \pdr \hrhok + (({\mathcal L}r)_j+\uk(r_j,t))\cdot\pdv (M\hrhok) \right) = 0,\\
~\hfill \mbox{for all $(r,v,t) \in \Omega^{J+1} \times \R^{(J+1)d} \times (0,T]$},\nonumber\\
\hrhok(r,v,0)=\hrho_0^{(k)}(r,v)~\hfill \qquad \mbox{for all $(r,v) \in \Omega^{J+1} \times \R^{(J+1)d}$},
\label{eq:FP-ini-k}
\end{alignat}
subject to a (weakly imposed) specular boundary condition with respect to the independent variable $r$. The precise specification of
the initial datum $\hrhok_0$ in terms of $\hrho_0$ will be detailed in the next subsection. Having determined
$\hrhok$ from this problem, we shall find the next velocity field iterate $u^{(k+1)}$ by solving, with $\hrhok$ fixed, the Oseen system
(cf. \eqref{eq:oseen-k} below). We shall prove that one can extract a subsequence from the sequence of iterates
$((u^{(k)},\hrhok))_{k \geq 1}$, which converges to a solution $(u,\hrho)$ of the coupled Oseen--Fokker--Planck system in the limit
of $k\rightarrow \infty$.

\subsection{Definition of the initial data}
First, we define the sequence of initial data $(\hrho_0^{(k)})_{k\geq 1}$ appearing in \eqref{eq:FP-ini-k}.
Given $\hrho_0$ as in \eqref{eq:ini-cond}, and letting
\[G_k(s) := \frac{s}{1+ k^{-\frac{1}{4}}\sqrt{s}},\qquad s \in [0,\infty),\]
we define
\[ \hrho_0^{(k)}:= G_k(\hrho_0),\qquad k = 1,2,\dots .\]
The purpose of this construction, which can be seen as a renormalization of the initial
datum $\hrho_0$, is to ensure that, under the original hypotheses, \eqref{eq:ini-cond}, on $\hrho_0$, the functions
$\hrho_0^{(k)}$ thus defined possess the following properties:
\begin{subequations}
\begin{alignat}{2}
\hrho_0^{(k)} & \in L^1_M(\Omega^{J+1} \times \R^{(J+1)d};\R_{\geq 0}),\qquad &&\mbox{for each fixed $k \geq 1$},\label{eq:inicond-a0}\\
M\mathcal{F}(\hrho_0^{(k)}) &\in L^1(\Omega^{J+1} \times \R^{(J+1)d};\R_{\geq 0}),\qquad &&\mbox{for each fixed $k \geq 1$},\label{eq:inicond-a00}\\
\int_{\Omega^{J+1} \times \R^{(J+1)d}} &M(v)\, \hrho_0^{(k)} \dd r \dd v \leq 1,\qquad &&\mbox{for each fixed $k \geq 1$},\label{eq:inicond-a000}\\
\hrho_0^{(k)} & \in L^2_M(\Omega^{J+1} \times \R^{(J+1)d};\R_{\geq 0})  \qquad && \mbox{for each fixed $k \geq 1$},\label{eq:inicond-a}
\end{alignat}
and, possibly for a subsequence only (not indicated),
\begin{alignat}{2}
\hrho_0^{(k)} &\rightarrow \hrho_0 \qquad & \mbox{strongly in $L^1_M(\Omega^{J+1} \times \R^{(J+1)d})$ as $k \rightarrow \infty$},\label{eq:inicond-c}\\
\mathcal{F}(\hrho_0^{(k)}) &\rightarrow \mathcal{F}(\hrho_0) \qquad & \mbox{strongly in $L^1_M(\Omega^{J+1} \times \R^{(J+1)d})$ as $k \rightarrow \infty$}.\label{eq:inicond-d}
\end{alignat}
\end{subequations}
We shall now proceed to show that these properties do indeed hold; having done so, we shall explain their relevance in the
proof of our main result.

\smallskip

That $\hrho_0^{(k)} \geq 0$ for all $k \geq 1$ is a direct consequence of its definition and the assumed nonnegativity of
$\hrho_0$ (cf. \eqref{eq:ini-cond}).
By \eqref{eq:ini-cond}, and noting that $0 \leq G_k(s) \leq s$,
\eqref{eq:inicond-a0} and \eqref{eq:inicond-a000} directly follow.
The assertion \eqref{eq:inicond-a00} is also immediate by noting that $\mathcal{F}(\hrho_0^{(k)}) =
\mathcal{F}(G_k(\hrho_0)) \leq \max\{1,\mathcal{F}(\hrho_0)\}$.
We therefore proceed to prove the inclusion \eqref{eq:inicond-a}.

We have, for each $k \geq 1$, that
\begin{equation}\label{eq:inicond-a1}
\|\hrho_0^{(k)}\|^2_{L^2_M(\Omega^{J+1} \times \R^{(J+1)d})} = \|G_k(\hrho_0)\|^2_{L^2_M(\Omega^{J+1} \times \R^{(J+1)d})} \leq k^{\frac{1}{2}} \|\hrho_0\|_{L^1_M(\Omega^{J+1} \times \R^{(J+1)d})}.
\end{equation}
Thus we have verified \eqref{eq:inicond-a}.

Next,
\begin{align}\label{eq:inicond-c1}
\|\hrho_0^{(k)} - \hrho_0\|_{L^1_M(\Omega^{J+1} \times \R^{(J+1)d})} &= \|G_k(\hrho_0) - \hrho_0\|_{L^1_M(\Omega^{J+1} \times \R^{(J+1)d})}
\leq \left\|\frac{\hrho_0\, k^{-\frac{1}{4}}\,\sqrt{\hrho_0}}{1+k^{-\frac{1}{4}}\sqrt{\hrho_0}}\right\|_{L^1_M(\Omega^{J+1} \times \R^{(J+1)d})}.
\end{align}
Clearly,
\[ \frac{\hrho_0\, \,\sqrt{\hrho_0}}{k^{\frac{1}{4}} + \sqrt{\hrho_0}} \rightarrow 0
\qquad \mbox{a.e. on $\Omega^{J+1} \times \R^{(J+1)d}$}.\]
Also, trivially,
\[0 \leq \frac{\hrho_0\, k^{-\frac{1}{4}}\,\sqrt{\hrho_0}}{1+k^{-\frac{1}{4}}\sqrt{\hrho_0}} \leq \hrho_0
\in L^1_M(\Omega^{J+1} \times \R^{(J+1)d}).\]
Hence, by the dominated convergence theorem,
\[ \lim_{k \rightarrow \infty} \left\|\frac{\hrho_0\, k^{-\frac{1}{4}}\,\sqrt{\hrho_0}}{1+k^{-\frac{1}{4}}\sqrt{\hrho_0}}\right\|_{L^1_M(\Omega^{J+1} \times \R^{(J+1)d})} = 0.\]
By passing to the limit $k \rightarrow \infty$ in \eqref{eq:inicond-c1} we then deduce that
\[ \lim_{k \rightarrow \infty} \|\hrho_0^{(k)} - \hrho_0\|_{L^1_M(\Omega^{J+1} \times \R^{(J+1)d})} = 0.\]
Thus we have shown \eqref{eq:inicond-c}.

To prove \eqref{eq:inicond-d}, thanks to \eqref{eq:inicond-c}, it suffices to show that, as $k \rightarrow \infty$,
\[ \frac{\hrho_0}{1 + k^{-\frac{1}{4}}\sqrt{\hrho_0}} \log \left( \frac{\hrho_0}{1 + k^{-\frac{1}{4}}\sqrt{\hrho_0}}\right) \rightarrow \hrho_0 \log \hrho_0
\qquad \mbox{strongly in $L^1_M(\Omega^{J+1} \times \R^{(J+1)d})$}.\]
To this end we write
\begin{align}\label{eq:proof-of-d}
\frac{\hrho_0}{1 + k^{-\frac{1}{4}}\sqrt{\hrho_0}} \log \left( \frac{\hrho_0}{1 + k^{-\frac{1}{4}}\sqrt{\hrho_0}}\right) = \frac{\hrho_0 \log \hrho_0}{1 + k^{-\frac{1}{4}}\sqrt{\hrho_0}} - \hrho_0\, \frac{\log \left(1 + k^{-\frac{1}{4}}\sqrt{\hrho_0}\right)}{1 + k^{-\frac{1}{4}}\sqrt{\hrho_0}}.
\end{align}
We shall show below that the first fraction on the right-hand side of the equality \eqref{eq:proof-of-d} converges to $\hrho_0 \log \hrho_0$
strongly in $L^1_M(\Omega^{J+1} \times \R^{(J+1)d})$ while the second fraction converges to $0$ strongly in
$L^1_M(\Omega^{J+1} \times \R^{(J+1)d})$, and that will complete the proof of \eqref{eq:inicond-d}.  Indeed,
that the second fraction on the right-hand side of \eqref{eq:proof-of-d} converges to $0$ strongly in $L^1_M(\Omega^{J+1} \times \R^{(J+1)d})$ follows directly from the dominated
convergence theorem by noting that
\[ \left|\frac{\log \left(1 + k^{-\frac{1}{4}}\sqrt{\hrho_0}\right)}{1 + k^{-\frac{1}{4}}\sqrt{\hrho_0}}\right|\leq \frac{1}{{\rm e}}\qquad\mbox{and}\qquad
\lim_{k \rightarrow \infty} \frac{\log \left(1 + k^{-\frac{1}{4}}\sqrt{\hrho_0}\right)}{1 + k^{-\frac{1}{4}}\sqrt{\hrho_0}} = 0\quad\mbox{a.e. on $\Omega^{J+1} \times \R^{(J+1)d}$.}\]
Focusing now on the first fraction on the right-hand side of \eqref{eq:proof-of-d}, we consider
\begin{align*}
\frac{\hrho_0 \log \left(\hrho_0\right)}{1 + k^{-\frac{1}{4}}\sqrt{\hrho_0}} - \hrho_0 \log \hrho_0
= - \hrho_0 \log \hrho_0 \frac{k^{-\frac{1}{4}}\sqrt{\hrho_0}}{1 + k^{-\frac{1}{4}}\sqrt{\hrho_0}}.
\end{align*}
The term on the right-hand side of this equality converges to $0$ strongly in $L^1_M(\Omega^{J+1} \times \R^{(J+1)d})$ as $k \rightarrow \infty$, thanks to the dominated convergence theorem.
That completes the proof of \eqref{eq:inicond-d}.

The significance of \eqref{eq:inicond-a0}--\eqref{eq:inicond-d} is that these are precisely the properties
which we used in the previous section to prove, for a fixed divergence-free velocity field $u$,
contained in $L^2(0,T;W^{1,\sigma}_0(\Omega)^d)$, $\sigma>d$, the existence of a solution $\hrho$ to the Fokker--Planck equation, subject to such initial data for $\hrho$.

\subsection{Existence of a solution to the initial-value problem \eqref{eq:FP-eq-k}, \eqref{eq:FP-ini-k}}
Having verified all of \eqref{eq:inicond-a0}--\eqref{eq:inicond-d}, the arguments
developed in Section \ref{sec:FP}
imply the existence of a weak solution $\hrho^{(k)}$ to the problem \eqref{eq:FP-eq-k} for a given
divergence-free $u^{(k)} \in L^2(0,T;W^{1,\sigma}_0(\Omega)^d)$ with $\sigma>d$. More precisely, there exists a
\[
\hrhok \in \mathcal{C}_w([0,T];L^1_M(\Omega^{J+1} \times \R^{(J+1)d};\R_{\geq 0})),
\]
with
\[\nabla_v \hrhok \in L^2(0,T;L^1_M(\Omega^{J+1} \times \R^{(J+1)d})), \quad
M\,\pd_t \hrhok \in L^2(0,T;(W^{s,2}(\Omega^{J+1} \times \R^{(J+1)d}))'),\quad s>(J+1)d + 1,
\]
and satisfying
\begin{alignat*}{2}
v_j\, \hrhok &\in L^2(0,T;L^1_M(\Omega^{J+1} \times \R^{(J+1)d})),\quad j=1,\dots,J+1,\\
(({\mathcal L}r)_j+u^{(k)}(r_j,\tau))\,\hrhok &\in L^2(0,T;L^1_M(\Omega^{J+1} \times \R^{(J+1)d})),\quad j=1,\dots,J+1,
\end{alignat*}
such that,  for all $t \in (0,T]$:
\begin{align}\label{eq:weak-duality-bb-k}
&\int_0^t \big\langle M\,\pd_\tau\hrhok(\cdot,\cdot,\tau),\varphi(\cdot,\cdot,\tau)\big\rangle \dd \tau
+ \frac{\beta^2}{\eps^2}\left(\sum_{j=1}^{J+1} \int_0^t \int_{\Omega^{J+1}} \int_{\R^{(J+1)d}} M(v)\,\pdv \hrhok \cdot \pdv \varphi \dd v \dd r \dd \tau \right)\nonumber\\
&\qquad- \frac{1}{\eps} \left(\sum_{j=1}^{J+1} \int_0^t \int_{\Omega^{J+1}} \int_{\R^{(J+1)d}} M(v)\, v_j \hrhok\cdot \pdr \varphi \dd v \dd r \dd \tau \right)\nonumber\\
&\qquad- \frac{1}{\eps} \left(\sum_{j=1}^{J+1} \int_0^t \int_{\Omega^{J+1}} \int_{\R^{(J+1)d}} M(v)\,(({\mathcal L}r)_j+u^{(k)}(r_j,\tau))\,\hrhok\cdot \pdv \varphi \dd v \dd r \dd \tau \right)
= 0 \nonumber \\
&\hspace{0.5in} \forall\, \varphi \in L^2(0,T; W^{1,2}_{*,M}(\Omega^{J+1} \times \R^{(J+1)d})\cap W^{s,2}_{*}(\Omega^{J+1} \times \R^{(J+1)d})), \quad s>(J+1)d+1.
\end{align}
Furthermore $\hrhok(\cdot,\cdot,0)=\hrhok_0(\cdot,\cdot)$ in the sense of $\mathcal{C}_w([0,T];L^1_M(\Omega^{J+1} \times \R^{(J+1)d};\R_{\geq 0}))$, and
\[ \int_{\Omega^{J+1} \times \R^{(J+1)d}} M(v)\,\hrhok(r,v,t)\dd r \dd v = \int_{\Omega^{J+1} \times \R^{(J+1)d}} M(v)\,\hrhok_0(r,v)\dd r \dd v \leq 1, \quad t \in (0,T].\]
In addition, thanks to \eqref{eq:energy-hrho}, $\hrho^{(k)}$ satisfies the following energy inequality:
%
\begin{align}\label{eq:energy-hrho-k}
&\int_{\Omega^{J+1}} \int_{\R^{(J+1)d}} M(v)\,\mathcal{F}(\hrhok(t)) \dd v \dd r
 + \frac{\beta^2}{2\eps^2} \sum_{j=1}^{J+1} \int_0^t \int_{\Omega^{J+1}} \int_{\R^{(J+1)d}} M(v)\,\frac{|\pdv \hrhok|^2}{\hrho^{(k)}}
\, \dd v \dd r \dd \tau
\nonumber\\
&\leq  \int_{\Omega^{J+1}}  \int_{\R^{(J+1)d}} M(v)\,\mathcal{F}(\hrho_{0}^{(k)}) \dd v \dd r + \frac{16}{\beta}\,(J+1)d\, L^2\, T +  \frac{1}{\beta}(J+1)\, \|u^{(k)}\|_{L^2(0,T;L^\infty({\Omega}))}^2.
\end{align}
%

\noindent
It is important to note here that the upper bound in \eqref{eq:energy-hrho-k} only involves
the $L^1_M(\Omega^{J+1} \times \R^{(J+1)d})$ norm of $\mathcal{F}(\hrhok_0)$, which,
thanks to \eqref{eq:inicond-d}, converges to $\mathcal{F}(\hrho_0)$ strongly in $L^1_M(\Omega^{J+1} \times \R^{(J+1)d})$, as $k \rightarrow \infty$;
and on the $L^2(0,T;L^\infty(\Omega)^d)$ norm of $u^{(k)}$,
which we shall now bound by a constant, independent of $k$ and of $\eps$. Once we have done so,
\eqref{eq:energy-hrho-k}
will yield a uniform-in-$k$ (and $\eps$-uniform) bound on the $L^\infty(0,T;L^1_M(\Omega^{J+1} \times \R^{(J+1)d}))$
norm of $\mathcal{F}(\hrho^{(k)})$ and the $L^2(0,T;L^2_M(\Omega^{J+1} \times \R^{(J+1)d}))$
norm of $\nabla_v \sqrt{\hrhok}$, which will, together with the strong convergence of $\hrho^{(k)}$ to $\hrho$ in $L^1(0,T;L^1_M(\Omega^{J+1}\times \R^{(J+1)d}))$, which we shall also prove, yield the convergence results required to
pass to the limit in the weak form of \eqref{eq:FP-eq-k} as $k \rightarrow \infty$.

\medskip

\subsection{Existence of a solution to the Oseen system}
Having shown the existence of a solution $\hrhok$ to \eqref{eq:FP-eq-k}, \eqref{eq:FP-ini-k}
for a given divergence-free
$u^{(k)} \in L^2(0,T;W^{1,\sigma}_0(\Omega)^d)$ with $\sigma>d$,
we define $(u^{(k+1)},\pi^{(k+1)})$, with $u^{(k+1)} \in L^\infty(0,T;L^2(\Omega)^d) \cap
L^2(0,T;W^{1,2}_0(\Omega)^d)$, and $\pi^{(k+1)} \in \mathcal{D}'(0,T;L^2(\Omega)/\R)$
as the weak solution of the unsteady Oseen system:
\begin{alignat}{2}\label{eq:oseen-k}
\pd_t u^{(k+1)} + (b\cdot \nabla) u^{(k+1)} - \mu \triangle u^{(k+1)} + \nabla \pi^{(k+1)} &= \nabla \cdot \KK^{(k)} &&\qquad \mbox{for
$(x, t) \in \Omega \times (0,T]$}, \nonumber\\
\nabla \cdot u^{(k+1)} &= 0 &&\qquad \mbox{for $(x, t) \in \Omega \times (0, T]$},\\
u^{(k+1)}(x,0) & = u_0(x)&&\qquad \mbox{for $x \in \Omega$},\nonumber
\end{alignat}
where $u_0 \in W^{1-2/z,z}_0({\Omega})^d$, with $z=d+\vartheta$ for some $\vartheta \in (0,1)$, is divergence-free, and
\begin{align*}
\KK^{(k)}(x,t) &:=\frac{\int_{D^{J}\times \R^{(J+1)d}} \sum_{j=1}^J (F(q_j)\otimes q_j)\,M\,\hrhok\bigl(B(q,x),v,t\bigr)
\dd q  \dd v}
{\int_{D^{J}\times \R^{(J+1)d}} M\,\hrhok\,\bigl(B(q,x),v,t\bigr)
\dd q  \dd v},\qquad
(x,t) \in \Omega \times (0,T].
\end{align*}

Thanks to \eqref{eq:ce},
\begin{equation}\label{eq:Kbond}
\|\KK^{(k)}\|_{L^\infty(0,T;L^\infty(\Omega))} \leq C,
\end{equation}
where $C$ is a positive constant, independent of $k$. Thus, there exists a $\KK \in L^\infty(0,T;L^\infty(\Omega;\R^{d\times d}_{\rm symm}))$ (to be identified), and a subsequence, not indicated, such that
\begin{equation}\label{eq:oseen-k-K}
\KK^{(k)} \rightarrow \KK\qquad \mbox{weak$^*$ in $L^\infty(0,T;L^\infty(\Omega;\R^{d\times d}_{\rm symm}))$ as $k \rightarrow \infty$}.
\end{equation}

As $W^{1-2/z,z}_0(\Omega)^d \hookrightarrow L^2(\Omega)^d$ for $z=d+\vartheta$ for some $\vartheta \in (0,1)$, by standard arguments from the analysis of the incompressible Navier--Stokes equations (cf., for example, \cite{Temam}, Chpt. III)
we deduce from \eqref{eq:Kbond} that there exists a unique weak solution $(u^{(k+1)},\pi^{(k+1)})$ to the Oseen system
with  $u^{(k+1)} \in L^\infty(0,T;L^2(\Omega)^d)\cap L^2(0,T;W^{1,2}_0(\Omega)^d)$, and
\[\|u^{(k+1)}\|_{L^\infty(0,T;L^2(\Omega)) \cap L^2(0,T;W^{1,2}({\Omega}))} \leq C(1 + \|u_0\|_{L^2(\Omega)}),\]
where $C$ is independent of $k$. Hence, by interpolation,\footnote{By the Gagliardo--Nirenberg inequality,
$\|v\|_{L^4(\Omega)} \leq C \|v\|_{L^2(\Omega)}^{1/2} \|v\|_{W^{1,2}(\Omega)}^{1/2}$ for $d=2$, and
$\|v\|_{L^{10/3}(\Omega)} \leq C \|v\|_{L^2(\Omega)}^{2/5} \|v\|_{W^{1,2}(\Omega)}^{3/5}$ for $n=3$. Hence, by the application of H\"older's inequality,
$\|v\|_{L^4(0,T;L^4(\Omega))} \leq C \|v\|_{L^\infty(0,T;L^2(\Omega))}^{1/2}  \|v\|_{L^2(0,T;W^{1,2}(\Omega))}^{1/2}$
for $d=2$ and $\|v\|_{L^{10/3}(0,T;L^{10/3}(\Omega))} \leq C \|v\|_{L^\infty(0,T;L^2(\Omega))}^{2/5}  \| v\|_{L^2(0,T;W^{1,2}(\Omega))}^{3/5}$ for $d=3$.}
\[ \|u^{(k+1)}\|_{L^{\hat\sigma}(Q_T)} \leq C \quad \mbox{where}\quad \left\{ \begin{array}{ll}
                                                                     \mbox{$\hat\sigma=4$} & \mbox{when $d=2$},\\
                                                                     \mbox{$\hat\sigma=\frac{10}{3}$}
                                                                     & \mbox{when $d=3$,}\end{array}\right.\]
where $Q_T:=\Omega \times (0,T)$. Therefore, also,
\[ \|b \otimes u^{(k+1)}\|_{L^{\hat\sigma}(Q_T)} \leq C \quad \mbox{where}\quad \left\{ \begin{array}{ll}
                                                                     \mbox{$\hat\sigma=4$} & \mbox{when $d=2$},\\
                                                                     \mbox{$\hat\sigma=\frac{10}{3}$}
                                                                     & \mbox{when $d=3$.}\end{array}\right.\]

\begin{remark}\label{rem:relax}
We note here in passing that the regularity hypothesis $b \in L^\infty(0,T; L^\infty(\Omega)^d)$ that was used here
to deduce the last inequality can be weakened to assuming instead that $b \in L^s(0,T;L^s(\Omega)^d)$ for some $s>2d(d+2)/(2(d+2)-d^2)$, $d=2,3$. The latter weaker assumption on $b$ results in $\|b \otimes u^{(k+1)}\|_{L^{\hat\sigma}(Q_T)} \leq C$ for some $\hat\sigma>d$, which then still
suffices to draw the same conclusions to the ones below.
\end{remark}

Continuing with our stronger but simpler assumption
that $b \in L^\infty(0,T; L^\infty(\Omega)^d)$, we have that
\[ \|\KK^{(k)}- b\otimes u^{(k+1)}\|_{L^{\hat\sigma}(Q_T)} \leq C \quad \mbox{where}\quad \left\{ \begin{array}{ll}
                                                                     \mbox{$\hat\sigma=4$} & \mbox{when $d=2$},\\
                                                                     \mbox{$\hat\sigma=\frac{10}{3}$}
        & \mbox{when $d=3$.}\end{array}\right.\]
Clearly, $\hat\sigma=2 + \frac{4}{d}$, $d=2,3$.

We shall now show that the divergence-free function
$u^{(k+1)}$ possesses additional regularity, in the sense that
$u^{(k+1)} \in L^2(0,T;W^{1,\sigma}_0(\Omega)^d)$, with $\sigma:=\min(\hat{\sigma},z)$; we note that this
fixes the value of $\sigma$, and it is clear that $\sigma>d$, as is required by the arguments
contained in Sections \ref{sec:use} and \ref{sec:FP}.
To do so, we shall move the convective term in the Oseen equation to the right-hand side of the equation,
resulting in an unsteady Stokes system with source term $\nabla\cdot(\KK^{(k)}- b\otimes u^{(k+1)})$.
This then enables us to apply the regularity result for the unsteady Stokes system
stated in \cite{KS2001} (cf. pp.~3067--3069 therein, in particular),
which guarantees the existence of a positive constant $C=C_\sigma$, independent of $k$, such that
\[
\|u^{(k+1)}\|_{W^{1,\frac{1}{2}}_\sigma(Q_T)}
\leq C\left(\|\KK^{(k)}- b \otimes u^{(k+1)}\|_{L^\sigma(Q_T)} + \|u_0\|_{W^{1-\frac{2}{\sigma},\sigma}(\Omega)}\right),
\]
where $\sigma=\min(\hat{\sigma},z)>d$, $\hat\sigma:=2 + \frac{4}{d}$, with $z=d+\vartheta$ for some $\vartheta \in (0,1)$, and
$$W^{1,\frac{1}{2}}_\sigma (Q_T) := L^\sigma(0,T;W^{1,\sigma}_0(\Omega)^d)
\cap W^{1/2,\sigma}(0,T;L^\sigma(\Omega)^d).$$

As $W^{1,\frac{1}{2}}_\sigma (Q_T) \hookrightarrow L^2(0,T;W^{1,\sigma}_0({\Omega})^d)$, it follows that
\begin{align}\label{eq:uniform-u}
 \|u^{(k+1)}\|_{L^2(0,T;W^{1,\sigma}({\Omega}))}
\leq C(1 + \|u_0\|_{W^{1-\frac{2}{\sigma},\sigma}(\Omega)}),
\end{align}
where $\sigma=\min(\hat{\sigma},z)>d$, $\hat\sigma:=2 + \frac{4}{d}$, $d=2,3$, and $z=d+\vartheta$ for some $\vartheta \in (0,1)$.

\subsection{Passage to the limit $k \rightarrow \infty$}\label{sec:limit-oseen}
We deduce from \eqref{eq:uniform-u} and \eqref{eq:oseen-k} that
\begin{alignat}{2}
u^{(k)} & \rightarrow u \qquad && \mbox{weakly in $L^2(0,T;W^{1,\sigma}_0({\Omega})^d)$ as $k \rightarrow \infty$},\qquad \sigma>d,\nonumber\\
u^{(k)} & \rightarrow u \qquad && \mbox{weakly in $W^{1,2}(0,T;W^{-1,\sigma}({\Omega})^d)$ as $k \rightarrow \infty$},\qquad \sigma>d, 
\label{eq:NS-conv}\\
u^{(k)} & \rightarrow u \qquad && \mbox{strongly in $L^2(0,T;\mathcal{C}^{0,\gamma}(\overline\Omega)^d)$ as $k \rightarrow \infty$},\qquad 0<\gamma < 1-{\textstyle\frac{d}{\sigma}},\qquad \sigma>d, \nonumber
\end{alignat}
where the last result follows, via the Aubin--Lions lemma, thanks to the compact embedding of the Sobolev space
$W^{1,\sigma}_0({\Omega})^d$ into the H\"older space $\mathcal{C}^{0,\gamma}(\overline{\Omega})^d$ for $0<\gamma<1-\frac{d}{\sigma}$, $\sigma>d$.
Using \eqref{eq:oseen-k-K} and \eqref{eq:NS-conv} it is now straightforward to pass to the limit in \eqref{eq:oseen-k}.

All that remains to be done is to identify the weak$^*$ limit $\KK$ of the sequence $(\KK^{(k)})_{k \geq 0}$ in terms of
the limit $\hrho$ of the sequence $(\hrho^{(k)})_{k \geq 0}$. As $\KK^{(k)}$ has the form
\[ \frac{{\mathfrak A}^{(k)}}{{\mathfrak B}^{(k)}}, \quad k=0,1,\dots, \]
the limit $\KK$ is anticipated to be of the form
\[ \frac{\mathfrak A}{\mathfrak B},\]
where
\begin{align*}
{\mathfrak A}^{(k)} &:= \int_{D^{J}\times \R^{(J+1)d}} \sum_{j=1}^{J}(F(q_j)\otimes q_j)\,M\,\hrhok\bigl(B(q,x),v,t\bigr)
\dd q  \dd v,
\\
{\mathfrak B}^{(k)} &:= \int_{D^{J}\times \R^{(J+1)d}} M\,\hrhok\,\bigl(B(q,x),v,t\bigr)
\dd q  \dd v,
\\
\mathfrak{A} &:= \int_{D^{J}\times \R^{(J+1)d}} \sum_{j=1}^{J}(F(q_j)\otimes q_j)\,M\,\hrho\bigl(B(q,x),v,t\bigr)
\dd q  \dd v,
\\
\mathfrak{B} &:= \int_{D^{J}\times \R^{(J+1)d}} M\,\hrho\,\bigl(B(q,x),v,t\bigr)
\dd q  \dd v.
\end{align*}

The identification of the limit $\KK$ proceeds as follows. First we need to prove strong convergence of the sequence $(\hrho^{(k)})_{k\geq 0}$.
As we are now required to work under the original hypotheses on the initial
condition, stated in \eqref{eq:ini-cond}, rather than the stronger assumption used for the parabolic regularization of the Fokker--Planck
equation, we can no longer use our earlier argument.
In other words, the only piece of information we are allowed to
use at this point is the energy inequality \eqref{eq:energy-hrho-k}, in conjunction with the bound on
$(u^{(k)})_{k \geq 0}$ supplied by \eqref{eq:uniform-u}.

We therefore argue as follows. Since we have by now already passed to the limit $\alpha \rightarrow 0_+$, and have
thereby removed the $r$-diffusion term from the Fokker--Planck equation, we can rewrite \eqref{eq:FP-eq-k} as
\begin{alignat}{2}
\label{eq:FP-eq-k1}
M\pd_t \hrhok - \frac{\beta^2}{\eps^2}\left(\sum_{j=1}^{J+1} \pdv\!\cdot(M \pdv \hrhok) \right) + \frac{1}{\eps}
\left(\sum_{j=1}^{J+1} M v_j\!\cdot \pdr \hrhok\right) \!=\!  -\frac{1}{\eps}\sum_{j=1}^{J+1}\left((({\mathcal L}r)_j+\uk(r_j,t))\cdot\pdv (M\hrhok) \right)\!,\nonumber\\
~\hfill \mbox{in $\mathcal{D}'(\Omega^{J+1} \times \R^{(J+1)d} \times (0,T))$}
\end{alignat}
(i.e., in the sense of distributions on $\Omega^{J+1} \times \R^{(J+1)d} \times (0,T)$),
and we can exploit the fact that the differential operator appearing on the left-hand side of \eqref{eq:FP-eq-k1} is hypoelliptic.
Thus we can replicate
the argument appearing in the Appendix of the work of DiPerna \& Lions \cite{DL}, concerning strong $L^1$
compactness of a sequence of solutions to a hypoelliptic equation driven by a sequence of source terms that is equibounded in $L^1$
and has uniform decay as $|v| \rightarrow \infty$ in a sense to be made precise below.
Having done so, we will deduce the strong convergence of the sequence $(\hrho^{(k)})_{k \geq 0}$ in
the function space $L^1(0,T;L^1_M(\Omega^{J+1} \times \R^{(J+1)d}))$; i.e.
\begin{alignat}{2}\label{eq:FP-k-strong}
\hrho^{(k)} &\rightarrow \hrho \qquad &&\mbox{strongly in $L^1(0,T;L^1_M(\Omega^{J+1} \times \R^{(J+1)d}))\quad$ as $k \rightarrow \infty$}.
\end{alignat}
To this end, we will first show that the expression appearing on the right-hand side of \eqref{eq:FP-eq-k1} is bounded
in the norm of $L^1(0,T;L^1(\Omega^{J+1} \times \R^{(J+1)d}))$, uniformly with respect to $k$. Clearly, for any $j \in \{1,\dots,J+1\}$,
and $k \geq 1$,
\begin{align}\label{eq:rhs-k-bound}
&\|(({\mathcal L}r)_j+\uk)\cdot\pdv (M\hrhok)\|_{L^1(0,T;L^1(\Omega^{J+1} \times \R^{(J+1)d}))} \nonumber\\
&\qquad \leq
C \int_0^T \left(1 + \|u^{(k)}(\cdot,t)\|_{L^\infty(\Omega)}\right) \|\pdv (M\hrhok(\cdot,\cdot,t))\|_{L^1(\Omega^{J+1} \times \R^{(J+1)d})}
\dd t\nonumber\\
&\qquad \leq C \int_0^T \left(1 + \|u^{(k)}(\cdot,t)\|_{L^\infty(\Omega)}\right) \|M\, |v_j|\,\hrhok(\cdot,\cdot,t)\|_{L^1(\Omega^{J+1} \times \R^{(J+1)d})}\dd t\nonumber\\
&\qquad\quad + C \int_0^T \left(1 + \|u^{(k)}(\cdot,t)\|_{L^\infty(\Omega)}\right) \|M\, \pdv \hrhok(\cdot,\cdot,t)\|_{L^1(\Omega^{J+1} \times \R^{(J+1)d})} \dd t.
 \end{align}
The first term on the right-hand side of \eqref{eq:rhs-k-bound} is bounded, using \eqref{eq:ab}, \eqref{eq:energy-hrho-k}, \eqref{eq:inicond-d}, and
\eqref{eq:uniform-u}, as follows:
\begin{align*}
 \|M\, |v_j|\,\hrhok(\cdot,\cdot,t)\|_{L^1(\Omega^{J+1} \times \R^{(J+1)d})} &\leq \int_{\Omega^{J+1}} \int_{\R^{(J+1)d}}(M(v)\,({\rm e}^{|v_j|} -1) + M(v)\,\mathcal{F}(\hrhok(t)) \dd v \dd r
\\&\leq C\left(1 + \int_{\Omega^{J+1}} \int_{\R^{(J+1)d}} M(v)\,\mathcal{F}(\hrhok(t)) \dd v \dd r\right)
\\&\leq C,
\end{align*}
where $C$ is a positive constant, independent of $k$; hence, noting \eqref{eq:uniform-u},
\begin{align}\label{eq:firstterm}
\int_0^T \left(1 + \|u^{(k)}(\cdot,t)\|_{L^\infty(\Omega)}\right) \|M\, |v_j|\,\hrhok(\cdot,\cdot,t)\|_{L^1(\Omega^{J+1} \times \R^{(J+1)d})}\dd t
\leq C \int_0^T \left(1 + \|u^{(k)}(\cdot,t)\|_{L^\infty(\Omega)}\right) \dd t \leq C,
\end{align}
where $C$ is a positive constant, independent of $k$.

The second term on the right-hand side of \eqref{eq:rhs-k-bound} is bounded as follows. First, using
the Cauchy--Schwarz inequality with respect to $r$ and $v$ we have that
\begin{align*}
\|M\, \pdv \hrhok(\cdot,\cdot,t)\|_{L^1(\Omega^{J+1} \times \R^{(J+1)d})} &= \bigg\|M\, \pdv \bigg(\sqrt{\hrhok(\cdot,\cdot,t)}\bigg)^2\bigg\|_{L^1(\Omega^{J+1} \times \R^{(J+1)d})}
\\
&\leq 2\, \|M\,\hrhok(\cdot,\cdot,t)\|_{L^1(\Omega^{J+1} \times \R^{(J+1)d})}\, \bigg\|M\, \pdv \sqrt{\hrhok(\cdot,\cdot,t)}\,\bigg\|_{L^2(\Omega^{J+1} \times \R^{(J+1)d})}.
\end{align*}
Hence, now using the Cauchy--Schwarz inequality with respect to $t$, we deduce that
{
\begin{align*}
&\int_0^T \left(1 + \|u^{(k)}(\cdot,t)\|_{L^\infty(\Omega)}\right) \|M\, \pdv \hrhok(\cdot,\cdot,t)\|_{L^1(\Omega^{J+1} \times \R^{(J+1)d})} \dd t
\nonumber\\
&\qquad \leq 2 \,\|M\,\hrhok\|_{L^\infty(0,T;L^1(\Omega^{J+1} \times \R^{(J+1)d}))}\,
\int_0^T \left(1 + \|u^{(k)}(\cdot,t)\|_{L^\infty(\Omega)}\right)  \bigg\|M\, \pdv \sqrt{\hrhok(\cdot,\cdot,t)}\,\bigg\|_{L^2(\Omega^{J+1} \times \R^{(J+1)d})}
\dd t\nonumber\\
&\qquad \leq 2\, \|M\,\hrhok\|_{L^\infty(0,T;L^1(\Omega^{J+1} \times \R^{(J+1)d}))}\, \|1 + \|u^{(k)}(\cdot,t)\|_{L^\infty(\Omega)}\|_{L^2(0,T)}\,
\bigg\|M\, \pdv \sqrt{\hrhok}\,\bigg\|_{L^2(0,T;L^2(\Omega^{J+1} \times \R^{(J+1)d}))}.
\end{align*}
}

\noindent
Thus, by noting the uniform bounds \eqref{eq:energy-hrho-k}, \eqref{eq:inicond-d}, and \eqref{eq:uniform-u}, we have that
\begin{align}\label{eq:secondterm}
&\int_0^T \left(1 + \|u^{(k)}(\cdot,t)\|_{L^\infty(\Omega)}\right) \|M\, \pdv \hrhok(\cdot,\cdot,t)\|_{L^1(\Omega^{J+1} \times \R^{(J+1)d})} \dd t \leq C,
\end{align}
where $C$ is a positive constant, independent of $k$.

Using \eqref{eq:firstterm} and \eqref{eq:secondterm} in \eqref{eq:rhs-k-bound}, we then deduce that the expression on the right-hand side of \eqref{eq:FP-eq-k1}
is bounded in $L^1(0,T;L^1(\Omega^{J+1} \times \R^{(J+1)d}))$, uniformly with respect to $k$.

Next, we show that the sequence of functions appearing on the right-hand side of \eqref{eq:FP-eq-k1} has the following additional (`equiboundedness') property: for each $j \in \{1,\dots,J+1\}$,
\begin{equation}\label{eq:additional}
\lim_{R \rightarrow \infty} \sup_{k \geq 1} \|\chi_{|v|\geq R}(\cdot)\, (({\mathcal L}r)_j+\uk)\cdot\pdv
(M\hrhok)\|_{L^1(0,T;L^1(\Omega^{J+1} \times \R^{(J+1)d}))}  = 0,
\end{equation}
where $\chi_{|v| \geq R}$ is the characteristic function of the set of all $v \in \R^{(J+1)d}$ such that $|v|\geq R$, with $|\cdot|$ signifying
the Euclidean norm on
$\R^{(J+1)d}$. Similarly as in \eqref{eq:rhs-k-bound}, we have that
\begin{align}\label{eq:rhs-k-bound1}
&\|\chi_{|v|\geq R}(\cdot)\,(({\mathcal L}r)_j+\uk)\cdot\pdv (M\hrhok)\|_{L^1(0,T;L^1(\Omega^{J+1} \times \R^{(J+1)d}))} \nonumber\\
&\qquad \leq
C \int_0^T \left(1 + \|u^{(k)}(\cdot,t)\|_{L^\infty(\Omega)}\right) \|\chi_{|v|\geq R}(\cdot)\,\pdv (M\hrhok(\cdot,\cdot,t))\|_{L^1(\Omega^{J+1} \times \R^{(J+1)d})}
\dd t\nonumber\\
&\qquad \leq C \int_0^T \left(1 + \|u^{(k)}(\cdot,t)\|_{L^\infty(\Omega)}\right) \|\chi_{|v|\geq R}(\cdot)\,M\, |v_j|\,\hrhok(\cdot,\cdot,t)\|_{L^1(\Omega^{J+1} \times \R^{(J+1)d})}\dd t\nonumber\\
&\qquad\quad + C \int_0^T \left(1 + \|u^{(k)}(\cdot,t)\|_{L^\infty(\Omega)}\right) \|\chi_{|v|\geq R}(\cdot)\,M\, \pdv \hrhok(\cdot,\cdot,t)\|_{L^1(\Omega^{J+1} \times \R^{(J+1)d})} \dd t.
 \end{align}
The first term on the right-hand side of \eqref{eq:rhs-k-bound1} is bounded as follows. We first note that, for $|v|\geq R>0$, by \eqref{eq:ab},
\[0 \leq  M(v)\, |v_j|\,  \hrhok \leq M(v)\, |v|\,  \hrhok \leq \frac{4 \beta}{R} M(v)\, \frac{|v|^2}{4\beta}\,  \hrhok
\leq \frac{4 \beta}{R}\left(M(v) \big({\rm e}^{\frac{|v|^2}{4\beta}} - 1\big) + M(v) \mathcal{F}(\hrhok)\right).
\]
Therefore, using \eqref{eq:energy-hrho-k}, \eqref{eq:inicond-d}, and
\eqref{eq:uniform-u}, we have that
\begin{align*}
 \|\chi_{|v|\geq R}(\cdot)\,M\, |v_j|\,\hrhok(\cdot,\cdot,t)\|_{L^1(\Omega^{J+1} \times \R^{(J+1)d})}
 &\leq \frac{4 \beta}{R} \int_{\Omega^{J+1}} \int_{\R^{(J+1)d}}(M(v)\,({\rm e}^{\frac{|v|^2}{4\beta}} -1) + M(v)\,\mathcal{F}(\hrhok(t)) \dd v \dd r
\\&\leq \frac{C}{R}\left(1 + \int_{\Omega^{J+1}} \int_{\R^{(J+1)d}} M(v)\,\mathcal{F}(\hrhok(t)) \dd v \dd r\right)
\\&\leq \frac{C}{R},
\end{align*}
where $C$ is a positive constant, independent of $k$; hence, noting \eqref{eq:uniform-u} again,
\begin{align}\label{eq:firstterm1}
\int_0^T \left(1 + \|u^{(k)}(\cdot,t)\|_{L^\infty(\Omega)}\right) \|\chi_{|v|\geq R}(\cdot)\,M\, |v_j|\,\hrhok(\cdot,\cdot,t)\|_{L^1(\Omega^{J+1} \times \R^{(J+1)d})}\dd t
\leq \frac{C}{R},
\end{align}
where $C$ is a positive constant, independent of $k$.

The second term on the right-hand side of \eqref{eq:rhs-k-bound1} is bounded as follows. First, using
the Cauchy--Schwarz inequality with respect to $r$ and $v$ we have that
\begin{align*}
&\|\chi_{|v|\geq R}(\cdot)\, M\, \pdv \hrhok(\cdot,\cdot,t)\|_{L^1(\Omega^{J+1} \times \R^{(J+1)d})} = \bigg\|\chi_{|v|\geq R}(\cdot)\,M\, \pdv \bigg(\sqrt{\hrhok(\cdot,\cdot,t)}\bigg)^2\bigg\|_{L^1(\Omega^{J+1} \times \R^{(J+1)d})}
\\
&\qquad \leq 2\, \|\chi_{|v|\geq R}(\cdot)\, M\,\hrhok(\cdot,\cdot,t)\|_{L^1(\Omega^{J+1} \times \R^{(J+1)d})}\, \bigg\|M\, \pdv \sqrt{\hrhok(\cdot,\cdot,t)}\,\bigg\|_{L^2(\Omega^{J+1} \times \R^{(J+1)d})}.
\end{align*}
Hence, now using the Cauchy--Schwarz inequality with respect to $t$, we deduce that
{\footnotesize
\begin{align*}
&\int_0^T \left(1 + \|u^{(k)}(\cdot,t)\|_{L^\infty(\Omega)}\right) \|\chi_{|v|\geq R}(\cdot)\, M\, \pdv \hrhok(\cdot,\cdot,t)\|_{L^1(\Omega^{J+1} \times \R^{(J+1)d})} \dd t
\nonumber\\
&\quad\leq 2 \,\|\chi_{|v|\geq R}(\cdot)\,M\,\hrhok\|_{L^\infty(0,T;L^1(\Omega^{J+1} \times \R^{(J+1)d}))}\,
\int_0^T \left(1 + \|u^{(k)}(\cdot,t)\|_{L^\infty(\Omega)}\right)  \bigg\|M\, \pdv \sqrt{\hrhok(\cdot,\cdot,t)}\,\bigg\|_{L^2(\Omega^{J+1} \times \R^{(J+1)d})}
\dd t
\nonumber\\
&\quad\leq 2\, \|\chi_{|v|\geq R}(\cdot)\,M\,\hrhok\|_{L^\infty(0,T;L^1(\Omega^{J+1} \times \R^{(J+1)d}))}\, \|1 + \|u^{(k)}(\cdot,t)\|_{L^\infty(\Omega)}\|_{L^2(0,T)}\,
\bigg\|M\, \pdv \sqrt{\hrhok}\,\bigg\|_{L^2(0,T;L^2(\Omega^{J+1} \times \R^{(J+1)d}))}.
\end{align*}
}

However, for $|v| \geq R>0$, by \eqref{eq:ab},
\[ 0 \leq M(v)\, \hrhok \leq \frac{4 \beta}{R^2} M(v)\, \frac{|v|^2}{4\beta}\,  \hrhok
\leq \frac{4 \beta}{R^2}\left(M(v) \big({\rm e}^{\frac{|v|^2}{4\beta}} - 1\big) + M(v) \mathcal{F}(\hrhok)\right),
\]
and therefore, by noting the uniform bounds \eqref{eq:energy-hrho-k}, \eqref{eq:inicond-d}, and \eqref{eq:uniform-u}, we have that
\[  \|\chi_{|v|\geq R}(\cdot)\,M\,\hrhok\|_{L^\infty(0,T;L^1(\Omega^{J+1} \times \R^{(J+1)d}))} \leq \frac{C}{R^2},\]
where $C$ is a positive constant, independent of $k$. Thus, by noting the uniform bound  \eqref{eq:uniform-u}, we have that
\begin{align}\label{eq:secondterm1}
&\int_0^T \left(1 + \|u^{(k)}(\cdot,t)\|_{L^\infty(\Omega)}\right) \|\chi_{|v|\geq R}(\cdot)\,M\, \pdv
\hrhok(\cdot,\cdot,t)\|_{L^1(\Omega^{J+1} \times \R^{(J+1)d})} \dd t \leq \frac{C}{R^2},
\end{align}
where $C$ is a positive constant, independent of $k$. Hence, using \eqref{eq:firstterm1} and \eqref{eq:secondterm1} in \eqref{eq:rhs-k-bound1},
we obtain
\[ \|\chi_{|v|\geq R}(\cdot)\,(({\mathcal L}r)_j+\uk)\cdot\pdv (M\hrhok)\|_{L^1(0,T;L^1(\Omega^{J+1} \times \R^{(J+1)d}))} \leq \frac{C}{R},\]
where $C$ is a positive constant, independent of $k$, and therefore \eqref{eq:additional} directly follows.

Furthermore, we note that, similarly to the argument preceding \eqref{eq:firstterm1}, for $|v|\geq R>0$, by \eqref{eq:ab}, we have that
\[ 0 \leq M(v)\,  \hrhok_0 \leq \frac{4 \beta}{R^2} M(v)\, \frac{|v|^2}{4\beta}\,  \hrhok_0
\leq \frac{4 \beta}{R^2}\left(M(v) \big({\rm e}^{\frac{|v|^2}{4\beta}} - 1\big) + M(v) \mathcal{F}(\hrhok_0)\right).
\]
Therefore, using \eqref{eq:energy-hrho-k}, \eqref{eq:inicond-d}, and
\eqref{eq:uniform-u}, we have that
\begin{align}\label{eq:k-bound-ini}
 \|\chi_{|v|\geq R}(\cdot)\,M\,\hrhok_0\|_{L^1(\Omega^{J+1} \times \R^{(J+1)d})}
 &\leq \frac{4 \beta}{R^2} \int_{\Omega^{J+1}} \int_{\R^{(J+1)d}}(M(v)\,({\rm e}^{\frac{|v|^2}{4\beta}} -1) + M(v)\,\mathcal{F}(\hrhok_0) \dd v \dd r
\nonumber
 \\&\leq \frac{C}{R^2}\left(1 + \int_{\Omega^{J+1}} \int_{\R^{(J+1)d}} M(v)\,\mathcal{F}(\hrhok_0) \dd v \dd r\right)
\nonumber
\\&\leq \frac{C}{R^2},
\end{align}
where $C$ is a positive constant, independent of $k$.

To summarize, we have shown that the sequence on the right-hand side of \eqref{eq:FP-eq-k1} is bounded in the norm of $L^1(0,T;L^1(\Omega^{J+1} \times \R^{(J+1)d}))$,
uniformly with respect to $k$. We have also shown that \eqref{eq:additional} and \eqref{eq:k-bound-ini} hold. Having done so, we have verified
the conditions stated under (A.4) and (A.5) in the Appendix of DiPerna \& Lions \cite{DL}. The properties listed under (A.1)--(A.3) in \cite{DL}
follow from properties of the fundamental solution of the hypoelliptic operator on the
left-hand side of  \eqref{eq:FP-eq-k1}, and can be verified by recalling the explicit expression for the
fundamental solution (see, for example, Section II.1 in \cite{Bouchut}). Having checked each of (A.1)--(A.5) in \cite{DL},
an identical argument to the one in the Appendix of \cite{DL} yields the strong convergence of $(M\hrhok)_{k \geq 0}$
to $M\hrho$ in the norm of $L^1(0,T;L^1(\Omega^{J+1} \times \R^{(J+1)d}))$, as stated in \eqref{eq:FP-k-strong}, and hence, thanks to the boundedness of this sequence in $L^\infty(0,T;L^1(\Omega^{J+1} \times \R^{(J+1)d}))$ (which follows from
\eqref{eq:energy-hrho-k}, \eqref{eq:inicond-d} and \eqref{eq:NS-conv}$_3$), strong convergence of $(M\hrhok)_{k \geq 0}$ to $M\hrho$ in $L^p(0,T;L^1(\Omega^{J+1} \times \R^{(J+1)d}))$  also follows, for all $p \in [1,\infty)$; equivalently,
$(\hrhok)_{k \geq 0}$ converges to $\hrho$ in $L^p(0,T;L^1_M(\Omega^{J+1} \times \R^{(J+1)d}))$  for all $p \in [1,\infty)$.

We are now ready for the identification of the weak$^*$ limit $\KK$ of the sequence $(\KK^{(k)})_{k\geq 0}$
in terms of $\hrho$. The argument consists of the following six steps.

\begin{itemize}
\item[(i)] The strong convergence \eqref{eq:FP-k-strong} of the sequence $(\hrho^{(k)})_{k \geq 0}$ in $L^1(0,T;L^1_M(\Omega^{J+1} \times \R^{(J+1)d}))$ implies a.e.
convergence of (a subsequence, not indicated, of)
${\mathfrak A}^{(k)}$ to ${\mathfrak A}$ on  $\Omega \times (0,T)$. Let us show that this is indeed the case: since the Jacobian
$|\mbox{det }\! B|$ is constant and $F \in L^\infty(D^J;\R^d)$, it follows from \eqref{eq:FP-k-strong} by performing
the change of variables $r=B(q,x)$ that, for any $j \in \{1,\dots, J\}$, also
\[~\qquad \int_0^T \int_{\Omega} \int_{D^J \times \R^{(J+1)d}} |F(q_j) \otimes q_j|\,|\hrho^{(k)}(B(q,x),v,t) - \hrho(B(q,x),v,t)|\,M(v) \dd q \dd v \dd x \dd t \rightarrow 0.\]
This then implies that there exists a subsequence, not indicated, such that
\[\int_{D^J \times \R^{(J+1)d}} |F(q_j) \otimes q_j|\, |\hrho^{(k)}(B(q,x),v,t)-\hrho(B(q,x),v,t)|\,M(v) \dd q \dd v \rightarrow 0\]
for a.e. $(x,t) \in \Omega \times (0,T)$. Indeed, by defining, for each $j \in \{1,\dots, J\}$,
\[ \delta_{kj}(x,t):=\int_{D^J \times \R^{(J+1)d}}|F(q_j) \otimes q_j|\,|\hrho^{(k)}(B(q,x),v,t)-\hrho(B(q,x),v,t)|\,M(v) \dd q \dd v,\]
Tonelli's theorem yields that
$\delta_{kj} \in L^1(\Omega\times (0,T);\R_{\geq 0})$ for all $k \geq 1$. As,
\[ \|\delta_{kj}\|_{L^1(\Omega\times (0,T))}=\int_0^T \int_{\Omega} \delta_{kj}(q,t)\dd q \dd t \rightarrow 0,\]
there exists a subsequence of $(\delta_{kj})_{k \geq 1}$, not indicated, such that $\delta_{kj}(x,t) \rightarrow 0_+$ for a.e. $(x,t) \in \Omega\times (0,T)$, for each $j \in \{1,\dots, J\}$,
\item[(ii)] Analogously, ${\mathfrak B}^{(k)}$ converges to $\mathfrak{B}$ a.e. on  $\Omega \times (0,T)$.
\item[(iii)] Now (i) and (ii) imply that ${\mathfrak A}^{(k)}/{\mathfrak B}^{(k)}$ converges to $\mathfrak{A}/\mathfrak{B}$ a.e. on  $\Omega \times (0,T)$.
\item[(iv)] Since $|{\mathfrak A}^{(k)}/{\mathfrak B}^{(k)}| \leq C$, where $C$ is a positive constant, independent of $k$,
the dominated convergence theorem yields that
$\int_{E} {\mathfrak A}^{(k)}/{\mathfrak B}^{(k)} \dd x \dd t$
converges to $\int_{E} \mathfrak{A}/\mathfrak{B} \dd x \dd t$ for every measurable set $E \subset \Omega \times (0,T)$.
\item[(v)] Now (iv), together with the fact that $({\mathfrak A}^{(k)}/{\mathfrak B}^{(k)})_{k \geq 0}$ is bounded in $L^\infty(\Omega \times (0,T))$, implies weak$^*$ convergence in
$L^\infty(\Omega \times (0,T))$ of ${\mathfrak A}^{(k)}/{\mathfrak B}^{(k)}$ to
$\mathfrak{A}/\mathfrak{B}$ thanks to Corollary 2.49 in \cite{Fonseca_Leoni}.
\item[(vi)] However, \eqref{eq:oseen-k-K} states that ${\mathfrak A}^{(k)}/{\mathfrak B}^{(k)}$ converges
weakly$^*$ to $\KK$, in $L^\infty(\Omega \times (0,T))$. Therefore, by uniqueness of the weak$^*$ limit, $\KK = \mathfrak{A}/\mathfrak{B}$.
\end{itemize}
Thus we have shown that
\begin{align*} \KK = \frac{\mathfrak A}{\mathfrak B} =
\frac{\int_{D^{J}\times \R^{(J+1)d}} \sum_{j=1}^J (F(q_j)\otimes q_j)\,M\,\hrho\bigl(B(q,x),v,t\bigr)
\dd q  \dd v}{\int_{D^{J}\times \R^{(J+1)d}} M\,\hrho\,\bigl(B(q,x),v,t\bigr)
\dd q  \dd v}.
\end{align*}

\smallskip

Finally, we can pass to the limit $k \rightarrow \infty$ in the sequence of Fokker--Planck equations \eqref{eq:weak-duality-bb-k}. As this part of the proof is very similar to the passage to the
limit $\alpha \rightarrow 0_+$ expounded in the previous section, we confine ourselves to summarizing
the main points.

The strong convergence result \eqref{eq:FP-k-strong} and the energy inequality \eqref{eq:energy-hrho-k}
imply the existence of
\[
\hrho \in L^\infty(0,T;L^1_M(\Omega^{J+1} \times \R^{(J+1)d};\R_{\geq 0})),
\]
with
\[ \nabla_v \sqrt{\hrho} \in L^2(0,T;L^2_M(\Omega^{J+1} \times \R^{(J+1)d})),\]
\[\nabla_v \hrho \in L^2(0,T;L^1_M(\Omega^{J+1} \times \R^{(J+1)d}))\quad \mbox{and}\quad
M\,\pd_t \hrho \in L^2(0,T;(W^{s,2}(\Omega^{J+1} \times \R^{(J+1)d}))'),\quad s>(J+1)d + 1,
\]
such that, as $k \rightarrow \infty$,
{\small
\begin{alignat*}{2}
\hrhok &\rightarrow \hrho &&\!\!\left\{\begin{array}{ll}\mbox{weakly$^*$ in $L^\infty(0,T;L^1_M(\Omega^{J+1} \times \R^{(J+1)d}))$},\\
\mbox{strongly in $L^p(0,T;L^1_M(\Omega^{J+1} \times \R^{(J+1)d}))\quad$ for all $p \in [1,\infty)$,}
\end{array}\right. \\
\nabla_v \hrhok &\rightharpoonup \nabla_v \hrho &&\quad \mbox{weakly in $L^2(0,T;L^1_M(\Omega^{J+1} \times \R^{(J+1)d}))$},\\
M\,\pd_t \hrhok &\rightharpoonup M\, \pd_t \hrho &&\quad \mbox{weakly in $L^2(0,T;(W^{s,2}(\Omega^{J+1} \times \R^{(J+1)d}))'),\!\quad\! s > (J+1)d\! +\! 1$},\\
v_j\, \hrhok &\rightharpoonup v_j\, \hrho &&\quad
\mbox{weakly in $L^2(0,T;L^1_M(\Omega^{J+1} \times \R^{(J+1)d}))$,\quad $j=1,\dots,J+1$},\\
(({\mathcal L}r)_j+u^{(k)}(r_j,\tau))\,\hrhok &\rightharpoonup (({\mathcal L}r)_j+u(r_j,\tau))\,\hrho &&\quad
\mbox{weakly in $L^2(0,T;L^1_M(\Omega^{J+1} \times \R^{(J+1)d}))$,\quad $j=1,\dots,J+1$.}
\end{alignat*}
}

Using these convergence results, passage to the limit $k \rightarrow \infty$ in \eqref{eq:weak-duality-bb-k}
implies the existence
of
\[
\hrho \in L^\infty(0,T;L^1_M(\Omega^{J+1} \times \R^{(J+1)d};\R_{\geq 0})),
\]
with
\[ \nabla_v \sqrt{\hrho} \in L^2(0,T;L^2_M(\Omega^{J+1} \times \R^{(J+1)d})),\]
\[\nabla_v \hrho \in L^2(0,T;L^1_M(\Omega^{J+1} \times \R^{(J+1)d}))\quad \mbox{and}\quad
M\,\pd_t \hrho \in L^2(0,T;(W^{s,2}(\Omega^{J+1} \times \R^{(J+1)d}))'),\quad s>(J+1)d + 1,
\]
satisfying the following weak form of the Fokker--Planck equation:
\begin{align}\label{eq:weak-duality-bblim}
&\int_0^t \big\langle M\,\pd_\tau\hrho(\cdot,\cdot,\tau),\varphi(\cdot,\cdot,\tau)\big\rangle \dd \tau
+ \frac{\beta^2}{\eps^2}\left(\sum_{j=1}^{J+1} \int_0^t \int_{\Omega^{J+1}} \int_{\R^{(J+1)d}} M(v)\,\pdv \hrho \cdot \pdv \varphi \dd v \dd r \dd \tau \right)\nonumber\\
&\qquad- \frac{1}{\eps} \left(\sum_{j=1}^{J+1} \int_0^t \int_{\Omega^{J+1}} \int_{\R^{(J+1)d}} M(v)\, v_j \hrho\cdot \pdr \varphi \dd v \dd r \dd \tau \right)\nonumber\\
&\qquad- \frac{1}{\eps} \left(\sum_{j=1}^{J+1} \int_0^t \int_{\Omega^{J+1}} \int_{\R^{(J+1)d}} M(v)\,(({\mathcal L}r)_j+u(r_j,\tau))\,\hrho\cdot \pdv \varphi \dd v \dd r \dd \tau \right)
= 0 \nonumber \\
&\hspace{0.5in} \forall\, \varphi \in L^2(0,T; W^{1,2}_{*,M}(\Omega^{J+1} \times \R^{(J+1)d})\cap W^{s,2}_*(\Omega^{J+1} \times \R^{(J+1)d})), \quad s>(J+1)d+1,  \quad \forall\,t \in (0,T].
\end{align}
Furthermore $\hrho(\cdot,\cdot,0)=\hrho_0(\cdot,\cdot)$ in the sense of $\mathcal{C}_w([0,T];L^1_M(\Omega^{J+1} \times \R^{(J+1)d};\R_{\geq 0}))$, and
\[ \int_{\Omega^{J+1} \times \R^{(J+1)d}} M\,\hrho(r,v,t)\dd r \dd v = \int_{\Omega^{J+1} \times \R^{(J+1)d}} M\, \hrho_0(r,v)\dd r \dd v = 1.\]
In addition, $\hrho$ satisfies the following energy inequality:
%
\begin{align*}
\int_{\Omega^{J+1}} \int_{\R^{(J+1)d}} M(v)\,\mathcal{F}(\hrho(t)) \dd v \dd r
 + \frac{\beta^2}{2\eps^2} \sum_{j=1}^{J+1} \int_0^t \int_{\Omega^{J+1}} \int_{\R^{(J+1)d}} M(v)\,\frac{|\pdv \hrho|^2}{\hrho}
\, \dd v \dd r \dd \tau
\nonumber\\
\leq  C\bigg[1 + \int_{\Omega^{J+1}} \int_{\R^{(J+1)d}} M(v)\,\mathcal{F}(\hrho_{0}) \dd v \dd r \bigg],
\end{align*}
where $C=C(\|u_0\|_{W^{1-\frac{2}{\sigma},\sigma}(\Omega)},\|b\|_{L^\infty(0,T;L^\infty(\Omega))})$,
$\sigma:=\min(\hat{\sigma},z)>d$, with $\hat\sigma:=2 + \frac{4}{d}$ and $z=d+\vartheta$ for some $\vartheta \in (0,1)$. In particular, $C$ is independent of $\eps>0$.

This then completes the proof of the existence of large-data global weak solutions to the coupled Oseen--Fokker--Planck system under consideration, for all $\eps>0$.

\section{Trace theorems for the solution of the Fokker--Planck equation}
\label{sec:trace}
In this section, by using similar arguments as in \cite{onthetraceMischler}, we prove that the solution to the Fokker--Planck equation has a unique trace on the boundary of our domain,
which is defined thanks to a Green's formula. We then use this result to prove that the specular boundary condition is attained in a strong sense by the solution. To this end, given the vector
\[ E_j=E_j(r,v,t):=\frac{1}{\eps} (({\mathcal L}r)_j+u(r_j,t)) - \frac{\beta^2}{\eps^2} v_j\]
and a weak solution $\varrho=\varrho(r,v,t) $ of the Fokker--Planck equation
\begin{alignat}{2}
\label{eq:FP-eq1}
 \Lambda_{E_j}(\varrho) := \pd_t \varrho + \sum_{j=1}^{J+1} E_j(r,v,t) \cdot \pdv \varrho + \sum_{j=1}^{J+1} (\pdv \cdot E_j(r,v,t)) \varrho - \frac{\beta^2}{\eps^2}\, \sum_{j=1}^{J+1} \pdv^2 \varrho + \frac{1}{\eps}
\sum_{j=1}^{J+1} v_j\cdot \pdr \varrho = 0,\\
~\hfill \mbox{for all $(r,v,t) \in \Omega^{J+1} \times \R^{(J+1)d} \times (0,T]$},\nonumber\\
\varrho(r,v,0)=\varrho_0(r,v)~\hfill \qquad \mbox{for all $(r,v) \in \Omega^{J+1} \times \R^{(J+1)d}$},
\label{eq:FP-ini1}
\end{alignat}
satisfying the specular boundary condition in a weak sense, we show that $\varrho$ has a trace $\gamma \varrho$ on the boundary $ \partial \Omega^{(j)} \times \R^{(J+1)d} \times (0,T)$, $j=1,\dots,J+1$,
and a trace $\gamma_t \varrho = \varrho(\cdot,t)$ on the section $\Omega ^{J+1} \times \R^{(J+1)d} \times \{t\} $ for all $t \in [0,T]$. These trace functions will be shown to be
well-defined thanks to a Green's formula, which we shall now discuss.

In the previous section we showed that $\varrho \in L^\infty(0,T;L^1(\Omega^{J+1} \times \R^{(J+1)d};\R_{\geq 0}))$ is
a solution to \eqref{eq:FP-eq1} in the sense of distributions, i.e.,
%
\begin{alignat}{2}
\label{conjug-eq:FP-eq}
\int_0^T \int_{\Omega^{J+1}} \int_{\R^{(J+1)d}}
\varrho \, \Lambda^*_{E_j} (\varphi) \dd v \dd r \dd \tau = 0,
\end{alignat}
for all test functions $ \varphi \in \cD(\overline{D}) := \mathcal{C}^{\infty}_0( \overline{\Omega^{J+1}} \times \R^{(J+1)d} \times [0,T])$, where we have set:
\begin{alignat}{2}
\label{conjug_op:FP-eq}
\Lambda^*_{E_j}(\varrho)=\pd_t \varphi + \sum_{j=1}^{J+1} E_j(r,v,t) \cdot \pdv \varphi + \sum_{j=1}^{J+1} (\pdv \cdot E_j(r,v,t)) \, \varphi + \frac{\beta^2}{\eps^2}\, \sum_{j=1}^{J+1} \pdv^2 \varphi + \frac{1}{\eps}
\sum_{j=1}^{J+1} v_j\cdot \pdr \varphi.
\end{alignat}
From the previous section we know that $u \in L^{2}(0,T;W^{1,\sigma}(\Omega)^d)$, with $\sigma>d$.
Since, by Morrey's inequality, $W^{1,\sigma}(\Omega) \hookrightarrow L^{\infty}(\Omega)$,
we have in particular that $u \in L^1(0,T;L^{\infty}(\Omega)^d)$. We thereby deduce that
$E_j \in L^1(0,T;L^\infty_{loc}(\Omega;W^{1,\infty}_{loc}(\R^{d})))^d$ and
$\pdv \cdot E_j \in L^{\infty}(0,T;L^{\infty}_{loc}(\Omega \times \R^{d}))$ for all
$j=1,\dots, J+1$.

We shall suppose henceforth that the initial datum $\varrho_0$ for the Fokker--Planck equation has the following
factorized form: $\varrho_0(r,v) = M(v)\, \hrho_0(r)$, where $\hrho_0$ is a nonnegative function of $r$ only, such that
$\int_{\Omega^{J+1}} \hrho_0(r) \dd r = 1$,
and $$\hrho_0 \in L^2(\Omega^{J+1} \times \R^{{(J+1)}d};\R_{\geq 0}).$$
Under this hypothesis it directly follows that
\[ \hrho \in L^\infty(0,T; L^2_M(\Omega^{J+1} \times \R^{(J+1)d};\R_{\geq 0})),\]
and
\[ \varrho \in L^\infty(0,T; L^2_{M^{-1}}(\Omega^{J+1} \times \R^{(J+1)d};\R_{\geq 0})),\]
 and consequently, since $M^{-1}(v) \ge (2\pi \beta)^{\frac{1}{2}(J+1)}$ for all $v \in \R^{(J+1)d}$, that
 \[ \varrho \in L^\infty(0,T; L^2(\Omega^{J+1} \times \R^{(J+1)d};\R_{\geq 0})).\]

 \begin{remark}
 To show that $\hrho_0 \in L^2(\Omega^{J+1} \times \R^{{(J+1)}d};\R_{\geq 0})$ implies $ \hrho \in L^\infty(0,T; L^2_M(\Omega^{J+1} \times \R^{(J+1)d};\R_{\geq 0}))$, one has to follow a similar line of argument as in Section \ref{sec:FP}, Subsection \ref{subsec_2:FP}. Indeed, it suffices to test equation \eqref{eq:weak-duality-b} with the function
\[\hrho_\alpha
\in L^2(0,T;W^{1,2}_{*,M}(\Omega^{J+1} \times \mathbb{R}^{(J+1)d})),\]
rather than
 \[\mathcal{F}'(\hrho_\alpha+\gamma) = \log(\hrho_\alpha + \gamma)
\in L^2(0,T;W^{1,2}_{*,M}(\Omega^{J+1} \times \mathbb{R}^{(J+1)d})),\]
\noindent where $\hrho_\alpha \geq 0$ and $\gamma>0$, and pass to the limit $\alpha \rightarrow 0_+$ in the equation
satisfied by $\hrho_\alpha$ using the bounds resulting from the corresponding energy estimate.
 \end{remark}

\subsection{Statement of the Trace Theorem}
\label{sec:use1}

\begin{theorem}
\label{theorem:trace}
Let $\varrho \in L^\infty(0,T;L^2(\Omega^{J+1} \times \R^{(J+1)d};\R_{\geq 0}))$ be a solution of equation \eqref{conjug-eq:FP-eq}.
Then, for every $t \in [0,T]$, there exists a $\gamma_t \varrho \in L^1(\Omega^{J+1} \times \R^{(J+1)d})$ and a
$\gamma \varrho$ defined on $ \partial \Omega^{(j)} \times \R^{(J+1)d} \times (0,T)$ for $j=1,\dots,J+1$, such that:
\[
\gamma_t \varrho \in C([0,T]; L^1_{loc}(\Omega^{J+1} \times \R^{(J+1)d})) \quad \mbox{and} \quad \gamma \varrho \in L^1_{loc}\bigl(\partial  \Omega^{(j)} \times \R^{(J+1)d} \times [0,T], (v_j
\cdot n(r_j))^2 \, \dd v \dd s(r) \dd \tau \bigr),
\]
for $j=1,\dots,J+1$, and which satisfy the Green's formula
{
\begin{align}
\label{Green:FP-eq}
&\int_{t_0}^{t_1} \int_{\Omega^{J+1}} \int_{\R^{(J+1)d}}
\varrho \,  \Lambda^*_{E_j} (\varphi)  \dd v \dd r \dd \tau  \nonumber \\  & \qquad \qquad =\Bigl[ \int_{\Omega^{J+1}} \int_{\R^{(J+1)d}}
\varrho(\cdot, \tau) \, \varphi \dd v \dd r\Bigr]_{t_0}^{t_1} + \sum_{j=1}^{J+1} \int_{t_0}^{t_1} \int_{\partial \Omega^{(j)}} \int_{\R^{(J+1)d}} (v_j
\cdot n(r_j)) \,
\gamma \varrho \, \varphi \dd v \dd s(r) \dd \tau
\end{align}
}

\noindent
for all $t_0, \, t_1 \in [0,T]$ and for all test functions $\varphi \in \cD_0(\overline{D})$, the space of functions
$\varphi \in \cD(\overline{D})$ such that $\varphi=0$ on $\Sigma_0 \times (0,T)$, where
$\Sigma_0:= \bigcup_{j=1}^{J+1} \Bigl \{ (r,v) \in \partial \Omega^{(j)} \times \R^{(J+1)d}, \, v_j \cdot n(r_j)=0 \Bigr \}$ and we have used the notation $\cD(\overline{D}) :=  \mathcal{C}^{\infty}_0(\overline{\Omega^{J+1}} \times \R^{(J+1)d} \times [0,T])$.
\end{theorem}

Let us first introduce some additional notation. Since $\partial \Omega$ is $\mathcal{C}^2$, $\Omega$ is locally on one side of
$\partial \Omega$ and there exists a function $d=d_{\Omega} \in W^{2,\infty}(\R^{d})$ such that for all $z$ in an interior
neighbourhood of $\partial \Omega$ one has
$$d(z)=-\mathrm{dist}(z,\partial \Omega).$$

We define in $\overline{\Omega}$ the gradient field
$$n(z)=\nabla_z d(z),$$
which coincides with the unit outward normal vector to $\Omega$ at every point of $\partial \Omega$.
Hence, the unit outward normal (column-)vector to $\partial \Omega_j$ at $r_j \in \partial\Omega_j$, for $j=1,\dots, J+1$, is  $$n(r_j)=\nabla_{r_j} d(r_j)= \partial_{r_j} d(r_j).$$
Here, the set $\Omega_j$ still denotes $\Omega$; by assigning it the index $j$, however, we wish to emphasize by our notation that in
the consideration of the distance to the boundary of $\Omega$, the distance of the coordinate $r_j \in \Omega_j$ is measured to the
boundary $\partial \Omega_j$ of the set that contains it.

We consider
\[\dd \mu_i = |n(r_j) \cdot v_j|^{i}\dd v \dd s(r) \dd t, \qquad i=1,2,\]
which are measures defined on $\partial \Omega^{(j)} \times \R^{(J+1)d} \times (0,T)$.
For a given real $R>0$, we define the sets
\[ B_R=\bigl \{ y \in \R^{d} \,  :  \, |y| < R \bigr \}, \qquad \cO := \Omega^{J+1} \times \R^{(J+1)d}, \qquad \cD := \Omega^{J+1} \times \R^{(J+1)d} \times (0,T),\]
\[ \cO_R := (\Omega \cap B_R)^{J+1} \times B_R^{J+1},\qquad \cD_R:= \cO_R \times (0,T).\]
We shall also use the abbreviation $L_R^{a,b}$ for the function space $L^{a}(0,T;L^{b}(\cO_R))$, and $L^{a,b}_{loc}$ for the
function space $L^{a}(0,T;L^{b}_{loc}(\cO))$.
\begin{proof}[Proof of Theorem \ref{theorem:trace}]
The proof of the theorem will be performed in three steps. First, we obtain two a priori estimates assuming that the solution of
equation \eqref{eq:FP-eq1} is smooth. Then, following the method proposed by DiPerna \& Lions in \cite{DiPerna1989}, we approach
the weak solution $\varrho$ of equation \eqref{eq:FP-eq1} by a sequence of regular functions $(\varrho_k)_{k\geq 1}$, which are
solutions of equation \eqref{eq:FP-eq1} with an error term $r_k$ that vanishes at infinity; these regular functions satisfy
the two a priori estimates from the first step. Finally, we deduce the existence of a trace by passing to the limit.

{\sc{Step 1: A priori estimates.}} In this step, we derive two a priori estimates. We first assume that
\[\varrho \in W^{1,1}_{loc}\bigl(\Omega^{J+1} \times \R^{(J+1)d} \times (0,T) \bigr)\]
so that the following manipulations are admissible. We consider three functions that we shall specify later: $\psi \in C^1(\R)$
nondecreasing with $\psi(0)=0$, 
$\Phi=\Phi(r,v,t) \in \mathcal{C}^{\infty}_0 \bigl(\R^{(J+1)d} \times \R^{(J+1)d} \times [0,T]\bigr)$ and $\tilde \beta \in C^1(\R)$, and
we fix $t_0, t_1 \in [0,T]$. Below, we shall write $\psi$ for $\psi(v_j \cdot n(r_j))$.
We use Green's formula together with equation \eqref{eq:FP-eq1} to get
\begin{align}
\begin{split}
\label{Green-apriori:FP-eq}
& \Biggl[ \int_{\Omega^{J+1}} \int_{\R^{(J+1)d}}
\tilde \beta(\varrho) \, \psi \, \Phi  \dd v \dd r \Biggr]_{t_0}^{t_1} + \sum_{j=1}^{J+1} \int_{t_0}^{t_1} \int_{\partial \Omega^{(j)}} \int_{\R^{(J+1)d}} (v_j
\cdot n(r_j)) \, \tilde \beta(\varrho) \, \psi \, \Phi \dd v \dd s(r) \dd \tau \\
& \qquad \qquad = \int_{t_0}^{t_1} \int_{\Omega^{J+1}} \int_{\R^{(J+1)d}} \Lambda_{E_j} \bigl( \tilde \beta(\varrho(r,v,\tau)) \, \Phi(r,v,\tau) \, \psi(v_j \cdot n(r_j)) \bigr) \dd v \dd r \dd \tau  \\
& \qquad \qquad = \int_{t_0}^{t_1} \int_{\Omega^{J+1}} \int_{\R^{(J+1)d}} \Biggl \{ \Phi \, \psi \, \Lambda_{E_j} \bigl( \tilde \beta(\varrho(r,v,\tau)) \bigr) + \tilde \beta(\varrho) \, \Phi \, \Lambda_{E_j} \bigl(  \psi(v_j \cdot n(r_j)) \bigr) \\
& \qquad \qquad \qquad + \tilde \beta(\varrho) \, \psi \, \Lambda_{E_j}\Phi(r,v,\tau) \Biggl \} \dd v \dd r \dd \tau  \\
& \qquad \qquad = \int_{t_0}^{t_1} \int_{\Omega^{J+1}} \int_{\R^{(J+1)d}} \Biggl \{ \Phi \, \psi \, \tilde \beta'(\varrho) \, \Lambda_{E_j} \varrho + \tilde \beta(\varrho) \, \Phi \, \psi'(v_j \cdot n(r_j)) \, \Lambda_{E_j} \bigl(v_j \cdot n(r_j) \bigr) \\
& \qquad \qquad \qquad + \tilde \beta(\varrho) \, \psi \, \Lambda_{E_j}\Phi(r,v,\tau) \Biggl \} \dd v \dd r \dd \tau \\
& \qquad \qquad = \int_{t_0}^{t_1} \int_{\Omega^{J+1}} \int_{\R^{(J+1)d}} \Biggl \{ \tilde \beta(\varrho)\,\Phi \Biggl[\psi'(v_j \cdot n(r_j)) \Biggl(\frac{1}{\eps}
\sum_{j=1}^{J+1} v_j^{\rm T} \, D^2 d_{\Omega} \, v_j + \sum_{j=1}^{J+1} E_j(r,v,\tau)\cdot n(r_j) \Biggr)  \\
& \qquad \qquad \qquad  + \sum_{j=1}^{J+1} (\pdv \cdot E_j(r,v,\tau)) \psi(v_j \cdot n(r_j)) \Biggr] + \tilde \beta(\varrho) \, \psi \, \Lambda_{E_j} \Phi \Biggr \} \dd v \dd r \dd \tau.
\end{split}
\end{align}
We now fix $t_0 \in [0,T]$, a compact set $K$ of $\cO := \Omega^{J+1} \times \R^{(J+1)d}$, $\psi(z)=1$ and
$\tilde \beta=\tilde\beta_{\tilde \eps}$ where $\tilde \beta_{\tilde \eps}$ is a sequence of smooth even and nonnegative functions such
that $\tilde \beta_{\tilde \eps}(0)=0$ and $\tilde \beta_{\tilde \eps}(y) \rightarrow |y|$, for all $y \in \R$. We can then choose
$\Phi \in \mathcal{C}^{\infty}_0(\Omega^{J+1} \times \R^{(J+1)d})$ in such a way that $0 \le \Phi \le 1$ in $\Omega^{J+1} \times \R^{(J+1)d} \times (0,T)$, $\Phi = 1$ on $K$ and we denote by $R>0$ a real number satisfying  $\mathrm{supp} \, \Phi \subset \cO_R$.
The identity \eqref{Green-apriori:FP-eq} then implies that, for all $t \in [0,T]$,
\begin{align*}
\int_{\Omega^{J+1}} \int_{\R^{(J+1)d}}
\tilde \beta_{\tilde \eps}(\varrho(\cdot, t_1))\, \Phi  \dd v \dd r & = \int_{\Omega^{J+1}} \int_{\R^{(J+1)d}} \tilde \beta_{\tilde \eps}(\varrho(\cdot, t_0))\, \Phi  \dd v \dd r
\\
& \quad + \int_{t_0}^{t_1}  \int_{\Omega^{J+1}} \int_{\R^{(J+1)d}}
\tilde \beta_{\tilde \eps}(\varrho) \, \sum_{j=1}^{J+1} (\pdv \cdot E_j(r,v,\tau)) \dd v \dd r \dd \tau
\\
& \quad + \int_{t_0}^{t_1}  \int_{\Omega^{J+1}} \int_{\R^{(J+1)d}} \tilde \beta_{\tilde \eps}(\varrho) \, \Lambda_{E_j} \Phi \dd v \dd r \dd \tau\\
\end{align*}
\begin{align*}
& \le \|\tilde \beta_{\tilde \eps}(\varrho(\cdot, t_0)) \|_{L^1_R} + C_{R} \int_0^T  \sum_{j=1}^{J+1} \|\pdv \cdot E_j(\cdot,\cdot,\tau) \|_{L^{\infty}_R} \| \tilde \beta_{\tilde \eps}(\varrho) (\cdot,\tau) \|_{L^1_R} \dd \tau \\
& \quad + C_{R,\eps,\beta,J} \| \nabla \Phi \|_{L^{\infty}} \int_0^T \Biggl( 1 + \sum_{j=1}^{J+1}  \| E_j(\cdot,\cdot,\tau) \|_{L^{\infty}_R} \Biggr) \| \tilde \beta_{\tilde \eps}(\varrho)(\cdot,\tau) \|_{L^1_R} \dd \tau  \\
& \quad + C_{R,\eps,\beta,J} \, \| \Delta \Phi \|_{L^{\infty}} \int_0^T \| \tilde \beta_{\tilde \eps}(\varrho)(\cdot, \tau) \|_{L^1_R} \dd \tau.
\end{align*}
Letting $\tilde \eps$ tend to $0$, we deduce our first a priori estimate:

\begin{align}
\label{apriori1:FP-eq}
\begin{split}
\sup_{\tau \in [0,T]}\| \varrho(\cdot,\tau))\|_{L^1(K)} & \le  \|\varrho(\cdot,t_0) \|_{L^1_R} + C_{R} \int_0^T  \sum_{j=1}^{J+1} \|\pdv \cdot E_j(\cdot,\cdot,\tau)\|_{L^{\infty}_R} \| \varrho (\cdot,\tau) \|_{L^1_R} \dd \tau \\
& \quad + C_{R,\eps,\beta,J} \int_0^T \Biggl( 1 + \sum_{j=1}^{J+1}  \| E_j(\cdot,\cdot,\tau) \|_{L^{\infty}_R} \Biggr) \| \varrho(\cdot,\tau) \|_{L^1_R} \dd \tau.
\end{split}
\end{align}

Let us now fix a compact subset $K$ of $\partial \Omega^{(j)} \times \R^{(J+1)d}$, $\psi(z)=z$, $t_0=0$, $t_1=T$, with $\tilde \beta$ as before.
We choose $\Phi \in \mathcal{C}^{\infty}_0(\overline{\Omega^{J+1}} \times \R^{(J+1)d})$ 
in such a way that $0 \le \Phi \le 1$ in $\cO$, $\Phi \equiv 1$ on $K$, and we 
denote by $R>0$ a real number satisfying $\mathrm{supp} \, \Phi \subset B_R \times B_R$. 

We then deduce from the identity \eqref{Green-apriori:FP-eq} a second a priori estimate:
\begin{align}
\begin{split}
& \sum_{j=1}^{J+1} \int_{0}^{T} \int_{\partial \Omega^{(j)}} \int_{\R^{(J+1)d}} (v_j
\cdot n(r_j))^2 \, \tilde \beta_{\tilde \eps}(\varrho) \, \Phi \dd v \dd s(r) \dd \tau \\
&\qquad = -  \Biggl[ \int_{\Omega^{J+1}} \int_{\R^{(J+1)d}}
( v_j \cdot n(r_j)) \tilde \beta_{\tilde \eps}(\varrho) \, \Phi  \dd v \dd r \Biggr]_{0}^{T} \\
& \qquad \qquad + \int_{0}^{T} \int_{\Omega^{J+1}} \int_{\R^{(J+1)d}} \Biggl \{ \tilde \beta_{\tilde \eps}(\varrho)\Phi \Biggl[\frac{1}{\eps}
\sum_{j=1}^{J+1} v_j^{\rm T} \, D^2 d_{\Omega} \, v_j + \sum_{j=1}^{J+1} E_j(r,v,\tau)\cdot n(r_j)  \\
& \qquad \qquad \qquad  + \sum_{j=1}^{J+1} ( v_j \cdot n(r_j)) \, (\pdv \cdot E_j(r,v,\tau)) \Biggr] + (v_j
\cdot n(r_j)) \, \tilde \beta_{\tilde \eps}(\varrho)  \, \Lambda_{E_j}\Phi \Biggr \} \dd v \dd r \dd \tau
\\
& \le R \, (\|\tilde \beta_{\tilde \eps}(\varrho(\cdot,T)) \|_{L^1} + \|\tilde \beta_{\tilde \eps}(\varrho(\cdot,0)) \|_{L^1})
\\
& \qquad \qquad + \int_{0}^{T} \int_{\Omega^{J+1}} \int_{\R^{(J+1)d}} \Biggl \{ \tilde \beta_{\tilde \eps}(\varrho) \, \Phi \, \Biggl[\frac{1}{\eps}
\sum_{j=1}^{J+1} v_j^{\rm T} \, D^2 d_{\Omega} \, v_j + \sum_{j=1}^{J+1} E_j(r,v,\tau)\cdot n(r_j)   \\
& \qquad \qquad  + \sum_{j=1}^{J+1} ( v_j \cdot n(r_j)) \, (\pdv \cdot E_j(r,v,\tau)) \Biggr] + (v_j
\cdot n(r_j)) \tilde \beta_{\tilde \eps}(\varrho) \, \Lambda_{E_j} \Phi \Biggr \} \dd v \dd r \dd \tau \\
 & \qquad \qquad  + C_{R} \int_0^T  \sum_{j=1}^{J+1} \|\pdv \cdot E_j(\cdot,\cdot,\tau) \|_{L^{\infty}_R} \| \tilde \beta_{\tilde \eps}(\varrho)  (\cdot,\tau) \|_{L^1_R} \dd \tau  \nonumber \\
& \qquad \qquad + C_{R,\eps,\beta,J} \int_0^T \Biggl( 1 + \sum_{j=1}^{J+1}  \| E_j(\cdot,\cdot,\tau) \|_{L^{\infty}_R} \Biggr) \| \tilde \beta_{\tilde \eps}(\varrho) (\cdot,\tau) \|_{L^1_R} \dd \tau.
\end{split}
\end{align}
Letting $\tilde \eps$ tend to $0$, we then have that
\begin{align}
\label{apriori2}
&\| \varrho\|_{L^1([0,T] \times K, \dd \mu_2)} \nonumber\\
&\qquad  \le R \, \bigl( \|\varrho(\cdot,T) \|_{L^1} + \|\varrho(\cdot,0) \|_{L^1} \bigr) + C_{R} \int_0^T  \sum_{j=1}^{J+1} \| \pdv \cdot E_j(\cdot,\cdot,\tau) \|_{L^{\infty}_R} \| \varrho (\cdot,\tau) \|_{L^1_R} \dd \tau  \\
& \qquad \qquad + C_{R,\eps,\beta,J} \int_0^T \Biggl( 1 + \sum_{j=1}^{J+1}  \| E_j(\cdot,\cdot,\tau) \|_{L^{\infty}_R} \Biggr) \| \varrho(\cdot,\tau) \|_{L^1_R} \dd \tau. \nonumber
\end{align}

{\sc{Step 2: Regularization.}} In this step, we prove the following lemma, which states that $\varrho$ can be approximated by a sequence
$\varrho_k$ of regular functions, defined on $\overline{\cO} \times [0,T]$, and we solve \eqref{eq:FP-eq1} with an error term $r_k$, which
tends to $0$ as $k \rightarrow \infty$. Given the sequence of mollifiers $(\omega_k)_{k \geq 1}$ defined by
\[
\omega_k(z)=k^d \omega(kz), \quad k \in \bbN, \quad \omega \in C^{\infty}(\R^d;\R_{\geq 0}), \quad \mathrm{supp} \, \omega \subset B_1,
\quad \int_{\R^d} \omega(z) \dd z = 1,
\]
where $\bbN$ is the set of all positive integers,
we introduce the sequence of regularized functions $$\tilde \varrho_k = \varrho \star_{r,k} \omega_k *_v \omega_k,$$ where $*_v$
denotes the usual convolution; thus,
\begin{align*}
(u *_v H_k)(v) &:= \int_{\R^{(J+1)d}} u(\eta) \, H_{k}\bigl(v-\eta \bigr) \dd \eta,
\\ &= \int_{\R^{(J+1)d}} u(\eta) \, \prod_{j=1}^{J+1} \, h_{k}\bigl(v_j-\eta_j \bigr) \dd \eta \\
&= \int_{\R^d} \dots \int_{\R^d} u(\eta_1,\dots, \eta_{J+1})\, \prod_{j=1}^{J+1} \, h_{k}\bigr(v_j-\eta_j \bigr) \dd \eta_1 \cdots \dd \eta_{J+1},
\end{align*}
for any function $u \in L^1(\R^{(J+1)d})$ and a function $H_k(v) :=\prod_{j=1}^{J+1} \, h_{k}\bigl(v_j\bigr)$,
$v:=(v_1^{\rm T}, \dots , v_{J+1}^{\rm T})^{\rm T} \in \R^{(J+1)d}$, where $v_j \in \R^d$ for $j=1,\dots, J+1$, $h_{k} \in L^1(\R^{d})$,
$\mathrm{supp} \, h_{k} \subset B_{\frac{1}{k}}$.
We have that the convolution above is well-defined since $ H_k \in L^1(\R^{(J+1)d})$. Hence, by Young's inequality for convolutions,
$u *_v H_k \in  L^1(\R^{(J+1)d})$ and
$$ \| u *_v H_k \|_{L^1(\R^{(J+1)d})} \leq \| u  \|_{L^1(\R^{(J+1)d})} \| H_k \|_{L^1(\R^{(J+1)d})}.$$
Now, let $u \in L^1_{loc}(\overline{\Omega^{J+1}})$. We extend $u$ by $0$ to the complement of $\Omega^{J+1}$ and we denote by $\star_{r,k}$
the convolution--translation defined by:
\begin{align*}
(u \star_{r,k} H_k)(r) &:= \int_{\R^{(J+1)d}} u(y) \, H_{k}\bigg(r-\frac{2}{k}n(r)-y \bigg) \dd y \\
&= \int_{\R^{(J+1)d}} u(y) \, \prod_{j=1}^{J+1} \, h_{k}\bigg(r_j-\frac{2}{k}n(r_j)-y_j \bigg) \dd y \\
&= \int_{\R^d} \dots \int_{\R^d} u(y_1,\dots, y_{J+1})\, \prod_{j=1}^{J+1} \, h_{k}\bigg(r_j-\frac{2}{k}n(r_j)-y_j \bigg) \dd y_1 \cdots \dd y_{J+1},
\end{align*}
where $h_{k} \in L^1(\R^{d})$, $\mathrm{supp} \, h_{k} \subset B_{\frac{1}{k}}$, $r:=(r_1^{\rm T}, \dots, r_{J+1}^{\rm T})^{\rm T} \in \Omega^{J+1}$,
$r_j \in \overline{\Omega_j} \subset \R^d$ for all $j=1,\dots, J+1$.
The point of using a convolution--translation is to ensure that the variable $y$ stays in the interior of the domain $\Omega^{J+1}$,
so that we do not create bad discontinuities in the derivatives of $u$ at the boundary of the domain. Indeed, since the mollifiers $h_k$ are
compactly supported in $B_{\frac{1}{k}}$, we have that $y_j \in B(r_j-\frac{2}{k}n(r_j),\frac{1}{k})$, for all $j=1,\dots, J+1$.
Set $\tilde r_j := r_j-\frac{2}{k}n(r_j)$ and $d(y_j, \partial \Omega_j):= \inf \{|y_j-z| \, : \, z \in \partial \Omega_j \}$ for all $j=1,\dots ,J+1$. Hence $|y_j-\tilde r_j|< \frac{1}{k}$ and for $z \in \partial \Omega_j$:
\begin{align*}
    |y_j-z| &= |y_j-\tilde r_j + \tilde r_j-z| \\
    & \geq \bigl||y_j-\tilde r_j| - |z-\tilde r_j|\bigr|.
\end{align*}
Since $|z-\tilde r_j|=d(\tilde r_j, \partial \Omega_j) > \frac{2}{k}$, we obtain that $|y_j-\tilde r_j| - |z-\tilde r_j|
< \frac{1}{k} - \frac{2}{k}= -\frac{1}{k}$. Thus we deduce that $|y_j-z|>\frac{1}{k}>0$.
This implies $d(y_j, \partial \Omega_j)>0$, for all $j=1,\dots,J+1$.
Hence, $y_j$ is in the interior of $\Omega_j$ for all $j=1,\dots, J+1$, which implies that $y$ is in the interior of $\Omega^{J+1}$.
\begin{lemma}\label{lem:5.2}
For each $k \in \bbN$ there exists a function
\[\varrho_k \in C(\overline{\Omega^{J+1}} \times \R^{(J+1)d} \times [0,T]) \cap W^{1,1} (0,T; W^{1,\infty}_{loc}(\overline{\cO})),\]
such that the sequence $\varrho_k$ satisfies:
\begin{align}
\label{convergence:FP}
\begin{split}
\mbox{$\varrho_k$ is bounded in $L^{\infty}(0,T;L^2_{loc}(\overline{\cO}))$, $\hspace{0.1in}$} \\
\mbox{$\varrho_k \rightarrow \varrho$ in $L^{a}(0,T;L^2_{loc}(\overline{\cO})) \quad \forall\, a \in [1, \infty)$}
\end{split}
\end{align}
and
\begin{align}
\label{rest:FP}
\Lambda_{E_j}\varrho_k = r_k  \mbox{ in $\mathcal{D}'\bigl(\Omega^{J+1} \times \R^{(J+1)d} \times (0,T)\bigr)$},
\end{align}
where $r_k$ converges to $0$ in $L^1_{loc}(\cO \times [0,T])$.
\end{lemma}
\begin{proof}
The proof of this lemma is inspired by the work \cite{DiPerna1989} of DiPerna \& Lions.
By considering $\varrho$ as a function of $t$, $y$ and $\eta$, i.e.,  $\varrho = \varrho(y,\eta,t)$, we multiply equation  \eqref{eq:FP-eq1} by the test function
\[ \prod_{l=1}^{J+1} \omega_k\bigg(r_l-\frac{2}{k}n(r_l)-y_l \bigg)\,\prod_{m=1}^{J+1} \omega_k(v_m-\eta_m) \, \in \mathcal{C}^\infty_0(\Omega^{J+1}_{y} \times \R^{(J+1)d}_{\eta})\]
for fixed $r \in \overline{\Omega^{J+1}}$ and $v \in \R^{(J+1)d}$, and integrate over $y$ and $\eta$. We get
\begin{align}
\begin{split}
 \pd_t \tilde \varrho_k &= - \Biggl(\sum_{j=1}^{J+1} E_j(r,v,t) \cdot \pdv \varrho \Biggr)\star_{r,k} \omega_k *_v \omega_k -
 \Biggl(\sum_{j=1}^{J+1} (\pdv \cdot E_j(r,v,t)) \varrho \Biggr)\star_{r,k} \omega_k *_v \omega_k \\ &\quad + \Biggl( \frac{\beta^2}{\eps^2}\, \sum_{j=1}^{J+1} \pdv^2 \varrho \Biggr)\star_{r,k} \omega_k *_v \omega_k - \Biggl(\frac{1}{\eps}
\sum_{j=1}^{J+1} v_j\cdot \pdr \varrho \Biggr)\star_{r,k} \omega_k *_v \omega_k \in L^1(0,T;W^{1,\infty}_{loc}(\cO)),
\end{split}
\end{align}
and, in particular, $\tilde \varrho_k \in W^{1,1}(0,T;W^{1,\infty}_{loc}(\cO))$.

Let us define $\varrho_k$ to be the continuous representative of $\tilde \varrho_k$ in the class of functions almost everywhere equal to  $\tilde \varrho_k$. Then $\varrho_k$ solves \eqref{rest:FP} with
$$r_k=r^1_k(\varrho)+r^2_k(\varrho)+r^3_k(\varrho)+r^4_k(\varrho),$$
where
\begin{align*}
r^1_k(\varrho)&=\sum_{j=1}^{J+1} E_j(r,v,t) \cdot \pdv \varrho_k - \Biggl(\sum_{j=1}^{J+1} E_j(r,v,t) \cdot \pdv \varrho \Biggr)\star_{r,k} \omega_k *_v \omega_k,\\
r^2_k(\varrho)&=\sum_{j=1}^{J+1} (\pdv \cdot E_j(r,v,t)) \varrho_k - \Biggl( \sum_{j=1}^{J+1} (\pdv \cdot E_j(r,v,t)) \varrho \Biggr)\star_{r,k} \omega_k *_v \omega_k,\\
r^3_k(\varrho)&=\Biggl( \frac{\beta^2}{\eps^2}\, \sum_{j=1}^{J+1} \pdv^2 \varrho \Biggr)\star_{r,k} \omega_k *_v \omega_k - \frac{\beta^2}{\eps^2}\, \sum_{j=1}^{J+1} \pdv^2 \varrho_k,\\
r^4_k(\varrho)&= \frac{1}{\eps}
\sum_{j=1}^{J+1} v_j\cdot \pdr \varrho_k - \Biggl(\frac{1}{\eps}
\sum_{j=1}^{J+1} v_j\cdot \pdr \varrho \Biggr)\star_{r,k} \omega_k *_v \omega_k.
\end{align*}
We have to prove that $r^1_k(\varrho)$, $r^2_k(\varrho)$, $r^3_k(\varrho)$ and $r^4_k(\varrho)$ all converge to $0$ in $L^1_{loc}$.
Let us remark that if $\varrho$ is smooth then one has
$$\pdr (\varrho \star_{r,k}\omega_k)=\bigg(I-\frac{2}{k}D^2 d_{\Omega}\bigg)(\pdr \varrho) \star_{r,k} \omega_k,$$
and therefore
\begin{align}
\label{limit_rest4}
r^4_k(\varrho) \rightarrow_{k \rightarrow \infty} 0  \quad \mbox{ in }  L^1_{loc}\bigl(\Omega^{J+1} \times \R^{(J+1)d} \times (0,T)\bigr).
\end{align}
Indeed, we have that
\begin{align*}
\pdr (\varrho \star_{r,k}\omega_k) &=  \pdr \biggl( \int_{\Omega^{J+1}}  \varrho(y,v,t) \, \prod_{l=1}^{J+1} \omega_k\bigg(r_l-\frac{2}{k}n(r_l)-y_l\bigg) \dd y \biggr) \\
& = \int_{\Omega^{J+1}}  \varrho(y,v,t) \, \pdr \biggl( \prod_{l=1}^{J+1}  \omega_k\Bigl(r_l-\frac{2}{k}n(r_l)-y_l\Bigr) \biggr) \dd y \\
& = \int_{\Omega^{J+1}}  \varrho(y,v,t) \, \pdr  \omega_k \Bigl(r_j-\frac{2}{k}n(r_j)-y_j\Bigr) \prod_{l=1,\, l \neq j}^{J+1}  \omega_k \Bigl(r_l-\frac{2}{k}n(r_l)-y_l\Bigr) \dd y \\
& = \bigg(I-\frac{2}{k}D^2 d_{\Omega}\bigg) \int_{\Omega^{J+1}}  \varrho(y,v,t) \, \nabla \omega_k\bigg(r_j-\frac{2}{k}n(r_j)-y_j \bigg)  \prod_{l=1,\, l \neq j}^{J+1}  \omega_k \Bigl(r_l-\frac{2}{k}n(r_l)-y_l\Bigr)  \dd y \\
& = - \bigg(I-\frac{2}{k}D^2 d_{\Omega}\bigg) \int_{\Omega^{J+1}}  \varrho(y,v,t) \, \partial_{y_j} \biggl( \omega_k\Bigl(r_j-\frac{2}{k}n(r_j)-y_j \Bigr) \biggr)  \prod_{l=1,\, l \neq j}^{J+1}  \omega_k \Bigl(r_l-\frac{2}{k}n(r_l)-y_l\Bigr)  \dd y \\
& = \bigg(I-\frac{2}{k}D^2 d_{\Omega}\bigg) \int_{\Omega^{J+1}}  \partial_{y_j} \varrho(y,v,t) \, \omega_k\bigg(r_j-\frac{2}{k}n(r_j)-y_j \bigg)  \prod_{l=1,\, l \neq j}^{J+1}  \omega_k \Bigl(r_l-\frac{2}{k}n(r_l)-y_l\Bigr) \dd y \\
& = \bigg(I-\frac{2}{k}D^2 d_{\Omega}\bigg) \int_{\Omega^{J+1}}  \partial_{y_j} \varrho(y,v,t) \, \prod_{l=1}^{J+1}  \omega_k \Bigl(r_l-\frac{2}{k}n(r_l)-y_l\Bigr) \dd y \\
&=\bigg(I-\frac{2}{k}D^2 d_{\Omega}\bigg)(\pdr \varrho) \star_{r,k} \omega_k.
\end{align*}
To deal with a general $\varrho \in L^{\infty,1}_{loc}$ we begin by proving an a priori estimate. One has
\begin{align*}
r^4_k(\varrho) & = \frac{1}{\eps}
\sum_{j=1}^{J+1} \int_{\R^{(J+1)d}}  \int_{\Omega^{J+1}}  \biggl \{ v_j \cdot \pdr \Bigl[\varrho(y,\eta,t) \, \prod_{l=1}^{J+1} \omega_k \Bigl(r_l-\frac{2}{k}n(r_l)-y_l \Bigr) \, \prod_{m=1}^{J+1} \omega_k(v_m-\eta_m) \Bigr] \\
& \quad - \eta_j \cdot \partial_{y_j} \varrho(y,\eta,t) \, \prod_{l=1}^{J+1} \omega_k \Bigl(r_l-\frac{2}{k}n(r_l)-y_l \Bigr) \, \prod_{m=1}^{J+1} \omega_k(v_m-\eta_m)  \biggr \} \dd y \dd \eta.
\end{align*}
By differentiation with respect to $r_j$ in the first integrand and integration by parts in the second integrand we obtain
\begin{align*}
r^4_k(\varrho) & = \frac{1}{\eps}
\sum_{j=1}^{J+1} \int_{\R^{(J+1)d}} \int_{\Omega^{J+1}}  \biggl \{ v_j \cdot \biggl[\varrho(y,\eta,t) \, \biggl(I-\frac{2}{k}D^2 d_{\Omega}\biggr)\nabla \omega_k \biggl(r_j-\frac{2}{k}n(r_j)-y_j \biggr) \, \\
& \qquad \qquad \times \prod_{l=1,l \neq j}^{J+1} \omega_k \Bigl(r_l-\frac{2}{k}n(r_l)-y_l \Bigr) \, \prod_{m=1}^{J+1}\omega_k(v_m-\eta_m) \biggr] - \eta_j \cdot \nabla  \omega_k \biggl(r_j-\frac{2}{k}n(r_j)-y_j \biggr) \, \varrho(y,\eta,t) \\
& \qquad \qquad \times \prod_{l=1,\, l \neq j}^{J+1} \omega_k \Bigl(r_l-\frac{2}{k}n(r_l)-y_l \Bigr) \, \prod_{m=1}^{J+1} \omega_k(v_m-\eta_m)  \biggr \} \dd y \dd \eta
\end{align*}
\begin{align*}
& = \frac{1}{\eps}
\sum_{j=1}^{J+1} \int_{\R^{(J+1)d}} \int_{\Omega^{J+1}}  \varrho(y,\eta,t) \, \prod_{m=1}^{J+1} \omega_k(v_m-\eta_m)\prod_{l=1,\, l \neq j}^{J+1} \omega_k \Bigl(r_l-\frac{2}{k}n(r_l)-y_l \Bigr) \\
&  \qquad \qquad \times \biggl\{ (v_j-\eta_j) \cdot \nabla \omega_k \bigg(r_j-\frac{2}{k}n(r_j)-y_j\bigg)
- \frac{2}{k} \, v_j \cdot \, \left(D^2 d_{\Omega}\right) \nabla \omega_k \bigg(r_j-\frac{2}{k}n(r_j)-y_j\bigg) \biggr \} \dd y \dd \eta.
\end{align*}
\begin{lemma}
There exists a constant $C$, which only depends on $R$ and $d_{\Omega}$, such that the following bound holds:
$$\| r^4_k(\varrho)\|_{L^1(\cD_R)} \le C \, \| \varrho \|_{L^1(\cD_{R+1})}.$$
\end{lemma}
\begin{proof} The proof proceeds as follows. We note that
{\small
\begin{align}
\begin{split}
\label{r_4_bound}
      \| r^4_k(\varrho) \|_{L^1(\cD_R)} &= \int_0^T \int_{(\Omega \cap B_{R})^{J+1}} \int_{B_{R}^{J+1}} \biggl | \frac{1}{\eps}
\sum_{j=1}^{J+1} \int_{\R^{(J+1)d}} \int_{\Omega^{J+1}}  \varrho(y,\eta,t) \, \prod_{m=1}^{J+1} \omega_k(v_m-\eta_m)\prod_{l=1,\, l \neq j}^{J+1} \omega_k \Bigl(r_l-\frac{2}{k}n(r_l)-y_l \Bigr)
 \\
&  \qquad   \times\biggl\{(v_j-\eta_j) \cdot \nabla \omega_k \bigg(r_j-\frac{2}{k}n(r_j)-y_j\bigg)
- \frac{2}{k} v_j \cdot (D^2 d_{\Omega})\nabla \omega_k \bigg(r_j-\frac{2}{k} \, n(r_j)-y_j\bigg) \biggr \} \dd y \dd \eta \,\biggr | \dd v \dd r \dd t \\
& \leq \frac{1}{\eps}
\sum_{j=1}^{J+1} \int_0^T \int_{(\Omega \cap B_R)^{J+1}} \int_{B_R^{J+1}} \int_{\R^{(J+1)d}} \int_{\Omega^{J+1}} \biggl | \varrho(y,\eta,t) \, \prod_{m=1}^{J+1} \omega_k(v_m-\eta_m)\prod_{l=1,\, l \neq j}^{J+1}
\omega_k \Bigl(r_l-\frac{2}{k}n(r_l)-y_l \Bigr)
 \\
&  \qquad  \times  \biggl\{(v_j-\eta_j) \cdot \nabla \omega_k \bigg(r_j-\frac{2}{k}n(r_j)-y_j\bigg)
- \frac{2}{k} \, v_j \cdot (D^2 d_{\Omega})\nabla \omega_k \bigg(r_j-\frac{2}{k} \, n(r_j)-y_j\bigg) \biggr \} \biggr | \dd y \dd \eta \dd v \dd r \dd t \\
& \leq \frac{1}{\eps}
\sum_{j=1}^{J+1} \int_0^T \int_{(\Omega \cap B_R)^{J+1}} \int_{B_R^{J+1}} \int_{\R^{(J+1)d}} \int_{\Omega^{J+1}} \biggl | \varrho(y,\eta,t) \, \prod_{m=1}^{J+1} \omega_k(v_m-\eta_m)\prod_{l=1,\, l \neq j}^{J+1} \omega_k \Bigl(r_l-\frac{2}{k}n(r_l)-y_l \Bigr)
 \\
&  \qquad \qquad  \times \biggl\{(v_j-\eta_j) \cdot \nabla \omega_k \bigg(r_j-\frac{2}{k}n(r_j)-y_j\bigg)\biggr\} \biggr | \dd y \dd \eta \dd v \dd r \dd t \\
& + \frac{1}{\eps} \sum_{j=1}^{J+1} \int_0^T \int_{(\Omega \cap B_R)^{J+1}} \int_{B_R^{J+1}} \int_{\R^{(J+1)d}}
\int_{\Omega^{J+1}} \biggl | \frac{2}{k} \, \varrho(y,\eta,t) \, \prod_{m=1}^{J+1} \omega_k(v_m-\eta_m)\prod_{l=1,\, l \neq j}^{J+1}
\omega_k \Bigl(r_l-\frac{2}{k}n(r_l)-y_l \Bigr) \\
& \qquad \qquad \times \biggl\{v_j \cdot (D^2 d_{\Omega})\nabla \omega_k \bigg(r_j-\frac{2}{k}n(r_j)-y_j\bigg) \biggr\} \biggr | \dd y \dd \eta \dd v \dd r \dd t \\
& \leq I_1 + I_2,
\end{split}
\end{align}
}
where we set
\begin{align*}
 I_1 &:= \frac{1}{\eps}
\sum_{j=1}^{J+1} \int_0^T \int_{(\Omega \cap B_R)^{J+1}} \int_{B_R^{J+1}} \int_{\R^{(J+1)d}} \int_{\Omega^{J+1}}
\biggl | \varrho(y,\eta,t) \, \prod_{m=1}^{J+1} \omega_k(v_m-\eta_m)\prod_{l=1,\, l \neq j}^{J+1} \omega_k \Bigl(r_l-\frac{2}{k}n(r_l)-y_l \Bigr)
 \\
&  \qquad \qquad  \times \biggl\{ (v_j-\eta_j) \cdot \nabla \omega_k \bigg(r_j-\frac{2}{k}n(r_j)-y_j\bigg) \biggr \} \biggr | \dd y \dd \eta \dd v \dd r \dd t
\end{align*}
and
\begin{align*}
 I_2 &:=  \frac{1}{\eps} \sum_{j=1}^{J+1} \int_0^T \int_{(\Omega \cap B_R)^{J+1}} \int_{B_R^{J+1}} \int_{\R^{(J+1)d}} \int_{\Omega^{J+1}} \biggl | \frac{2}{k} \, \varrho(y,\eta,t) \, \prod_{m=1}^{J+1} \omega_k(v_m-\eta_m)\prod_{l=1,\, l \neq j}^{J+1} \omega_k \Bigl(r_l-\frac{2}{k} \, n(r_l)-y_l \Bigr) \\
& \qquad \qquad \times \biggl\{v_j \cdot (D^2 d_{\Omega})\nabla \omega_k \bigg(r_j-\frac{2}{k}n(r_j)-y_j\bigg) \biggr\}\biggr | \dd y \dd \eta \dd v \dd r \dd t.
\end{align*}
We have the following upper bound on $I_1$:
\begin{align*}
 I_1 & \leq \frac{1}{\eps}
\sum_{j=1}^{J+1} \int_0^T \int_{(\Omega \cap B_R)^{J+1}} \int_{B_R^{J+1}} \int_{\R^{(J+1)d}} \int_{\Omega^{J+1}}  \varrho(y,\eta,t) \, \prod_{m=1, \,m \neq j}^{J+1} \omega_k(v_m-\eta_m) \prod_{l=1,\, l \neq j}^{J+1} \omega_k \Bigl(r_l-\frac{2}{k}n(r_l)-y_l \Bigr)
 \\
&  \qquad \qquad  \times \biggl|  \omega_k(v_j-\eta_j) (v_j-\eta_j) \biggr| \biggl|\nabla \omega_k \bigg(r_j-\frac{2}{k}n(r_j)-y_j\bigg) \biggr|  \dd y \dd \eta \dd v \dd r \dd t.
\end{align*}
Now, using the Fubini--Tonelli theorem, we obtain:
\begin{align}
\label{I_1:bound}
\begin{split}
&I_1  \leq \frac{1}{\eps}
\sum_{j=1}^{J+1} \int_0^T \int_{\Omega \cap B_R} \int_{B_R} \int_{\R^{(J+1)d}} \int_{\Omega^{J+1}}  \varrho(y,\eta,t) \, \biggl( \prod_{m=1, \,m \neq j}^{J+1} \int_{B_R} \omega_k(v_m-\eta_m) \dd v_m \biggr)
 \\
&\!\times \biggl(\prod_{l=1,\, l \neq j}^{J+1} \int_{\Omega \cap B_R}\!\omega_k \Bigl(r_l-\frac{2}{k}n(r_l)-y_l \Bigr) \dd r_l \biggr) \biggl|  \omega_k(v_j-\eta_j) (v_j-\eta_j) \biggr| \biggl|\nabla \omega_k \bigg(r_j-\frac{2}{k}n(r_j)-y_j\bigg) \biggr|  \dd y \dd \eta \dd v_j \dd r_j \dd t.
\end{split}
\end{align}
First, using the change of variable $z=v_m-\eta_m$, which implies that $\dd z=\dd v_m$, we obtain the following upper bound:
\begin{align}
\begin{split}
\label{upp_bound_1}
\prod_{m=1, \,m \neq j}^{J+1} \int_{B_R} \omega_k(v_m-\eta_m) \dd v_m & = \prod_{m=1, \,m \neq j}^{J+1} \int_{\R^d} \omega_k(z) \dd z \\
& = \prod_{m=1, \,m \neq j}^{J+1} \int_{\R^d} k^d \omega(kz) \dd z \\
& = \prod_{m=1, \,m \neq j}^{J+1} \int_{\R^d} \omega(z) \dd z \\
& =1.
\end{split}
\end{align}
\begin{remark}
\label{remark1_bound}
Let us also remark that since $|v_m-\eta_m| \le \frac{1}{k}$ for all $m=1,\dots, J+1$ and  $k \ge 1$, i.e., $\frac{1}{k} \le 1$,
which imply that $|\eta_m| \le |v_m| + \frac{1}{k} \le R + 1$, for all $m=1,\dots, J+1$, we have that $\eta \in B^{J+1}_{R+1}$.
\end{remark}
Then, we perform the change of variable $s_l=r_l-\frac{2}{k}n(r_l)-y_l$, which implies that
$\dd s_l = (I-\frac{2}{k} \nabla_{r_l}n(r_l)) \dd r_l$, i.e.,  $\dd r_l = (I-\frac{2}{k} \nabla_{r_l}n(r_l))^{-1} \dd s_l$ for all $l=1,\dots, J+1$.
For $l \in \{1,\dots, J+1\}$ fixed (and therefore not explicitly indicated),  we set $A:=\frac{2}{k} \nabla_{r_l}n(r_l)$. We have that
$$|A| \leq \frac{2 \, \|\nabla n\|_{L^{\infty}(\R^d)}}{k}.$$
Hence $|A|<1$ for all $k > 2 \, \|\nabla n\|_{L^{\infty}(\R^d)}=:b$. In that case,
\begin{align*}
   |(I-A)^{-1}| &= \biggl|\sum_{n=0}^{\infty} A^n \biggr| \\
   & \leq \sum_{n=0}^{\infty} \bigl| A \bigr|^n \\
   & \leq \sum_{n=0}^{\infty} \biggl(\frac{2 \, \|\nabla n\|_{L^{\infty}(\R^d)}}{k} \biggr)^n \\
   & = \frac{1}{1- \frac{2 \, \|\nabla n\|_{L^{\infty}(\R^d)}}{k}} \\
   & = \frac{k}{k-2 \, \|\nabla n\|_{L^{\infty}(\R^d)}}.
\end{align*}
We have that the function $g : x \mapsto \frac{x}{x-b}$, where $x>b$, is strictly monotonic decreasing.
Indeed, $g'(x)=\frac{-b}{(x-b)^2} < 0$ for all $x >b$. Hence, taking $x \geq 2b$, the maximum of $g$ is achieved at $x=2b$, where $g(x)=2$.

\smallskip

Therefore, by choosing $k \ge 2 b$, we get
$$  |(I-A)^{-1}| \leq 2;$$
thus, the change of variable gives
\begin{align}
\begin{split}
\label{upp_bound_2}
    \prod_{l=1,\, l \neq j}^{J+1} \int_{\Omega \cap B_R} \omega_k \Bigl(r_l-\frac{2}{k}n(r_l)-y_l \Bigr) \dd r_l & \leq    \prod_{l=1,\, l \neq j}^{J+1} 2 \int_{\R^d} \omega_k(s_l) \dd s_l \\
     & = \prod_{l=1,\, l \neq j}^{J+1} 2 \int_{\R^d} k^d \omega(k \, s_l) \dd s_l \\
     & = \prod_{l=1,\, l \neq j}^{J+1} 2 \int_{\R^d} \omega(z) \dd z \\
     & \leq  2^J.
    \end{split}
\end{align}
\begin{remark}
\label{remark2_bound}
Let us also remark that since $|r_l-\frac{2}{k}n(r_l)-y_l| \le \frac{1}{k}$, for all $m=1,\dots, J+1$, and for $k \ge 3$,
i.e., $\frac{3}{k} \le 1$, which imply that $|y_l| \le |r_l| + \frac{3}{k} \le R + 1$, we have that $y \in (\Omega \cap B_{R+1})^{J+1}$.
\end{remark}
Hence, using estimates \eqref{upp_bound_1} and \eqref{upp_bound_2} together with Remarks \ref{remark1_bound} and \ref{remark2_bound} in \eqref{I_1:bound},
we get for all $k \ge \max(2 b,3)$ that
\begin{align}
\begin{split}
\label{I_1:bound2}
I_1 & \leq \frac{2^J}{\eps}
\sum_{j=1}^{J+1} \int_0^T \int_{\Omega \cap B_R} \int_{B_R} \int_{B_{R+1}^{J+1}} \int_{(\Omega \cap B_{R+1})^{J+1}}  \varrho(y,\eta,t)\,
\biggl|  \omega_k(v_j-\eta_j) (v_j-\eta_j) \biggr|  \\
&
\qquad \qquad \qquad \times \biggl|\nabla \omega_k \bigg(r_j-\frac{2}{k}n(r_j)-y_j\bigg) \biggr|  \dd y \dd \eta \dd v_j \dd r_j \dd t \\
& \leq \frac{2^J}{\eps}
\sum_{j=1}^{J+1} \int_0^T \int_{B_{R+1}^{J+1}} \int_{(\Omega \cap B_{R+1})^{J+1}}  \varrho(y,\eta,t) \biggl( \int_{B_R}
\biggl|  \omega_k(v_j-\eta_j) (v_j-\eta_j) \biggr| \dd v_j \biggr)  \\
&
\qquad \qquad \qquad \times \biggl( \int_{\Omega \cap B_R} \biggl|\nabla \omega_k \bigg(r_j-\frac{2}{k}n(r_j)-y_j\bigg) \biggr| \dd r_j\biggr)  \dd y \dd \eta \dd t.
\end{split}
\end{align}
Again, by performing the same change of variable as in \eqref{upp_bound_2}, we obtain, for a constant $C_1>0$, independent of $k$, the
following bound:
\begin{align}
\label{upp_bound_3}
    \begin{split}
    \int_{\Omega \cap B_R} \biggl|\nabla \omega_k \bigg(r_j-\frac{2}{k}n(r_j)-y_j\bigg) \biggr| \dd r_j & \leq 2  \int_{\R^d} \bigl|\nabla \omega_k(\xi) \bigr| \dd \xi \\
    & = 2 \int_{\R^d} k^{d+1} |\nabla \omega(k \, \xi) \bigr| \dd \xi \\
    & = 2 k \int_{\R^d}  |\nabla \omega(z) \bigr| \dd z \\
    & \leq  2 \, k \, C_1.
    \end{split}
\end{align}
Finally, noting that $w_k \in C_0^{\infty}(\R^d)$, we obtain, for a constant $C_2>0$, independent of $k$, that
\begin{align}
\label{upp_bound_4}
    \begin{split}
       \int_{B_R}
\biggl|  \omega_k(v_j-\eta_j) (v_j-\eta_j) \biggr| \dd v_j & \leq \int_{\R^d} \omega_k(\xi) |\xi| \dd \xi \\
& = \int_{\R^d} k^d \omega(k \,\xi) |\xi| \dd \xi \\
& = \frac{1}{k} \int_{\R^d} \omega(z) |z| \dd z \\
& \leq \frac{C_2}{k}.
    \end{split}
\end{align}
We use \eqref{upp_bound_3} and \eqref{upp_bound_4} in \eqref{I_1:bound2}, for a constant $C>0$, independent of $k$,
we get the following bound on the term $I_1$:
\begin{align}
\begin{split}
\label{I_1:bound_final}
I_1 & \leq \frac{2^J}{\eps}
\sum_{j=1}^{J+1} \int_0^T \int_{B_{R+1}^{J+1}} \int_{(\Omega \cap B_{R+1})^{J+1}} \varrho(y,\eta,t) \biggl( \int_{B_R}
\biggl|  \omega_k(v_j-\eta_j) (v_j-\eta_j) \biggr| \dd v_j \biggr)  \\
&
\qquad \qquad \qquad \times \biggl( \int_{\Omega \cap B_R} \biggl|\nabla \omega_k \bigg(r_j-\frac{2}{k}n(r_j)-y_j\bigg) \biggr| \dd r_j\biggr)  \dd y \dd \eta \dd t \\
 & \leq \frac{2^{J+1} \, C_1 \, C_2}{\eps}
\sum_{j=1}^{J+1} \int_0^T \int_{B_{R+1}^{J+1}} \int_{(\Omega \cap B_{R+1})^{J+1}}  \varrho(y,\eta,t) \dd y \dd \eta \dd t \\
& \leq C \| \varrho\|_{L^1(\cD_{R+1})}.
\end{split}
\end{align}
Now, using the same technique as for the bound on $I_1$, we get the following bound on the term $I_2$:
\begin{align}
\begin{split}
\label{I_2:bound_final}
I_2 & \leq \frac{4 \, |B_R|}{\eps \, k}
\sum_{j=1}^{J+1} \int_0^T \int_{\Omega \cap B_R} \int_{\R^{(J+1)d}} \int_{\Omega^{J+1}} \bigl|  D^2 d_{\Omega} \bigr| \varrho(y,\eta,t) \, \biggl( \prod_{m=1}^{J+1} \int_{B_R} \omega_k(v_m-\eta_m) \dd v_m \biggr)
 \\
& \qquad \qquad\times \biggl(\prod_{l=1,\, l \neq j}^{J+1} \omega_k \Bigl(r_l-\frac{2}{k}n(r_l)-y_l \Bigr) \dd r_l \biggr) \, \biggl|\nabla \omega_k \bigg(r_j-\frac{2}{k}n(r_j)-y_j\bigg) \biggr|  \dd y \dd \eta \dd r_j \dd t \\
& \leq  \frac{2^{J+2} \, |B_R| \, \| D^2 d_{\Omega}\|_{L^{\infty}(\R^d)}}{\eps \, k}
\sum_{j=1}^{J+1} \int_0^T \int_{B_{R+1}^{J+1}} \int_{(\Omega \cap B_{R+1})^{J+1}}  \varrho(y,\eta,t) \\
& \qquad \qquad \times \biggl( \int_{\Omega \cap B_R} \biggl|\nabla \omega_k \bigg(r_j-\frac{2}{k}n(r_j)-y_j\bigg) \biggr| \dd r_j \biggr) \dd y \dd \eta \dd t \\
& \leq \frac{2^{J+2} \, C_1 \, |B_R| \, \| D^2 d_{\Omega}\|_{L^{\infty}(\R^d)}}{\eps}
\sum_{j=1}^{J+1} \int_0^T \int_{B_{R+1}^{J+1}} \int_{(\Omega \cap B_{R+1})^{J+1}}  \varrho(y,\eta,t) \dd y \dd \eta \dd t \\
& \leq C \, \| \varrho\|_{L^1(\cD_{R+1})}.
\end{split}
\end{align}
In conclusion, using \eqref{I_1:bound_final} and \eqref{I_2:bound_final} in \eqref{r_4_bound}, we get,
with a constant $C$ that only depends on $R$ and $d(r)$, the following bound:
\begin{align}
\label{apriori3}
\begin{split}
\| r^4_k(\varrho) \|_{L^1(\cD_R)}  & \le I_1 + I_2\\
&\le C \| \varrho\|_{L^1(\cD_{R+1})}.
\end{split}
\end{align}
That completes the proof of the lemma.
\end{proof}

Next, for $ \varrho \in L^\infty(0,T;L^1(\Omega^{J+1} \times \R^{(J+1)d};\R_{\geq 0}))$ we argue by density; in other words, we consider a
sequence $(\varrho_{\eps})_{\eps>0}$ of smooth functions such that $\varrho_{\eps} \rightarrow \varrho$ in $L^1_{loc}(\overline{\cD})$,
and we note the following obvious decomposition of the function $r^4_k(\varrho)$:
\[
r^4_k(\varrho)= r^4_k(\varrho_{\eps})+ r^4_k(\varrho - \varrho_{\eps}),
\]
which obviously converges to $0$ in $L^1_{loc}(\overline{\cD})$ as $k \rightarrow \infty$ and $\eps \rightarrow 0$,
thanks to \eqref{limit_rest4} and \eqref{apriori3}.

The convergence of the sequence $r^1_k(\varrho)$ to $0$, as $k \rightarrow \infty$, has been proved, in a simpler case, in the article of DiPerna \& Lions \cite{DiPerna1989},
Lemma II.1. By following a similar line of argument as in \cite{DiPerna1989}, we proceed as follows. If $\varrho$ and $E_j$ are sufficiently smooth, we have that
\begin{align*}
r^1_k(\varrho)&=\sum_{j=1}^{J+1} E_j(r,v,t) \cdot \pdv \varrho_k - \Biggl(\sum_{j=1}^{J+1} E_j(r,v,t) \cdot \pdv \varrho \Biggr)\star_{r,k} \omega_k *_v \omega_k \\
&= \sum_{j=1}^{J+1}\int_{\R^{(J+1)d}} \int_{\Omega^{J+1}}  \Bigl \{ E_j(r,v,t) \cdot \pdv \Bigl[\varrho(y,\eta,t) \, \prod_{l=1}^{J+1}\omega_k\bigg(r_l-\frac{2}{k}n(r_l)-y_l \bigg) \, \prod_{m=1}^{J+1} \omega_k(v_m-\eta_m) \Bigr] \\
& \qquad - E_j(y,\eta,t)  \cdot \partial_{\eta_j} \varrho(y,\eta,t) \, \prod_{l=1}^{J+1} \omega_k\bigg(r_l-\frac{2}{k}n(r_l)-y_l \bigg) \, \prod_{m=1}^{J+1} \omega_k(v_m-\eta_m)  \Bigr \} \dd y \dd \eta.
\end{align*}
By performing integration by parts on the second integrand, we get
\begin{align*}
r^1_k(\varrho)&= \sum_{j=1}^{J+1}\int_{\R^{(J+1)d}} \int_{\Omega^{J+1}}  \biggl \{ E_j(r,v,t) \cdot \nabla \omega_k(v_j-\eta_j) \, \prod_{l=1}^{J+1} \omega_k\Bigl(r_l-\frac{2}{k}n(r_l)-y_l \Bigr) \\
& \qquad  \times \varrho(y,\eta,t) \prod_{m=1, \,m \neq j}^{J+1} \omega_k(v_m-\eta_m) + \partial_{\eta_j}  \cdot \Bigl[ E_j(y,\eta,t) \, \omega_k(v_j-\eta_j) \Bigr] \\
& \qquad \times \varrho(y,\eta,t) \, \prod_{m=1, \,m \neq j}^{J+1} \omega_k(v_m-\eta_m) \, \prod_{l=1}^{J+1} \omega_k\Bigl(r_l-\frac{2}{k}n(r_l)-y_l \Bigr) \biggr \} \dd y \dd \eta  \\
&= \sum_{j=1}^{J+1}\int_{\R^{(J+1)d}} \int_{\Omega^{J+1}}  \biggl \{ E_j(r,v,t) \cdot \nabla \omega_k(v_j-\eta_j) \, \prod_{l=1}^{J+1} \omega_k\Bigl(r_l-\frac{2}{k}n(r_l)-y_l \Bigr) \\
& \qquad \times  \varrho(y,\eta,t) \, \prod_{m=1, \,m \neq j}^{J+1} \omega_k(v_m-\eta_m) + \partial_{\eta_j}  \cdot [E_j(y,\eta,t)] \, \varrho(y,\eta,t) \\
& \qquad \times \prod_{m=1}^{J+1} \omega_k(v_m-\eta_m)  \prod_{l=1}^{J+1} \omega_k\bigg(r_l-\frac{2}{k}n(r_l)-y_l\bigg)  -  E_j(y,\eta,t) \cdot \nabla \omega_k(v_j-\eta_j) \\
& \qquad \times \varrho(y,\eta,t) \, \prod_{l=1}^{J+1} \omega_k\bigg(r_l-\frac{2}{k}n(r_l)-y_l \bigg) \prod_{m=1, \,m \neq j}^{J+1} \omega_k(v_m-\eta_m) \biggr \} \dd y \dd \eta \\
&= \sum_{j=1}^{J+1}\int_{\R^{(J+1)d}} \int_{\Omega^{J+1}}  \varrho(y,\eta,t) \, \prod_{l=1}^{J+1} \omega_k\bigg(r_l-\frac{2}{k}n(r_l)-y_l\bigg)  \prod_{m=1, \,m \neq j}^{J+1} \omega_k(v_m-\eta_m)\\
& \qquad  \times \Bigl \{ \bigl[  E_j(r,v,t)- E_j(y,\eta,t) \bigr] \cdot \nabla \omega_k(v_j-\eta_j) \Bigr \} \dd y \dd \eta
\\
&\qquad\qquad + \Bigl(\sum_{j=1}^{J+1} \varrho(r,v,t)  \, (\pdv \cdot E_j(r,v,t)) \Bigr)\star_{r,k} \omega_k *_v \omega_k.
\end{align*}
The second term on the right-hand side converges to $$\sum_{j=1}^{J+1} \varrho(r,v,t) (\pdv \cdot E_j(r,v,t))$$ in $L^1_{loc}$,
as $k$ tends to $\infty$, by standard results for convolutions.

For the first integral, on the one hand, we have that
%
\begin{align*}
& \sum_{j=1}^{J+1}\int_{\R^{(J+1)d}} \int_{\Omega^{J+1}}  \varrho(y,\eta,t) \,  \prod_{l=1}^{J+1} \omega_k\bigg(r_l-\frac{2}{k}n(r_l)-y_l\bigg) \, \prod_{m=1, \,m \neq j}^{J+1} \omega_k(v_m-\eta_m) \,  E_j(r,v,t) \cdot \nabla \omega_k(v_j-\eta_j)  \dd y \dd \eta \\
&= - \sum_{j=1}^{J+1}\int_{\R^{(J+1)d}} \int_{\Omega^{J+1}}  \varrho(y,\eta,t)  \,  \prod_{l=1}^{J+1} \omega_k\bigg(r_l-\frac{2}{k}n(r_l)-y_l\bigg)
\\
&\qquad \qquad \times \prod_{m=1, \,m \neq j}^{J+1} \omega_k(v_m-\eta_m) \,  E_j(r,v,t) \cdot \partial_{\eta_j} \bigl(\omega_k(v_j-\eta_j) \bigr) \dd y \dd \eta \\
&= \sum_{j=1}^{J+1}\int_{\R^{(J+1)d}} \int_{\Omega^{J+1}}  \partial_{\eta_j} \varrho(y,\eta,t) \cdot  E_j(r,v,t) \,  \omega_k(v_j-\eta_j) \,  \prod_{l=1}^{J+1} \omega_k\bigg(r_l-\frac{2}{k}n(r_l)-y_l\bigg) \\
&\qquad \qquad \times  \prod_{m=1, \,m \neq j}^{J+1} \omega_k(v_m-\eta_m) \dd y \dd \eta \\
&= \sum_{j=1}^{J+1}\int_{\R^{(J+1)d}} \int_{\Omega^{J+1}}  \partial_{\eta_j} \varrho(y,\eta,t) \cdot  E_j(r,v,t) \, \prod_{l=1}^{J+1} \omega_k\bigg(r_l-\frac{2}{k}n(r_l)-y_l\bigg) \, \prod_{m=1}^{J+1} \omega_k(v_m-\eta_m) \dd y \dd \eta,
\end{align*}
which converges to $$\sum_{j=1}^{J+1} \partial_{v_j} \varrho(r,v,t) \cdot  E_j(r,v,t)$$ in $L^1_{loc}$ as $k$ tends to $\infty$.
On the other hand,
{\small
\begin{align*}
&\quad - \sum_{j=1}^{J+1}\int_{\R^{(J+1)d}} \int_{\Omega^{J+1}}  \varrho(y,\eta,t)  \,  \prod_{l=1}^{J+1} \omega_k\bigg(r_l-\frac{2}{k}n(r_l)-y_l\bigg) \, \prod_{m=1, \,m \neq j}^{J+1} \omega_k(v_m-\eta_m) \, E_j(y,\eta,t) \cdot \nabla \omega_k(v_j-\eta_j)  \dd y \dd \eta \\
&=  \sum_{j=1}^{J+1}\int_{\R^{(J+1)d}} \int_{\Omega^{J+1}}  \varrho(y,\eta,t) \,  \,  \prod_{l=1}^{J+1} \omega_k\bigg(r_l-\frac{2}{k}n(r_l)-y_l\bigg) \, \prod_{m=1, \,m \neq j}^{J+1} \omega_k(v_m-\eta_m) \, \,  E_j(y,\eta,t) \cdot \partial_{\eta_j} \bigl(\omega_k(v_j-\eta_j) \bigr) \dd y \dd \eta \\
&= - \sum_{j=1}^{J+1}\int_{\R^{(J+1)d}} \int_{\Omega^{J+1}}  \partial_{\eta_j}  \cdot \bigl[\varrho(y,\eta,t) \, E_j(y,\eta,t)\bigr]  \, \omega_k(v_j-\eta_j) \,  \prod_{l=1}^{J+1} \omega_k\bigg(r_l-\frac{2}{k}n(r_l)-y_l\bigg) \, \prod_{m=1, \,m \neq j}^{J+1} \omega_k(v_m-\eta_m) \dd y \dd \eta \\
&= - \sum_{j=1}^{J+1}\int_{\R^{(J+1)d}} \int_{\Omega^{J+1}} \partial_{\eta_j}  \cdot \bigl[\varrho(y,\eta,t) \, E_j(y,\eta,t)\bigr] \, \prod_{l=1}^{J+1} \omega_k\bigg(r_l-\frac{2}{k}n(r_l)-y_l\bigg) \, \prod_{m=1}^{J+1} \omega_k(v_m-\eta_m) \dd y \dd \eta
\\
& =- \sum_{j=1}^{J+1}\int_{\R^{(J+1)d}} \int_{\Omega^{J+1}}  \partial_{\eta_j} \varrho(y,\eta,t) \cdot E_j(y,\eta,t) \, \prod_{l=1}^{J+1} \omega_k\bigg(r_l-\frac{2}{k}n(r_l)-y_l\bigg) \, \prod_{m=1}^{J+1} \omega_k(v_m-\eta_m) \dd y \dd \eta \\
&  \quad - \sum_{j=1}^{J+1}\int_{\R^{(J+1)d}} \int_{\Omega^{J+1}}  \varrho(y,\eta,t) \, \partial_{\eta_j}  \cdot E_j(y,\eta,t) \, \prod_{l=1}^{J+1} \omega_k\bigg(r_l-\frac{2}{k}n(r_l)-y_l\bigg) \, \prod_{m=1}^{J+1} \omega_k(v_m-\eta_m) \dd y \dd \eta,
\end{align*}
}

\noindent
which converges to $$ - \sum_{j=1}^{J+1} \partial_{v_j} \varrho(r,v,t) \cdot  E_j(r,v,t) - \sum_{j=1}^{J+1} \varrho(r,v,t) \, (\pdv \cdot E_j(r,v,t))$$
in $L^1_{loc}$ as $k$ tends to $\infty$. Hence $r^1_k$ converges in $L^1_{loc}$ to $0$ as $k$ tends to $\infty$.
The general case follows by means of a density argument, using the inequality stated in the next lemma.
\begin{lemma}
There exists a constant $C$, which only depends on $R$ and $d_{\Omega}$, such that the following bound holds:
$$ \| r^1_k(\varrho) \|_{L^1(\cD_R)} \leq  C  \| \varrho\|_{L^2(\cD_{R+1})}.$$
\end{lemma}
\begin{proof} We begin by noting that
\begin{align*}
r^1_k(\varrho) &= \sum_{j=1}^{J+1}\int_{\R^{(J+1)d}} \int_{\Omega^{J+1}}  \varrho(y,\eta,t) \, \prod_{l=1}^{J+1} \omega_k\biggl(r_l-\frac{2}{k}n(r_l)-y_l \biggr)  \prod_{m=1, \,m \neq j}^{J+1} \omega_k(v_m-\eta_m)
\\
&\qquad \qquad \times \Bigl \{ \bigl[  E_j(r,v,t)
- E_j(y,\eta,t)  \bigr] \cdot \nabla \omega_k(v_j-\eta_j) \Bigr \} \dd y \dd \eta \\
& \quad + \sum_{j=1}^{J+1}\int_{\R^{(J+1)d}} \int_{\Omega^{J+1}}  \partial_{\eta_j}  \cdot E_j(y,\eta,t) \\
& \qquad \qquad \times \prod_{l=1}^{J+1} \omega_k\biggl(r_l-\frac{2}{k}n(r_l)-y_l \biggr) \, \prod_{m=1}^{J+1} \omega_k(v_m-\eta_m)  \, \varrho(y,\eta,t) \dd y \dd \eta \\
& =: I_1 + I_2,
\end{align*}
where we set
\begin{align}
\begin{split}
        \label{I_1:def2}
 I_1 & :=   \sum_{j=1}^{J+1}\int_{\R^{(J+1)d}} \int_{\Omega^{J+1}}  \varrho(y,\eta,t) \, \prod_{l=1}^{J+1} \omega_k\biggl(r_l-\frac{2}{k}n(r_l)-y_l \biggr)  \prod_{m=1, \,m \neq j}^{J+1} \omega_k(v_m-\eta_m) \\
 &\qquad \qquad \times \Bigl \{ \bigl[  E_j(r,v,t) - E_j(y,\eta,t)  \bigr] \cdot \nabla \omega_k(v_j-\eta_j) \Bigr \} \dd y \dd \eta,
 \end{split}
\end{align}
and
\begin{align}
\begin{split}
        \label{I_2:def2}
    I_2 & := \sum_{j=1}^{J+1}\int_{\R^{(J+1)d}} \int_{\Omega^{J+1}}  \partial_{\eta_j}  \cdot E_j(y,\eta,t) \prod_{l=1}^{J+1} \omega_k\biggl(r_l-\frac{2}{k}n(r_l)-y_l \biggr) \, \prod_{m=1}^{J+1} \omega_k(v_m-\eta_m)  \, \varrho(y,\eta,t) \dd y \dd \eta.
    \end{split}
\end{align}
Let us recall that
\[ E_j=E_j(r,v,t):=\frac{1}{\eps} (({\mathcal L}r)_j+u(r_j,t)) - \frac{\beta^2}{\eps^2} v_j.\]
Hence,
\begin{align}
    \begin{split}
        \label{E_j_decomposition}
        E_j(r,v,t)-E_j(y,\eta,t)= \frac{1}{\eps} \bigl( ({\mathcal L}r)_j - ({\mathcal L}y)_j \bigr) + \frac{1}{\eps} \bigl( u(r_j,t) - u(y_j,t) \bigr) - \frac{\beta^2}{\eps^2} (v_j - \eta_j).
    \end{split}
\end{align}
Using \eqref{E_j_decomposition} in \eqref{I_1:def2}, we get
\begin{align}
\begin{split}
 I_1 & :=   \sum_{j=1}^{J+1}\int_{\R^{(J+1)d}} \int_{\Omega^{J+1}}  \varrho(y,\eta,t) \, \prod_{l=1}^{J+1} \omega_k\biggl(r_l-\frac{2}{k}n(r_l)-y_l \biggr)  \prod_{m=1, \,m \neq j}^{J+1} \omega_k(v_m-\eta_m) \\
 &\qquad \qquad \times \biggl \{ \biggl[   \frac{1}{\eps} \bigl( ({\mathcal L}r)_j  - ({\mathcal L}y)_j \bigr)
 + \frac{1}{\eps} \bigl( u(r_j,t) - u(y_j,t) \bigr) - \frac{\beta^2}{\eps^2} (v_j - \eta_j)  \biggr] \cdot \nabla \omega_k(v_j-\eta_j) \biggr \} \dd y \dd \eta \\
&=: I^1_1+ I_1^2 +I_1^3,
 \end{split}
\end{align}
where we set
\begin{align*}
     I^1_1 &:= \sum_{j=1}^{J+1}\int_{\R^{(J+1)d}} \int_{\Omega^{J+1}}  \varrho(y,\eta,t) \, \prod_{l=1}^{J+1} \omega_k\biggl(r_l-\frac{2}{k}n(r_l)-y_l \biggr)  \prod_{m=1, \,m \neq j}^{J+1} \omega_k(v_m-\eta_m) \\
        & \qquad \qquad \quad \times \frac{1}{\eps} \bigl( ({\mathcal L}r)_j  - ({\mathcal L}y)_j \bigr) \cdot
        \nabla \omega_k(v_j-\eta_j) \dd y \dd \eta,
\end{align*}
\begin{align*}
     I^2_1 &:=  \sum_{j=1}^{J+1}\int_{\R^{(J+1)d}} \int_{\Omega^{J+1}}  \varrho(y,\eta,t) \, \prod_{l=1}^{J+1} \omega_k\biggl(r_l-\frac{2}{k}n(r_l)-y_l \biggr)  \prod_{m=1, \,m \neq j}^{J+1} \omega_k(v_m-\eta_m) \\
        & \qquad \qquad \quad \times \frac{1}{\eps} \, \bigl( u(r_j,t) - u(y_j,t) \bigr) \cdot \nabla \omega_k(v_j-\eta_j) \dd y \dd \eta,
\end{align*}
and
\begin{align*}
     I^3_1 &:= - \sum_{j=1}^{J+1}\int_{\R^{(J+1)d}} \int_{\Omega^{J+1}}  \varrho(y,\eta,t) \, \prod_{l=1}^{J+1} \omega_k\biggl(r_l-\frac{2}{k}n(r_l)-y_l \biggr)  \prod_{m=1, \,m \neq j}^{J+1} \omega_k(v_m-\eta_m) \\
        & \qquad \qquad \quad \times \frac{\beta^2}{\eps^2} \, (v_j - \eta_j)  \cdot \nabla \omega_k(v_j-\eta_j) \dd y \dd \eta.
\end{align*}
On the one hand, we have that
\begin{align*}
    \| I^3_1 \|_{L^1(\cD_R)} &= \int_0^T  \int_{(\Omega \cap B_R)^{J+1}} \int_{B_R^{J+1}} \biggl |  \sum_{j=1}^{J+1}\int_{\R^{(J+1)d}} \int_{\Omega^{J+1}}  \varrho(y,\eta,t) \, \prod_{l=1}^{J+1} \omega_k\biggl(r_l-\frac{2}{k}n(r_l)-y_l \biggr)  \\
        & \qquad \times \prod_{m=1, \,m \neq j}^{J+1} \omega_k(v_m-\eta_m) \, \frac{\beta^2}{\eps^2} \, (v_j - \eta_j)  \cdot \nabla \omega_k(v_j-\eta_j) \dd y \dd \eta
        \biggr | \dd v \dd r \dd t \\
    & \leq \sum_{j=1}^{J+1} \int_0^T \int_{(\Omega \cap B_R)^{J+1}} \int_{B_R^{J+1}} \int_{\R^{(J+1)d}} \int_{\Omega^{J+1}} \biggl | \varrho(y,\eta,t) \, \prod_{l=1}^{J+1} \omega_k\biggl(r_l-\frac{2}{k}n(r_l)-y_l \biggr) \\
        & \qquad \times  \prod_{m=1, \,m \neq j}^{J+1} \omega_k(v_m-\eta_m) \frac{\beta^2}{\eps^2} \, (v_j - \eta_j)  \cdot \nabla \omega_k(v_j-\eta_j) \biggr | \dd y \dd \eta
         \dd v \dd r \dd t.
\end{align*}
Noting that $|v_j-\eta_j| \leq \frac{1}{k}$ for all $j=1,\dots,J+1$, we then deduce that
\begin{align*}
    \| I^3_1 \|_{L^1(\cD_R)}
    & \leq \frac{\beta^2}{k \, \eps^2} \sum_{j=1}^{J+1} \int_0^T \int_{(\Omega \cap B_R)^{J+1}} \int_{B_R^{J+1}} \int_{\R^{(J+1)d}} \int_{\Omega^{J+1}} \biggl | \varrho(y,\eta,t) \, \prod_{l=1}^{J+1} \omega_k\biggl(r_l-\frac{2}{k}n(r_l)-y_l \biggr)  \\
        & \qquad \qquad \quad \times \prod_{m=1, \,m \neq j}^{J+1} \omega_k(v_m-\eta_m) \, \nabla \omega_k(v_j-\eta_j) \biggr | \dd y \dd \eta
         \dd v \dd r \dd t\\
       & \leq \frac{\beta^2}{k \, \eps^2} \sum_{j=1}^{J+1} \int_0^T \int_{(\Omega \cap B_R)^{J+1}} \int_{B_R^{J+1}} \int_{\R^{(J+1)d}} \int_{\Omega^{J+1}} \biggl | \varrho(y,\eta,t) \, \prod_{l=1}^{J+1} \omega_k\biggl(r_l-\frac{2}{k}n(r_l)-y_l \biggr)  \\
        & \qquad \qquad \quad \times \prod_{m=1, \,m \neq j}^{J+1} \omega_k(v_m-\eta_m) \, \nabla \omega_k(v_j-\eta_j) \biggr | \dd y \dd \eta
         \dd v \dd r \dd t\\
    & \leq \frac{\beta^2}{k \, \eps^2} \sum_{j=1}^{J+1} \int_0^T \int_{\R^{(J+1)d}} \int_{\Omega^{J+1}} \varrho(y,\eta,t) \, \biggl( \prod_{l=1}^{J+1} \int_{\Omega \cap B_R}  \omega_k\biggl(r_l-\frac{2}{k}n(r_l)-y_l \biggr) \dd r_l \biggr)  \\
        & \qquad \qquad \quad \times  \biggl(\prod_{m=1, \,m \neq j}^{J+1} \int_{B_R}  \omega_k(v_m-\eta_m) \dd v_m \biggr) \, \biggl( \int_{B_R} \bigl | \nabla \omega_k(v_j-\eta_j) \bigr | \dd v_j \biggr) \dd y \dd \eta \dd t\\
        & \leq \frac{2^{J+1} \, \beta^2}{k \, \eps^2} \sum_{j=1}^{J+1} \int_0^T \int_{B_{R+1}^{J+1}} \int_{(\Omega \cap B_{R+1})^{J+1}} \varrho(y,\eta,t) \, \biggl( \int_{B_R} \bigl | \nabla \omega_k(v_j-\eta_j) \bigr | \dd v_j \biggr) \dd y \dd \eta \dd t\\
           & \leq \frac{2^{J+1} \, \beta^2}{k \, \eps^2} \sum_{j=1}^{J+1} \int_0^T \int_{B_{R+1}^{J+1}} \int_{(\Omega \cap B_{R+1})^{J+1}} \varrho(y,\eta,t) \, \biggl( \int_{\R^d} \bigl | \nabla \omega_k(\xi) \bigr | \dd \xi \biggr) \dd y \dd \eta \dd t\\
           & = \frac{2^{J+1} \, \beta^2}{k \, \eps^2} \sum_{j=1}^{J+1} \int_0^T \int_{B_{R+1}^{J+1}} \int_{(\Omega \cap B_{R+1})^{J+1}} \varrho(y,\eta,t) \, \biggl( \int_{\R^d} k^{d+1} \bigl | \nabla \omega(k \xi) \bigr | \dd \xi \biggr) \dd y \dd \eta \dd t\\
\end{align*}
\begin{align*}
           & = \frac{2^{J+1} \, \beta^2}{k \, \eps^2} \sum_{j=1}^{J+1} \int_0^T \int_{B_{R+1}^{J+1}} \int_{(\Omega \cap B_{R+1})^{J+1}} \varrho(y,\eta,t) \, \biggl( \int_{\R^d} k \bigl | \nabla \omega(z) \bigr | \dd z \biggr) \dd y \dd \eta \dd t\\
           & \leq \frac{2^{J+1} \, (J+1) \, \beta^2\, C}{\eps^2} \int_0^T \int_{B_{R+1}^{J+1}} \int_{(\Omega \cap B_{R+1})^{J+1}} \varrho(y,\eta,t) \dd y \dd \eta \dd t\\
           & \leq C \| \varrho\|_{L^1(\cD_{R+1})},
\end{align*}
where $C>0$ is a constant independent of $k$.
On the other hand, we have that
\begin{align*}
    \| I^1_1 \|_{L^1(\cD_R)} & = \int_0^T  \int_{(\Omega \cap B_R)^{J+1}} \int_{B_R^{J+1}} \biggl |  \sum_{j=1}^{J+1} \int_{\R^{(J+1)d}} \int_{\Omega^{J+1}} \varrho(y,\eta,t) \prod_{l=1}^{J+1} \omega_k\biggl(r_l-\frac{2}{k}n(r_l)-y_l \biggr) \\
        & \qquad \qquad \quad  \times \prod_{m=1, \,m \neq j}^{J+1} \omega_k(v_m-\eta_m) \bigl( ({\mathcal L}r)_j  - ({\mathcal L}y)_j \bigr)
        \cdot
        \nabla \omega_k(v_j-\eta_j)  \dd y \dd \eta\,
        \biggr | \dd v \dd r \dd t
    \\
    & \leq  \frac{1}{\eps} \sum_{j=1}^{J+1} \int_0^T \int_{(\Omega \cap B_R)^{J+1}} \int_{B_R^{J+1}} \int_{\R^{(J+1)d}} \int_{\Omega^{J+1}} \biggl | \varrho(y,\eta,t) \prod_{l=1}^{J+1} \omega_k\biggl(r_l-\frac{2}{k}n(r_l)-y_l \biggr) \\
        & \qquad \qquad \quad  \times \prod_{m=1, \,m \neq j}^{J+1} \omega_k(v_m-\eta_m) \bigl( ({\mathcal L}r)_j  - ({\mathcal L}y)_j \bigr) \cdot
        \nabla \omega_k(v_j-\eta_j) \biggr | \dd y \dd \eta
         \dd v \dd r \dd t,
\end{align*}
and noting that $|({\mathcal L}r)_j  - ({\mathcal L}y)_j | \leq |r_j-y_j| \leq \frac{3}{k}$ for all $j=1,\dots,J+1$, we get
\begin{align*}
    \| I^1_1 \|_{L^1(\cD_R)}
    & \leq  \frac{3}{k \, \eps} \sum_{j=1}^{J+1} \int_0^T \int_{(\Omega \cap B_R)^{J+1}} \int_{B_R^{J+1}} \int_{\R^{(J+1)d}} \int_{\Omega^{J+1}} \varrho(y,\eta,t) \prod_{l=1}^{J+1} \omega_k\biggl(r_l-\frac{2}{k}n(r_l)-y_l \biggr) \\
        & \qquad \qquad \quad  \times \prod_{m=1, \,m \neq j}^{J+1} \omega_k(v_m-\eta_m) \,
        \Bigl | \nabla \omega_k(v_j-\eta_j) \Bigr | \dd y \dd \eta
         \dd v \dd r \dd t \\
   & \leq \frac{2^{J+1} \, (J+1) \, C}{\eps} \int_0^T \int_{B_{R+1}^{J+1}} \int_{(\Omega \cap B_{R+1})^{J+1}} \varrho(y,\eta,t) \dd y \dd \eta \dd t\\
           & \leq C \| \varrho\|_{L^1(\cD_{R+1})},
\end{align*}
where $C>0$ is a constant independent of $k$.

In order to show that, for a constant $C>0$ independent of $k$, we have $ \| I_1 \|_{L^1(\cD_R)} \leq C \| \varrho\|_{L^2(\cD_{R+1})}$, it remains to prove that $\| I^2_1 \|_{L^1(\cD_R)} \leq C \| \varrho\|_{L^2(\cD_{R+1})}$.
Using the fundamental theorem of calculus we have, for any sufficiently smooth real-valued function $u$, that
$$ u(r_j,t) - u(y_j,t) = \int_0^1 \nabla u \Bigl(y_j + h \, (r_j - y_j)\Bigr) (r_j-y_j) \dd h.$$

Now, for $u \in L^2({0,T; W^{1,\sigma}}(\Omega)^d)$, with $\sigma>d$, a solution to the Oseen system \eqref{eq:NSe}, we define
$$\tilde u (\cdot,t) = \left\{
    \begin{array}{cl}
        u(\cdot,t) &\quad \mbox{in $\Omega$} \\
        0  &\quad \mbox{in $\R^d \setminus \Omega$.}
    \end{array}
\right.
$$
Let us first study the smoothed functions
$$u^{\delta}(\cdot,t) := \tilde u *_{x} \omega_{\delta}(\cdot,t) \qquad \mbox{for $\delta>0$},$$
where $\omega_{\delta}(x):= \frac{1}{\delta^d} \, \omega(\frac{x}{\delta})$ denotes the usual mollifier.

We first claim that
$$u^{\delta} \rightarrow \tilde u \, \, \mbox{in} \, \, L^2(\R^{d} \times (0,T)) \, \, \mbox{as} \, \, \delta \rightarrow 0.$$
To prove this, we first note that if $\tilde u$ is smooth, then, for $x \in \R^d$,
\begin{align*}
    u^{\delta}(x,t) - \tilde u(x,t) & = \frac{1}{\delta^d} \int_{B(x,\delta)} \omega \Bigl(\frac{x-z}{\delta}\Bigr)
    ( \tilde u(z,t) - \tilde u (x,t)) \dd z \\
    &= \int_{B(0,1)} \omega(y) \bigl( \tilde u(x+ \delta\,  y, t) - \tilde u(x,t) \bigr) \dd y \\
    & = \int_{B(0,1)} \omega(y) \int_0^1 \frac{\dd}{\dd h}\bigl( \tilde u(x+ h \, \delta\,  y, t) \bigr) \dd h \dd y
    \\
    & = \delta \int_{B(0,1)} \omega(y) \int_0^1 \nabla \tilde u(x+ h \, \delta \, y , t) \cdot y \dd h \dd y.
\end{align*}
Hence,
\begin{align*}
    |u^{\delta}(x,t) - \tilde u(x,t)| & \leq \delta \int_0^1 \int_{B(0,1)}  |\nabla \tilde u(x+ h \, \delta \, y, t) |  \dd y \dd h,
\end{align*}
and using the change of variable $z=x+ h \, \delta \, y$,
\begin{align*}
    |u^{\delta}(x,t) - \tilde u(x,t)| & \leq \delta \int_0^1 \, \frac{1}{(h \, \delta)^d} \, \int_{B(x,\,h\delta)}  |\nabla \tilde u(z,t) |  \dd z \dd h \\
    & \leq C \, |B| \, \delta \, \mathcal{M}_{|\nabla \tilde u|}(x,t),
\end{align*}
\noindent where $\mathcal{M}_f(x):= \sup_{r>0} \dashint_{B(x,r)}{|f(y)|} \dd y$ is the Hardy--Littlewood maximal function.

Thus
\begin{align}
\label{estim_tild_u}
\begin{split}
     \int_{\R^d}  |u^{\delta}(x,t) - \tilde u(x,t)|^2 \dd x & \leq C \delta^2 \int_{\R^d}  |\mathcal{M}_{|\nabla \tilde u|}(x,t)|^2 \dd x \\
    & \leq C \delta^2 \int_{\R^d} |\nabla \tilde u(x,t)|^2 \dd x,
\end{split}
\end{align}
where we have used a classical strong $(p,p)$-bound on the maximal function, which asserts that the maximal operator
is bounded in $L^p(\R^d)$ for $1 < p \leq \infty$, i.e.,
$$ \| \mathcal{M}_f\|_p \leq c \, \| f \|_p,$$
where $c = c(d, p)>0$ is a constant. Now, for the smooth functions
$$ \tilde{\tilde{u}}(\cdot,t) := \tilde u *_{x} \omega_{\lambda}(\cdot,t) \qquad \mbox{with $\lambda>0$}$$ and
$$ u^{\delta}(\cdot,t) := \tilde{\tilde{u}} *_{x} \omega_{\delta}(\cdot,t) \qquad \mbox{with $\delta>0$}, $$
from estimate \eqref{estim_tild_u}, we still have
\begin{align*}
     \int_{\R^d}  |u^{\delta}(x,t) - \tilde{\tilde{u}}(x,t)|^2 \dd x &
    \leq C \delta^2 \int_{\R^d} |\nabla \tilde{\tilde{u}} (x,t)|^2 \dd x.
\end{align*}
Thus,
\begin{align}
\label{estim_tild_tild_u}
\begin{split}
     \int_{\R^d}  |\bigl((\tilde u *_{x} \omega_{\lambda})*_{x} \omega_{\delta} \bigr)(x,t) - \tilde u *_{x} \omega_{\lambda}(x,t)|^2 \dd x &
    \leq C \delta^2 \int_{\R^d} |\nabla  (\tilde u *_{x} \omega_{\lambda})(x,t)|^2 \dd x \\
    &  \leq C \delta^2 \int_{\R^d} |\nabla \tilde u (x,t)|^2 \dd x.
\end{split}
\end{align}
Also, since for a.e. $t \in (0,T)$
$$\tilde u(\cdot, t) \in L^{\sigma}(\R^d),$$
\noindent there exists a subsequence $(\lambda_{k_t})_{t \geq 0}$ such that
$$\tilde u  *_{x} \omega_{\lambda_{k_t}}(\cdot,t) \rightarrow_{\lambda_{k_t} \rightarrow 0^{+}} \tilde u (\cdot, t) \qquad \mbox{a.e. in $\R^d$}.$$
Hence, on the one hand,
$$(\tilde u *_{x} \omega_{\delta}) *_{x} \omega_{\lambda_{k_t}}(\cdot,t) \rightarrow_{\lambda_{k_t} \rightarrow 0^{+}} \tilde u *_{x} \omega_{\delta} (\cdot, t) \qquad \mbox{a.e. in $\R^d$},$$
on the other hand, since
$$(\tilde u  *_{x} \omega_{\lambda}) *_{x} \omega_{\delta} = (\tilde u  *_{x} \omega_{\delta}) *_{x} \omega_{\lambda},$$
using Fatou's lemma in \eqref{estim_tild_tild_u}, we get
\small{
\begin{align*}
       \int_{\R^d} \liminf_{\lambda_{k_t} \rightarrow 0^+} |\bigl((\tilde u *_{x} \omega_{\lambda})*_{x} \omega_{\delta} \bigr)(x,t) - \tilde u *_{x} \omega_{\lambda}(x,t)|^2 \dd x & \leq
     \liminf_{\lambda_{k_t} \rightarrow 0^+} \int_{\R^d} |\bigl((\tilde u *_{x} \omega_{\lambda})*_{x} \omega_{\delta} \bigr)(x,t) - \tilde u *_{x} \omega_{\lambda}(x,t)|^2 \dd x
    \\
    &  \leq C \delta^2 \int_{\R^d} |\nabla \tilde u (x,t)|^2 \dd x,
\end{align*}
}
that is,
\begin{align*}
     \int_{\R^d} |(\tilde u *_{x} \omega_{\delta})(x,t) - \tilde u (x,t)|^2 \dd x
    &  \leq C \delta^2 \int_{\R^d} |\nabla \tilde u (x,t)|^2 \dd x \\
    & \leq C \delta^2 (1+ \| u_0 \|_{L^2(\Omega)}^2).
\end{align*}
However, since
$$   C \delta^2 (1+ \| u_0 \|_{L^2(\Omega)}^2) \rightarrow_{\delta \rightarrow 0^+} 0,$$
which implies that
$$  \int_{\R^d} |(\tilde u *_{x} \omega_{\delta})(x,t) - \tilde u (x,t)|^2 \dd x \rightarrow_{\delta \rightarrow 0^+} 0 $$
uniformly in $t \in [0,T]$, by Lebesgue's dominated convergence theorem we then have that
$$  \int_0^T \int_{\R^d} |(\tilde u *_{x} \omega_{\delta})(x,t) - \tilde u (x,t)|^2 \dd x \dd t \rightarrow_{\delta \rightarrow 0^+} 0;$$
that is,
$$u^{\delta} \rightarrow \tilde u \, \, \mbox{in} \, \, L^2(\R^{d} \times (0,T)) \, \, \mbox{as} \, \, \delta \rightarrow 0,$$
as was asserted above.

Hence, we deduce that there exists a subsequence $(\delta_p)_{p>0}$ such that
$$ (\tilde u *_{x} \omega_{\delta_p})(x,t) - \tilde u \rightarrow_{\delta_p \rightarrow 0^+} 0 \quad \mbox{for a.e. $(x,t) \in \R^d \times (0,T)$}. $$
Set $u_p(x,t) := (\tilde u *_{x} \omega_{\delta_p})(x,t). \,$
Since
\begin{align*}
    I^2_1 &:=  \sum_{j=1}^{J+1}\int_{\R^{(J+1)d}} \int_{\Omega^{J+1}}  \varrho(y,\eta,t) \, \prod_{l=1}^{J+1} \omega_k\biggl(r_l-\frac{2}{k}n(r_l)-y_l \biggr)  \prod_{m=1, \,m \neq j}^{J+1} \omega_k(v_m-\eta_m) \\
        & \qquad \qquad \quad \times \frac{1}{\eps} \, \bigl( u(r_j,t) - u(y_j,t) \bigr) \cdot \nabla \omega_k(v_j-\eta_j) \dd y \dd \eta,
\end{align*}
this implies, by Fatou's lemma, that
\begin{align}
\label{fatou_1}
\begin{split}
    | I^2_1 |
        & \leq \sum_{j=1}^{J+1}\int_{\R^{(J+1)d}} \int_{\Omega^{J+1}}  \varrho(y,\eta,t) \, \prod_{l=1}^{J+1} \omega_k\biggl(r_l-\frac{2}{k}n(r_l)-y_l \biggr)  \prod_{m=1, \,m \neq j}^{J+1} \omega_k(v_m-\eta_m) \\
        & \qquad \qquad \quad \times \frac{1}{\eps} \, |  u(r_j,t) - u(y_j,t) | \,\, \bigl| \nabla \omega_k(v_j-\eta_j) \bigr| \dd y \dd \eta \\
        & \leq \liminf_{\delta_p \rightarrow 0}  \sum_{j=1}^{J+1}\int_{\R^{(J+1)d}} \int_{\Omega^{J+1}}  \varrho(y,\eta,t) \, \prod_{l=1}^{J+1} \omega_k\biggl(r_l-\frac{2}{k}n(r_l)-y_l \biggr)  \prod_{m=1, \,m \neq j}^{J+1} \omega_k(v_m-\eta_m) \\
        & \qquad \qquad \quad \times \frac{1}{\eps} \, \bigl | u_p(r_j,t) - u_p(y_j,t) \bigr | \cdot \nabla \omega_k(v_j-\eta_j) \dd y \dd \eta \\
        & :=  \liminf_{\delta_p \rightarrow 0} \, I^{2,p}_1.
        \end{split}
\end{align}
Moreover,
\begin{align*}
     I^{2,p}_1
        & \leq \frac{3}{k}\sum_{j=1}^{J+1}\int_{\R^{(J+1)d}} \int_{\Omega^{J+1}}  \varrho(y,\eta,t) \, \prod_{l=1}^{J+1} \omega_k\biggl(r_l-\frac{2}{k}n(r_l)-y_l \biggr)  \prod_{m=1, \,m \neq j}^{J+1} \omega_k(v_m-\eta_m) \\
        & \qquad \qquad \quad \times \frac{1}{\eps} \, \biggl| \biggl(\int_0^1 \nabla u_p \Bigl(y_j + h \, (r_j - y_j)\Bigr) \dd h \biggr) \biggr | \,\, \bigl| \nabla \omega_k(v_j-\eta_j) \bigr| \dd y \dd \eta.
\end{align*}
By noting that $(\nabla \omega)_k(z)= k^d \, \nabla \omega(kz)$, which implies that $k (\nabla \omega)_k(z)= \nabla \omega_k(z)$, we get
\begin{align*}
    | I^{2,p}_1 |
        & \leq 3 \sum_{j=1}^{J+1}\int_{\R^{(J+1)d}} \int_{\Omega^{J+1}}  \varrho(y,\eta,t) \, \prod_{l=1}^{J+1} \omega_k\biggl(r_l-\frac{2}{k}n(r_l)-y_l \biggr)  \prod_{m=1, \,m \neq j}^{J+1} \omega_k(v_m-\eta_m) \\
        & \qquad \qquad \quad \times \frac{1}{\eps} \, \biggl| \biggl(\int_0^1 \nabla u_p \Bigl(y_j + (r_j - y_j)\Bigr) \dd h \biggr) \biggr | \,\, \bigl| (\nabla \omega)_k(v_j-\eta_j) \bigr| \dd y \dd \eta.
\end{align*}
Hence we have that
\begin{align}
\begin{split}
\label{I_2:bound2}
    \| I^{2,p}_1 \|_{L^1(\cD_R)}
        & \leq \frac{3 \cdot 2^J}{\eps} \sum_{j=1}^{J+1} \int_0^T \int_{B_{R+1}^{J+1}} \int_{(\Omega \cap B_{R+1})^{J+1}}  \varrho(y,\eta,t) \, \biggl( \int_{\Omega \cap B_{R}} \biggl[ \omega_k\biggl(r_j-\frac{2}{k}n(r_j)-y_j \biggr)  \\
        & \qquad \qquad \times \int_0^1 \Bigl| \nabla u_p \Bigl(y_j + h \, (r_j - y_j)\Bigr) \Bigr |  \dd h \biggr] \dd r_j \biggr) \, \biggl( \int_{B_R} \bigl| (\nabla \omega)_k(v_j-\eta_j) \bigr| \dd v_j \biggr) \dd y \dd \eta \dd t.
        \end{split}
\end{align}
On the one hand,
\begin{align}
    \begin{split}
    \label{I_2:bound3}
        \int_{B_R} \bigl| (\nabla \omega)_k(v_j-\eta_j) \bigr| \dd v_j & \leq \int_{\R^d}\bigl| (\nabla \omega)_k(\xi)\bigr| \dd \xi \\
        & = \int_{\R^d} k^d \bigl|\nabla \omega(k \, \xi) \bigr| \dd \xi \\
         &  = \int_{\R^d} \bigl|\nabla \omega(z) \bigr| \dd z \\
         & \leq C,
    \end{split}
\end{align}
where $C$ is a positive constant, independent of $k$. Furthermore, we have that
\begin{align*}
    &\int_{\Omega \cap B_{R}} \omega_k\biggl(r_j-\frac{2}{k}n(r_j)-y_j \biggr) \,  \biggl[ \int_0^1 \Bigl| \nabla u_p \Bigl(y_j + h \, (r_j - y_j)\Bigr) \Bigr |  \dd h \biggr] \dd r_j \\
    & \qquad \qquad =  \int_0^1 \int_{\Omega \cap B_{R}} \omega_k\biggl(r_l-\frac{2}{k}n(r_l)-y_l \biggr) \, \Bigl| \nabla u_p \Bigl(y_j + h \, (r_j - y_j)\Bigr) \Bigr | \dd r_j \dd h.
\end{align*}
Thus, by using the change of variable $z_j=r_j-y_j$, which implies that $\dd z_j = \dd r_j$, where $r_j$ and $y_j \in B_{R+1}$, we get that
\begin{align*}
    &\int_{\Omega \cap B_{R}} \omega_k\biggl(r_j-\frac{2}{k}n(r_j)-y_j \biggr) \,  \biggl[ \int_0^1 \Bigl| \nabla u_p \Bigl(y_j + h \, (r_j - y_j)\Bigr) \Bigr |  \dd h \biggr] \dd r_j \\
    & \qquad \qquad \leq  \int_0^1 \int_{B_{2(R+1)}} \omega_k\biggl(z_j-\frac{2}{k}n(z_j+y_j) \biggr) \, \Bigl| \nabla u_p \Bigl(y_j + h \, z_j \Bigr) \Bigr | \dd z_j \dd h \\
     & \qquad \qquad =  \int_0^1 \int_{B_{2(R+1)}} k^d \, \omega \biggl(k \, z_j-\ 2 \, n(z_j+y_j) \biggr) \, \Bigl| \nabla u_p \Bigl(y_j + h \, z_j \Bigr) \Bigr | \dd z_j \dd h.
\end{align*}
Now, by the change of variable $\xi_j = k \, z_j$, which implies that $\dd \xi_j = k^d \dd z_j$ and $|\xi_j| \leq 1 + 2 \| n \|_{L^\infty(\R^d)} := R'$ since $|\xi_j - 2 \, n \Bigl(\frac{\xi_j}{k} + y_j\Bigr)| \leq 1$, we obtain
\begin{align*}
    &\int_{\Omega \cap B_{R}} \omega_k\biggl(r_j-\frac{2}{k}n(r_j)-y_j \biggr) \,  \biggl[ \int_0^1 \Bigl| \nabla u_p \Bigl(y_j + h \, (r_j - y_j)\Bigr) \Bigr |  \dd h \biggr] \dd r_j \\
     & \qquad \qquad \leq  \int_0^1 \int_{B_{1 + 2 \| n \|_{L^\infty(\R^d)}}}  \omega \biggl(\xi_j - 2 \, n \Bigl(\frac{\xi_j}{k}+y_j\Bigr) \biggr) \, \Bigl| \nabla u_p \Bigl(y_j + \frac{h}{k} \, \xi_j \Bigr) \Bigr | \dd \xi_j \dd h \\
     & \qquad \qquad \leq  \int_0^1 \int_{B_{R'}} \Bigl| \nabla u_p \Bigl(y_j + \frac{h}{k} \, \xi_j \Bigr) \Bigr | \dd \xi_j \dd h.
\end{align*}
Next, by performing the change of variable $s=y_j + \frac{h}{k} \, \xi_j$, which implies that $\dd s = \bigl(\frac{h}{k}\bigr)^d \, \dd \xi_j $ and $s \in B(y_j,\frac{h}{k} \, R')$, we get
\begin{align*}
    \int_{\Omega \cap B_{R}} \omega_k\biggl(r_j-\frac{2}{k}n(r_j)-y_j \biggr) \,  \biggl[ \int_0^1 \Bigl| \nabla u_p \Bigl(y_j + h \, (r_j - y_j)\Bigr) \Bigr |  \dd h \biggr] \dd r_j & \leq  \int_0^1 \int_{B(y_j,\frac{h}{k} \, R')} \Bigl| \nabla u_p (s) \Bigr | \frac{\dd s}{\bigl(\frac{h}{k} \bigr)^d} \dd h \\
     & \leq  C_{R'}\int_0^1 \dashint_{B(y_j,\frac{h}{k} \, R')} \Bigl| \nabla u_p (s) \Bigr | \dd s \dd h.
\end{align*}
Consider again the Hardy--Littlewood maximal function $\mathcal{M}_f(x):= \sup_{r>0} \dashint_{B(x,r)}{|f(y)|} \dd y$.
We then have that
 \\
     \begin{align}
     \label{I_2:bound4}
    \int_{\Omega \cap B_{R}} \omega_k\biggl(r_j-\frac{2}{k}n(r_j)-y_j \biggr) \,  \biggl[ \int_0^1 \Bigl| \nabla u_p \Bigl(y_j + h \, (r_j - y_j)\Bigr) \Bigr |  \dd h \biggr] \dd r_j & \leq C_{R'} \mathcal{M}_{|\nabla u_p|} (y_j,t).
\end{align}

Hence, by noting inequalities \eqref{I_2:bound3} and \eqref{I_2:bound4} in \eqref{I_2:bound2}, we obtain
\begin{align*}
    \| I^{2,p}_1 \|_{L^1(\cD_R)}
        & \leq \frac{3 \cdot 2^J \, C_{R'} \, C}{\eps} \sum_{j=1}^{J+1} \int_0^T \int_{B_{R+1}^{J+1}} \int_{(\Omega \cap B_{R+1})^{J+1}}  \varrho(y,\eta,t) \, \mathcal{M}_{|\nabla u_p|} (y_j,t) \dd y \dd \eta \dd t.
\end{align*}
Using the Cauchy--Schwarz inequality and the boundedness of the maximal operator
in $L^2(\R^d)$, we get
\begin{align}
\label{bound_I_2,p}
\begin{split}
    \| I^{2,p}_1 \|_{L^1(\cD_R)}
        & \leq \frac{3 \cdot 2^J \, C_{R'} \, C}{\eps} \sum_{j=1}^{J+1} \int_0^T  \| \varrho (\cdot,\cdot,t) \|_{L^2 \bigl((\Omega \cap B_{R+1})^{J+1} \times B_{R+1}^{J+1}\bigr)} \,  \|\mathcal{M}_{|\nabla u_p|} (y_j,t) \|_{{L^2}(\R^d)} \dd t \\
        & \leq \frac{3 \cdot 2^J \, C_{R'} \, C \, {(J+1)}}{\eps}\int_0^T  \| \varrho (\cdot,\cdot,t) \|_{L^2 \bigl((\Omega \cap B_{R+1})^{J+1} \times B_{R+1}^{J+1}\bigr)} \,  \| \nabla u_p(\cdot,t) \|_{{L^2}(\R^d)} \dd t \\
        & \leq \frac{3 \cdot 2^J \, C_{R'} \, C \, {(J+1)}}{\eps} \, \| \varrho  \|_{L^2((\Omega \cap B_{R+1})^{J+1} \times B_{R+1}^{J+1} \times (0,T))} \, \| \nabla u_p\|_{L^2 (0,T;{L^2}(\R^d))} \\
         & \leq \frac{3 \cdot 2^J \, C_{R'} \, C \, {(J+1)}}{\eps} \, \| \varrho  \|_{L^2((\Omega \cap B_{R+1})^{J+1} \times B_{R+1}^{J+1} \times (0,T))} \, \| \nabla u\|_{L^2 (0,T;{L^2}(\R^d))} \\
        & \leq C \, (\| u_0\|_{L^2(\R^d)} + 1) \, \| \varrho  \|_{L^2((\Omega \cap B_{R+1})^{J+1} \times B_{R+1}^{J+1} \times (0,T))} \\
        & \leq C \, \| \varrho  \|_{L^2((\Omega \cap B_{R+1})^{J+1} \times B_{R+1}^{J+1} \times (0,T))} \\
        & = C  \| \varrho  \|_{L^2(\cD_{R+1})}.
\end{split}
\end{align}
Now, since from \eqref{fatou_1} we have that
\begin{align*}
    | I^2_1 |
        & \leq \liminf_{\delta_p \rightarrow 0} \, I^{2,p}_1,
\end{align*}
we deduce that
\begin{align*}
    \| I^2_1 \|_{L^1(\cD_R)} & = \int_0^T  \int_{(\Omega \cap B_R)^{J+1}} \int_{B_R^{J+1}} \bigl | I^2_1
        \bigr | \dd v \dd r \dd t \\
        & \leq \int_0^T  \int_{(\Omega \cap B_R)^{J+1}} \int_{B_R^{J+1}} \liminf_{\delta_p \rightarrow 0} \, I^{2,p}_1 \dd v \dd r \dd t \\
         & \leq \liminf_{\delta_p \rightarrow 0} \, \int_0^T  \int_{(\Omega \cap B_R)^{J+1}} \int_{B_R^{J+1}} I^{2,p}_1 \dd v \dd r \dd t \\
         &= \liminf_{\delta_p \rightarrow 0}    \| I^{2,p}_1 \|_{L^1(\cD_R)}.
\end{align*}
Finally, from \eqref{bound_I_2,p}, we get
\begin{align*}
    \| I^2_1 \|_{L^1(\cD_R)}
         & \leq  C  \| \varrho  \|_{L^2(\cD_{R+1})},
\end{align*}
as has been asserted.

\begin{remark}
We note in passing that in bounding $\| I^2_1 \|_{L^1(\cD_R)}$ we could have used a more direct argument.
Indeed, by a standard property of the maximal function (cf., for example, Corollary 4.3 in \cite{max_func_aalto}),
we have that if $u \in W^{1,p}(\R^d)$, with $1 \leq p \leq \infty$ then there is a set $E$ of measure zero such that the following inequality holds
\begin{align}
   \label{ineq:max_fct}
|u(x) - u(y)| \leq c |x - y| \bigl( \mathcal{M}_{\nabla u}(x) + \mathcal{M}_{\nabla u}(y) \bigr),
\end{align}
\noindent for all $x,y \in \R^{d} \setminus E$.
Since
\begin{align*}
    I^2_1 &:=  \sum_{j=1}^{J+1}\int_{\R^{(J+1)d}} \int_{\Omega^{J+1}}  \varrho(y,\eta,t) \, \prod_{l=1}^{J+1} \omega_k\biggl(r_l-\frac{2}{k}n(r_l)-y_l \biggr)  \prod_{m=1, \,m \neq j}^{J+1} \omega_k(v_m-\eta_m) \\
        & \qquad \qquad \quad \times \frac{1}{\eps} \, \bigl( u(r_j,t) - u(y_j,t) \bigr) \cdot \nabla \omega_k(v_j-\eta_j) \dd y \dd \eta,
\end{align*}
we have that
\begin{align*}
    \| I^2_1 \|_{L^1(\cD_R)} & = \int_0^T  \int_{(\Omega \cap B_R)^{J+1}} \int_{B_R^{J+1}} \biggl |\sum_{j=1}^{J+1}\int_{\R^{(J+1)d}} \int_{\Omega^{J+1}}  \varrho(y,\eta,t) \, \prod_{l=1}^{J+1} \omega_k\biggl(r_l-\frac{2}{k}n(r_l)-y_l \biggr)  \prod_{m=1, \,m \neq j}^{J+1} \omega_k(v_m-\eta_m) \\
        & \qquad \qquad \quad \times \frac{1}{\eps} \, \bigl( u(r_j,t) - u(y_j,t) \bigr) \cdot \nabla \omega_k(v_j-\eta_j) \dd y \dd \eta \biggr| \dd v \dd r \dd t \\
        & \leq \sum_{j=1}^{J+1} \int_0^T  \int_{(\Omega \cap B_R)^{J+1}} \int_{B_R^{J+1}} \int_{\R^{(J+1)d}} \int_{\Omega^{J+1}} \biggl | \varrho(y,\eta,t) \, \prod_{l=1}^{J+1} \omega_k\biggl(r_l-\frac{2}{k}n(r_l)-y_l \biggr)  \prod_{m=1, \,m \neq j}^{J+1} \omega_k(v_m-\eta_m) \\
        & \qquad \qquad \quad \times \frac{1}{\eps} \, \bigl( u(r_j,t) - u(y_j,t) \bigr) \cdot \nabla \omega_k(v_j-\eta_j) \biggr| \dd y \dd \eta  \dd v \dd r \dd t \\
        & = \sum_{j=1}^{J+1} \int_0^T  \int_{(\Omega \cap B_R)^{J+1}} \int_{B_R^{J+1}} \int_{\R^{(J+1)d}} \int_{\Omega^{J+1}} \biggl | \varrho(y,\eta,t) \, \prod_{l=1}^{J+1} \omega_k\biggl(r_l-\frac{2}{k}n(r_l)-y_l \biggr)  \prod_{m=1, \,m \neq j}^{J+1} \omega_k(v_m-\eta_m) \\
        & \qquad \qquad \quad \times \frac{1}{\eps} \, \bigl| u(r_j,t) - u(y_j,t) \bigr| \, \bigl| \nabla \omega_k(v_j-\eta_j) \bigr| \dd y \dd \eta  \dd v \dd r \dd t.
\end{align*}
From inequality \eqref{ineq:max_fct} it then directly follows that
\begin{align*}
&    \| I^2_1 \|_{L^1(\cD_R)}  \leq \sum_{j=1}^{J+1} \int_0^T  \int_{(\Omega \cap B_R)^{J+1}} \int_{B_R^{J+1}} \int_{\R^{(J+1)d}} \int_{\Omega^{J+1}} \, \varrho(y,\eta,t) \, \prod_{l=1}^{J+1} \omega_k\biggl(r_l-\frac{2}{k}n(r_l)-y_l \biggr)  \prod_{m=1, \,m \neq j}^{J+1} \omega_k(v_m-\eta_m) \\
        & \qquad \qquad \times \frac{c}{\eps} \, |r_j - y_j| \, \bigl| \mathcal{M}_{\nabla u}(r_j,t) + \mathcal{M}_{\nabla u}(y_j , t) \bigr| \, \bigl| \nabla \omega_k(v_j-\eta_j) \bigr| \dd y \dd \eta  \dd v \dd r \dd t
        \\
        & \qquad \qquad \times \Bigl(\mathcal{M}_{\nabla u}(r_j,t) + \mathcal{M}_{\nabla u}(y_j , t)\Bigr) \biggr] \dd r_j \biggr) \, \biggl( \int_{B_R} \bigl| (\nabla \omega)_k(v_j-\eta_j) \bigr| \dd v_j \biggr) \dd y \dd \eta \dd t \\
%
    %
    & \leq \frac{3 \cdot 2^J \, c \, C}{\eps} \sum_{j=1}^{J+1} \int_0^T \int_{B_{R+1}^{J+1}} \int_{(\Omega \cap B_{R+1})^{J+1}}  \varrho(y,\eta,t) \, \biggl( \int_{\Omega \cap B_{R}} \biggl[ \omega_k\biggl(r_j-\frac{2}{k}n(r_j)-y_j \biggr)  \\
        & \qquad \qquad \times \Bigl(\mathcal{M}_{\nabla u}(r_j,t) + \mathcal{M}_{\nabla u}(y_j , t)\Bigr) \biggr] \dd r_j \biggr) \dd y \dd \eta \dd t
    \\
    & = \frac{3 \cdot 2^J \, c \, C}{\eps} \Biggl[ \sum_{j=1}^{J+1} \int_0^T \int_{B_{R+1}^{J+1}} \int_{(\Omega \cap B_{R+1})^{J+1}}  \varrho(y,\eta,t) \, \biggl( \int_{\Omega \cap B_{R}} \biggl[ \omega_k\biggl(r_j-\frac{2}{k}n(r_j)-y_j \biggr)
    \mathcal{M}_{\nabla u}(r_j,t)  \biggr] \dd r_j \biggr) \dd y \dd \eta \dd t\\
     &\qquad \qquad + \, \sum_{j=1}^{J+1} \int_0^T \int_{B_{R+1}^{J+1}} \int_{(\Omega \cap B_{R+1})^{J+1}}  \varrho(y,\eta,t)
         \biggl( \int_{\Omega \cap B_{R}} \omega_k\biggl(r_j-\frac{2}{k}n(r_j)-y_j \biggr)  \dd r_j \biggr) \, \mathcal{M}_{\nabla u}(y_j,t) \, \dd y \dd \eta \dd t \Biggr]
        \\
        & \leq
        \frac{3 \cdot 2^J \, c \, C}{\eps} \Biggl[ \sum_{j=1}^{J+1} \int_0^T \int_{B_{R+1}^{J+1}} \int_{(\Omega \cap B_{R+1})^{J+1}}  \varrho(y,\eta,t) \, \biggl( \int_{\Omega \cap B_{R}} \biggl[ \omega_k\biggl(r_j-\frac{2}{k}n(r_j)-y_j \biggr)
        \mathcal{M}_{\nabla u}(r_j,t)  \biggr] \dd r_j \biggr) \dd y \dd \eta \dd t\\
        &\qquad \qquad  + \, 2 \, \sum_{j=1}^{J+1} \int_0^T \int_{B_{R+1}^{J+1}} \int_{(\Omega \cap B_{R+1})^{J+1}}  \varrho(y,\eta,t) \mathcal{M}_{\nabla u}(y_j,t)  \, \dd y \dd \eta \dd t. \Biggr]
        \end{align*}

       Then, using changes of variable, we get

        \begin{align*}
      \| I^2_1 \|_{L^1(\cD_R)} & \leq
         \frac{3 \cdot 2^J \, c \, C}{\eps} \Biggl[ \sum_{j=1}^{J+1} \int_0^T \int_{B_{R+1}^{J+1}} \int_{(\Omega \cap B_{R+1})^{J+1}}  \varrho(y,\eta,t) \, \biggl( \int_{B_{2(R+1)}} \biggl[ \omega_k\biggl(z_j-\frac{2}{k}n(z_j+y_j) \biggr)  \\
        & \qquad \qquad \times \mathcal{M}_{\nabla u}(z_j+y_j,t)  \biggr] \dd z_j \biggr) \dd y \dd \eta \dd t
        \, + \, 2\, C \, \| \varrho  \|_{L^2{(\cD_{R+1})}} \Biggr]
        \\
        & =
         \frac{3 \cdot 2^J \, c \, C}{\eps} \Biggl[ \sum_{j=1}^{J+1} \int_0^T \int_{B_{R+1}^{J+1}} \int_{(\Omega \cap B_{R+1})^{J+1}}  \varrho(y,\eta,t) \, \biggl( \int_{B_{2(R+1)}} \biggl[k^d \,  \omega \biggl(k \, z_j- 2 \, n(z_j+y_j) \biggr)  \\
        & \qquad \qquad \times \mathcal{M}_{\nabla u}(z_j+y_j,t)  \biggr] \dd z_j \biggr) \dd y \dd \eta \dd t
        \, + \, 2\, C \, \| \varrho  \|_{L^2{(\cD_{R+1})}} \Biggr]    \\
        & =
         \frac{3 \cdot 2^J \, c \, C}{\eps} \Biggl[ \sum_{j=1}^{J+1} \int_0^T \int_{B_{R+1}^{J+1}} \int_{(\Omega \cap B_{R+1})^{J+1}}  \varrho(y,\eta,t) \, \biggl( \int_{B_{R'}} \biggl[ \omega \biggl(\xi_j- 2 \, n(\frac{\xi_j}{k} + y_j) \biggr)  \\
        & \qquad \qquad \times \mathcal{M}_{\nabla u}\bigl(\frac{1}{k} \, \xi_j + y_j,t \bigr)  \biggr] \dd \xi_j \biggr) \dd y \dd \eta \dd t
        \, + \, 2\, C \, \| \varrho  \|_{L^2{(\cD_{R+1})}} \Biggr]
        \\
        & \leq
         \frac{3 \cdot 2^J \, c \, C}{\eps} \Biggl[ \sum_{j=1}^{J+1} \int_0^T \int_{B_{R+1}^{J+1}} \int_{(\Omega \cap B_{R+1})^{J+1}}  \varrho(y,\eta,t) \, \biggl( \int_{B_{R'}} \biggl[
         \mathcal{M}_{\nabla u}\bigl(\frac{1}{k} \, \xi_j + y_j,t \bigr)  \biggr] \dd \xi_j \biggr) \dd y \dd \eta \dd t
       \\
        & \qquad \qquad  + \, 2\, C \, \| \varrho  \|_{L^2{(\cD_{R+1})}} \Biggr].
      \end{align*}
Finally, we have
        \begin{align*}
         \| I^2_1 \|_{L^1(\cD_R)} & \leq
         \frac{3 \cdot 2^J \, c \, C}{\eps} \Biggl[ \sum_{j=1}^{J+1} \int_0^T \int_{B_{R+1}^{J+1}} \int_{(\Omega \cap B_{R+1})^{J+1}}  \varrho(y,\eta,t) \, \biggl( \int_{B_{\bigl(y_j,\frac{1}{k} R'\bigr)}}
         \mathcal{M}_{\nabla u}(s,t )  \, \, \frac{\dd s}{\frac{1}{k^d}} \biggr) \dd y \dd \eta \dd t
       \\
        & \qquad \qquad
       + \, 2\, C \, \| \varrho  \|_{L^2{(\cD_{R+1})}} \Biggr]
        \\
         & \leq
         \frac{3 \cdot 2^J \, c \, C}{\eps} \Biggl[ \sum_{j=1}^{J+1} \int_0^T \int_{B_{R+1}^{J+1}} \int_{(\Omega \cap B_{R+1})^{J+1}}  \varrho(y,\eta,t) \, \,
          \mathcal{M}_{\mathcal{M}_{\nabla u}}(y_j,t)  \, \, \dd y \dd \eta \dd t
       \\
        & \qquad \qquad  + \, 2\, C \, \| \varrho  \|_{L^2{(\cD_{R+1})}} \Biggr]
        \\
         & \leq
         \frac{3 \cdot 2^J \, c \, C}{\eps} \Biggl[ \sum_{j=1}^{J+1} \int_0^T \| \varrho (\cdot,\cdot,t) \|_{L^2{\bigl( (\Omega \cap B_{R+1})^{J+1} \times B_{R+1}^{J+1} \bigr)}} \, \,
          \| \mathcal{M}_{\mathcal{M}_{\nabla u}}(y_j,t)  \|_{L^2{(\R^d)}} \, \dd t
         + \, 2\, C \, \| \varrho  \|_{L^2{(\cD_{R+1})}} \Biggr]
        \\
         & \leq
         \frac{3 \cdot 2^J \, c \, C}{\eps} \Biggl[ \sum_{j=1}^{J+1} \int_0^T \| \varrho (\cdot,\cdot,t) \|_{L^2{\bigl( (\Omega \cap B_{R+1})^{J+1} \times B_{R+1}^{J+1} \bigr)}} \, \,
          \| \mathcal{M}_{\nabla u}(y_j,t)  \|_{L^2{(\R^d)}} \, \dd t
         + \, 2\, C \, \| \varrho  \|_{L^2{(\cD_{R+1})}} \Biggr]
        \\
        & \leq  C \, \| \varrho  \|_{L^2{(\cD_{R+1})}}.
\end{align*}
\end{remark}
Therefore, we have that
$$\| I_1 \|_{L^1(\cD_R)} \leq C \, \| \varrho  \|_{L^2{(\cD_{R+1})}},$$

\medskip

Now, we shall show that $\| I_2 \|_{L^1(\cD_R)} \leq C \| \varrho  \|_{L^1(\cD_{R+1})}$. Indeed,
\begin{align*}
    \| I_2 \|_{L^1(\cD_R)} &=\int_0^T\int_{(\Omega \cap  B_R)^{J+1}} \int_{{B_R}^{J+1}} \biggl | \sum_{j=1}^{J+1}\int_{\R^{(J+1)d}} \int_{\Omega^{J+1}}  \partial_{\eta_j}  \cdot E_j(y,\eta,t) \prod_{l=1}^{J+1} \omega_k\biggl(r_l-\frac{2}{k}n(r_l)-y_l \biggr) \\
    & \qquad \times \prod_{m=1}^{J+1} \omega_k(v_m-\eta_m)  \, \varrho(y,\eta,t) \dd y \dd \eta \biggr | \dd v \dd r \dd t \\
    & \leq \sum_{j=1}^{J+1} \int_0^T \int_{(\Omega \cap  B_R)^{J+1}} \int_{{B_R}^{J+1}} \int_{\R^{(J+1)d}} \int_{\Omega^{J+1}} \bigl |\partial_{\eta_j}  \cdot E_j(y,\eta,t) \bigr |  \prod_{l=1}^{J+1} \omega_k\biggl(r_l-\frac{2}{k}n(r_l)-y_l \biggr) \\
    & \qquad \times \prod_{m=1}^{J+1} \omega_k(v_m-\eta_m)  \, \varrho(y,\eta,t) \dd y \dd \eta \dd v \dd r \dd t \\
    & \leq C \, 2^{J+1} \sum_{j=1}^{J+1} \int_0^T \int_{B_{R+1}^{J+1}} \int_{(\Omega \cap B_{R+1})^{J+1}} \varrho(y,\eta,t)  \, \bigl |\partial_{\eta_j}  \cdot E_j(y,\eta,t) \bigr | \dd y \dd \eta \dd t \\
    & \leq C \, \| \varrho  \|_{L^1(\cD_{R+1})}.
\end{align*}
In conclusion, we have that
$$\| r^1_k(\varrho) \|_{L^1(\cD_R)} \leq C \, \| \varrho \|_{L^2(\cD_{R+1})},$$
as has been asserted. That completes the proof of the lemma.
\end{proof}
Then, for $ \varrho \in L^\infty(0,T;L^2(\Omega^{J+1} \times \R^{(J+1)d};\R_{\geq 0}))$ we argue again by density:
we consider a sequence $(\varrho_{\eps})_{\eps>0}$ of smooth functions, such that $ \varrho_{\eps} \rightarrow \varrho$ in $L^2_{loc}(\overline{\cD})$, and we write:
\[
r^1_k(\varrho)= r^1_k(\varrho_{\eps})+ r^1_k(\varrho - \varrho_{\eps}),
\]
which obviously converges to $0$ in $L^1_{loc}(\overline{\cD})$ as $k \rightarrow \infty$ and $\eps \rightarrow 0$.

\smallskip

Next, we consider the term $r^3_k(\varrho)$. We begin by observing that the following equalities hold:
\begin{align*}
r^3_k(\varrho)&=\Biggl( \frac{\beta^2}{\eps^2}\, \sum_{j=1}^{J+1} \pdv^2 \varrho \Biggr)\star_{r,k} \omega_k *_v \omega_k - \frac{\beta^2}{\eps^2}\, \sum_{j=1}^{J+1} \pdv^2 \varrho_k \\
\qquad &= \frac{\beta^2}{\eps^2} \sum_{j=1}^{J+1} \int_{\R^{(J+1)d}} \int_{\Omega^{J+1}}  \biggl \{ \partial_{\eta_j}^2 \varrho(y,\eta,t) \, \prod_{l=1}^{J+1} \omega_k\Bigl(r_l-\frac{2}{k}n(r_l)-y_l \Bigr) \, \prod_{m=1}^{J+1} \omega_k(v_m-\eta_m) \\
\qquad & \quad - \pdv^2 \Bigl[ \varrho(y,\eta,t) \, \prod_{l=1}^{J+1} \omega_k\Bigl(r_l-\frac{2}{k}n(r_l)-y_l \Bigr) \, \prod_{m=1}^{J+1} \omega_k(v_m-\eta_m) \Bigr]  \biggr \} \dd y \dd \eta,
\end{align*}
and hence, by using an integration by parts on the first integrand, we obtain
\begin{align*}
r^3_k(\varrho) &= \frac{\beta^2}{\eps^2} \sum_{j=1}^{J+1} \int_{\R^{(J+1)d}} \int_{\Omega^{J+1}}  \biggl \{  \varrho(y,\eta,t) \, \partial_{\eta_j}^2[\omega_k(v_j-\eta_j)] \, \prod_{l=1}^{J+1} \omega_k\biggl(r_l-\frac{2}{k}n(r_l)-y_l \biggr) \, \prod_{m=1, \,m \neq j}^{J+1} \omega_k(v_m-\eta_m)  \\
\qquad & \quad - \varrho(y,\eta,t) \pdv^2 [\omega_k(v_j-\eta_j)] \, \prod_{l=1}^{J+1} \omega_k\biggl(r_l-\frac{2}{k}n(r_l)-y_l \biggr) \, \prod_{m=1, \,m \neq j}^{J+1} \omega_k(v_m-\eta_m) \biggr \} \dd y \dd \eta,\\
 &= \frac{\beta^2}{\eps^2} \sum_{j=1}^{J+1} \int_{\R^{(J+1)d}} \int_{\Omega^{J+1}}  \biggl \{  \varrho(y,\eta,t) \, (\Delta \omega)_k(v_j-\eta_j) \prod_{l=1}^{J+1} \omega_k\biggl(r_l-\frac{2}{k}n(r_l)-y_l \biggr) \, \prod_{m=1, \,m \neq j}^{J+1} \omega_k(v_m-\eta_m) \\
\qquad & \quad - \varrho(y,\eta,t) \, (\Delta \omega)_k(v_j-\eta_j) \prod_{l=1}^{J+1} \omega_k\biggl(r_l-\frac{2}{k}n(r_l)-y_l \biggr) \, \prod_{m=1, \,m \neq j}^{J+1} \omega_k(v_m-\eta_m) \biggr \} \dd y \dd \eta\\
\qquad & = 0.
\end{align*}

Having dealt with the terms $r^1_k(\varrho)$, $r^3_k(\varrho)$ and $r^4_k(\varrho)$, we are now left with the task of considering the remaining term, $r^2_k(\varrho)$. We begin by noting that
\begin{align*}
r^2_k(\varrho)&=\sum_{j=1}^{J+1} (\pdv \cdot E_j(r,v,t)) \varrho_k - \Biggl( \sum_{j=1}^{J+1} (\pdv \cdot E_j(r,v,t)) \varrho \Biggr)\star_{r,k} \omega_k *_v \omega_k \\
&= \sum_{j=1}^{J+1}\int_{\R^{(J+1)d}} \int_{\Omega^{J+1}}  \biggl \{(\pdv \cdot E_j(r,v,t)) \, \varrho(y,\eta,t) \, \prod_{l=1}^{J+1} \omega_k\biggl(r_l-\frac{2}{k}n(r_l)-y_l \biggr) \, \prod_{m=1}^{J+1} \omega_k(v_m-\eta_m) \\
&\quad - \partial_{\eta_j}  \cdot E_j(y,\eta,t)  \,  \varrho(y,\eta,t) \, \prod_{l=1}^{J+1} \omega_k\biggl(r_l-\frac{2}{k}n(r_l)-y_l \biggr) \, \prod_{m=1}^{J+1} \omega_k(v_m-\eta_m) \biggr \} \dd y \dd \eta,
\end{align*}
which, as long as $\varrho$ and $E_j$ are sufficiently smooth, converges to $0$ in $L^1_{loc}$ as $k$ tends to $\infty$ by standard results
on convolutions. The general case then follows by using a density argument using the inequality which we shall next prove.
For a constant $C>0$, independent of $k$, we have that
\begin{align*}
      \| r^2_k (\varrho)\|_{L^1(\cD_R)} &= \int_0^T \int_{(\Omega \cap  B_R)^{J+1}} \int_{{B_R}^{J+1}} \biggl | \sum_{j=1}^{J+1}\int_{\R^{(J+1)d}} \int_{\Omega^{J+1}}  \biggl \{ \bigl[ (\pdv \cdot E_j(r,v,t)) - \partial_{\eta_j}  \cdot E_j(y,\eta,t) \bigr] \, \varrho(y,\eta,t) \\
      & \qquad \times \prod_{l=1}^{J+1} \omega_k\biggl(r_l-\frac{2}{k}n(r_l)-y_l \biggr) \, \prod_{m=1}^{J+1} \omega_k(v_m-\eta_m) \biggr\} \dd y
      \dd \eta\, \biggr | \dd v \dd r \dd t \\
& \leq \frac{1}{\eps}
\sum_{j=1}^{J+1} \int_0^T \int_{(\Omega \cap B_R)^{J+1}} \int_{B_R^{J+1}} \int_{\R^{(J+1)d}} \int_{\Omega^{J+1}} \biggl | \biggl \{ \bigl[ (\pdv \cdot E_j(r,v,t)) - \partial_{\eta_j}  \cdot E_j(y,\eta,t) \bigr] \, \varrho(y,\eta,t) \\
      & \qquad \times \prod_{l=1}^{J+1} \omega_k\biggl(r_l-\frac{2}{k}n(r_l)-y_l \biggr) \, \prod_{m=1}^{J+1} \omega_k(v_m-\eta_m) \biggr\} \dd y \dd \eta\, \biggr | \dd y \dd \eta \dd v \dd r \dd t \\
    & \leq C \, 2^{J+1} \sum_{j=1}^{J+1} \int_0^T \int_{B_{R+1}^{J+1}} \int_{(\Omega \cap B_{R+1})^{J+1}} \bigl | (\pdv \cdot E_j(r,v,t)) - \partial_{\eta_j}  \cdot E_j(y,\eta,t) \bigr | \, \varrho(y,\eta,t)  \dd y \dd \eta \dd t \\
     & \leq C \, \sum_{j=1}^{J+1}  \|\pdv \cdot E_j \|_{L^{\infty}\bigl(0,T;L^{\infty}{\bigl((\Omega \cap B_{R+1}) \times B_{R+1}\bigr)\bigr)}} \,\int_0^T \int_{B_{R+1}^{J+1}} \int_{(\Omega \cap B_{R+1})^{J+1}}  \varrho(y,\eta,t)  \dd y \dd \eta \dd t
    \\
    & \leq C \, \| \varrho  \|_{L^1(\cD_{R+1})}.
\end{align*}
Then, for $ \varrho \in L^\infty(0,T;L^1(\Omega^{J+1} \times \R^{(J+1)d};\R_{\geq 0}))$ we argue again by density: we consider a sequence $(\varrho_{\eps})_{\eps}$ of smooth functions, such that $ \varrho_{\eps} \rightarrow \varrho$ in $L^1_{loc}(\overline{\cD})$, and we write:
\[
r^2_k(\varrho)= r^2_k(\varrho_{\eps})+ r^2_k(\varrho - \varrho_{\eps}),
\]
which obviously converges to $0$ in $L^1_{loc}(\overline{\cD})$ as $k\rightarrow \infty$ and $\eps \rightarrow 0$.
That completes Step 2.
\end{proof}
{\sc{Step 3: Passing to the limit.}} Thanks to \eqref{convergence:FP} we have that $\varrho_k(\cdot, t)$ converges to $\varrho(\cdot, t)$ in $L^1_{loc}(\cO)$ for almost all $t \in [0,T]$ and we denote by $t_0$ such a time.
For all $k,l \in \bbN$ the difference $\varrho_k - \varrho_l$ belongs to $W^{1,1}(0,T;W^{1,\infty}_{loc}(\cO))$ and solves
\[
\Lambda_{E_j}(\varrho_k - \varrho_l)=r_k-r_l \quad in \quad  \mathcal{D}'(\Omega^{J+1} \times \R^{(J+1)d} \times (0,T)).   
\]
The estimate \eqref{apriori1:FP-eq} applied to $\varrho_k - \varrho_l$ and Lemma \ref{lem:5.2} imply that, for all compact sets $K \subset \cO$, one has
\begin{equation}
\label{cauchyseq}
\sup_{\tau \in [0,T]}\| (\varrho_k - \varrho_l)(\cdot, \tau) \|_{L^1(K)} \rightarrow_{k,l \rightarrow \infty} 0.
\end{equation}
We then deduce from \eqref{cauchyseq} that there exists for all $t \in [0,T]$ a function $\gamma_t \varrho$ such that $\varrho_k(\cdot,t)$ converges to  $\gamma_t \varrho$ in $C([0,T];L^1_{loc}(\cO)),$ and in particular from \eqref{convergence:FP}, by uniqueness of the limit we get
\begin{equation}
\label{trace:equality2}
\varrho(r,v,t)= \gamma_t \varrho(r,v) \,
\mbox{ for almost every $(r,v,t)$ in $\Omega^{J+1} \times \R^{(J+1)d} \times (0,T)$}.
\end{equation}
Moreover, for all $t \in [0,T]$ and $R>0$ we have from \eqref{convergence:FP} and \eqref{trace:equality2}, by lower semi-continuity and thanks to Lebesgue's dominated convergence theorem, since $\varrho_k$ is bounded in $L^{\infty}(0,T;L^1_{loc}(\cO))$, that
\[
\| \gamma_t \varrho \|_{L^1_R} \le \lim_{k \rightarrow \infty} \sup_{t \in [0,T]} \|\varrho_k(\cdot, t) \|_{L^1_R} = \|\varrho \|_{L^{\infty,1}_R}.
\]
We have that $\varrho_k(\cdot, t) = (\gamma_t \varrho)\star_{r,k} \omega_k *_v \omega_k$ a.e. in $\cO$ for all $k \in \bbN$
and $t \in [0,T]$, and since the two functions $\varrho_k(\cdot, t)$ and $(\gamma_t \varrho)\star_{r,k} \omega_k *_v \omega_k$
are continuous, this holds everywhere in $\cO$ and thus $\varrho_k (\cdot, t) \rightarrow \gamma_t \varrho$ in $L^1_{loc}(\cO)$
for all $t \in [0,T]$. We note that $\gamma_t \varrho = \varrho(\cdot, t)$. From \eqref{cauchyseq}, we deduce
that $\varrho \in C([0,T];L^1_{loc}(\cO))$.

The estimate \eqref{apriori2} applied to $\varrho_k - \varrho_l$, Lemma \ref{lem:5.2} and the convergence \eqref{convergence:FP} imply that for all compact subsets $K \subset \partial \Omega^{(j)} \times \R^{(J+1)d}$ one has
\begin{equation}
\label{trace:cv}
\int_0^T \int_K| \gamma \varrho_k - \gamma \varrho_l| \dd \mu_2(r,v,t) \rightarrow_{k,l \rightarrow \infty} 0.
\end{equation}
We deduce from \eqref{trace:cv} the existence of a function $\gamma \varrho \in L^1_{loc}(\partial \Omega^{(j)} \times \R^{(J+1)d} \times [0,T], \dd \mu_2)$, which is the limit of $(\gamma \varrho_k)$.

Finally, for a fixed $\varphi \in \mathcal{C}^{\infty}_0( \overline{\Omega^{J+1}} \times \R^{(J+1)d} \times [0,T])$ there exists a constant $C>0$ such that $|\varphi(r,v,t)| \le C |n(r_j) \cdot v_j|$  on $\partial \Omega^{(j)} \times \R^{(J+1)d} \times (0,T)$, to ensure that the integral $$\sum_{j=1}^{J+1} \int_{t_0}^{t_1} \int_{\partial \Omega^{(j)}} \int_{\R^{(J+1)d}} (v_j
\cdot n(r_j)) \,
\gamma \varrho \, \varphi \dd v \dd s(r) \dd \tau,$$ appearing on the right-hand side of \eqref{Green:FP-eq}, is finite (since $\gamma \varrho \in L^1_{loc}(\partial \Omega^{(j)} \times \R^{(J+1)d} \times [0,T], \dd \mu_2)$, where $\dd \mu_2 = |n(r_j) \cdot v_j|^{2}\dd v \dd s(r) \dd \tau$). Therefore, the Green's formula \eqref{Green:FP-eq} is established by writing it first for $\varrho_k$ and then passing to the limit $k \rightarrow \infty$. Uniqueness of the trace follows from Green's formula.
That completes the proof.
\end{proof}

\subsection{Fokker-Planck equation with specular reflection on the boundary}

We show in this section that the specular boundary condition is attained in a strong sense by the solution of equation \eqref{eq:FP-eq1}.
In the previous section we showed that $\varrho \in L^\infty(0,T;L^1(\Omega^{J+1} \times \R^{(J+1)d};\R_{\geq 0}))$
is a solution to the problem \eqref{eq:FP-eq1}, \eqref{eq:FP-ini1} in the sense of distributions, i.e.,
\begin{alignat}{2}
\label{conjug-eq:FP-equation}
\int_0^T \int_{\Omega^{J+1}} \int_{\R^{(J+1)d}}
\varrho \, \Lambda^*_{E_j} (\varphi) \dd v \dd r \dd \tau = 0,
\end{alignat}
for all test functions $ \varphi \in W^{1,1}_0(0,T;W^{s,2}_*(\Omega^{J+1} \times \R^{(J+1)d}))$ with $s>(J+1)d+1$. Now, we want to prove that the solution $\varrho$ satisfies the following \textit{specular boundary
condition} on $\partial \Omega^{(j)}$, $j=1,\dots, J+1$:
\begin{align}
\label{eq:rho-specular}
\varrho(r,v,t) = \varrho(r,v_*^{(j)},t)\qquad \mbox{for all $(r,v,t) \in \partial\Omega^{(j)} \times \R^{(J+1)d} \times
 (0,T]$, with $v \cdot \nu^{(j)}(r)<0$},
\end{align}
where
\[ v_*^{(j)} =v_*^{(j)}(r,v) := v - 2(v\cdot \nu^{(j)}(r))\,\nu^{(j)}(r), \qquad j=1,\dots, J+1.\]
To do so, let introduce some notational conventions. We define the field $\Pi_{r_j}$ of projection operators on the hyperplane, which is orthogonal to $\nu(r_j)$, in such a way that
\[ v_j = (\nu(r_j) \cdot v_j)\,\nu(r_j) + \Pi_{r_j} v_j, \]
and
\[\nu(r_j) \cdot \Pi_{r_j} v_j=0, \qquad
\mbox{for all $v_j \in \R^d$}.\]

Given three functions $\phi \in \mathcal{C}^{\infty}_0(\R^{(J+1)d} \times (0,T))$,  $\psi \in \mathcal{C}^{\infty}_0([0,\infty))$ with $\psi(0)=0$, and $\Psi \in \mathcal{C}^{\infty}_0(\R^{d-1})$, we set
\begin{equation}
\label{testfunction_RS}
\varphi(r,v,t) = \phi(r,t) \, \psi((\nu(r_j) \cdot v_j)^2)\, \Psi(\Pi_{r_j} v_j),
\end{equation}
and we define, following  \cite{onthetraceMischler}, the class $\cR \cS$ (standing for \textit{r\'eflexion sp\'eculaire}) as the space of functions $\varphi$ which can be expressed in the form \eqref{testfunction_RS}.
We now show that $\varphi$ satisfies the specular boundary condition. By replacing $v$ in \eqref{testfunction_RS} with $v_*^{(j)}$, we have
\begin{equation}
\label{particulartestfunction}
\varphi(r,v_*^{(j)},t) = \phi(r,t) \, \psi((\nu(r_j) \cdot v_*^{(j)})^2)\, \Psi(\Pi_{r_j}v_*^{(j)}).
\end{equation}
Since
\begin{align*}
v_*^{(j)} &=v_j-2(\nu(r_j) \cdot v_j)\nu(r_j) \\
&= \Pi_{r_j} v_j - (\nu(r_j) \cdot v_j)\nu(r_j),
\end{align*}
we have that $\Pi_{r_j}v_*^{(j)} = \Pi_{r_j} v_j$ and $(\nu(r_j) \cdot v_*^{(j)})^2= (\nu(r_j) \cdot v_j)^2 |\nu(r_j)|^2=(\nu(r_j) \cdot v_j)^2$. In particular, we get
\begin{equation*}
\varphi(r,v_*^{(j)},t) = \varphi(r,v_j,t) .
\end{equation*}
Therefore, thanks to \eqref{conjug-eq:FP-equation} and the Green's formula \eqref{Green:FP-eq}, the trace $\gamma \varrho$ is well-defined and satisfies
\begin{equation}
\sum_{j=1}^{J+1} \int_{0}^{T} \int_{\partial \Omega^{(j)}} \int_{\R^{(J+1)d}}  (v_j
\cdot \nu(r_j)) \,
\gamma \varrho (r,v,\tau) \, \varphi(r,v,\tau) \dd v \dd s(r) \dd \tau=0 \quad \forall\, \varphi \in \cR \cS.
\end{equation}
Hence, for almost every $(r,t) \in \partial \Omega^{(j)} \times (0,T)$, for all $\tilde \psi$ odd, such that $|\tilde \psi(z)| \le C \, z^2$, for all $\Psi$, by summing twice the same integral we have that
\begin{equation}
\sum_{j=1}^{J+1} \int_{v^{'} \in \Pi_{r_j}(\R^d)} \int_{v^{''} \in \R_{\geq 0}}
\Bigl[\gamma \varrho (r,v'+v'' \, \nu(r_j) ,\tau) +
\gamma \varrho (r,v'+v'' \,\nu(r_j) ,\tau) \Bigr]\, \Psi(v') \, \tilde \psi
(v'') \dd v_j' \dd v''=0.
\end{equation}
Hence, by performing a change of variable in the second integral ($v''$ becomes $-v''$ and we use the fact that $\tilde \psi$ is an odd function), we get
\begin{equation}
\sum_{j=1}^{J+1} \int_{v^{'} \in \Pi_{r_j}(\R^d)} \int_{v^{''} \in \R_{\geq 0}}
\Bigl[\gamma \varrho (r,v'+v'' \, \nu(r_j) ,\tau) -
\gamma \varrho (r,v'-v'' \,\nu(r_j) ,\tau) \Bigr]\, \Psi(v') \, \tilde \psi
(v'') \dd v_j' \dd v''=0,
\end{equation}
which is equivalent to $\gamma \varrho (r,v,t) =
\gamma \varrho (r,v_*^{(j)},t)$ for almost every $(r,v,t) \in \partial \Omega^{(j)} \times \R^{(J+1)d} \times
 (0,T]$, i.e.,  $\varrho$ satisfies the specular reflection boundary condition \eqref{eq:rho-specular}.


\section{The small-mass limit and equilibration in momentum space}
\label{sec:SM}

In the previous section we showed the existence of functions $u = u_\eps$ and $\hrho=\hrho_\eps$,
such that
\[ u_\eps \in \mathcal{C}([0,T];L^\sigma(\Omega)^d) \cap L^2(0,T;W^{1,\sigma}_0(\Omega)^d) \cap W^{1,2}(0,T;W^{-1,\sigma}(\Omega)^d),\]
with $\sigma=\min(\hat\sigma,z)>d$, $\hat{\sigma}:=2 + \frac{4}{d}$ and $z=d+\vartheta$ for some $\vartheta \in (0,1)$,
is a weak solution to the Oseen system \eqref{eq:NSe}, and $\hrho_\epsilon$ with
\[
\mathcal{F}(\hrho_\eps) \in L^\infty(0,T;L^1_M(\Omega^{J+1} \times \R^{(J+1)d};\R_{\geq 0})),
\]
\[ \nabla_v \sqrt{\hrho_\eps} \in L^2(0,T;L^2_M(\Omega^{J+1} \times \R^{(J+1)d})),\]
\[\nabla_v \hrho_\eps \in L^2(0,T;L^1_M(\Omega^{J+1} \times \R^{(J+1)d}))\quad \mbox{and}\quad
M\,\pd_t \hrho_\eps \in L^2(0,T;(W^{s,2}(\Omega^{J+1} \times \R^{(J+1)d}))'),\quad s>(J+1)d + 1,
\]
satisfies the following weak form of the Fokker--Planck equation: for all $t \in (0,T]$,
\begin{align}\label{eq:weak-duality-eps}
&\int_0^t \big\langle M\,\pd_\tau\hrho_\eps(\cdot,\cdot,\tau),\varphi(\cdot,\cdot,\tau)\big\rangle \dd \tau
+ \frac{\beta^2}{\eps^2}\left(\sum_{j=1}^{J+1} \int_0^t \int_{\Omega^{J+1}} \int_{\R^{(J+1)d}} M(v)\,\pdv \hrho_\eps \cdot \pdv \varphi \dd v \dd r \dd \tau \right)\nonumber\\
&\qquad- \frac{1}{\eps} \left(\sum_{j=1}^{J+1} \int_0^t \int_{\Omega^{J+1}} \int_{\R^{(J+1)d}} M(v)\, v_j \hrho_\eps\cdot \pdr \varphi \dd v \dd r \dd \tau \right)\nonumber\\
&\qquad- \frac{1}{\eps} \left(\sum_{j=1}^{J+1} \int_0^t \int_{\Omega^{J+1}} \int_{\R^{(J+1)d}} M(v)\,(({\mathcal L}r)_j+u_\eps(r_j,\tau))\,\hrho_\eps\cdot \pdv \varphi \dd v \dd r \dd \tau \right)
= 0 \nonumber \\
&\hspace{0.5in} \forall\, \varphi \in L^2(0,T; W^{1,2}_{*,M}(\Omega^{J+1} \times \R^{(J+1)d})\cap W^{s,2}_*(\Omega^{J+1} \times \R^{(J+1)d})), \quad s>(J+1)d+1.
\end{align}
Furthermore $\hrho_\eps(\cdot,\cdot,0)=\hrho_0(\cdot,\cdot)$ in the sense of $\mathcal{C}_w([0,T];L^1_M(\Omega^{J+1} \times \R^{(J+1)d};\R_{\geq 0}))$, and
\begin{equation}\label{eq:conservation}
 \int_{\Omega^{J+1} \times \R^{(J+1)d}} M\,\hrho_\eps(r,v,t)\dd r \dd v = \int_{\Omega^{J+1} \times \R^{(J+1)d}} M\, \hrho_0(r,v)\dd r \dd v = 1
 \qquad \forall\, t \in (0,T].
\end{equation}
In addition, $\hrho_\eps$ satisfies the following energy inequality:
{
\begin{align}\label{eq:energy-eps}
\int_{\Omega^{J+1}} \int_{\R^{(J+1)d}} M(v)\,\mathcal{F}(\hrho_\eps(t)) \dd v \dd r
 + \frac{\beta^2}{2\eps^2} \sum_{j=1}^{J+1} \int_0^t \int_{\Omega^{J+1}} \int_{\R^{(J+1)d}} M(v)\,\frac{|\pdv \hrho_\eps|^2}{\hrho_\eps}
\, \dd v \dd r \dd \tau
\nonumber\\
\leq  C\bigg[1 + \int_{\Omega^{J+1}} \int_{\R^{(J+1)d}} M(v)\,\mathcal{F}(\hrho_{0}) \dd v \dd r \bigg],
\end{align}
}

\noindent
where $C=C(\|u_0\|_{W^{1-\frac{2}{\sigma},\sigma}(\Omega)},\|b\|_{L^\infty(0,T;L^\infty(\Omega))})$, $\sigma=\min(\hat\sigma,z)>d$, $\hat{\sigma}:=2 + \frac{4}{d}$ and
$z=d+\vartheta$ for some $\vartheta \in (0,1)$, as in the previous section;
in particular, $C$ is independent of $\eps>0$. Motivated by the ideas in \cite{PS08}, the aim of this section is to rigorously identify the small-mass limit of the system, corresponding to passage to the limit $\eps \rightarrow 0_+$.

We begin by noting that $(\mathcal{F}(\hrho_\eps))_{\eps>0}$ is bounded in $L^\infty(0,T;L^1_M(\Omega^{J+1} \times \R^{(J+1)d}))$,
and the sequence $(\nabla_v \sqrt{\hrho_\eps})_{\eps}$ is bounded in $L^2(0,T;L^2_M(\Omega^{J+1} \times \R^{(J+1)d}))$.
Hence from equation \eqref{eq:weak-duality-eps} we have that the sequence $(M\,\pd_t \hrho_\eps)_{\eps>0}$ is bounded in
$L^2(0,T;(W^{s,2}(\Omega^{J+1} \times \R^{(J+1)d}))')$, for $s>(J+1)d + 1$.

We now proceed analogously as in the paragraph following \eqref{eq:bound8}. We consider the Maxwellian-weighted Orlicz space $L^\Phi_M(\Omega^{J+1} \times \R^{(J+1)d})$, with Young's function $\Phi(r) = \mathcal{F}(1+|r|)$
(cf. Kufner, John \& Fu\v{c}ik \cite{KJF}, Sec. 3.18.2). This has a separable predual $E^\Psi_M(\Omega^{J+1} \times \R^{(J+1)d})$, with Young's function
 $\Psi(r) = {\rm e}^{|r|} - |r| - 1$; the space $E^\Psi_M(\Omega^{J+1} \times \R^{(J+1)d})$ is defined as the closure of all bounded measurable functions
 in the norm of the Orlicz space $L^\Psi_M(\Omega^{J+1} \times \R^{(J+1)d})$. As there exists a constant $K$ such that $\mathcal{F}(1+r) \leq K (1+\mathcal{F}(r))$ for all $r \geq 0$, it
 follows from \eqref{eq:bound1} that the sequence $(\mathcal{F}(1+\hrho_\eps))_{\eps >0}$ is bounded in $L^\infty(0,T;L^1_M(\Omega^{J+1} \times \R^{(J+1)d}))$. Hence,
 $\hrho_\eps$ is bounded in $L^\infty(0,T;L^\Phi_M(\Omega^{J+1} \times \R^{(J+1)d})) = L^\infty(0,T;(E^\Psi_M(\Omega^{J+1} \times \R^{(J+1)d}))')$.
 By the Banach--Alaoglu theorem, there exists a subsequence (not indicated) of the sequence $(\hrho_\eps)_{\eps>0}$ and a
 \[\hrho_{(0)} \in L^\infty(0,T;L^\Phi_M(\Omega^{J+1} \times \R^{(J+1)d};\mathbb{R}_{\geq 0}))\quad (\mbox{whereby also $\mathcal{F}(\hrho_{(0)}) \in L^\infty(0,T;L^1_M(\Omega^{J+1} \times \R^{(J+1)d};\mathbb{R}_{\geq 0}))$})\]
(not to be confused with the initial datum $\hrho_0$) such that, as $\eps \rightarrow 0_+$,
\begin{alignat}{2}\label{eq:weak-orl1-eps}
\hrho_\eps &\rightharpoonup \hrho_{(0)} \geq 0 &&\quad \mbox{weakly$^*$ in $L^\infty(0,T;L^\Phi_M(\Omega^{J+1} \times \R^{(J+1)d})) = L^\infty(0,T;(E^\Psi_M(\Omega^{J+1} \times \R^{(J+1)d}))')$}.
\end{alignat}
As, by definition, $L^\infty(\Omega^{J+1} \times \R^{(J+1)d}) \subset E^\Psi_M(\Omega^{J+1} \times \R^{(J+1)d})$, it follows in particular that
\begin{alignat}{2}\label{eq:weak-orl2-eps}
\hrho_\eps &\rightharpoonup \hrho_{(0)} &&\quad \mbox{weakly$$ in $L^p(0,T;L^1_M(\Omega^{J+1} \times \R^{(J+1)d}))$}\qquad \forall\, p \in [1,\infty),\nonumber
\\
M\,\pd_t \hrho_\eps &\rightharpoonup M\,\pd_t \hrho_{(0)} &&\quad \mbox{weakly in $L^2(0,T;(W^{s,2}(\Omega^{J+1} \times \R^{(J+1)d}))'), \quad s>(J+1)d + 1$}.
\\
v_j\, \hrho_\eps &\rightharpoonup v_j\, \hrho_{(0)} &&\quad
\mbox{weakly in $L^2(0,T;L^1_M(\Omega^{J+1} \times \R^{(J+1)d}))$,\quad $j=1,\dots,J+1$},\nonumber
\end{alignat}
After multiplying \eqref{eq:pre-lim-1} by $\eps^2$, taking $t=T$, omitting the first and third term from the left-hand side,
passing to the limit $\alpha \rightarrow 0_+$ for a fixed $\gamma \in (0,1]$ using the weak lower-semicontinuity of the
second term on the left-hand side, and then passing to the limit $\eps\rightarrow 0_+$, noting, as in
\eqref{eq:uniform-u}, that
\begin{align}\label{eq:uniform-u-eps}
\|u_\eps\|_{L^2(0,T;W^{1,\sigma}({\Omega})) \cap W^{1,2}(0,T;W^{-1,\sigma}({\Omega}))}
\leq C(1 + \|u_0\|_{W^{1-\frac{2}{\sigma},\sigma}(\Omega)}),
\end{align}
with $\sigma>d$, whereby
$$\|u_\epsilon\|_{L^2(0,T;L^\infty({\Omega}))} \leq C(1 + \|u_0\|_{W^{1-\frac{2}{\sigma},\sigma}(\Omega)}),$$ where $C$ is a positive constant independent of $\epsilon$, we have that
\begin{align}\label{eq:energy-0}
\sum_{j=1}^{J+1} \int_0^T \int_{\Omega^{J+1}} \int_{\R^{(J+1)d}} M(v)\,\frac{|\pdv \hrho_{(0)}\,|^2}{\hrho_{(0)} + \gamma}
\, \dd v \dd r \dd \tau \leq 0.
\end{align}
Hence, $\pdv \hrho_{(0)} = 0$ a.e. in $\Omega^{J+1} \times \R^{(J+1)d} \times (0,T)$ for all $j \in \{1,\dots,J+1\}$. As $\hrho_{(0)}$ has vanishing weak
derivatives with respect to all coordinates of $v_j$ for all $j \in \{1,\dots,J+1\}$ it follows that $\hrho_{(0)}$ is constant with respect to all
$v_j$, $j \in \{1,\dots,J+1\}$. In other words, $\hrho_{(0)}(r,v,t) = \eta(r,t)$ for a function $\eta \in L^\infty(0,T; L^1(\Omega^{J+1}))$, to be determined.

An identical argument to the one following Lemma \ref{lemma-strauss} implies that
\[ \hrho_{(0)} \in \mathcal{C}_w([0,T]; L^1_M(\Omega^{J+1} \times \R^{(J+1)d};\mathbb{R}_{\geq 0})).\]
It then follows from \eqref{eq:conservation} that
\begin{equation}\label{eq:conservation0}
\int_{\Omega^{J+1} \times \R^{(J+1)d}} \hrho_{(0)}(r,v,t) \dd r \dd v = 1\qquad \forall\, t \in (0,T].
\end{equation}

We deduce from \eqref{eq:uniform-u-eps} that
\begin{alignat}{2} \label{eq:NS-conv-eps}
u_\eps & \rightharpoonup u_{(0)} \qquad && \mbox{weakly in $L^2(0,T;W^{1,\sigma}_0({\Omega})^d)$ as $\eps \rightarrow 0_+$},\qquad \sigma>d,\nonumber\\
u_\eps & \rightharpoonup u_{(0)} \qquad && \mbox{weakly in $W^{1,2}(0,T;W^{-1,\sigma}({\Omega})^d)$ as $\eps \rightarrow 0_+$},\qquad \sigma>d,\\
u_\eps & \rightarrow u_{(0)} \qquad && \mbox{strongly in $L^2(0,T;\mathcal{C}^{0,\gamma}(\overline{\Omega})^d)$ as $\eps \rightarrow 0_+$},\qquad 0<\gamma < 1-{\textstyle\frac{d}{\sigma}},\qquad \sigma>d, \nonumber
\end{alignat}
where the last result follows, via the Aubin--Lions lemma, thanks to the compact embedding of the Sobolev space
$W^{1,\sigma}({\Omega})^d$ into the H\"older space $\mathcal{C}^{0,\gamma}(\overline{\Omega})^d$ for $0<\gamma<1-\frac{d}{\sigma}$, $\sigma>d$.
Hence also
\begin{alignat}{2}\label{eq:rho-u-eps}
 (({\mathcal L}r)_j+u_\eps(r_j,\tau))\,\hrho_\eps &\rightharpoonup (({\mathcal L}r)_j+u_{(0)}(r_j,\tau))\,\hrho_{(0)} &&\quad
\mbox{weakly in $L^2(0,T;L^1_M(\Omega^{J+1} \times \R^{(J+1)d}))$,}
\end{alignat}
for each $j=1,\dots,J+1$.
Using \eqref{eq:weak-orl1-eps}, \eqref{eq:weak-orl2-eps}, \eqref{eq:NS-conv-eps} and \eqref{eq:rho-u-eps} we can now pass to the limit $\eps\rightarrow 0_+$
in \eqref{eq:weak-duality-eps} to deduce that, for all $t \in (0,T]$,
\begin{align}\label{eq:weak-duality-0}
&\sum_{j=1}^{J+1} \int_0^t \int_{\Omega^{J+1}} \int_{\R^{(J+1)d}} M(v)\,\pdv \hrho_{(0)} \cdot \pdv \varphi \dd v \dd r \dd \tau= 0 \nonumber\\
&\hspace{0.5in} \forall\, \varphi \in L^2(0,T; W^{1,2}_{*,M}(\Omega^{J+1} \times \R^{(J+1)d})\cap W^{s,2}_*(\Omega^{J+1} \times \R^{(J+1)d})), \quad s>(J+1)d+1.
\end{align}
Thus,
\begin{align}\label{eq:rho0-eq}
\sum_{j=1}^{J+1} \pdv\! \cdot (M(v)\,\pdv \hrho_{(0)}) = 0
\qquad\mbox{in $\mathcal{D}'(\Omega^{J+1} \times \R^{(J+1)d}\times (0,T))$}.
\end{align}
%
%
By defining
\begin{equation}\label{eq:eta}
\varrho_{(0)}:= M\,\hrho_{(0)} = M\, \eta,
\end{equation}
with $\eta \in L^\infty(0,T;L^1(\Omega^{J+1}))$, to be determined, it directly follows from \eqref{eq:rho0-eq} that
\[ \mathcal{L}_0^*\varrho_{(0)} = 0 \qquad\mbox{in $\mathcal{D}'(\Omega^{J+1} \times \R^{(J+1)d} \times (0,T))$}.\]
%

Using \eqref{eq:weak-duality-0}, \eqref{eq:weak-duality-eps} can now be rewritten in the following equivalent form:
for all $t \in (0,T]$,
\begin{align}\label{eq:weak-duality-eps1}
&\eps \int_0^t \big\langle M\,\pd_\tau\hrho_\eps(\cdot,\cdot,\tau),\varphi(\cdot,\cdot,\tau)\big\rangle \dd \tau
+ \left(\beta^2\sum_{j=1}^{J+1} \int_0^t \int_{\Omega^{J+1}} \int_{\R^{(J+1)d}} M(v)\,\pdv \left(\frac{\hrho_\eps - \hrho_{(0)}}{\eps}\right) \cdot \pdv \varphi \dd v \dd r \dd \tau \right)\nonumber\\
&\qquad- \left(\sum_{j=1}^{J+1} \int_0^t \int_{\Omega^{J+1}} \int_{\R^{(J+1)d}} M(v)\, v_j \hrho_\eps\cdot \pdr \varphi \dd v \dd r \dd \tau \right)\nonumber\\
&\qquad- \left(\sum_{j=1}^{J+1} \int_0^t \int_{\Omega^{J+1}} \int_{\R^{(J+1)d}} M(v)\,(({\mathcal L}r)_j+u_\eps(r_j,\tau))\,\hrho_\eps\cdot \pdv \varphi \dd v \dd r \dd \tau \right)
= 0 \nonumber \\
&\hspace{0.9cm} \forall\, \varphi \in L^2(0,T; W^{1,2}_{*,M}(\Omega^{J+1} \times \R^{(J+1)d})\cap W^{s,2}_*(\Omega^{J+1} \times \R^{(J+1)d})), \quad s>(J+1)d+1.
\end{align}

We now continue by performing some formal calculations, where the word `formal' refers to the fact that all manipulations with limits with respect to $\epsilon \rightarrow 0_+$ that we shall encounter will be assumed to be meaningful, without rigorous justification. The purpose of these formal calculations is to illuminate why the partial differential equation satisfied by $\eta$ is indeed the one that our subsequent rigorous, but less enlightening, argument will ultimately deliver.

First we let $\eps\rightarrow 0_+$ in \eqref{eq:weak-duality-eps1} and note \eqref{eq:weak-orl2-eps} and \eqref{eq:rho-u-eps} to deduce that, for all $t \in (0,T]$,
\begin{align}\label{eq:weak-duality-eps2}
&\lim_{\eps \rightarrow 0_{+}}\left(\beta^2 \sum_{j=1}^{J+1} \int_0^t \int_{\Omega^{J+1}} \int_{\R^{(J+1)d}} M(v)\,\pdv \left(\frac{\hrho_\eps - \hrho_{(0)}}{\eps}\right) \cdot \pdv \varphi \dd v \dd r \dd \tau \right)\nonumber\\
&\qquad- \left(\sum_{j=1}^{J+1} \int_0^t \int_{\Omega^{J+1}} \int_{\R^{(J+1)d}} M(v)\, v_j \hrho_{(0)}\cdot \pdr \varphi \dd v \dd r \dd \tau \right)\nonumber\\
&\qquad- \left(\sum_{j=1}^{J+1} \int_0^t \int_{\Omega^{J+1}} \int_{\R^{(J+1)d}} M(v)\,(({\mathcal L}r)_j+u_{(0)}(r_j,\tau))\,\hrho_{(0)}\cdot \pdv \varphi \dd v \dd r \dd \tau \right)
= 0 \nonumber \\
&\hspace{0cm} \forall\,
\varphi \in L^2(0,T; W^{1,2}_{*,M}(\Omega^{J+1} \times \R^{(J+1)d})\cap W^{s,2}_*(\Omega^{J+1} \times \R^{(J+1)d})), \quad s>(J+1)d+1,
\end{align}
and hence, also, for all test functions $\varphi \in \mathcal{C}^\infty_0(\Omega^{J+1} \times \R^{(J+1)d} \times (0,T))$.

We define $\hrho_{(1)} \in \mathcal{D}'(\Omega^{J+1} \times \R^{(J+1)d} \times (0,T))$ by
\[ \hrho_{(1)} := \lim_{\eps \rightarrow 0_{+}} \frac{\hrho_{\eps} - \hrho_{(0)}}{\eps},\]
with the limit understood in the sense of $\mathcal{D}'(\Omega^{J+1} \times \R^{(J+1)d} \times (0,T))$, and let
\[ \varrho_{(1)}:= M\,\hrho_{(1)}.\]
%

%
By taking $t=T$ in \eqref{eq:weak-duality-eps2} passage to the limit $\eps \rightarrow 0_+$ yields
\[\cL_{0}^*\varrho_{(1)} =-\cL_{1}(u_{(0)})^*\varrho_{(0)} \qquad \mbox{in $\mathcal{D}'(\Omega^{J+1} \times \R^{(J+1)d} \times (0,T))$}.\]
Expanding the right-hand side of this equality we have that
\begin{equation}
\label{eq:2a}
\cL_{0}^*\varrho_{(1)}=\sum_{j=1}^{J+1} M\, v_j \cdot \pdr\eta+(({\mathcal L}r)_j+u_{(0)}(r_j,t))\cdot
(\pdv M)\, \eta  \qquad \mbox{in $\mathcal{D}'(\Omega^{J+1} \times \R^{(J+1)d} \times (0,T))$}.
\end{equation}
As $v_j M=-\beta \pdv M$, we therefore have that
\begin{equation}
\label{eq:2b1}
\cL_{0}^*\varrho_{(1)}=-\sum_{j=1}^{J+1} \Bigl(\beta \,\pdr\eta - \eta\,(({\mathcal L}r)_j+u_{(0)}(r_j,t))\Bigr)\cdot
\pdv M \qquad \mbox{in $\mathcal{D}'(\Omega^{J+1} \times \R^{(J+1)d} \times (0,T))$}.
\end{equation}
By \eqref{eq:need}, $\cL^*_{0,j}(\pdv M) = - (\pdv M)$, and upon taking the inner product of this $d$-component equality with the $d$-component vector field $\beta\, \pdr\eta-\eta\, (({\mathcal L}r)_j+u_{(0)}(r_j,t))$, which is, clearly, independent of
$v_j$, and then summing through $j=1,\dots, J+1$, we deduce that one solution of \eqref{eq:2b1} is
\begin{equation}
\label{eq:2c}
\sum_{j=1}^{J+1}\Bigl(\beta\, \pdr\eta-\eta\, (({\mathcal L}r)_j+u_{(0)}(r_j,t))\Bigr)
\cdot \pdv M.
\end{equation}
Therefore the general solution of \eqref{eq:2b1} is
\begin{align*}
\varrho_{(1)} &= \sum_{j=1}^{J+1}\Bigl(\beta\, \pdr\eta-\eta\, (({\mathcal L}r)_j+u_{(0)}(r_j,t))\Bigr)
\cdot \pdv M + \eta_{(1)}(r,t)\, M\\
&= \frac{1}{\beta}\sum_{j=1}^{J+1}M\,\Bigl(\eta\, (({\mathcal L}r)_j+u_{(0)}(r_j,t))- \beta\, \pdr\eta\Bigr)
\cdot v_j + \eta_{(1)}(r,t)\, M,
\end{align*}
where $\eta_{(1)} \in \mathcal{D}'(\Omega^{J+1} \times (0,T))$ is arbitrary, because $\cL^*_0(\eta_{(1)} M) = \eta_{(1)} \cL^*_0(M) = 0$ thanks to $\cL^*_0(M) = 0$.
%
As it will transpire from the calculations that follow, the choice of $\eta_{(1)}$ does not affect $\eta$, and $\eta_{(1)}$ will be therefore, ultimately, set to $0$. Since $M$ and $\pdv M$, $j=1,\dots,J+1$, belong to the topological vector space $\mathcal{S}$
of rapidly decreasing functions defined on $\R^{(J+1)d}$ (the test space for the Schwarz space $\mathcal{S}'$ of tempered distributions), the structure of $\varrho_{(1)}$ implies that $\varrho_{(1)} \in \mathcal{D}'(\Omega^{J+1} \times (0,T)) \otimes \mathcal{S}$, where the latter is the linear space of all finite linear combinations of products of the form $a(r,t) \,b(v)$ with $a  \in \mathcal{D}'(\Omega^{J+1} \times (0,T))$ and $b \in \mathcal{S}$.

%

We now define $\hrho_{(2)} \in \mathcal{D}'(\Omega^{J+1} \times \R^{(J+1)d} \times (0,T))$ by
\[ \hrho_{(2)} := \lim_{\eps \rightarrow 0_{+}} \frac{\hrho_{\eps} - \hrho_{(0)} - \eps \hrho_{(1)}}{\eps^2},\]
with the limit understood to be in $\mathcal{D}'(\Omega^{J+1} \times \R^{(J+1)d} \times (0,T))$, and let
\[ \varrho_{(2)}:= M\,\hrho_{(2)}.\]

Next, we subtract the equality \eqref{eq:weak-duality-eps2} from \eqref{eq:weak-duality-eps1}, divide the difference by $\eps$,
and use test functions $\varphi \in \mathcal{C}^\infty_0(\Omega^{J+1} \times \R^{(J+1)d} \times (0,T))$, so as to rewrite the
resulting equality as one in $\mathcal{D}'(\Omega^{J+1} \times \R^{(J+1)d} \times (0,T))$, and then pass to the limit $\eps \rightarrow 0_+$, noting the definitions of $\hrho_{(0)}$, $\hrho_{(1)}$, $\hrho_{(2)}$, $u_{(0)}$, and
defining
\[ u_{(1)}:= \lim_{\eps \rightarrow 0_+} \frac{u_\eps - u_{(0)}}{\eps}\qquad \mbox{in $\mathcal{D}'(\Omega \times (0,T))$}.\]
Hence,
\begin{align*}
M \partial_t \hrho_{(0)} - \sum_{j=1}^{J+1} \pdv \cdot (M \pdv \hrho_{(2)}) + \sum_{j=1}^{J+1} M v_j \cdot \pdr \hrho_{(1)}
+ \sum_{j=1}^{J+1} (\mathcal{L}r)_j \cdot \pdv (M\hrho_{(1)})\\
+ \sum_{j=1}^{J+1} u_{(1)}(r_j,\cdot) \cdot \pdv (M \hrho_{(0)}) + u_{(0)}(r_j,\cdot) \cdot \pdv (M \hrho_{(1)}) = 0.
\end{align*}
Recalling that, by definition, $\varrho_{(i)}= M \hrho_{(i)}$, $i=0,1,2$, we then have that
\begin{align*}
\partial_t \varrho_{(0)} - \mathcal{L}_0^*\varrho_{(2)} + \sum_{j=1}^{J+1} v_j \cdot \pdr \varrho_{(1)}
+ \sum_{j=1}^{J+1} (\mathcal{L}r)_j \cdot \pdv \varrho_{(1)}
+ \sum_{j=1}^{J+1} u_{(1)}(r_j,\cdot) \cdot \pdv \varrho_{(0)} + u_{(0)}(r_j,\cdot) \cdot \pdv \varrho_{(1)} = 0.
\end{align*}
Equivalently,
\begin{align}\label{eq:varrho2}
\mathcal{L}_0^*\varrho_{(2)}  = \partial_t \varrho_{(0)} + \sum_{j=1}^{J+1} v_j \cdot \pdr \varrho_{(1)} +
\sum_{j=1}^{J+1}((\mathcal{L}r)_j + u_{(0)}(r_j,\cdot)) \cdot \pdv \varrho_{(1)} + \sum_{j=1}^{J+1} u_{(1)}(r_j,\cdot) \cdot \pdv \varrho_{(0)}.
\end{align}
Since both  $\varrho_{(0)}$ and $\varrho_{(1)}$ belong to $\mathcal{D}'(\Omega^{J+1} \times (0,T)) \otimes \mathcal{S}$,
the same is true of the right-hand side of \eqref{eq:varrho2}.
It is therefore meaningful to test both sides of \eqref{eq:varrho2} with $\I(v)$ (considered as an element of $\mathcal{S}'$); upon noting that
\[\tensor[_{\mathcal S'}]{\langle}{} \I(v) , \mathcal{L}_0^*\varrho_{(2)} \tensor[]{\rangle}{_{\mathcal{S}}}
= \tensor[_{\mathcal S'}]{\langle}{} \mathcal{L}_0(\I(v)) , \varrho_{(2)} \tensor[]{\rangle}{_{\mathcal{S}}}
= \tensor[_{\mathcal S'}]{\langle}{} 0  , \varrho_{(2)} \tensor[]{\rangle}{_{\mathcal{S}}} =0\]
because $\mathcal{L}_{0,j}(\I(v_j))=0$ for all $j=1,\dots, J+1$, we arrive at
\[ 0 = \tensor[_{\mathcal S'}]{\bigg\langle}{} \I(v) , \partial_t \varrho_{(0)} + \sum_{j=1}^{J+1} v_j \cdot \pdr \varrho_{(1)} + \sum_{j=1}^{J+1}
((\mathcal{L}r)_j + u_{(0)}(r_j,\cdot)) \cdot \pdv \varrho_{(1)} + \sum_{j=1}^{J+1} u_{(1)}(r_j,\cdot) \cdot \pdv \varrho_{(0)}\tensor[]{\bigg\rangle}{_{\mathcal{S}}},\]
as an equality in $\mathcal{D}'(\Omega^{J+1} \times (0,T))$, where $\tensor[_{\mathcal S'}]{\langle}{} \cdot, \cdot \tensor[]{\rangle}{_{\mathcal S}}$
denotes the duality pairing between $\mathcal{S}'$ and $\mathcal{S}$. Hence,
\begin{align*} 0 &= (\partial_t \eta)\,  \tensor[_{\mathcal S'}]{\bigg\langle}{} \I(v), M \tensor[]{\bigg\rangle}{_{\mathcal S}} +
\tensor[_{\mathcal S'}]{\bigg\langle}{} \I(v),  \sum_{j=1}^{J+1} v_j \cdot \pdr \varrho_{(1)}\tensor[]{\bigg\rangle}{_{\mathcal S}} \\
&\quad + \tensor[_{\mathcal S'}]{\bigg\langle}{} \I(v), \sum_{j=1}^{J+1}
((\mathcal{L}r)_j + u_{(0)}(r_j,\cdot)) \cdot \pdv \varrho_{(1)}\tensor[]{\bigg\rangle}{_{\mathcal S}} + \tensor[_{\mathcal S'}]{\bigg\langle}{} \I(v), \sum_{j=1}^{J+1} u_{(1)}(r_j,\cdot) \cdot \pdv \varrho_{(0)} \tensor[]{\bigg\rangle}{_{\mathcal S}},
\end{align*}
as an equality in $\mathcal{D}'(\Omega^{J+1} \times (0,T))$. Thanks to the definition of partial derivative of a tempered distribution the last two terms on the right-hand side vanish, while $\tensor[_{\mathcal S'}]{\langle}{} \I(v), M \tensor[]{\rangle}{_{\mathcal S}}=\int_{\R^{(J+1)d}} M(v) \dd v = 1$, resulting in
\[ \partial_t \eta + \tensor[_{\mathcal S'}]{\bigg\langle}{} \I(v),  \sum_{j=1}^{J+1} v_j \cdot \pdr \varrho_{(1)}\tensor[]{\bigg\rangle}{_{\mathcal S}} = 0.\]
In order to simplify the second term on the left-hand side, we consider
\[ v_j \cdot \pdr \varrho_{(1)} = v_j \cdot \pdr \left(\sum_{k=1}^{J+1}\Bigl(\beta\, \partial_{r_k}\eta-\eta\, (({\mathcal L}r)_k+u_{(0)}(r_k,\cdot))\Bigr)
\cdot \partial_{v_k} M \right) + v_j \cdot \pdr (\eta_{(1)}\, M).\]
As $v_j \cdot \pdr (\eta_{(1)}\, M) = - \beta \,\pdv (M(\pdr\eta_{(1)}))$, we have that $\tensor[_{\mathcal S'}]{\langle}{} \I(v), v_j \cdot \pdr (\eta_{(1)}\, M) \tensor[]{\rangle}{_{\mathcal S}} = 0$ for all $j=1,\dots,J+1$. Consequently, the precise choice of $\eta_{(1)}$ is immaterial, to the extent that
\begin{equation}\label{eq:eta-eq}
\partial_t \eta + \tensor[_{\mathcal S'}]{\bigg\langle}{} \I(v), \sum_{j=1}^{J+1} v_j \cdot \pdr\left( \sum_{k=1}^{J+1}\Bigl(\beta\, \partial_{r_k}\eta-\eta\, (({\mathcal L}r)_k+u_{(0)}(r_k,\cdot))\Bigr)
\cdot \partial_{v_k} M\right) \tensor[]{\bigg\rangle}{_{\mathcal S}} = 0,
\end{equation}
regardless of the specific choice of $\eta_{(1)}$.
Now (with the integral over $\R^{(J+1)d}$ considered below understood as a Gel'fand--Pettis integral of a function with values in a topological vector space, which is in our case $\mathcal{D}'(\Omega^{J+1} \times (0,T))$), we have that
\begin{align*}
&\tensor[_{\mathcal S'}]{\bigg\langle}{} \I(v), \sum_{j=1}^{J+1} v_j \cdot \pdr\left( \sum_{k=1}^{J+1}\Bigl(\beta\, \partial_{r_k}\eta-\eta\, (({\mathcal L}r)_k+u_{(0)}(r_k,\cdot))\Bigr)
\cdot \partial_{v_k} M\right) \tensor[]{\bigg\rangle}{_{\mathcal S}}\\
& \quad =
\int_{\R^{(J+1)d}}
\sum_{j=1}^{J+1} v_j \cdot \pdr\left( \sum_{k=1}^{J+1}\Bigl(\beta\, \partial_{r_k}\eta-\eta\, (({\mathcal L}r)_k+u_{(0)}(r_k,\cdot))\Bigr)\cdot \partial_{v_k} M\right) \dd v\\
& \quad = \sum_{j,k=1}^{J+1}
\int_{\R^{(J+1)d}}
 v_j \cdot \pdr\left(\Bigl(\beta\, \partial_{r_k}\eta-\eta\, (({\mathcal L}r)_k+u_{(0)}(r_k,\cdot))\Bigr)\cdot \partial_{v_k} M\right) \dd v\\
&\quad = - \sum_{j,k=1}^{J+1}
\int_{\R^{(J+1)d}} \partial_{v_k}\cdot \Bigl(A_{jk}v_j\Bigr)M \dd v = - \sum_{j=1}^{J+1}
\int_{\R^{(J+1)d}} \mbox{tr\! }(A_{jj}) M \dd v = - \sum_{j=1}^{J+1}\mbox{tr\! }(A_{jj}),
\end{align*}
where
\[ A_{jk} := \pdr\Bigl(\beta\, \partial_{r_k}\eta-\eta\, (({\mathcal L}r)_k+u_{(0)}(r_k,\cdot))\Bigr)
\in [\mathcal{D}'(\Omega^{J+1} \times (0,T))]^{d \times d},\qquad j,k =1, \dots, J+1. \]
Thus, \eqref{eq:eta-eq} yields the following partial differential equation satisfied by $\eta$:
\begin{equation}\label{eq:eta-pre}
\partial_t \eta - \sum_{j=1}^{J+1}  \Bigl(\beta\, \pdr^2\eta-\pdr\cdot \Bigl(\eta\, (({\mathcal L}r)_j+u_{(0)}(r_j,\cdot))\Bigr)\Bigr)=0,\qquad \mbox{in $\mathcal{D}'(\Omega^{J+1} \times (0,T))$.}
\end{equation}
This is the nonlinear Fokker--Planck equation associated with the McKean--Vlasov diffusion
\begin{equation}
\label{eq:ld2}
\dot{r}_j=({\mathcal L}r)_j+u_{(0)}(r_j,t;\eta)+\sqrt{2\beta}\,\dot{W}_j,
\end{equation}
where $u_{(0)}$ is the limit (cf. \eqref{eq:NS-conv-eps}) of the sequence $(u_\eps)_{\eps >0}$ defined above. We emphasize here that we are yet to show that $u_{(0)}$ is a solution of the Oseen equation, whose right-hand side is to be identified.

\smallskip

This concludes our formal calculations. The rest of the section is devoted to making the above formal
passage to the small-mass limit
$\epsilon \rightarrow 0_+$ rigorous, including the rigorous identification of the equation \eqref{eq:eta-pre} satisfied by
$\eta$.

We shall suppose henceforth that the initial datum $\varrho_0$ for the Fokker--Planck equation has the following
factorized form: $\varrho_0(r,v) = M(v)\, \hrho_0(r)$, where $\hrho_0$ is a nonnegative function of $r$ only, such that
$\int_{\Omega^{J+1}} \hrho_0(r) \dd r = 1$,
and $$\hrho_0 \in L^2(\Omega^{J+1};\R_{\geq 0}).$$
Under this hypothesis it directly follows that
\[ \hrho_\eps \in L^\infty(0,T; L^2_M(\Omega^{J+1} \times \R^{(J+1)d};\R_{\geq 0})) \cap L^2(0,T;L^2(\Omega^{J+1}; W^{1,2}_M(\R^{(J+1)d})))\]
and
\[ \partial_t \hrho_\eps \in L^2(0,T;(W^{1,2}_M(\Omega^{J+1} \times \R^{(J+1)d}))').\]
Consequently, by a density argument, \eqref{eq:weak-duality-eps} implies that
\begin{align}\label{eq:weak-duality-eps11}
&\int_0^t \big\langle M\,\pd_\tau\hrho_\eps(\cdot,\cdot,\tau),\varphi(\cdot,\cdot,\tau)\big\rangle \dd \tau
+ \frac{\beta^2}{\eps^2}\left(\sum_{j=1}^{J+1} \int_0^t \int_{\Omega^{J+1}} \int_{\R^{(J+1)d}} M(v)\,\pdv \hrho_\eps \cdot \pdv \varphi \dd v \dd r \dd \tau \right)\nonumber\\
&\qquad- \frac{1}{\eps} \left(\sum_{j=1}^{J+1} \int_0^t \int_{\Omega^{J+1}} \int_{\R^{(J+1)d}} M(v)\, v_j \hrho_\eps\cdot \pdr \varphi \dd v \dd r \dd \tau \right)\nonumber\\
&\qquad- \frac{1}{\eps} \left(\sum_{j=1}^{J+1} \int_0^t \int_{\Omega^{J+1}} \int_{\R^{(J+1)d}} M(v)\,(({\mathcal L}r)_j+u_\eps(r_j,\tau))\,\hrho_\eps\cdot \pdv \varphi \dd v \dd r \dd \tau \right)
= 0 \nonumber \\
&\hspace{3in} \forall\, \varphi \in L^2(0,T; W^{1,2}_{*,M}(\Omega^{J+1} \times \R^{(J+1)d})), \quad \forall\,t \in (0,T].
\end{align}
Further, $\hrho_\epsilon(\cdot,\cdot,0) = \hrho_0(\cdot,\cdot)$ in the sense of $\mathcal{C}_{w}([0,T];L^2_M(\Omega^{J+1} \times \R^{(J+1)d};\R_{\geq 0}))$.
\bigskip

The next step in our rigorous passage to the limit $\epsilon \rightarrow 0_+$ in \eqref{eq:weak-duality-eps11} is motivated by the proof of Lemma 2 on p.1374 in the work of Carrillo and Goudon \cite{CG}.
First, we formulate the `macroscopic' equations satisfied by the moments of $\varrho_\epsilon$. By taking $\varphi(r,v,t) = \phi(r,t)$ with $\phi \in L^2(0,T; W^{1,2}(\Omega^{J+1}))$ in \eqref{eq:weak-duality-eps11} and defining
\[ \orho_\epsilon(r,t):=\int_{\R^{(J+1)d}} M(v)\,\hrho_\epsilon(r,v,t) \dd v \quad\mbox{and}\quad \mathcal{J}_{\epsilon,j}(r,t):=\frac{1}{\epsilon}
\int_{\R^{(J+1)d}} M(v)\,v_j \, \hrho_\epsilon(r,v,t)\dd v,\quad j=1,\dots, J+1,\]
we have that
\begin{align}\label{eq:Moment1}
&\int_0^t \big\langle \pd_\tau\orho_\eps(\cdot,\tau),\phi(\cdot,\tau)\big\rangle \dd \tau
- \sum_{j=1}^{J+1} \int_0^t \int_{\Omega^{J+1}}  \mathcal{J}_{\epsilon,j} \cdot \pdr \phi \dd r \dd \tau = 0 \qquad
\forall\, \phi \in L^2(0,T; W^{1,2}(\Omega^{J+1})), \quad \forall\,t \in (0,T],
\end{align}
subject to the initial condition $\orho_\epsilon(\cdot,0) = \orho_0(\cdot):=\int_{\R^{(J+1)d}} M(v)\,\hrho_\epsilon(r,v,0) \dd v$.

Next, let $v_{i,\ell}$, $\ell=1,\dots,d$, denote the components of the vector $v_i \in \R^d$, and consider the test functions
$\varphi(r,v,t)=\phi(r,t)\,v_{i,\ell}$ in \eqref{eq:weak-duality-eps}, for $i=1,\dots,J+1$ and $\ell=1,\dots,d$, and let $\mathcal{J}_{\epsilon,i,\ell}$, for $\ell=1,\dots,d$, denote the components of the $d$-component vector-function $\mathcal{J}_{\epsilon,i}$ for $i=1,\dots,J+1$.
Hence, we are led, for $i=1,\dots, J+1$ and $\ell=1,\dots,d$, to
\begin{align*}
&\epsilon^2 \int_0^t \langle \partial_\tau \mathcal{J}_{\epsilon,i,\ell},\phi \rangle \dd \tau  + \beta^2 \int_0^t \int_{\Omega^{J+1}} \mathcal{J}_{\epsilon,i,\ell} \, \phi \dd r \dd \tau  \\
&\quad - \sum_{j=1}^{J+1} \int_0^t \int_{\Omega^{J+1}} \left(\int_{\R^{(J+1)d}} M(v)\, v_j \, v_{i,\ell}\, \hrho_\epsilon \dd v \right) \cdot
(\pdr \phi) \dd r \dd \tau\\
&\quad - \int_0^t \int_{\Omega^{J+1}} \orho_\epsilon\, \left((\mathcal{L} r)_{i,\ell} + u_{\epsilon,\ell}(r_i, \tau)\right)\,\phi\, \dd r \dd \tau = 0\qquad
\forall\, \phi \in L^2(0,T; \mathcal{C}^\infty_0(\Omega^{J+1})), \quad \forall\,t \in (0,T],
\end{align*}
where $(\mathcal{L} r)_{i,\ell}$, $\ell=1,\dots,d$, are the components of the vector-function
$(\mathcal{L} r)_i$, and $u_{\epsilon,\ell}$, $\ell=1,\dots,d$, are the components of the vector-function $u_\epsilon$.

We note that
\[ \sum_{j=1}^{J+1} v_j\,v_{i,\ell}\cdot \pdr \phi = \sum_{j=1}^{J+1} \sum_{k=1}^{J+1} v_{j,k}\,v_{i,\ell}\cdot \partial_{r_{j,k}} \phi = [(v \otimes v) \partial_r \phi]_{i,\ell},\qquad i=1,\dots, J+1,\quad \ell=1,\dots,d,\]
where $v$ and $r$ are $(J+1)d$-component vectors, whose components are denoted by
$v_{j,k}$, $j=1,\dots,J+1$, $k=1,\dots,d$ (or $v_{i,\ell}$,
$i=1,\dots,J+1$, $\ell=1,\dots,d$), and $r_{j,k}$, $j=1,\dots,J+1$, $k=1,\dots,d$, respectively.
Thus, by defining the $\mathbb{R}^{(J+1)d} \times \mathbb{R}^{(J+1)d}$-valued function $\mathbb{P}_\epsilon$ by
\[ [\mathbb{P}_\epsilon(r,t)]_{i,\ell,j,k}:= \int_{\mathbb{R}^{(J+1)d}} M(v)\, v_{i,\ell}\, v_{j,k}\,\hrho_\epsilon(r,v,t) \dd v,\]
for $i,j=1,\dots,J+1$ and $\ell,k = 1,\dots, d$, we have that
\begin{align}\label{eq:Moment2}
&\epsilon^2 \int_0^t \langle \partial_\tau \mathcal{J}_{\epsilon,i,\ell},\phi \rangle \dd \tau  + \beta^2 \int_0^t \int_{\Omega^{J+1}} \mathcal{J}_{\epsilon,i,\ell} \, \phi \dd r \dd \tau  \nonumber\\
&\quad - \int_0^t \int_{\Omega^{J+1}}  \left[\mathbb{P}_\epsilon\, \partial_r \phi  \right]_{i,\ell} \dd r \dd \tau \nonumber\\
&\quad - \int_0^t \int_{\Omega^{J+1}} \orho_\epsilon\, \left((\mathcal{L} r)_{i,\ell} + u_{\epsilon,\ell}(r_i, \tau)\right)\,\phi\, \dd r \dd \tau = 0\qquad
\forall\, \phi \in L^2(0,T; \mathcal{C}^\infty_0(\Omega^{J+1})), \quad \forall\,t \in (0,T],
\end{align}
where $[\mathbb{P}_\epsilon\,\partial_r \phi]_{i,\ell}:= \sum_{j=1}^{J+1}\sum_{k=1}^d [\mathbb{P}_\epsilon]_{i,\ell,j,k}\,\partial_{r_{j,k}}\phi$.

\begin{lemma}\label{lemma:bounds}
Let $ 0 < T < \infty$ and $0 < \epsilon < 1$; then, the following properties hold:
\begin{itemize}
\item[(i)] The sequence $(u_\epsilon)_{\epsilon>0}$ is bounded in $L^2(0,T;\mathcal{C}^{0,\gamma}(\overline{\Omega})^d)$,
with $0<\gamma<1-\frac{d}{\sigma}$, $\sigma>d$, $d=2,3$, and therefore also in $L^2(0,T;L^\infty(\Omega)^d)$;
\item[(ii)] $(\overline{\varrho}_\epsilon)_{\epsilon>0} $ and $({\mathcal L}r\, \overline{\varrho}_\epsilon)_{\epsilon>0}$ are bounded in $L^\infty(0,T;L^2(\Omega^{J+1}))$ and $L^\infty(0,T;L^2(\Omega^{J+1})^{(J+1)d})$, respectively;
\item[(iii)] Consider the $(J+1)d$-component vector-function $\orho_\epsilon u_\epsilon$, whose components are
\[\orho_\epsilon(r_1,\dots,r_{J+1},t)\, u_\epsilon(r_i,t),\qquad \mbox{for $i=1,\dots,J+1$}.\]
Then, $(\overline{\varrho}_\epsilon u_\epsilon)_{\epsilon>0}$ is bounded in $L^2(0,T;L^2(\Omega^{(J+1)d})^{(J+1)d})$;

\item[(iv)] The sequences of dissipation terms
\[ \bigg( \frac{1}{ \sqrt{\epsilon M(v)}} \bigg(\frac{v_j}{\beta} \varrho_\epsilon - \sqrt{\epsilon} M(v)\,u_\epsilon(r_j,\cdot)  + \pdv \varrho_\epsilon \bigg) =: \frac{1}{\sqrt{\epsilon}} D_{\epsilon,j} \bigg)_{\epsilon>0},\qquad j=1,\dots,J+1,\]
are bounded in $L^2(\Omega^{J+1} \times \R^{(J+1)d}  \times (0,T))^d;$
\item[(v)] The sequences $(\mathcal{J}_{\epsilon,j})_{\epsilon>0}$, $j=1,\dots, J+1$, are bounded in $L^2(\Omega^{J+1} \times (0,T))^d$;
\item[(vi)] $\mathbb{P}_\epsilon$ can be expressed as $\mathbb{P}_\epsilon = \beta\, \overline{\varrho}_\epsilon\, \mathbb{I} + \sqrt{\epsilon}\, \mathbb{R}_\epsilon$, with $(\mathbb{R}_\epsilon)_{\epsilon>0}$ bounded in $L^2(\Omega^{J+1} \times (0,T))^{(J+1)d \times (J+1)d}$,
    and $\mathbb{I}_{i,\ell,j,k}:=\delta_{i,j}\,\delta_{\ell,k}$ for $i,j=1,\dots,J+1$ and $\ell,k = 1,\dots,d$.
\end{itemize}
\end{lemma}

\begin{proof}
(i) In the previous section we showed that
\[ u_\eps \in \mathcal{C}([0,T];L^\sigma(\Omega)^d) \cap L^2(0,T;W^{1,\sigma}_0(\Omega)^d),\quad\mbox{with $\sigma>d$},\]
and $(u_\eps)_{\eps>0}$ is a bounded sequence in the norms of the function spaces appearing on the right-hand side of this
inclusion. Hence, using Morrey's inequality, we readily  deduce (i).

\smallskip

(ii) The Cauchy--Schwarz inequality implies the following bound:
\begin{equation*}
 |\overline{\varrho}_\epsilon(r,t)|^2 \le \left(\int_{\R^{(J+1)d}} M(v) \dd v\right) \left(\int_{\R^{(J+1)d}}|\hrho_\epsilon(r,v,t)|^2\,M(v) \dd v\right),
\end{equation*}
which then implies (ii), since $\hrho_\eps=\varrho_\epsilon/M$ is bounded in the function space
\[L^\infty(0,T; L^2_M(\Omega^{J+1} \times \R^{(J+1)d})) \cap L^2(0,T;L^2(\Omega^{J+1}; W^{1,2}_M(\R^{(J+1)d}))),\]
and $|{\mathcal L}r| \le C$ for all $r \in \overline{\Omega^{J+1}}$, where $C$ is a positive constant, independent of $\epsilon$.

\smallskip

(iii) Finally, we have that
\begin{align*}
\int_0^T \int_{\Omega^{J+1}}|\overline{\varrho}_\epsilon u_\epsilon|^2 \dd r \dd t & \le \int_0^T \|\overline{\varrho}_\epsilon(\cdot,t)\|^2_{L^2(\Omega^{J+1})} \|u_\epsilon(\cdot,t)\|^2_{L^\infty(\Omega)} \dd t \\
& \le  \|\overline{\varrho}_\epsilon\|^2_{L^\infty(0,T;L^2(\Omega^{J+1}))} \int_0^T \|u_\epsilon(\cdot,t)\|^2_{L^\infty(\Omega)} \dd t,
\end{align*}
which proves (iii) by using (i) and (ii).

\smallskip

(iv) Now, let us show that the sequence $(D_{\epsilon,j})_{\eps>0}$ is bounded in $L^2(\Omega^{J+1} \times \R^{(J+1)d}\times (0,T))^d$ for each $j=1,\dots,J+1$. On the one hand,
we know 
that $(\hrho_\eps)_{\eps>0}$ is a bounded sequence in the function space $L^\infty(0,T; L^2(\Omega^{J+1};L^2_M(\R^{(J+1)d}))) \cap L^2(0,T;L^2(\Omega^{J+1}; W^{1,2}_M(\R^{(J+1)d})))$; in particular,
\begin{equation}
\frac{\beta^2}{\eps^2}\int_0^T \int_{\Omega^{J+1}}\int_{\R^{(J+1)d}} \bigg| \frac{\pdv \varrho_\epsilon}{M(v)} + \frac{1}{\beta} \frac{v_j \varrho_\epsilon}{M(v)} \bigg|^2 M(v) \dd v \dd r \dd t
= \frac{\beta^2}{\eps^2}\int_0^T \int_{\Omega^{J+1}}\int_{\R^{(J+1)d}} \big| \pdv \, \hrho_\epsilon \big|^2 M(v) \dd v \dd r \dd t \leq C,
\label{eq:5.26}
\end{equation}
where $C$ is a positive constant, independent of $\epsilon$.

On the other hand, we write
\begin{align*}
& \int_0^T \int_{\Omega^{J+1}}\int_{\R^{(J+1)d}} \bigg| \frac{\pdv \varrho_\epsilon}{M(v)} + \frac{1}{\beta} \frac{v_j \varrho_\epsilon}{M(v)} - \sqrt{\epsilon}\, u_{\epsilon}(r_j,\cdot) \bigg|^2 M(v) \dd v \dd r \dd t \\
& \le 2 \int_0^T \int_{\Omega^{J+1}}\int_{\R^{(J+1)d}}  \bigg| \frac{\pdv \varrho_\epsilon}{M(v)} + \frac{1}{\beta} \frac{v_j \varrho_\epsilon}{M(v)} \bigg|^2 M(v) \dd v \dd r \dd t \\&\mathrel{\phantom{=}} +  2 \int_0^T \int_{\Omega^{J+1}}\int_{\R^{(J+1)d}}  \big| \sqrt{\epsilon}\, u_{\epsilon}(r_j,t) \big|^2 M(v) \dd v \dd r \dd t \\
& \le C\epsilon^2 +  2 |\Omega|^J \epsilon \int_0^T \int_{\Omega} |u_{\epsilon}(r_j,t)|^2 \dd r_j \dd t.
\end{align*}
Thus, using (i) and \eqref{eq:5.26} it follows that, for each $j=1,\dots,J+1$,
\[ \int_0^T \int_{\Omega^{J+1}}\int_{\R^{(J+1)d}}  |D_{\epsilon,j}|^2 \dd v \dd r \dd t \leq C \eps,\]
where $C$ is a positive constant, independent of $\epsilon$, which completes the proof of (iv).

\smallskip

(v) Next, we have that, since,
\[ \int_{\R^{(J+1)d}} M(v)\, v_j \dd v = 0,\qquad j=1,\dots,J+1,\]
also
\[  \int_{\R^{(J+1)d}} \overline{\varrho}_\epsilon(r,t)\, M(v)\, v_j \dd v  = \overline{\varrho}_\epsilon(r,t)  \int_{\R^{(J+1)d}} M(v)\, v_j \dd v = 0,\qquad j=1,\dots,J+1.\]
Therefore, by the Cauchy--Schwarz inequality and a Poincar\'{e}--Sobolev inequality with a Gaussian weight function,\footnote{See p.941 in Nash \cite{Nash}, p.533 in Chernoff \cite{Chernoff}, and p.397 in Beckner \cite{Beckner}.} we have that
\begin{align*}
|\mathcal{J}_{\epsilon,j}(r,t)|^2  &= \left| \int_{\R^{(J+1)d}} \frac{v_j}{\epsilon} \varrho_{\eps}(r,v,t)\dd v\right|^2 = \frac{1}{\epsilon^2}\left| \int_{\R^{(J+1)d}} v_j \, (\varrho_{\eps}(r,v,t) - \overline{\varrho}_\epsilon(r,t)\, M(v))\dd v\right|^2\\
&\leq  \frac{\overline{\varrho}^2_\epsilon(r,t)}{\epsilon^2} \left( \int_{\R^{(J+1)d}}|v_j|^2\,M(v) \dd v\right) \left(  \int_{\R^{(J+1)d}} \left|\frac{\varrho_{\eps}(r,v,t)}{\overline{\varrho}_\epsilon(r,t)\, M(v)} - 1\right|^2 M(v) \dd v\right)\\
& \leq \frac{\overline{\varrho}^2_\epsilon(r,t)}{\epsilon^2}\, \left( \int_{\R^{(J+1)d}}|v_j|^2\,M(v) \dd v\right) \left( \int_{\R^{(J+1)d}} \left|\nabla_v \left(\frac{\varrho_{\eps}(r,v,t)}{\overline{\varrho}_\epsilon(r,t)\, M(v)}\right)\right|^2 M(v) \dd v\right)\\
& = \frac{1}{\epsilon^2} \left( \int_{\R^{(J+1)d}}|v_j|^2\,M(v) \dd v\right) \left( \int_{\R^{(J+1)d}} \left|\nabla_v \left(\frac{\varrho_{\eps}(r,v,t)}{M(v)}\right)\right|^2 M(v) \dd v\right),\qquad j=1,\dots,J+1.
\end{align*}
Hence,
\[ | \mathcal{J}_{\epsilon,j}(r,t) |^2 \leq \frac{C}{\epsilon^2}   \int_{\R^{(J+1)d}} M\,|\nabla_v \hrho_{\eps}(r,v,t)|^2 \dd v,\qquad j=1,\dots,J+1,
\]
where $C$ is a positive constant, independent of $\epsilon$. Therefore,
\[ \int_0^T \int_{\Omega^{J+1}} |\mathcal{J}_{\epsilon,j}(r,t)|^2 \dd r \dd t \leq \frac{C}{\epsilon^2}\,\int_0^T \int_{\Omega^{J+1}}  \int_{\R^{(J+1)d}} M\, |\nabla_v \hrho_{\eps}(r,v,t)|^2 \dd v \dd r \dd t \leq C,\qquad j=1,\dots,J+1,\]
where $C$ is a positive constant, independent of $\epsilon$. That completes the proof of (v).

\smallskip

(vi) We recall from part (iv) the definition of $D_{\epsilon,j}$, $j=1,\dots,J+1$, and denote its $k$-th component by $D_{\epsilon,j,k}$, $k=1,\dots,d$. Analogously, let $u_{\epsilon,k}$ denote the $k$-th component of $u_\epsilon$, $k=1,\dots,d$. We then have that
\begin{align*}
[\mathbb{P}_\epsilon(r,t)]_{i,\ell,j,k} &= \int_{\R^{(J+1)d}} v_{i,\ell}\,v_{j,k} \,\varrho_\epsilon \dd v
\\
&= \beta \sqrt{\epsilon} \int_{\R^{(J+1)d}} \left(v_{i,\ell} \sqrt{M(v)} \, \frac{D_{\epsilon,j,k}(r,v,t)}{\sqrt{\epsilon}}\right) \dd v\\
&\quad + \beta \sqrt{\epsilon} \int_{\R^{(J+1)d}}  \left( v_{i,\ell} \, u_{\epsilon,k}(r_j,t)\right)  M(v)\dd v -\beta \int_{\R^{(J+1)d}} v_{i,\ell} \,\partial_{v_{j,k}} \varrho_\epsilon  \dd v.
\end{align*}
Focusing on the first two integrals, we define
\[ [\mathbb{R}_\epsilon(r,t)]_{i,\ell,j,k}: = \beta \int_{\R^{(J+1)d}} \left(v_{i,\ell} \sqrt{M(v)} \, \frac{D_{\epsilon,j,k}(r,v,t)}{\sqrt{\epsilon}}\right) \dd v
+ \beta  \int_{\R^{(J+1)d}}  \left( v_{i,\ell} \, u_{\epsilon,k}(r_j,t)\right)  M(v)\dd v.\]
The last of the three integrals in the expression for $\mathbb{P}_\epsilon$ is equal to $\delta_{i,j} \, \delta_{\ell,k}\,\overline{\varrho}_\epsilon$ by partial integration. Hence,
\[ [\mathbb{P}_\epsilon(r,t)]_{i,\ell,j,k} =  \beta\,\delta_{i,j}\, \delta_{\ell,k}\,\overline{\varrho}_\epsilon + \sqrt{\epsilon}\,[\mathbb{R}_\epsilon(r,t)]_{i,\ell,j,k}.\]
%
%
%
To complete the proof of (vi) it therefore remains to establish a uniform (with respect to $\epsilon$) bound on $[\mathbb{R}_\epsilon]_{i,\ell,j,k}$ in the norm of $L^2(\Omega^{J+1} \times (0,T))$, for $i,j=1,\dots, J+1$
and $\ell,k = 1, \dots, d$.

We have that
\begin{align*}
&\frac{1}{\beta^2}\int_0^T \int_{\Omega^{J+1}}|[\mathbb{R}_\epsilon]_{i,\ell,j,k}|^2 \dd r\dd t\\
&\quad= \int_0^T \int_{\Omega^{J+1}} \bigg|\int_{\R^{(J+1)d}} \left(v_{i,\ell} \sqrt{M(v)} \, \frac{D_{\epsilon,j,k}(r,v,t)}{\sqrt{\epsilon}}\right) \dd v
+ \int_{\R^{(J+1)d}}  \left( v_{i,\ell} \, u_{\epsilon,k}(r_j,t)\right)  M(v)\dd v\bigg|^2 \dd r \dd t \\
&\quad\le 2 \int_0^T \int_{\Omega^{J+1}}   \bigg|\int_{\R^{(J+1)d}} \left(v_{i,\ell} \sqrt{M(v)} \, \frac{D_{\epsilon,j,k}(r,v,t)}{\sqrt{\epsilon}}\right) \dd v\bigg|^2 \dd r \dd t\\
&\qquad + 2 \int_0^T \int_{\Omega^{J+1}} \bigg|\int_{\R^{(J+1)d}}  \left( v_{i,\ell} \, u_{\epsilon,k}(r_j,t)\right)  M(v)\dd v \bigg|^2 \dd r \dd t
\\
&\quad \le 2 \int_0^T \int_{\Omega^{J+1}} \bigg( \int_{\R^{(J+1)d}} |v_{i,\ell}|^2 M(v) \dd v \bigg) \bigg( \int_{\R^{(J+1)d}} \bigg|\frac{D_{\epsilon,j,k}(r,v,t)}{\sqrt{\epsilon}} \bigg|^2 \dd v \bigg) \dd r \dd t\\
&\qquad + 2 \bigg(\int_{\R^{(J+1)d}} |v_{i,\ell}|\, M(v) \dd v \bigg)^2 \bigg(\int_0^T \int_{\Omega^{J+1}}| u_{\epsilon,k}(r_j,t) |^2 \dd r \dd t \bigg) \\
&\quad \le C \int_0^T \int_{\Omega^{J+1}} \int_{\R^{(J+1)d}} \bigg|\frac{D_{\epsilon,j,k}(r,v,t)}{\sqrt{\epsilon}} \bigg|^2 \dd v \dd r \dd t + C\,|\Omega|^J \int_0^T \int_{\Omega}| u_{\epsilon,k}(r_j,t) |^2 \dd r_j \dd t,
\end{align*}
where $C$ is a positive constant, since moments of any order of $M$ are finite. Thus, the statement in part (vi) of the lemma follows from the assertions in parts (iv) and (i).
\end{proof}

Using the equations \eqref{eq:Moment1}, \eqref{eq:Moment2} together with the splitting of $\mathbb{P}_\epsilon$ introduced in part (vi) of Lemma \ref{lemma:bounds}, we arrive at the following system of moment equations:
\begin{eqnarray}
\left\{
\begin{array}{l}
\partial_t \overline{\varrho}_\epsilon + \mbox{div}_r\, \mathcal{J}_\epsilon = 0,\\
\beta\,\partial_r \overline{\varrho}_\epsilon = \sqrt{\epsilon} \big( - \epsilon \sqrt{\epsilon}\, \partial_t \mathcal{J}_\epsilon - \mbox{Div}_r \mathbb{R_\epsilon} \big) + \overline{\varrho}_\epsilon\,u_\epsilon  - \beta \mathcal{J}_\epsilon + {\mathcal L}r\, \overline{\varrho}_\epsilon.
\end{array}
\right.
\label{sys:Moments}
\end{eqnarray}

\begin{lemma}\label{le:strong-rho}
The sequence $(\overline{\varrho}_\epsilon)_\epsilon$ converges to $\overline{\varrho}= \eta$
weakly in the space $L^2(\Omega^{J+1} \times (0,T))$ and strongly in $L^p(\Omega^{J+1} \times (0,T))$ for all $p \in [1,2)$. Furthermore, we have that
\[ \lim_{\epsilon \rightarrow 0_+} \int_0^T \int_{\Omega^{J+1}} \int_{\R^{(J+1)d}} | \varrho_\epsilon- \overline{\varrho} M(v) | \dd v \dd r \dd t = 0.
\]
\end{lemma}

\begin{proof}
We begin by focusing on the first equation in the system \eqref{sys:Moments}. We observe that the sequence $(\mbox{div}_{(r,t)}(\mathcal{J}_\epsilon,\overline{\varrho}_\epsilon))_{\epsilon>0}$ (where $\mbox{div}_{(r,t)}(\mathcal{J}_\epsilon,\overline{\varrho}_\epsilon)$ is the divergence with respect to the $(r,t)$ variables
of the vector field $(\mathcal{J}_\epsilon,\overline{\varrho}_\epsilon)$, defined as
$(\mbox{div}_r, \partial_t)\cdot (\mathcal{J}_\epsilon,\overline{\varrho}_\epsilon)$,) is, thanks to \eqref{sys:Moments}$_1$, the zero-sequence $(0)_{\epsilon>0}$, and it is therefore, trivially, precompact in $W^{-1,2}(\Omega^{J+1} \times (0,T))$.

Next, we focus on the second equation in the system \eqref{sys:Moments}, which we restate here for clarity:
\begin{equation}\label{eq:rep-partial-varrho}
\beta\, \partial_r \overline{\varrho}_\epsilon = \sqrt{\epsilon} (- \epsilon \sqrt{\epsilon}\, \partial_t \mathcal{J}_\epsilon - \mbox{Div}_r\, \mathbb{R}_\epsilon) + \overline{\varrho}_\epsilon u_\epsilon - \beta \,\mathcal{J}_\epsilon + \mathcal{L}r\, \overline{\varrho}_\epsilon.
\end{equation}

Thanks to parts (iii), (v) and (ii) of Lemma \ref{lemma:bounds} the sequence $(\overline{\varrho}_\epsilon u_\epsilon(r_j,\cdot) - \beta \,\mathcal{J}_{\epsilon,j} + (\mathcal{L}r)_j\, \overline{\varrho}_\epsilon)_{\epsilon>0}$ is bounded in the function space $L^2(\Omega^{J+1} \times (0,T))^d$, and therefore, thanks to the compact embedding of the space $L^2(\Omega^{J+1} \times (0,T))$ into
$W^{-1,2}(\Omega^{J+1} \times (0,T))$, the sequence
$(\overline{\varrho}_\epsilon u_\epsilon(r_j,\cdot) - \beta \,\mathcal{J}_{\epsilon,j} + (\mathcal{L}r)_j\, \overline{\varrho}_\epsilon)_{\epsilon>0}$
is a precompact set in the space $W^{-1,2}(\Omega^{J+1} \times (0,T))^d$, for each $j=1,\dots, J+1$.

Furthermore, by parts (v) and (vi) of Lemma \ref{lemma:bounds} the sequences $(\mathcal{J}_\epsilon)_{\epsilon>0}$ and
$(\mathbb{R}_\epsilon)_{\epsilon>0}$ are bounded in the spaces $L^2(\Omega^{J+1} \times (0,T))^{(J+1)d}$
and $W^{-1,2}(\Omega^{J+1} \times (0,T))^{(J+1)d\times (J+1)d}$, respectively; therefore, the sequence
$(- \epsilon \sqrt{\epsilon} \partial_t \mathcal{J}_\epsilon - \mbox{Div}_r\, \mathbb{R}_\epsilon)_{\epsilon>0}$ is bounded
in $W^{-1,2}(\Omega^{J+1} \times (0,T))^{(J+1)d}$, whereby, upon multiplication by $\sqrt{\epsilon}$, we have that the sequence
$(\sqrt{\epsilon} (- \epsilon \sqrt{\epsilon} \partial_t \mathcal{J}_\epsilon - \mbox{Div}_r\, \mathbb{R}_\epsilon))_{\epsilon>0}$ is precompact in the space $W^{-1,2}(\Omega^{J+1} \times (0,T))^{(J+1)d}$;
more precisely, it converges to $0$ in $W^{-1,2}(\Omega^{J+1} \times (0,T))^{(J+1)d}$, as $\epsilon \rightarrow 0_+$.
Thus, since $\beta>0$, we deduce from \eqref{eq:rep-partial-varrho} that the sequence $(\partial_r \overline{\varrho}_\epsilon)_{\epsilon>0}$ is precompact in $W^{-1,2}(\Omega^{J+1} \times (0,T))^{(J+1)d}$. Hence, the sequence $(\mbox{curl}_{(r,t)}(0,\overline{\varrho}_\epsilon))_{\epsilon>0}$ (where
$\mbox{curl}_{(r,t)}(0,\overline{\varrho}_\epsilon)$ is the curl with respect to the $(r,t)$ variables, defined as $\partial_{(r,t)}-\partial_{(r,t)}^{\rm T}$, of the $((J+1)d+1)$-component vector field $(0,\overline{\varrho}_\epsilon)$, where $0$ is a $(J+1)d$-component zero-vector), is a precompact set in $W^{-1,2}(\Omega^{J+1} \times (0,T))^{((J+1)d+1) \times ((J+1)d+1)}$.

Hence, a direct application of the Div-Curl Lemma (cf. \cite{Tar}) yields that the weak limit of the
scalar product of the sequences $((\mathcal{J}_\epsilon,\overline{\varrho}_\epsilon))_{\epsilon>0}$
and $((0,\overline{\varrho}_\epsilon))_{\epsilon>0}$ is equal to the scalar product of their weak limits; i.e.,
\[ (\mathcal{J}_\epsilon,\overline{\varrho}_\epsilon)\cdot(0,\overline{\varrho}_\epsilon)
=\overline{\varrho}_\epsilon^2
\rightharpoonup
(\mathcal{J},\overline{\varrho})\cdot(0,\overline{\varrho})
=\overline{\varrho}^2 \quad \mbox{in $\mathcal{D}'(\Omega^{J+1} \times (0,T))$.}
\]
Combining this with the weak convergence result $\overline{\varrho}_\epsilon\rightharpoonup\overline{\varrho}$ in $L^2(\Omega^{J+1} \times (0,T))$, we have that
\begin{align*}
\int_{\Omega^{J+1} \times (0,T)} |\overline{\varrho}_\epsilon - \overline{\varrho}|^2 \phi \dd r \dd t =
\int_{\Omega^{J+1} \times (0,T)} [\overline{\varrho}_\epsilon]^2\, \phi \dd r \dd t +
\int_{\Omega^{J+1} \times (0,T)} [\overline{\varrho}]^2 \,\phi \dd r \dd t
- 2 \int_{\Omega^{J+1} \times (0,T)} \overline{\varrho}_\epsilon \,\overline{\varrho}\, \phi \dd r \dd t\\
= \langle [\overline{\varrho}_\epsilon]^2, \phi \rangle + \langle [\overline{\varrho}]^2, \phi \rangle
- 2\, \langle \overline{\varrho}_\epsilon , \overline{\varrho}\, \phi\rangle \rightarrow 0\qquad \mbox{as $\epsilon \rightarrow 0_+$, for all $\phi \in \mathcal{C}^\infty_0(\Omega^{J+1} \times (0,T))$}.
\end{align*}
This proves the strong convergence of $\overline{\varrho}_\epsilon$ to $\overline{\varrho}$ in $L^2_{loc}(\Omega^{J+1} \times (0,T))$. Thus, for any compact subset $\mathfrak{D}$ of $\Omega^{J+1} \times (0,T)$, we can extract a subsequence from the sequence $(\overline{\varrho}_\epsilon)_{\epsilon>0}$ that converges to $\overline{\varrho}$ a.e. on $\mathfrak{D}$. Hence,
by considering a countable nested family of compact sets $\mathfrak{D}_j \subset \Omega^{J+1} \times (0,T)$ with $\cup_{j \geq 1}
\mathfrak{D}_j = \Omega^{J+1} \times (0,T)$, by successive extraction of subsequences, there exists a subsequence of
$(\overline{\varrho}_\epsilon)_{\epsilon>0}$ (not indicated), which converges to $\overline{\varrho}$ a.e. on
$\Omega^{J+1} \times (0,T)$.

By combining the weak convergence $\overline{\varrho}_\epsilon\rightharpoonup\overline{\varrho}$ in $L^2(\Omega^{J+1} \times (0,T))$ (which implies the weak converge $\overline{\varrho}_\epsilon\rightharpoonup\overline{\varrho}$ in $L^1(\Omega^{J+1} \times (0,T))$, and thereby, thanks to the Dunford--Pettis theorem (cf. Theorem 2.54 in \cite{Fonseca_Leoni}), equiintegrability of $(\overline{\varrho}_\epsilon)_{\epsilon>0}$ on $\Omega^{J+1} \times (0,T)$) and the a.e. convergence of $\overline{\varrho}_\epsilon$ to $\overline{\varrho}$, Vitali's convergence theorem (cf. Theorem 2.24 in \cite{Fonseca_Leoni}) yields the strong convergence of $\overline{\varrho}_\epsilon$ to $\overline{\varrho}$ in $L^1(\Omega^{J+1} \times (0,T))$, and
therefore, thanks to the boundedness of the sequence $\overline{\varrho}_\epsilon$ in $L^p(\Omega^{J+1} \times (0,T))$, $1 \le p \leq 2$, we have
strong convergence $\overline{\varrho}_\epsilon\rightarrow\overline{\varrho}$ in $L^p(\Omega^{J+1} \times (0,T))$
for all $p \in [1,2)$.

Next, by the triangle inequality and noting that $\int_{\R^{(J+1)d}} M(v) \dd v =1$, we have that
\begin{equation*}
\int_0^T \int_{\Omega^{J+1}} \int_{\R^{(J+1)d}} | \varrho_\epsilon - M(v) \overline{\varrho} | \dd v \dd r \dd t \le \int_0^T \int_{\Omega^{J+1}} \int_{\R^{(J+1)d}} | \varrho_\epsilon - M(v) \overline{\varrho}_\epsilon | \dd v \dd r \dd t + \int_0^T \int_{\Omega^{J+1}} | \overline{\varrho} - \overline{\varrho}_\epsilon | \dd r \dd t.
\end{equation*}
We have already shown that the second integral on the right-hand side of this inequality tends to 0 as $\epsilon$ tends to $0$. For the first integral, using the Cauchy--Schwarz inequality and a Poincar\'e--Sobolev inequality with a Gaussian weight function (cf. the proof of item (v) in Lemma \ref{lemma:bounds}),
we obtain
\begin{align*}
\int_{\R^{(J+1)d}} | \varrho_\epsilon - M(v) \overline{\varrho}_\epsilon | \dd v & \le \bigg ( \int_{\R^{(J+1)d}}  | \varrho_\epsilon - M(v) \overline{\varrho}_\epsilon |^2 \frac{1}{M(v)} \dd v \bigg)^\frac{1}{2}\\
& \le \overline{\varrho}_\epsilon \bigg ( \int_{\R^{(J+1)d}} \bigg | \frac{\varrho_\epsilon}{\overline{\varrho}_\epsilon M(v)} - 1 \bigg |^2 M(v) \dd v \bigg)^\frac{1}{2} \\
& \le \bigg ( \int_{\R^{(J+1)d}} \bigg | \nabla_v \bigg(\frac{\varrho_\epsilon}{M(v)} \bigg ) \bigg |^2 M(v) \dd v \bigg)^\frac{1}{2}.
\end{align*}
Since
\begin{align*}
\bigg (\int_0^T \int_{\Omega^{J+1}} \int_{\R^{(J+1)d}} \bigg | \nabla_v \bigg(\frac{\varrho_\epsilon}{M(v)} \bigg ) \bigg |^2 M(v) \dd v  \dd r \dd t \bigg)^\frac{1}{2}  \le C \epsilon,
\end{align*}
we deduce by the Cauchy--Schwarz inequality that
\begin{align*}
\int_0^T \int_{\Omega^{J+1}}\int_{\R^{(J+1)d}} | \varrho_\epsilon - M(v) \overline{\varrho}_\epsilon | \dd v \dd r \dd t& \le C \epsilon,
\end{align*}
and therefore,
\begin{equation*}
\lim_{\epsilon \rightarrow 0_+} \int_0^T \int_{\Omega^{J+1}} \int_{\R^{(J+1)d}} | \varrho_\epsilon - M(v) \overline{\varrho} | \dd v \dd r \dd t = 0.
\end{equation*}
That completes the proof of the lemma.
\end{proof}

\begin{remark}\label{rem:equilib}
The strong convergence $\varrho_\epsilon \rightarrow M(v)\, \overline{\varrho} = M(v)\, \eta = \rho_{(0)}$ in
$L^1(\Omega^{J+1} \times \R^{(J+1)d} \times (0,T))$ in the small-mass limit $\epsilon \rightarrow 0_+$, which we have rigorously proved above,
is referred to in the chemical physics literature as \textit{equilibration in momentum} space (cf. p.71 in \cite{CBH}),
in the sense that the limiting
probability density function $\rho_{(0)}$ has the factorized form $M(v)\, \eta$, where $\eta=\eta(r,t)$
is completely independent of $v$, and satisfies a Fokker--Planck equation, which we shall carefully identify below;
furthermore, by noting part (vi) of Lemma \ref{lemma:bounds} and Lemma \ref{le:strong-rho}, we deduce that
\[ \lim_{\epsilon \rightarrow 0_+} \int_{\mathbb{R}^{(J+1)d}} v_{i,l}\, v_{j,k}\, \varrho_\epsilon  \dd v =  \beta\, \delta_{i,j}\,\delta_{\ell,k}\,\eta,\qquad \mbox{where
$\beta = k\mathtt{T} \zeta$,}
\]
strongly in $L^p(\Omega^{J+1} \times (0,T))$ for all $p \in [1,2)$ and weakly in $L^2(\Omega^{J+1} \times (0,T))$, which is yet
another manifestation of equilibration in momentum space, as a consequence of the small mass limit $\epsilon \rightarrow 0_+$.
For further details in this direction, we point the reader to the paper of Schieber and \"Ottinger \cite{SO}, and references therein.
\end{remark}

Having shown the strong convergence $\varrho_\epsilon \rightarrow M(v)\, \overline{\varrho} = M(v)\, \eta = \rho_{(0)}$ in
$L^1(\Omega^{J+1} \times \R^{(J+1)d} \times (0,T))$, we are now ready to pass to the limit $\epsilon \rightarrow 0_+$ in the Oseen equation. All that remains to be done in this respect is to identify the weak$^*$ limit $\KK_{(0)}$ of the sequence $(\KK_\epsilon)_{\epsilon>0}$ in terms of the limit $\eta$ of the sequence $(\hrho_\epsilon)_{\epsilon>0}$, where
\[ \KK_\epsilon:=\frac{{\mathfrak A}_\epsilon}{{\mathfrak B}_\epsilon}, \qquad \epsilon>0, \]
with
\begin{align*}
{\mathfrak A}_\epsilon &:= \int_{D^{J}\times \R^{(J+1)d}} \sum_{j=1}^{J}(F(q_j)\otimes q_j)\,M\,\hrho_\epsilon\bigl(B(q,x),v,t\bigr)
\dd q  \dd v,
\\
{\mathfrak B}_\epsilon &:= \int_{D^{J}\times \R^{(J+1)d}} M\,\hrho_\epsilon\bigl(B(q,x),v,t\bigr)
\dd q  \dd v.
\end{align*}

The limit $\KK_{(0)}$ is anticipated to be of the form
\[ \frac{\mathfrak A_{(0)}}{\mathfrak B_{(0)}},\]
where
\begin{align*}
\mathfrak{A}_{(0)} &:= \int_{D^{J}\times \R^{(J+1)d}} \sum_{j=1}^{J}(F(q_j)\otimes q_j)\,M\,\eta\bigl(B(q,x),t\bigr)
\dd q  \dd v = \int_{D^{J}} \sum_{j=1}^{J}(F(q_j)\otimes q_j)\,\eta\bigl(B(q,x),t\bigr)\dd q,
\\
\mathfrak{B}_{(0)} &:= \int_{D^{J}\times \R^{(J+1)d}} M\,\eta\bigl(B(q,x),t\bigr)
\dd q  \dd v = \int_{D^{J}}\eta\bigl(B(q,x),t\bigr) \dd q.
\end{align*}
The proof of this is identical to the proof, presented in Section \ref{sec:limit-oseen}, that the weak$^*$ limit $\KK$ of the sequence $(\KK^{(k)})_{k \geq 0}$, where $\KK^{(k)}=\frac{{\mathfrak A}^{(k)}}{{\mathfrak B}^{(k)}}$, $k=0,1,\dots$, considered in terms of the limit $\hrho$ of the sequence $(\hrho^{(k)})_{k \geq 0}$, is of the form $\frac{\mathfrak A}{\mathfrak B}$,
the key ingredient in the argument being the strong convergence $\varrho_\epsilon \rightarrow M(v)\, \overline{\varrho} = M(v)\, \eta = \rho_{(0)}$ in $L^1(\Omega^{J+1} \times \R^{(J+1)d} \times (0,T))$, guaranteed by Lemma
\ref{le:strong-rho}. We do not repeat the proof, therefore.


We now return to  \eqref{eq:Moment1}, and perform partial integration in the first term on the left-hand side, yielding
\begin{align}\label{eq:Moment11}
&\int_{\Omega^{J+1}} \orho_\eps(r,t)\,\phi(r,t) \dd r
-\int_0^t \int_{\Omega^{J+1}} \orho_\eps(r,\tau)\,\pd_\tau \phi(r,\tau) \dd r \dd \tau
- \sum_{j=1}^{J+1} \int_0^t \int_{\Omega^{J+1}}  \mathcal{J}_{\epsilon,j} \cdot \pdr \phi \dd r \dd \tau \nonumber\\
&\qquad =
\int_{\Omega^{J+1}} \orho_0(r)\,\phi(r,0) \dd r \quad
\forall\, \phi \in L^2(0,T; W^{1,2}(\Omega^{J+1}))
\cap W^{1,2}(0,T; L^{2}(\Omega^{J+1}))
, \quad \forall\,t \in (0,T],
\end{align}
since $\orho_\epsilon(\cdot,0) = \orho_0(\cdot):=\int_{\R^{(J+1)d}} M(v)\,\hrho_\epsilon(r,v,0) \dd v$. Passage to the limit
$\epsilon \rightarrow 0_+$ then gives
\begin{align}\label{eq:Moment12}
&\int_{\Omega^{J+1}} \eta(r,t)\, \phi(r,t) \dd r
-\int_0^t \int_{\Omega^{J+1}} \eta(r,\tau)\,\pd_\tau \phi(r,\tau) \dd r \dd \tau
- \sum_{j=1}^{J+1} \int_0^t \int_{\Omega^{J+1}}  \mathcal{J}_{j} \cdot \pdr \phi \dd r \dd \tau \nonumber\\
&\quad =
\int_{\Omega^{J+1}} \orho_0(r)\, \phi(r,0) \dd r \qquad
\forall\, \phi \in L^2(0,T; W^{1,2}(\Omega^{J+1}))\cap W^{1,2}(0,T; L^{2}(\Omega^{J+1})), \quad \forall\,t \in (0,T],
\end{align}
where $\mathcal{J}_{j} := - \beta\, \pdr\eta + \eta\, (({\mathcal L}r)_j+u_{(0)}(r_j,\cdot))$ for
$j=1,\dots,J+1$. To see that this is indeed the case, we recall from the proof of Lemma \ref{le:strong-rho} that
the sequence $$(\sqrt{\epsilon} (- \epsilon \sqrt{\epsilon}\, \partial_t \mathcal{J}_\epsilon - \mbox{Div}_r\, \mathbb{R}_\epsilon))_{\epsilon>0}$$ converges to $0$ in $W^{-1,2}(\Omega^{J+1} \times (0,T))^d$ as $\epsilon \rightarrow 0_+$.
It then follows from \eqref{eq:rep-partial-varrho} that, for each $j=1,\dots,J+1$,
\begin{equation}\label{eq:J-limit}
\lim_{\epsilon \rightarrow 0_+} (-\beta\, \partial_{r_j} \overline{\varrho}_\epsilon + \overline{\varrho}_\epsilon ((\mathcal{L}r)_j + u_\epsilon(r_j,\cdot))) = \mathcal{J}_{j}\qquad \mbox{in  $W^{-1,2}(\Omega^{J+1} \times (0,T))^d$}.
\end{equation}
Thanks to \eqref{eq:NS-conv-eps}$_3$ and since $\overline{\varrho}_\epsilon \rightharpoonup \overline{\varrho}= \eta$ weakly* in $L^\infty(0,T;L^2(\Omega^{J+1}))$, it follows that, for each $j=1,\dots,J+1$,
$$\overline{\varrho}_\epsilon ((\mathcal{L}r)_j + u_\epsilon(r_j,\cdot))
\rightharpoonup \eta ((\mathcal{L}r)_j + u(r_j,\cdot))\qquad \mbox{weakly in $L^2(0,T;L^2(\Omega^{J+1})^d)$}.$$
Also, $\beta\, \partial_{r_j} \overline{\varrho}_\epsilon  \rightharpoonup \beta\, \partial_{r_j} \eta$ weakly* in
$L^\infty(0,T;W^{-1,2}(\Omega^{J+1})^d)$. Hence,
\[\mathcal{J}_{j} := - \beta\, \pdr\eta + \eta\, (({\mathcal L}r)_j+u_{(0)}(r_j,\cdot)),\qquad
j=1,\dots,J+1,\]
as an equality in $W^{-1,2}(\Omega^{J+1} \times (0,T))^d$.

Now, since $\mathcal{J}_{j} \in L^2(\Omega^{J+1} \times (0,T))^d$ and $\eta\, (({\mathcal L}r)_j+u_{(0)}(r_j,\cdot)) \in L^2(\Omega^{J+1} \times (0,T))^d$ for all
$j=1,\dots,J+1$, it follows that $\pdr\eta \in L^2(\Omega^{J+1} \times (0,T))^d$ for all
$j=1,\dots,J+1$. Therefore,
\begin{equation}\label{eq:Ij}
\mathcal{J}_{j} := - \beta\, \pdr\eta + \eta\, (({\mathcal L}r)_j+u_{(0)}(r_j,\cdot)),\qquad
j=1,\dots,J+1,
\end{equation}
as an equality in $L^2(\Omega^{J+1} \times (0,T))^d$.

To summarize the main result of this section, we have shown that the small-mass limit of the coupled Oseen--Fokker--Planck
system under consideration satisfies the following coupled problem: the velocity-pressure pair $(u_{(0)},\pi_{(0)})$
solves the Oseen system
\begin{alignat}{2}\label{eq:oseen0}
\begin{aligned}
\pd_t u_{(0)} + (b\cdot \nabla) u_{(0)} - \mu \triangle u_{(0)} + \nabla \pi_{(0)} &= \nabla \cdot \KK_{(0)} &&\qquad \mbox{for
$(x, t) \in {\Omega} \times (0,T]$}, \\
\nabla \cdot u_{(0)} &= 0 &&\qquad \mbox{for $(x, t) \in {\Omega} \times (0, T]$},\\
u_{(0)}(x,t) &=0 &&\qquad \mbox{for $(x,t) \in \partial\Omega \times (0,T]$},\\
u_{(0)}(x,0) & = u_0(x)&&\qquad \mbox{for $x \in {\Omega}$},
\end{aligned}
\end{alignat}
with
\begin{alignat}{2}\label{eq:kramers}
\KK_{(0)}(x,t) &:= \frac{\int_{D^{J}} \sum_{j=1}^{J}(F(q_j)\otimes q_j)\,\eta\bigl(B(q,x),t\bigr)\dd q}{\int_{D^{J}}\eta\bigl(B(q,x),t\bigr) \dd q} &&\qquad\mbox{for
$(x, t) \in \Omega \times (0, T]$},
\end{alignat}
and the nonnegative function $\eta$, with $\int_{\Omega^{J+1}} \eta(r,t) \dd r = 1$ for all $t \in [0,T]$, solves the following parabolic initial-boundary-value problem:
\begin{alignat}{2}
\pd_t \eta &=
\sum_{j=1}^{J+1}  \Bigl(\beta\, \pdr^2\eta-\pdr\cdot \Bigl(\eta\, (({\mathcal L}r)_j+u_{(0)}(r_j,\cdot))\Bigr)\Bigr)&&
\qquad \mbox{in
$\Omega^{J+1} \times (0,T]$}, \label{eq:eta-1}\\
\eta(\cdot,0) &=\hrho_0 \in L^2(\Omega^{J+1};\R_{\geq 0}), \label{eq:eta-2}
\end{alignat}
subject to the weakly imposed boundary condition
$
\mathcal{J}_{j}\cdot \nu(r_j) = 0$ on $\partial\Omega^{(j)} \times (0,T]$ for $j=1,\dots, J+1$
(implied by the third term on the left-hand side of the equation \eqref{eq:Moment12}); i.e.,  by recalling the identity \eqref{eq:Ij},
we have the following zero-normal-flux boundary condition on $\eta$:
\begin{align}\label{eq:eta-3}
\big(\beta\, \pdr\eta - \eta\, (({\mathcal L}r)_j+u_{(0)}(r_j,\cdot))\big)\cdot \nu(r_j) = 0
\qquad \mbox{on $\partial\Omega^{(j)} \times  (0,T]$\, for\, $j=1,\dots, J+1$}.
\end{align}

We note that the partial differential equation \eqref{eq:eta-1} is of the form $\mbox{div}_{(r,t)}(\mathcal{J},\eta) = 0$,
where $\mbox{div}_{(r,t)}$ is the space-time divergence of the $((J+1)d + 1)$-component vector-function $(\mathcal{J},\eta)$
defined on $\Omega^{J+1} \times (0,T)$, with $(\mathcal{J},\eta) \in L^2(\Omega^{J+1} \times (0,T))^{(J+1)d} \times L^2(\Omega^{(J+1)d}
\times (0,T))$. Consequently, by a standard trace theorem for the function space
$H(\mbox{div}, \mathfrak{D})$, with $\mathfrak{D}=\Omega^{J+1} \times (0,T)$, the vector-function
$(\mathcal{J},\eta)$ has a well-defined normal trace on the boundary $\partial(\Omega^{J+1} \times (0,T))$
of the domain $\Omega^{J+1} \times (0,T)$, contained in
$W^{-\frac{1}{2},2}(\partial(\Omega^{J+1} \times (0,T)))$; see, for example, 
Theorem 18.7 in \cite{BaCa}.
Thus, the boundary condition \eqref{eq:eta-3} for \eqref{eq:eta-1} is meaningful,
as an equality in $W^{-\frac{1}{2},2}(\partial\Omega^{(j)} \times (0,T))$
(the dual space of $W^{\frac{1}{2},2}_{00}(\partial\Omega^{(j)} \times (0,T))$, $j=1, \dots, J+1$; cf., for example, Theorem 18.9 in \cite{BaCa}).

We complete this section by proving the existence of a unique solution to the parabolic initial-boundary-value problem satisfied by $\eta$.
To this end, we introduce the real-valued function $\tilde \eta$ defined on $\Omega^{J+1} \times [0,T]$ by
\begin{align*}
\tilde \eta(r,t) & := \eta(r,t) - \frac{1}{|\Omega|} \int_{\Omega^{J+1}} \eta(r,t) \dd r \\
& = \eta(r,t) - \frac{1}{|\Omega|}.
\end{align*}
Hence, we have that the function $\tilde \eta$, with $ \int_{\Omega^{J+1}} \tilde \eta(r,t) \dd r = 0$ for all $t \in [0,T]$, solves the following parabolic initial-boundary-value problem:
\begin{alignat}{2}
\pd_t \tilde \eta &=
\sum_{j=1}^{J+1}  \Bigl(\beta\, \pdr^2\tilde \eta-\pdr\cdot \Bigl( \Bigl(\tilde \eta + \frac{1}{|\Omega|} \Bigr) \, (({\mathcal L}r)_j + u_{(0)}(r_j,\cdot)) \Bigr)\Bigr)
\qquad &&\mbox{in $\Omega^{J+1} \times (0,T]$}, \label{eq:tilde_eta-1}\\
\tilde \eta_0 &:= \tilde \eta(\cdot,0) =\hrho_0 - \frac{1}{|\Omega|} \in L^2(\Omega^{J+1};\R_{\geq 0}),\qquad &&\int_{\Omega^{J+1}} \tilde \eta_0(r) \dd r = 0,
\label{eq:tilde_eta-2}
\\
\label{eq:tilde_eta-3}
&\hspace{-5mm}(\beta\, \pdr\tilde \eta - \Bigl(\tilde \eta + \frac{1}{|\Omega|} \Bigr) \, (({\mathcal L}r)_j+ u_{(0)}(r_j,\cdot))\Bigr)\cdot \nu(r_j) = 0
\quad &&\mbox{on $\partial\Omega^{(j)} \times  (0,T]$, $\;j=1,\dots,J+1$}.
\end{alignat}
Let us introduce the Hilbert space
$$ H^1_{\star}(\Omega^{J+1}) := \left\{
      \begin{aligned}
        \varphi \in H^1(\Omega^{J+1}):  \int_{\Omega^{J+1}} \varphi(r) \dd r = 0
      \end{aligned}
    \right\}$$
equipped with the norm of $H^1(\Omega^{J+1})$, with an analogous definition of $L^2_\star(\Omega^{J+1})$ equipped with the norm of $L^2(\Omega^{J+1})$.

By \eqref{eq:Moment12}, the weak formulation of the problem \eqref{eq:tilde_eta-1}--\eqref{eq:tilde_eta-3}
therefore amounts to seeking a function $$\tilde\eta \in \mathcal{C}([0,T];L^2_\star(\Omega^{J+1})) \cap L^2(0,T; H^1_\star(\Omega^{J+1}))$$ with $$\partial_t \tilde \eta \in L^2(0,T;H^1_{\star}(\Omega^{J+1})'),$$
such that $\tilde \eta (\cdot,0) = \tilde \eta_0(\cdot)$,
and
\begin{align*}
&\langle \partial_t \tilde \eta, \varphi \rangle_{H^1_{\star}(\Omega^{J+1})' \times H^1_{\star}(\Omega^{J+1})}
+ \sum_{j=1}^{J+1} \int_{\Omega^{J+1}} \bigl[\beta \, \pdr \tilde \eta - \tilde \eta \, (({\mathcal L}r)_j+u_{(0)}(r_j,\cdot)) \bigr] \cdot \pdr \varphi \dd r\\
\\& \qquad \qquad  = \sum_{j=1}^{J+1} \int_{\Omega^{J+1}} \frac{1}{|\Omega|} \, (({\mathcal L}r)_j +u_{(0)}(r_j,\cdot))\cdot \pdr \varphi \dd r\qquad \forall\,
\varphi \in H^1_\star(\Omega^{J+1}).
\end{align*}
%

We consider the bilinear form $a(\cdot,\cdot)$ defined on $H^1_\star(\Omega^{J+1}) \times H^1_\star(\Omega^{J+1})$ by
\[ a(\psi, \varphi):=   \sum_{j=1}^{J+1} \int_{\Omega^{J+1}} \bigl[\beta \, \pdr \psi - \psi \, (({\mathcal L}r)_j + u_{(0)}(r_j,\cdot))\bigr] \cdot \pdr \varphi \dd r, \qquad \psi, \varphi \in H^1_\star(\Omega^{J+1}),\]
and set
\[ \ell(\varphi) := \sum_{j=1}^{J+1} \int_{\Omega^{J+1}} \frac{1}{|\Omega|} \, (({\mathcal L}r)_j + u_{(0)}(r_j,\cdot))\cdot \pdr \varphi \dd r, \qquad \varphi \in H^1_\star(\Omega^{J+1}).\]

Because $u_{(0)} \in L^2(0,T;L^\infty(\Omega)^d)$, we have that $\ell \in L^2(0,T;H^1_{\star}(\Omega^{J+1})')$.
The  bilinear  form $a(\cdot, \cdot)$ is  obviously  well-defined for every $\psi, \varphi$ in $H^1_{\star}(\Omega^{J+1})$.
Moreover, by the Cauchy--Schwarz inequality, $a(\cdot,\cdot)$ is bounded (and therefore continuous); i.e.,
\begin{align*}
    |a(\psi, \varphi)| \leq C\, \|\psi\|_{H^1(\Omega^{J+1})}  \|\varphi\|_{H^1(\Omega^{J+1})} \qquad \forall\, \psi, \varphi \in H^1_\star(\Omega^{J+1}),
\end{align*}
for some positive constant $C$, independent of $t \in [0,T]$. Furthermore, $a(\cdot,\cdot)$ satisfies a G{\aa}rding inequality;
indeed, we have that
\begin{align*}
a(\psi, \psi) &= \beta \sum_{j=1}^{J+1} \int_{\Omega^{J+1}} |\pdr \psi|^2 - \sum_{j=1}^{J+1} \int_{\Omega^{J+1}} \psi \, (({\mathcal L}r)_j +  u_{(0)}(r_j,\cdot))\cdot \pdr  \psi \dd r \\
&\geq  \frac{\beta}{2}\sum_{j=1}^{J+1} \int_{\Omega^{J+1}} |\pdr \psi|^2 - \frac{1}{2\beta}\sum_{j=1}^{J+1} \int_{\Omega^{J+1}} \,  |({\mathcal L}r)_j + u_{(0)}(r_j,\cdot)|^2 \, |\psi| ^2 \dd r \\
& \geq  \frac{\beta}{2} \,|\psi|^2_{H^1(\Omega^{J+1})} - \frac{1}{2\beta}\,\Bigl(  {\rm{ess.sup}}_{r \in \Omega^{J+1}} \sum_{j=1}^{J+1}|({\mathcal L}r)_j+ u_{(0)}(r_j,\cdot)|^2 \Bigr)  \|\psi\|^2_{L^2(\Omega^{J+1})}\\
&  =  \frac{\beta}{2}\, \|\psi \|^2_{H^1(\Omega^{J+1})} - \left(\frac{\beta}{2} + \frac{1}{2\beta}\Bigl(  {\rm{ess.sup}}_{r \in \Omega^{J+1}} \sum_{j=1}^{J+1}|({\mathcal L}r)_j+ u_{(0)}(r_j,\cdot)|^2 \Bigr)  \right) \|\psi\|^2_{L^2(\Omega^{J+1})}
\end{align*}
for all $\psi \in H^1_\star(\Omega^{J+1})$, which leads to
\begin{align*}
a(\psi, \psi) & \geq \alpha\, \|\psi\|^2_{H^1(\Omega^{J+1})} - C\, \|\psi\|^2_{L^2(\Omega^{J+1})} \qquad \forall\, \psi \in H^1_\star(\Omega^{J+1}),
\end{align*}
where
\[ \alpha:=\beta/2\qquad \mbox{and}\qquad  C:= \frac{\beta}{2} + \frac{1}{2\beta}\Bigl(  {\rm{ess.sup}}_{r \in \Omega^{J+1}} \sum_{j=1}^{J+1}|({\mathcal L}r)_j+ u_{(0)}(r_j,\cdot)|^2 \Bigr)  \]
are positive constants.

A classical abstract result due to J.-L. Lions (cf. \cite{brezis2010functional}, Theorem
10.9) then implies that, for any initial datum $\tilde \eta_0  \in L^2_\star(\Omega^{J+1})$
(and $u_{(0)} \in L^2(0,T;W^{1,\sigma}(\Omega)^d)$, with $\sigma>d$, fixed), there exists a unique function $\tilde \eta$ satisfying:
$$ \tilde \eta \in \mathcal{C}([0,T]; L^2_\star(\Omega^{J+1}))\cap L^2(0,T;H^1_{\star}(\Omega^{J+1})), \qquad \partial_t \tilde \eta \in L^2(0,T;H^1_{\star}(\Omega^{J+1})'), $$
$$ \langle \partial_t \tilde \eta, \varphi \rangle_{H^1_{\star}(\Omega^{J+1})' \times H^1_{\star}(\Omega^{J+1})}+a(\tilde \eta, \varphi) = \ell(\varphi) \qquad \mbox{for a.e. } t \in (0,T),
\quad \forall\, \varphi \in H^1_{\star}(\Omega^{J+1}), $$
and
$$\tilde \eta(\cdot,0)= \tilde \eta_0(\cdot).$$
That concludes the proof of the existence of a unique weak solution to the parabolic initial-boundary-value problem \eqref{eq:tilde_eta-1}--\eqref{eq:tilde_eta-3} satisfied by $\tilde\eta$, which therefore
also establishes the existence of a unique weak solution to the parabolic initial-boundary-value problem satisfied by
$\eta=\tilde\eta + 1/|\Omega|$, for $u_{(0)} \in L^2(0,T;W^{1,\sigma}(\Omega)^d)$, with $\sigma>d$, fixed. Similarly, the Oseen system
has, for a given fixed $\eta$, a unique weak solution pair $(u_{(0)},\pi_{(0)})$ (with $\pi_{(0)}$ understood to be
unique up to an additive constant). The uniqueness of a solution triple $(u_{(0)},\pi_{(0)}, \eta)$
satisfying the coupled problem we have arrived at in the small-mass limit is of course not guaranteed, since $\mathbb{K}_{(0)}$ is a nonlinear function of $\eta$ and $u_{(0)}$
enters into the evolution equation for $\eta$, so the coupled system for the small-mass limit $(u_{(0)},\pi_{(0)}, \eta)$ is still very much nonlinear.

\section{The small mass limit problem and the classical Hookean bead-spring-chain model}

Our aim in this final section is to explore the connection between the small-mass-limit problem \eqref{eq:oseen0}--\eqref{eq:eta-3} and the classical Hookean bead-spring-chain model for dilute polymeric fluids.
We begin by recalling that
\[ x = \frac{1}{J+1}\big(r_1 + \cdots + r_{j+1}\big)\quad \mbox{and}\quad q_j = r_{j+1}-r_j\quad \mbox{for $j=1,\dots,J$},\]
and perform a change of variables in order to transform the partial derivatives in \eqref{eq:eta-1} with respect to the variables $r_j$, $j=1,\dots, J+1$, into partial derivatives with respect to $x$ and $q_j$, $j=1,\dots, J$. To this end, note that
\begin{align*}
\partial_{r_1} &= -\partial_{q_1} + \frac{1}{J+1} \partial_x,\\
\partial_{r_{j+1}} &=  \partial_{q_j} - \partial_{q_{j+1}} + \frac{1}{J+1} \partial_x, \qquad j=1,\dots,J-1,\\
\partial_{r_{J+1}} & = \partial_{q_J} + \frac{1}{J+1}\partial_x.
\end{align*}
Thus,
\[ \partial_{r_1}^2 + \cdots + \partial_{r_{J+1}}^2 = (-\partial_{q_1})^2 + (\partial_{q_1} - \partial_{q_2})^2 + \cdots + (\partial_{q_{J-1}} - \partial_{q_J})^2 + (\partial_{q_J})^2 + \frac{1}{J+1} \partial_x^2.\]
Consider the matrix $\mathcal{B} \in \mathbb{R}^{(J+1)d \times Jd}$, called the \textit{incidence matrix}, which is a $(J+1) \times J$ block matrix with $d\times d $ blocks, defined by
\begin{eqnarray*}
\mathcal{B}:= \left(
\begin{array}{ccccc}
    -\mathbb{I} & \mathbb{O}  & \mathbb{O} & \dots  &\mathbb{O} \\
    \mathbb{I} & -\mathbb{I} & \mathbb{O} & \ddots &\mathbb{O} \\
    \mathbb{O} & \mathbb{I} & -\mathbb{I} & \mathbb{O} & \mathbb{O}\\
    \vdots & \ddots & \ddots & \ddots & \ddots \\
    \mathbb{O} & \dots & \mathbb{O} & \mathbb{I}&-\mathbb{I}\\
    \mathbb{O} & \dots & \mathbb{O} & \mathbb{O}&\mathbb{I}
\end{array}
\right).
\end{eqnarray*}
The $d\times d$ block at position $(i,j)$ in $\mathcal{B}$ is equal $-\mathbb{I}$ if the $j$th spring starts at bead $i$, it is equal to $\mathbb{I}$ if the $j$th spring ends at bead $i$, and it is equal to $\mathbb{O}$ otherwise, for $i=1, \dots, J+1$
and $j=1,\dots,J$. Note that
\[ \mathcal{R}:=\mathcal{B}^{\rm T} \mathcal{B} =
\left(
\begin{array}{ccccc}
    2\mathbb{I} & -\mathbb{I}  & \mathbb{O} & \dots  &\mathbb{O} \\
    -\mathbb{I} & 2\mathbb{I} & -\mathbb{I} & \ddots &\mathbb{O} \\
    \mathbb{O} & -\mathbb{I} & 2\mathbb{I} & -\mathbb{I} & \mathbb{O}\\
    \vdots & \ddots & \ddots & \ddots & \ddots \\
    \mathbb{O} & \dots & -\mathbb{I} & 2\mathbb{I}&-\mathbb{I}\\
    \mathbb{O} & \dots & \mathbb{O} & -\mathbb{I}&2\mathbb{I}
\end{array}
\right).
\]
The symmetric positive definite block matrix $\mathcal{R}:=\mathcal{B}^{\rm T}\mathcal{B}$ of size $Jd \times Jd$ is referred to as the \textit{Rouse matrix}. In terms of the Rouse matrix we have
\begin{align}\label{eq:term1}
\partial_{r_1}^2 + \cdots + \partial_{r_{J+1}}^2 = \partial_q^{\rm T} \mathcal{B}^{\rm T} \mathcal{B} \partial_q + \frac{1}{J+1} \partial_x^2 = \partial_q^{\rm T} \mathcal{R}\, \partial_q + \frac{1}{J+1} \partial_x^2,
\end{align}
where $\partial_q := (\partial_{q_1}^{\rm T}, \dots, \partial_{q_J}^{\rm T})^{\rm T}$.

Next, note that
\begin{align*}
\partial_{r_1}\cdot(\eta(\mathcal{L}r)_1) &=(\partial_{r_1}\eta)\cdot(\mathcal{L}r)_1 + \eta \partial_{r_1} \cdot (\mathcal{L}r)_1 = (\partial_{r_1}\eta)\cdot q_1 - d\eta.
\end{align*}
We define, with $r = B(q,x)$, where $q=(q_1^{\rm T}, \dots, q_J^{\rm T})^{\rm T} \in D^J$ and $x \in \Omega$,
\[ \psi(x,q,t) := \eta(B(q,x),t) = \eta(r,t).\]
Hence,
\[ \partial_{r_1}\cdot(\eta(\mathcal{L}r)_1) = \bigg(-\partial_{q_1}\psi + \frac{1}{J+1} \partial_x\psi\bigg) \cdot q_1 - d\psi.\]
Similarly,
\[\partial_{r_{j+1}}\cdot(\eta(\mathcal{L}r)_{j+1}) = \bigg(\partial_{q_j}\psi-\partial_{q_{j+1}}\psi + \frac{1}{J+1} \partial_x\psi\bigg) \cdot (q_{j+1}-q_j) - 2d\psi, \qquad j=1,\dots, J-1,
\]
and
\[ \partial_{r_{J+1}}\cdot(\eta(\mathcal{L}r)_{J+1}) = \bigg(\partial_{q_J}\psi + \frac{1}{J+1} \partial_x\psi\bigg) \cdot (-q_J) - d\psi.
\]
Thus we have that
\begin{align}\label{eq:term2}
\sum_{j=1}^{J+1} \pdr \cdot (\eta(\mathcal{L}r)_j) &= - [\mathcal{B} \partial_q\psi \cdot \mathcal{B}q + 2dJ\psi]
= -  [(\partial_q\psi)^{\rm T}\mathcal{B}^{\rm T} \mathcal{B}q + 2dJ\psi]\nonumber\\
& = -  [(\partial_q\psi)^{\rm T}\mathcal{B}^{\rm T} \mathcal{B}q + (\partial_q^{\rm T}(\mathcal{B}^{\rm T}\mathcal{B}q))\psi] = - \partial_q^{\rm T}(\psi\,\mathcal{R}\,q).
\end{align}
By combining \eqref{eq:term1} and \eqref{eq:term2} we deduce that
\begin{align}\label{eq:term3a}
- \sum_{j=1}^{J+1}(\beta \pdr^2 \eta - \pdr \cdot(\eta(\mathcal{L}r)_j)
&= - \left[\beta \partial_q^{\rm T} \mathcal{R}\, \partial_q\psi   + \partial_q^{\rm T}(\mathcal{R}\,q\,\psi) \right]-  \frac{\beta}{J+1} \partial_x^2 \psi \nonumber\\
&=- \beta \partial_q \cdot \left[\mathcal{R}\, \left(\partial_q\psi   + \frac{1}{\beta}q\,\psi\right) \right]-  \frac{\beta}{J+1} \partial_x^2 \psi.
\end{align}
Let
\[\mathfrak{M}(q):= (2\pi\beta)^{-\frac{1}{2}Jd}\,\mbox{exp}\left(-|q|^2/2\beta\right), \qquad \mbox{where $q=(q_1^{\rm T}, \dots, q_J^{\rm T})^{\rm T} \in D^{J}$}.\]
Hence, \eqref{eq:term3a} yields
\begin{align}\label{eq:term3}
- \sum_{j=1}^{J+1}(\beta \pdr^2 \eta - \pdr \cdot(\eta(\mathcal{L}r)_j)
&=- \beta \partial_q \cdot \left[\mathcal{R}\,\, \mathfrak{M}(q)\,\partial_q\left(\frac{\psi}{\mathfrak{M}(q)} \right)\right]-  \frac{\beta}{J+1} \partial_x^2 \psi.
\end{align}

Next, observe that
\begin{align*}
\pdr \cdot  (\eta u_{(0)}(r_j,t)) = u_{(0)}(r_j,t)\cdot \pdr\eta = u_{(0)}(r_j,t)\cdot \left\{
\begin{array}{ll}
- \partial_{q_1}\psi + \frac{1}{J+1} \partial_x \psi, &\qquad j=1, \nonumber\\
\partial_{q_{j-1}}\psi - \partial_{q_{j}}\psi + \frac{1}{J+1}\partial_x \psi, &\qquad j=2,\dots, J,\\
\partial_{q_J}\psi + \frac{1}{J+1}\partial_x \psi, &\qquad j=J+1.
\end{array}
\right.
\end{align*}
Thus we have that
\begin{align*}
\sum_{j=1}^{J+1} \pdr \cdot  (\eta u_{(0)}(r_j,t)) = \left(\frac{1}{J+1} \sum_{j=1}^{J+1} u_{(0)}(r_j,t)\right)\cdot \partial_x \psi + \sum_{j=1}^J (u_{(0)}(r_{j+1},t) - u_{(0)}(r_j,t))\cdot \partial_{q_j}\psi.
\end{align*}
By performing the approximations
\[ \left(\frac{1}{J+1} \sum_{j=1}^{J+1} u_{(0)}(r_j,t)\right) \approx u_{(0)}(x,t)\]
and
\[ (u_{(0)}(r_{j+1},t) - u_{(0)}(r_j,t)) \approx (\nabla u_{(0)}(x,t))(r_{j+1} - r_j) = (\nabla u_{(0)})(x,t)q_j,\]
we obtain
\begin{align}\label{eq:term4}
\begin{aligned}
\sum_{j=1}^{J+1} \pdr \cdot  (\eta u_{(0)}(r_j,t))  &\approx u_{(0)}(x,t) \cdot \partial_x \psi + \sum_{j=1}^J  (\nabla u_{(0)})(x,t)q_j \cdot \partial_{q_j}\psi\\
&= u_{(0)}(x,t) \cdot \partial_x \psi + \sum_{j=1}^J  \partial_{q_j}\cdot ((\nabla u_{(0)})(x,t)q_j\psi),
\end{aligned}
\end{align}
where the last equality is a consequence of the fact that
\[\partial_{q_j}\cdot ((\nabla u_{(0)})q_j) = \mbox{tr}(\nabla u_{(0)}) = \nabla \cdot u_{(0)} = 0.\]

By substituting \eqref{eq:term3} and \eqref{eq:term4} into \eqref{eq:eta-1} and writing $\nabla$ instead of $\partial_x$
and $\Delta$ instead of $\partial_x^2$, we have that
\begin{align*}
\partial_t \psi + u_{(0)} \cdot \nabla \psi + \sum_{j=1}^J  \partial_{q_j}\cdot ((\nabla u_{(0)})q_j\psi)
- \beta \partial_q \cdot \left[\mathcal{R}\,\, \mathfrak{M}(q)\,\partial_q\left(\frac{\psi}{\mathfrak{M}(q)} \right)\right]-  \frac{\beta}{J+1} \Delta \psi \approx 0,
\end{align*}
which can also be written as
\begin{align}\label{eq:term5}
\partial_t \psi + u_{(0)} \cdot \nabla \psi + \sum_{j=1}^J  \partial_{q_j}\cdot ((\nabla u_{(0)})q_j\psi)
- \beta \sum_{i,j=1}^J \partial_{q_j} \cdot \left[\mathcal{R}_{ij}\, \mathfrak{M}(q)\,\partial_{q_i}\left(\frac{\psi}{\mathfrak{M}(q)} \right)\right]-  \frac{\beta}{J+1} \Delta \psi \approx 0.
\end{align}
Upon replacing the approximate equality
in \eqref{eq:term5} by equality we arrive at the Fokker--Planck equation for the classical Hookean
bead-spring-chain model with centre-of-mass diffusion:
\begin{align}\label{eq:term6}
\partial_t \psi + u_{(0)} \cdot \nabla \psi + \sum_{j=1}^J  \partial_{q_j}\cdot ((\nabla u_{(0)})q_j\psi)
- \beta \sum_{i,j=1}^J \partial_{q_j} \cdot \left[\mathcal{R}_{ij}\, \mathfrak{M}(q)\,\partial_{q_i}\left(\frac{\psi}{\mathfrak{M}(q)} \right)\right]-  \frac{\beta}{J+1} \Delta \psi =  0.
\end{align}

The equation \eqref{eq:term6} is supplemented by the initial condition
\begin{align}\label{eq:term7}
\psi(x,q,0) = \psi_0(x,q),
\end{align}
where $\psi_0(x,q):=\hat{\varrho}_0(B(q,x))$ (cf. \eqref{eq:eta-2}).

Since \eqref{eq:term6} is now posed on the domain
$\Omega \times D^J \times (0,T]$ rather than on $\Omega^{J+1} \times (0,T]$, it is natural to replace
the zero-normal-flux boundary condition \eqref{eq:eta-3}
on $\partial\Omega^{(J+1)} \times (0,T]$ by zero-normal-flux boundary conditions on
$\partial\Omega \times D^J \times (0,T]$ and $\Omega \times \partial D^{J} \times (0,T]$;
i.e.,
\begin{align}\label{eq:term8a}
\nabla \psi (x,q,t) \cdot n_x(x) = 0\qquad \mbox{for all  $(x,q,t)\in \partial\Omega \times D^J \times (0,T]$},
\end{align}
where $n_x$ is the unit outward normal vector to $\partial\Omega$, and
\begin{align}\label{eq:term8}
\sum_{i=1}^J \left[\beta \,\mathcal{R}_{ij}\, \mathfrak{M}(q)\,\partial_{q_i}\left(\frac{\psi}{\mathfrak{M}(q)} \right)
- ((\nabla u_{(0)})q_j\psi)\right]\cdot n_{q_j} = 0
\end{align}
for all $(x,q,t) \in \Omega \times (D \times \dots \times \partial D \times \cdots \times D) \times (0,T]$, $j=1,\dots,J$,
where $n_{q_j}$ is the unit outward normal vector to $\partial D$ for the $j$th copy of the domain $D$ in the Cartesian product $D^J = D \times \cdots \times D$.

By integrating the Fokker--Planck equation \eqref{eq:term6} over $D^J$ and using the boundary condition \eqref{eq:term8},
and then integrating both the boundary condition \eqref{eq:term8a} and the initial condition \eqref{eq:term7} over $D^J$, we obtain
\begin{alignat}{2}\label{eq:term9}
\partial_t \left(\int_{D^J} \psi\, \mathrm{d}q\right)  +
 u_{(0)} \cdot \nabla \left(\int_{D^J} \psi \,\mathrm{d}q\right) -  \frac{\beta}{J+1} \Delta \left(\int_{D^J} \psi \,\mathrm{d}q\right)  &=  0&&\quad \mbox{in $\Omega \times (0,T]$},\nonumber\\
\nabla \left(\int_{D^J} \psi\, \mathrm{d}q\right) \cdot n_x &= 0&&\quad \mbox{on $\partial\Omega \times (0,T]$},\\
\left(\int_{D^J} \psi\, \mathrm{d}q\right)(\cdot,0) &= \left(\int_{D^J} \psi_0(\cdot,q)\, \mathrm{d}q\right)&&\quad \mbox{in $\Omega$}.\nonumber
\end{alignat}

If the initial datum $\psi_0$ is such that, for some constant $n>0$,
\[ \int_{D^J} \psi_0(x,q)\, \mathrm{d}q =  n^{-1} \qquad \mbox{for a.e. $x \in \Omega$},\]
then, by uniqueness of the solution to the initial-boundary-value problem
\eqref{eq:term9}, it follows that
\[ \int_{D^J} \psi(x,q,t)\, \mathrm{d}q =  n^{-1} \qquad \mbox{for a.e. $(x,t) \in \Omega\times [0,T]$};\]
that is
\[ \int_{D^J} \eta(B(q,x),t) \dd q  = \int_{D^J} \psi(x,q,t) \dd q = n^{-1} \qquad \mbox{for a.e. $(x,t) \in \Omega\times [0,T]$}, \]
whereby the expression for the tensor $\mathbb{K}_{(0)}$ stated in \eqref{eq:kramers} simplifies to
\begin{align}\label{eq:ourK0}
\mathbb{K}_{(0)} = n \int_{D^J} \sum_{j=1}^{J}(F(q_j)\otimes q_j)\, \psi(x,q,t) \dd q.
\end{align}
%

In this form, $\mathbb{K}_{(0)}$ is referred to as \textit{Kramers' expression} for the polymeric extra stress tensor for the bead-spring-chain model with
$J$ springs.
We highlight one small but relevant difference between the classical Kramers expression and \eqref{eq:ourK0}: in the classical Kramers expression the integral in $q$ is taken over the whole of $\mathbb{R}^{Jd}$, whereas in our case the integral in $q$ is 
over $D^J \subset \mathbb{R}^{Jd}$, where $D:=\Omega - \Omega$. In this respect the formula \eqref{eq:ourK0} is more consistent with the definition of the configuration vectors $q_j:=r_{j+1}-r_j$, $j=1,\dots,J$,
than its classical counterpart; it also avoids the nonphysical feature of the classical Hookean model that springs in a linear bead-spring-chain are allowed to stretch out to infinity even though their
endpoints are confined to a bounded flow domain $\Omega$. In our case, in contrast, if $\Omega$ is bounded, then so is $D^J$. Of course, if $\Omega$ happens to be the whole of $\mathbb{R}^d$ then
$D^J = \mathbb{R}^{Jd}$, so \eqref{eq:ourK0} and its classical counterpart will coincide.

The main obstacle in proving the existence of global weak solutions to the Hookean bead-spring-chain
model (cf. \cite{BS2010-hookean}, for example,) where integration in the Kramers expression is over $\mathbb{R}^{Jd}$, 
is lack of weak compactness of the sequence of approximating solutions to the Fokker--Planck equation
in the $|q|^2$-weighted $L^1$ space $L^1_{|q|^2}(\mathbb{R}^{Jd})$
(even though the sequence of approximating solutions is strongly convergent in
$L^1_{\rm{loc}}(\mathbb{R}^{Jd})$), which then obstructs passage to the limit in the classical Kramers expression
precisely because integration with respect to the configuration space variable $q$ there is over the whole of $\mathbb{R}^{Jd}$ rather than a bounded subset of $\mathbb{R}^{Jd}$. This difficulty was ultimately overcome in \cite{BS2018}
in the case of $d=2$ through a rigorous proof of the fact that the macroscopic closure of the Hookean dumbbell model ($J=1$) is the Oldroyd-B model, for which a global existence result is available (cf. \cite{barrett-boyaval-09}).
The existence of global weak solutions to the Hookean dumbbell model in the case of $d=3$, with the Kramers expression in its classical form (i.e. with integration over $q \in \mathbb{R}^{Jd}$) however
remains an open problem. With the Kramers expression defined by $\eqref{eq:ourK0}$ now, the situation is radically different: the technical difficulties caused by loss of compactness disappear, enabling completion
of the proof of existence of global weak solutions to the Hookean bead-spring-chain model in both two and three space dimensions by replicating the proof contained in  \cite{BS2010-hookean}.

\medskip

\noindent{\bf Acknowledgements.}
The authors thank Andrew Stuart (Caltech) for stimulating discussions, his help with the formulation of the McKean--Vlasov diffusion problem studied in this paper, and for suggesting to them 
the small-mass-limit problem explored in the penultimate section. Ghozlane Yahiaoui's work was supported by the UK Engineering and Physical Sciences Research Council [EP/L015811/1].
Endre S\"uli is grateful to John W. Barrett for numerous helpful discussions on the subject of this paper. 
\vspace{0.1in}

\bibliographystyle{amsplain}
\bibliography{ref}

\end{document}